\documentclass[11pt]{article}
\usepackage{graphicx} %
\usepackage{subcaption}
\usepackage{adjustbox}
\usepackage[margin=1.0in]{geometry}
\usepackage[utf8]{inputenc}            %
\usepackage[T1]{fontenc}               %
\usepackage[dvipsnames]{xcolor}
\usepackage[
    pdftex=true,
    colorlinks=true,
    linkcolor=RoyalPurple,   %
    citecolor=Blue,     %
    urlcolor=Violet,     %
    linktoc=page
]{hyperref}
\usepackage[mathscr]{euscript}
\usepackage[T1]{fontenc}

\usepackage{multirow}
\usepackage{url}                        %
\usepackage{booktabs}                   %
\usepackage{amsfonts}                   %
\usepackage{nicefrac}                   %
\usepackage{bbm}
\usepackage{microtype}                  %
\usepackage[boxruled]{algorithm2e}
\usepackage{xcolor}                     %
\usepackage{lmodern}
\usepackage{enumitem}
\usepackage[numbers, sort]{natbib}
\usepackage{amsmath, amssymb, amsfonts}
\usepackage{amsthm, thmtools, thm-restate}
\usepackage{mathtools, nccmath, extarrows}
\usepackage{enumitem}
\usepackage{subcaption}
\usepackage{bm}
\usepackage{scalerel}
\usepackage[capitalise]{cleveref}
\usepackage{array}
\usepackage[most]{tcolorbox}
\usepackage{graphicx}
\usepackage{subcaption}
\usepackage[percent]{overpic} %
\usepackage{rotating} %
\usepackage{array}    %

\usepackage{etoolbox}     %
\newbool{showSquare}      %
\setbool{showSquare}{true}%

\newcounter{imagerow}[figure]
\usepackage{wrapfig}

\newcounter{imagecolumn}[imagerow]

\newlength{\imagewidth}

\usepackage{stackengine}
\stackMath

\DeclareFontEncoding{LS1}{}{}
\DeclareFontSubstitution{LS1}{stix}{m}{n}

\makeatletter
\newcommand*\textmathversion{\csname textmv@\math@version\endcsname}
\newcommand*\textmv@normal{m}
\newcommand*\textmv@bold{b}
\makeatother
\newtcbtheorem{defn}{Definition}%
{	colframe=blue!30!gray, theorem name, fonttitle=\bfseries,
	colback=blue!5, top=1mm,bottom=1mm, boxrule=0.5pt}{Example}

\newtcbtheorem[use counter=theorem, number within=section]{THM}{Theorem}%
{	colframe=blue!30!gray, fonttitle=\bfseries,
	colback=blue!5,  fontupper=\itshape, top=1mm,bottom=1mm, boxrule=0.5pt}{Theorem}
\usepackage{varwidth}
\newtcbox{\mymath}[1][]{%
	nobeforeafter, math upper, tcbox raise base,
	enhanced, colframe=blue!30!black,
	colback=blue!5, boxrule=0.5pt, top=1mm,bottom=1mm,
	#1}
 \newtcbox{\mybox}{colback=blue!5,
	colframe=blue!30!black, center, enhanced, varwidth upper}
    
\newtheorem{theorem}{Theorem}[section]
\newtheorem{proposition}[theorem]{Proposition}
\newtheorem{lemma}[theorem]{Lemma}
\newtheorem{corollary}[theorem]{Corollary}
\newtheorem{claim}[theorem]{Claim}

\newtheorem{definition}[theorem]{Definition}
\newtheorem{assumption}{Assumption}
\newtheorem{pconfig}{Configuration}

\newcommand{\changelocaltocdepth}[1]{%
  \addtocontents{toc}{\protect\setcounter{tocdepth}{#1}}%
  \setcounter{tocdepth}{#1}%
}

\newcommand\MD[1]{{\color{magenta}Mateo: ``#1''}}
\newcommand\MDA[1]{
}

\usepackage{bm}
\newcommand{\EEE}{{\bf E}}
\newcommand{\YY}{{\bf Y}}

\newcommand{\cA}{\mathcal{A}}

\newcommand{\cI}{\mathcal{I}}

\newcommand{\cM}{\mathcal{M}}

\newcommand{\cQ}{\mathcal{Q}}

\newcommand{\cS}{{\mathcal{S}}}

\newcommand{\cX}{\mathcal{X}}

\newcommand{\cZ}{\mathcal{Z}}

\newcommand{\R}{\mathbf{R}}
\newcommand{\B}{\mathbf{B}}

\newcommand{\RR}{{\R}}

\newcommand{\SSS}{\cS}

\newcommand{\Ima}{\operatorname{Im}}

\newcommand{\xs}{x^\star}
\newcommand{\Ts}{T^\star}

\newcommand{\rank}{\operatorname{rank}}
\newcommand{\rankk}[1]{\operatorname{rank}\left( #1 \right)}

\newcommand{\dimm}[1]{\operatorname{dim}\left( #1 \right)}

\newcommand{\tr}{\operatorname{tr}}
\newcommand{\diag}{\operatorname{diag}}

\newcommand{\dist}{\operatorname{dist}}

\newcommand{\Vect}{\operatorname{vec}}

\newcommand{\EE}{\operatorname{\mathbb E}}

\newcommand{\supp}{\operatorname{supp}}

\newcommand{\argmin}{\operatornamewithlimits{argmin}}
\newcommand{\argmax}{\operatornamewithlimits{argmax}}

\newcommand{\dom}{\operatorname{dom}}

\def\shortdisplay{\setlength{\abovedisplayskip}{5pt}%
	\setlength{\belowdisplayskip}{5pt}%
	\setlength{\abovedisplayshortskip}{2pt}%
	\setlength{\belowdisplayshortskip}{2pt}}

\let\oldselectfont\selectfont
\def\selectfont{\oldselectfont\shortdisplay}

\usepackage{xcolor}   %
\usepackage{pifont}   %

\newcommand{\cmark}{\textcolor{olive}{\ding{51}}}  %
\newcommand{\xmark}{\textcolor{purple}{\ding{55}}}    %
\newcommand{\yes}{{\color{olive}\bf Yes}}
\newcommand{\no}{{\color{purple} \bf No}}
\newcommand{\h}{h}
\newcommand{\zs}{z^{\star}}
\renewcommand{\c}{F}
\renewcommand{\P}{\Pi}

\newcommand{\mus}{\alpha} %
\newcommand{\Lhs}{\beta} %
\newcommand{\Ws}{{W^\star}}
\newcommand{\Xs}{{X^\star}}

\newcommand{\Ys}{{Y^\star}}

\newcommand{\Zs}{Z^\star}
\newcommand{\Ms}{M^\star}
\newcommand{\alignvarepsilon}{{\varepsilon_{\xs}}}
\newcommand{\jacvarepsilon}{{\varepsilon_{\nabla F}}}

\newcommand{\epsassymmatrix}{
 \frac{1}{16\sqrt{2}} \frac{ \min  \left\{  \sigma_{\rs}^2\left(\Xs\right), \sigma_{\rs}^2\left(\Ys\right) \right\}}  {\max \left\{ {\sigma_1\left(\Xs\right)}, {\sigma_1\left(\Ys\right)}  \right\}}
}

\newcommand{\epsassymmatrixhalf}{  \frac{1}{32\sqrt{2}} \frac{ \min  \left\{  \sigma_{\rs}^2\left(\Xs\right), \sigma_{\rs}^2\left(\Ys\right) \right\}}  {\max \left\{ {\sigma_1\left(\Xs\right)}, {\sigma_1\left(\Ys\right)}  \right\}}}

\newcommand{\epssymtensor}{\min\left\{R, \frac{\sigma_{dr}\left( \nabla \Fsym(\Xs)  \right)}{24\norm{\Xs}{F}}, \|\Xs\|_F  \right\}}

\newcommand{\epsAsymtensor}{\min\left\{R, \frac{\sigma_{(d_1+d_2+d_3-2)r}\left( \nabla \Fasym(\Ws,\Xs,\Zs)  \right)}{8\sqrt{3}\norm{({\Ws},{\Xs},{\Ys})}{F}}, \norm{({\Ws},{\Xs},{\Ys})}{F}  \right\}}

\newcommand{\epsAsymtensorhalf}{\frac{1}{2}\epsAsymtensor}

\newcommand{\proj}{{\rm proj}}
\newcommand{\norm}[2]{ \left\| #1 \right\|_{#2}}
\newcommand{\card}[1]{\# #1}

\newcommand{\inner}[2]{ \left\langle #1, #2 \right\rangle }
\newcommand{\vect}[1]{\operatorname{vec} \left( #1\right)}
\newcommand{\diagg}[1]{\operatorname{diag} \left( #1\right)}
\newcommand{\spann}[1]{\operatorname{span} \left( #1\right)}

\newcommand{\clb}{C_{\texttt{lb}}}
\newcommand{\cub}{C_{\texttt{ub}}}
\newcommand{\lc}{L_{\nabla \c}}
\newcommand{\lh}{L}

\newcommand{\sig}{s}
\newcommand{\rloc}{\delta}

\newcommand{\otimeskron}{\overset{}{\otimes}_{\rm{Kr}}}
\newcommand{\Fsym}{\ifbool{showSquare}{F_{\operatorname{sym}}}{F_{\operatorname{sym}}}}
\newcommand{\Fasym}{F_{\operatorname{asym}}}

\newcommand{\ball}[2]{
	\B_{#2}\left(#1\right)}

\newcommand{\minn}[1]{
	\min\left\{#1\right\}}

\newcommand{\rs}{r^\star}
\newcommand{\svd}[1]{U^{#1} \Sigma^{#1} \left( V^{#1} \right)^\top}
\newcommand{\um}{U^M}
\newcommand{\un}{U^N}
\newcommand{\ux}{U^{X}}
\newcommand{\uxprime}{U^{\Xprime}}

\newcommand{\uy}{U^Y}

\newcommand{\lambdax}[1]{\lambda_{#1}\left( XX^\top \right)}

\newcommand{\lambdaa}[2]{\lambda_{#2} \left( {#1} {#1}^\top \right)  }

\newcommand{\Wprime}{\tilde{W}}
\newcommand{\Xprime}{\tilde{X}}
\newcommand{\Mprime}{\tilde{M}}
\newcommand{\Rprime}{\tilde{R}}

\newcommand{\Yprime}{\tilde{Y}}
\newcommand{\Zprime}{\tilde{Z}}
\newcommand{\rprime}{\tilde{r}}

\newcommand{\ext}[2]{\psi( #1, #2 )}
\newcommand{\extone}[1]{\psi( #1, #1 )}
\newcommand{\trace}[1]{ \operatorname{trace} \left( #1 \right) }
\newcommand{\nulll}[1]{ \operatorname{null} \left( #1 \right) }

\newcommand{\abs}[1]{ \left| #1 \right|}

\newcommand{\dotp}[1]{\left\langle #1\right\rangle}
\newcolumntype{I}{>{\refstepcounter{imagecolumn}}{c}<{}}

\newcommand{\dl}{\omega_1}
\newcommand{\du}{\omega_2}
\newcommand{\constI}{\omega_0}
\newcommand{\pfail}{p_{\rm{fail}}}

\newcommand{\ik}{\underline{i}}
\newcommand{\jk}{\underline{j}}

\title{Preconditioned subgradient method for composite optimization:\\ overparameterization and fast convergence}
\author{Mateo D\'{i}az\thanks{Department of Applied Mathematics and Statistics, Johns Hopkins University, Baltimore, MD 21218, USA;	\texttt{https://mateodd25.github.io}. MD was partially supported by NSF awards CCF 2442615 and DMS 2502377.} \and  Liwei Jiang\thanks{Edwardson School of Industrial Engineering, Purdue University, West Lafayette, IN 47906, USA;	 \texttt{https://liwei-jiang97.github.io}.} \and Abdel Ghani Labassi\thanks{Department of Applied Mathematics and Statistics, Johns Hopkins University, Baltimore, MD 21218, USA;	\texttt{https://aglabassi.github.io/website/}.}}
\date{}

\begin{document}

	\maketitle
\begin{abstract}
  Composite optimization problems involve minimizing the composition of a smooth map with a convex function. Such objectives arise in numerous data science and signal processing applications, including phase retrieval, blind deconvolution, and collaborative filtering.
  The subgradient method achieves local linear convergence when the composite loss is well-conditioned. However, if the smooth map is, in a certain sense, ill-conditioned or overparameterized, the subgradient method exhibits much slower sublinear convergence even when the convex function is well-conditioned. To overcome this limitation, we introduce a Levenberg-Morrison-Marquardt subgradient method that converges linearly under mild regularity conditions at a rate determined solely by the convex function. Further, we demonstrate that these regularity conditions hold for several problems of practical interest, including square-variable formulations, matrix sensing, and tensor factorization. Numerical experiments illustrate the benefits of our method.
\end{abstract}
\newpage    \tableofcontents
\newpage
\newpage

\changelocaltocdepth{1}
\section{Introduction}
The goal of \emph{composite optimization problems} is to minimize  
\begin{equation}\label{eq:core-problem}
	\min_{x \in \RR^d} f(x) \quad \text{with} \quad f = \h \circ \c,
\end{equation}
where $\h \colon \RR^m \rightarrow \RR$ is a---possibly nonsmooth---convex function and $\c \colon \RR^d \rightarrow \RR^m$ is a smooth mapping. Taking $h$ or $\c$ as the identity map recovers convex and smooth optimization; thus, this formulation strictly extends both and amounts to a much richer class of nonsmooth nonconvex problems. Classical nonlinear least squares are a prominent example of this framework \cite{nocedal1999numerical, LH95, bjorck2024numerical}. Recently, composite optimization has gained renewed interest due to its applications in data science, including phase retrieval, matrix completion, and tensor factorization \cite{charisopoulos2021low, drusvyatskiy2017composite, duchi2017solving, tong2022scaling}. 

First-order methods, such as gradient descent, are the dominant algorithmic solution for large-scale composite problems. Under favorable growth conditions of the loss function, these methods converge linearly towards solutions provided good initialization \cite{chi2019nonconvex, charisopoulos2021low}. For instance, when the objective function $f$ is $\beta$-smooth and is locally $\mus$-strongly convex, gradient descent converges linearly at a rate that depends on the condition number $\beta/\mus$. 
Inconveniently, this condition number might worsen drastically depending on the choice of the smooth map $\c$. 
To illustrate this point, it is useful to think of the smooth map $\c$ as a \emph{parameterization}: minimizing \eqref{eq:core-problem} is akin to solving a constrained problem 
\begin{equation}\label{eq:constrained-problem}
    \min_{x \in \RR^d} f(x) = \min_{z \in \Ima \c} h(z),
\end{equation}
where $\Ima \c$ denotes the image of $\c$. Intuitively, when $\Ima \c$ is sufficiently ``benign,'' the intrinsic complexity of \eqref{eq:constrained-problem} should be dictated by the conditioning of $h$ restricted to $\Ima\c$ and not by the specific parameterization $\c$. 

Two factors concerning the parameterization $\c$ cause the conditioning of $f$ to differ from that of $h$ on $\Ima\c$: $(i)$ ill-conditionedness, or, even worse, $(ii)$ an excess of parameters.
For concreteness, consider a simple example,  
suppose we want to factorize a rank-$r^\star$ positive semidefinite (PSD) matrix $M^\star.$ 
In large-scale settings---where direct eigen- or singular-value decompositions are prohibitively costly---researchers turn to iterative schemes on low-rank parameterizations. The celebrated Burer-Monteiro approach \cite{burer2003nonlinear, burer2005local} parameterizes low-rank matrices via an explicit factorization, $\c(U) = UU^\top$ with $U \in \RR^{d \times r}$ and $r \ge r^\star$, and aims to solve
\begin{equation}\label{eqn: sensing_obj}
	\min_{U \in \RR^{d \times r}} \frac{1}{2}\|UU^\top - M^\star\|_F^2.
\end{equation} 
This can be seen as a nonconvex composite problem where $\h(M) = \frac{1}{2}\|M - M^\star\|_F^2.$  
A straightforward computation reveals that even though the condition number of the convex function $\h$ is one, the condition number of the composition $f = \h \circ \c$ near minimizers scales 
like $\sigma_1(M^\star)/\sigma_r(M^\star)$.
This leads to two potential issues for the convergence of gradient descent.  On the one hand, in the exactly parameterized regime, i.e., $r = r^\star$, the condition number of $f$ is proportional to the condition number $ \kappa(M^\star),$ which could lead to arbitrarily slow linear convergence depending on the matrix $M^\star$.\footnote{The condition number of a matrix $A$ is given by $\kappa(A) := \sigma_{\max}(A)/\sigma_{\min}(A)$ where $\sigma_{\min}(A)$ is the smallest nonzero singular value.} On the other, in the overparameterized regime, i.e., $r> r^\star$, 
$\sigma_r(M^\star)$ is zero, leading to an infinite condition number, which in turn results in sublinear convergence~\cite{zhuo2024computational}.
Both of these situations happen in practice since it is common to encounter ill-conditioned matrices with unknown rank.  
\begin{figure}
  \centering
  \includegraphics[width=0.5\linewidth]{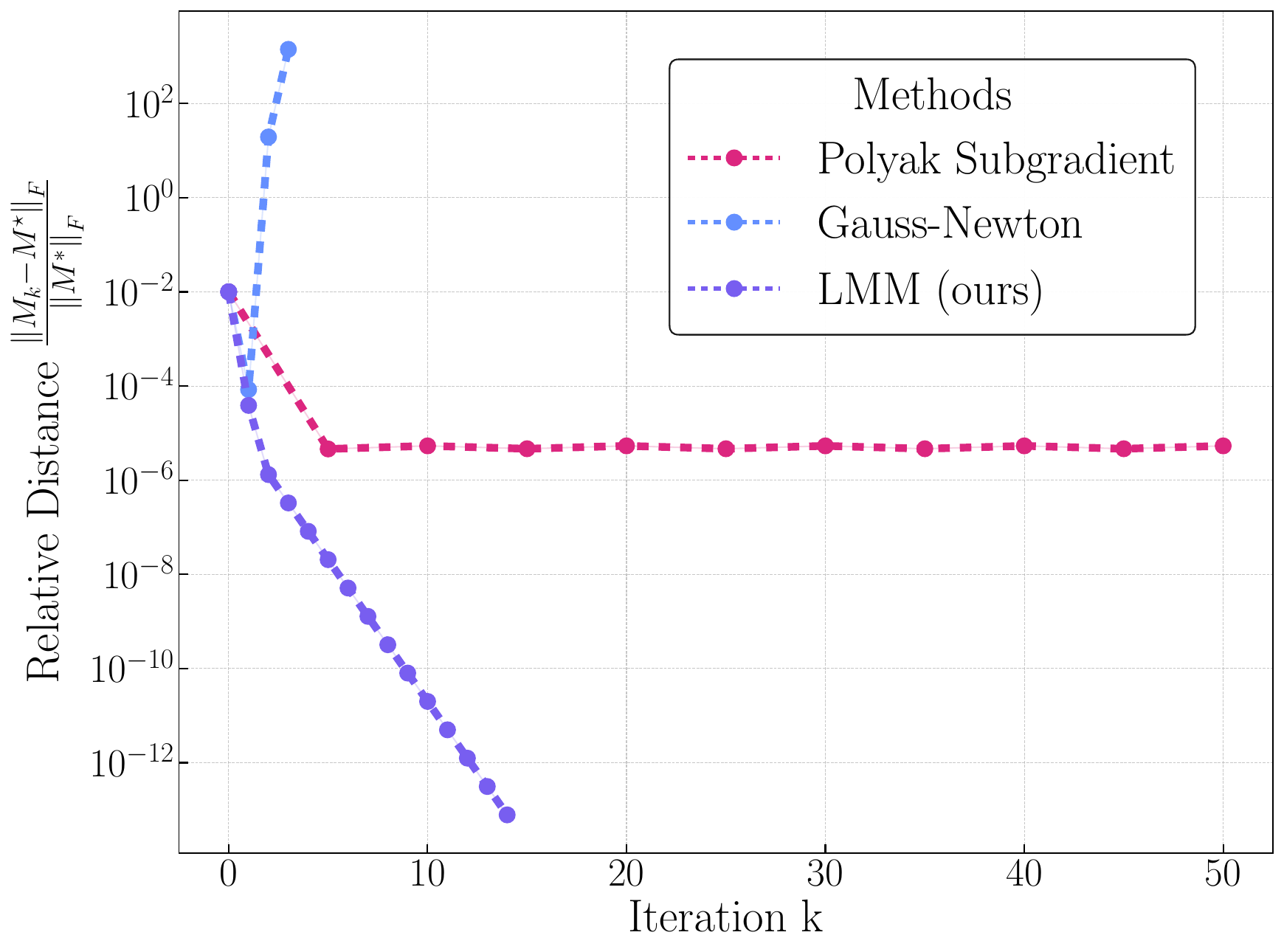}
  \caption{
  Relative distance to the solution against
    iteration count for Algorithm~\ref{eq:main-update} applied to an overparameterized nonsmooth matrix factorization problem with $F(U) = UU^{\top}, h(M) = \|M - M^{\star}\|_F,  M^{\star} \in \cS^{50}_{+}, \text{ and }U \in \RR^{50 \times 3}$  with $\rank(M^{\star}) = 2 < r = 3$ and $\kappa(M^\star) = 1$. All algorithms use the Polyak stepsize.}
  \label{fig:intro-example}
\end{figure}

These issues go beyond this simple, smooth problem. Indeed, for nonsmooth functions that are sharp and Lipschitz, the Polyak subgradient method converges at a linear rate, yet the rate might drastically slow down depending on the parameterization \cite{davis2022linearly, chi2019nonconvex}, and may even decay exponentially when overparametrization occurs; see Figure~\ref{fig:intro-example}.
Motivated by these drawbacks, numerous works have proposed ad-hoc preconditioned (sub)gradient methods that exhibit linear convergence at a rate independent of the parameterization $\c$ \cite{tong2021accelerating,xu2023power,zhang2021preconditioned,davis2022linearly}. Despite the breadth of this line of work, much of it focuses on concrete formulations, e.g., smooth low-rank matrix recovery, and the proposed methods do not systematically generalize to composite optimization problems, which motivates the main question of this work.

\mybox{\it\centering Is there a preconditioned subgradient method for general composite optimization problems that exhibits local linear convergence 
depending only on the convex outer function $h$
?}
We answer this question in the affirmative under mild assumptions. Borrowing ideas that date back to the work of Levenberg \cite{levenberg1944method}, Morrison~\cite{morrison1960methods}, and Marquardt \cite{marquardt1963algorithm} on nonlinear least squares, we propose a preconditioned subgradient method that updates
\begin{align}\label{eq:main-update}
	x_{k+1} &\leftarrow x_k - \gamma_k (\nabla \c(x_k)^\top \nabla \c(x_k) + \lambda_k I)^{-1} \nabla \c(x_k)^\top v_k,
\end{align}
with $v_k \in \partial \h ( \c(x_k))$ where $\partial \h$ denotes the convex subdifferential of $\h$. Let us comment on this algorithm and its underlying motivation. 
The method applies to both smooth and nonsmooth composite problems. 
The term $\nabla \c(x_k)^\top v_k$ corresponds to a subgradient of $f$, thus $(\nabla \c(x_k)^\top \nabla \c(x_k) + \lambda_k I)^{-1}$ acts as a preconditioner.
For structured problems, the cost of solving the linear system involved at each iteration is low.
For instance, for low-rank matrix recovery problems, the cost is proportional to that of solving an $r \times r$ linear system. 
When 
the convex function $\h$ corresponds to the $\ell_{2}$ norm squared,
update \eqref{eq:main-update} recovers the classical Levenberg-Morrison-Marquadt (LMM) method;\footnote{LMM is often only attributed to Levenberg and Marquardt.} moreover, when $\lambda_k = 0$, it reduces to the Gauss-Newton method. Recently, Davis and Jiang \cite{davis2022linearly} introduced a Gauss-Newton subgradient method (GNP) for general composite problems, which was the main inspiration for this work. Davis and Jiang showed that if $\nabla \c$ has full rank near a minimizer, then GNP converges at a linear rate that only depends on the conditioning of $\h$. However, for overparameterized problems, $\nabla \c$ does not have full rank near minimizers, leading to potentially ill-posed preconditioners. Indeed, even mild overparameterization in low-rank matrix factorization leads to the divergence of GNP; see Figure~\ref{fig:intro-example}. To overcome this issue, we regularize the preconditioner
, which improves numerical stability. 
 \paragraph{Main contributions.} Let us summarize our three core contributions.
 \begin{itemize}[leftmargin=5mm]
   \item[] (\textbf{Method}) We propose a \emph{Levenberg–Morrison-Marquardt} subgradient method (Algorithm~\ref{alg:LM}) along with a concrete choice of stepsizes, $\gamma_k$, inspired by the Polyak stepsize \cite{polyak1969minimization}, and damping coefficients, $\lambda_k$, that displays rapid local convergence {universally} across all combinations of smooth and nonsmooth, overparameterized and exactly parameterized settings. \ifbool{showSquare}{Since our parameter choice heavily relies on information about the function that may not be readily available to practitioners, we also present another parameter configuration based on geometrically decaying schedules \cite{goffin1977convergence, davis2018subgradient}, which only requires rough bounds on the function parameters.}{} %

\item[] (\textbf{General-purpose convergence guarantees}) Under mild assumptions, we show that our parameter configurations guarantee linear convergence at a rate depending solely on the convex function $\h$. Our results rely on nearly decoupled assumptions for $\h$ and the smooth map $\c$, allowing one to combine these functions freely while still achieving rapid convergence. 
\ifbool{showSquare}{In particular, we require that $\h$ is in some sense well-conditioned on the image of $\c$---quadratic growth with Lipschitz gradient in smooth settings, or sharp and Lipschitz in nonsmooth settings---while $\c$ must satisfy that its image and Jacobian are in a certain sense aligned near minimizers.}{} %

\item[] (\textbf{Consequences for statistical recovery problems})
         To complement our convergence guarantees, we study their implications for various data science tasks. In particular, we show that the geometric assumptions required for our general-purpose convergence results hold for \ifbool{showSquare}{$(i)$ nonnegative least squares formulations, $(ii)$ (overparameterized) low-rank matrix recovery, and $(iii)$ canonical polyadic (CP) tensor factorization problems.}{$(i)$ (overparameterized) low-rank matrix recovery and $(ii)$ canonical polyadic (CP) tensor factorization problems.} \ifbool{showSquare}{As a result, we establish the linear convergence of the Levenberg–Morrison-Marquardt subgradient method~\eqref{eq:main-update} for all these problems at a rate that only depends on the conditioning of the convex function $\h$. This recovers existing convergence guarantees in certain settings and provides the first such results for others.}{}
\end{itemize}
\paragraph{Outline of the paper.} 
\ifbool{showSquare}{We conclude this section with related work. Section~\ref{sec:prelims} sets out notation and necessary background.}{} In Section~\ref{sec:algorithm}, we formally introduce composite problems, our key assumptions, and the algorithm we propose. After that, Section~\ref{sec:guarantees} provides general-purpose convergence guarantees under suitable regularity conditions. In Section~\ref{sec:examples}, we verify these conditions for several statistical recovery problems. 
Section~\ref{sec:experiments} contains numerical experiments showcasing the benefits of our method. \ifbool{showSquare}{We defer long and technical proofs to the appendix.}{}

\ifbool{showSquare}{
\subsection{Related work}
\paragraph{Nonlinear least squares.}
Nonlinear least squares problems \cite{bjorck2024numerical} form a widely studied instance of~\eqref{eq:core-problem}, where the outer function is the squared Euclidean norm. Although Newton’s method enjoys local quadratic convergence under mild regularity, forming or even applying the Hessian is often prohibitively expensive at scale. The Gauss-Newton algorithm is a computationally cheaper alternative that enjoys similar guarantees \cite{nocedal1999numerical, nesterov2007modified} and has been widely used in the sciences and engineering \cite{gn1, gn2, gn3}. Gauss--Newton can fail when the iteration’s linear system is singular, making the update ill-defined.  
To overcome this issue, Levenberg \cite{levenberg1944method}, Morrison \cite{morrison1960methods}, and Marquardt \cite{marquardt1963algorithm} independently introduced an additional damping term ensuring invertability. %
The LMM method has been extensively analyzed for nonlinear least squares \cite{more2006levenberg,lm2,lm3,fischer2024levenberg} and is widely used in applications \cite{lm1a,lm3a,lm5a,lm6a}. %

\paragraph{Composite optimization.} Splitting methods are a popular alternative for composite objectives. The term `composite optimization' is often used to refer to the subclass of additive composite problems where the loss can be expressed as a sum of a convex and a smooth function. Classical schemes for this subclass include the forward-backward (proximal-gradient) splitting
\cite{lions1979splitting,chen1997convergence,drusvyatskiy2016error} and optimal accelerated versions
\cite{beck2009fast,lan2012optimal,lan2016gradient}. %
For general composite objectives, the {prox-linear} method linearizes $F$ and computes a proximal step of composition of the linearization with the convex outer function each iteration \cite{nesterov2007modified,cartis2011evaluation,lewis2015proximal,drusvyatskiy2016error,drusvyatskiy2017composite}. This scheme is closely related to classical trust region variants of Gauss-Newton \cite{fletcher2009model, burke1985descent, burke1995gauss, wright1990convergence, cartis2011evaluation}. Only recently has the local and global convergence of subgradient methods for composite problems been established \cite{davis2018stochastic,davis2018subgradient}.

\begin{table}[t]
    \centering
    \resizebox{\textwidth}{!}{
    \begin{tabular}{lccccccc}
        \toprule
        \multirow{2}{*}{\textbf{Algorithm}}
        & \multicolumn{5}{c}{\textbf{Low-rank matrix recovery}}
        & \multirow{2}{*}{\textbf{Converges to}}
        & \multirow{2}{*}{\parbox{3.6cm}{\centering \textbf{Applicable beyond}\\\textbf{matrix recovery}}} \\
        \cmidrule(lr){2-6}
        & Overparam. & Symm. & Asymm. & Smooth & Nonsmooth
        &  &  \\
        \midrule
        ScaledGD \cite{tong2021accelerating}
        & \xmark & \cmark & \cmark & \cmark & \xmark
        & Solution
        & \no\textsuperscript{*} \\
        ScaledGD($\lambda$) \cite{xu2023power}
        & \cmark & \cmark & \cmark & \cmark & \xmark
        & Neighborhood\textsuperscript{$\sharp$}
        & \no\textsuperscript{*} \\

        ScaledSM \cite{tong2021scaledsubgradient}
        & \xmark & \cmark & \cmark & \xmark & \cmark
        & Solution
        & \no\textsuperscript{*} \\
      PreconditionedGD \cite{zhang2021preconditioned}
        & \cmark & \cmark & \xmark & \cmark & \xmark
        & Solution        & \no \\
             Asymmetric PreconditionedGD \cite{cheng2024accelerating}
        & \cmark & \xmark & \cmark & \cmark & \xmark
        & Solution        & \no \\
        OPSA \cite{giampouras2024guarantees}
        & \cmark & \xmark & \cmark & \xmark & \cmark
        & Solution\textsuperscript{\ddag}
        & \no \\
        APGD~\cite{liu2025efficient} &\cmark &\xmark & \cmark &\cmark &\xmark & Solution & \no\\
         Approximated GN \cite{jiaglobally} & \cmark & \cmark & \cmark & \cmark & \xmark & Solution & \no \\
        GNP \cite{davis2022linearly}
        & \xmark & \cmark & \cmark & \cmark & \cmark
        & Solution
        & \yes \\
        Algorithm~\ref{alg:LM} (ours)
        & \cmark & \cmark & \cmark & \cmark & \cmark
        & Solution
        & {\color{olive} \it \bf Yes} \\
        \bottomrule
    \end{tabular}}
  \caption[Caption for LOF]{Comparison of methods for low-rank matrix recovery. These are problems where $\c(U) = U U^{\top}$ (symmetric) or $\c(U, V)= UV^{\top}$ (asymmetric). A check mark \cmark\, indicates that the method exhibits local linear convergence (depending only on $h$) for that particular setting.\\
    {\small \textsuperscript{*} The same authors modified ScaledGD to extend to tensor problems \cite{tong2022scaling}.\\}
    {\small \textsuperscript{$\sharp$} Converges arbitrarily close to a solution, with the final distance controlled by a parameter.}\\
  {\small \textsuperscript{\ddag} The method converges to the solution of a regularized problem, which might differ from the original one.}}
    \label{tab:comparison}
\end{table}

\paragraph{(Sub)gradient methods for matrix recovery.}
Low-rank matrix recovery via the factorization approach has been the subject of intensive study over the past decade
\cite{tu2016low, li2020nonconvex, zhang2021general, ma2023geometric, ye2021global}. 
For the exactly parameterized regime, it is known that the optimization landscape of smooth objectives is benign and randomly initialized gradient descent finds global minimizers \cite{ge2017no, zhu2018global, chen2015fast, bhojanapalli2016global, chi2019nonconvex}. %
Yet, all local convergence rates depend on the condition number of the ground truth \cite{charisopoulos2021low, ma2019implicit}. %
Recent work has focused on the rank-overparameterized setting, where subgradient methods still find global minimizers, yet they exhibit a sublinear local rate of convergence due to flattened local geometry caused by overparameterization~\cite{zhuo2024computational,ding2021rank}.  
Several works have proposed strategies to accelerate the convergence of these methods based on small initialization with early stopping and alternating small and long steps \cite{jiang2023algorithmic, wind2023asymmetric, li2018algorithmic, stoger2021small,ma2023global,ding2022validation, jin2023understanding, xu2023power, soltanolkotabi2025implicit, xiong2023over, davis2024gradient}. Even though these methods achieve linear convergence, their rates still depend on the condition number
of the solution matrix. %

\paragraph{Overparameterized matrix recovery problems.} The seminal works \cite{tong2021accelerating, tong2021scaledsubgradient} proposed preconditioned (sub)gradient methods to overcome the dependence on the conditioning of the ground truth matrix in the exactly parameterized setting. Soon after, \cite{zhang2021preconditioned} introduced a preconditioned gradient method that additionally handles overparameterization for smooth objectives with PSD matrices. Subsequently, many other works introduced methods based on preconditioning together with small overparameterization \cite{xu2023power}, alternating minimization \cite{liu2025efficient}, and Gauss-Newton type methods \cite{jiaglobally, laufer2025rgnmr}.  Closer to our work, \cite{giampouras2024guarantees} introduced the Overparameterized Preconditioned Subgradient Algorithm (OPSA), which achieves local convergence rates for nonsmooth, overparameterized, asymmetric objectives. However, OPSA converges to the solution to a regularized problem, which might differ from the ground truth.
Beyond matrix problems, recent literature has also studied preconditioned methods for low Tucker-rank tensor recovery \cite{tong2022accelerating, tong2022scaling, luo2023low, wu2025guaranteed}. 
Given that much of the existing literature focuses on matrix recovery problems, we include a comparison in Table~\ref{tab:comparison}.\footnote{This table is an oversimplification for pedagogical purposes; each statement holds under potential additional assumptions of the respective paper.}

}{ 
} %

\section{Preliminaries}\label{sec:prelims}

\paragraph{Linear algebra.} The symbol $[m]$ will be shorthand for the set $\{1, \dots, m\}.$ Further, for a finite set $S$, the symbol $\# S$ denote its cardinality.  We will use the symbol $\EEE$ to denote a finite-dimensional Euclidean space with inner product $\langle\cdot,\cdot \rangle$ and the induced norm $\|x\|=\sqrt{\langle x,x\rangle}$.  %
The closed ball of radius $r > 0$ around $x \in \EEE$ will be denoted by $\B_{r}(x).$
For any point $x\in \EEE$ and a set  $Q\subset\EEE$,  the distance and the nearest-point projection (with respect to the Euclidean norm) are defined by
\begin{align*}
\dist(x;Q)=\inf_{y\in Q} \|x-y\|\qquad \textrm{and}\qquad\proj_Q(x)=\argmin_{y\in Q} \|x-y\|,
\end{align*}
respectively.
Given a linear map between Euclidean spaces, $\cA\colon\EEE\to{\bf Y}$, its adjoint map will be written as $\cA^*\colon {\bf Y}\to \EEE$. We will use $I_d$ for the $d$-dimensional identity matrix, while ${\bm 0}_{d}$ denotes the $d$-dimensional origin. We use $\R$ and $\R_{+}$ to denote the reals and the positive reals, respectively. We always endow the Euclidean space of vectors $\R^d$ with the usual dot-product $\langle x,y \rangle=x^{\top}y$ and the induced $\ell_2$-norm $\|x\|_2 = \sqrt{\dotp{x,x}}$. For any $x \in \R^{d}$, we use $\supp(x)$ to denote the indices of nonzero entries of $x$. Given two vectors $x, y \in \RR^{d}$, we let $x \odot y \in \RR^{d}$ denote their Hadamard or component-wise product.

Similarly, we will equip the space of rectangular matrices $\R^{d_1\times d_2}$ with the trace product $\langle X,Y\rangle=\tr(X^\top Y)$ and the induced Frobenius norm $\|X\|_F=\sqrt{\tr(X^\top X)}$.
The operator norm of a matrix $X\in \R^{d_1\times d_2}$ will be written as $\|X\|_{\textrm{op}}$. The symbol $\sigma(X)$ will denote the vector of singular values of a matrix $X$ in nonincreasing order. We use $\sigma_{\max}(X)$ and $\sigma_{\min}(X)$ to denote the largest and smallest nonzero singular values. Similarly, for a given symmetric matrix %
The symbols $\mathcal{S}^{d}$ and $\mathcal{S}^{d}_+$ denote the sets of $d \times d$ symmetric and positive semidefinite, respectively. We use $O(d, r)$ to denote the set of matrices with orthogonal columns, i.e., $Q \in \RR^{d \times r}$ such that $Q^{\top}Q = I.$ Given a matrix $X \in \mathcal{S}^{d}$, the symbol $\lambda(X)$ denotes the vector of eigenvalues of $X$ in nonincreasing order.
Given two matrices $A, B$ of potentially different sizes, we let $A \otimeskron B$ denote their Kroncker product. On the other hand, $\otimes$ denotes the tensor product. Note that the inputs and outputs of Kronecker products are always matrices, while the inputs and outputs of the tensor product might be higher-order tensors. Given a matrix $X \in \RR^{d \times r}$, we use $X_{i:}$ and $X_{j}$ to denote its $i$th row and column, respectively, and $X_S$ with $S\subseteq[r]$ to denote the submatrix of $X$ with columns indexed by $S$.

\paragraph{Nonsmooth analysis.} Nonsmooth functions are central to this work. Therefore, we will utilize a few basic constructions of generalized differentiation; we refer the interested reader to the monographs \cite{clarke2008nonsmooth, rockafellar2009variational,Mord_1,cov_lift}.
Consider a function $f\colon\EEE\to\R\cup\{+\infty\}$ and a point $x$, with $f(x)$ finite. We use the convention $\frac{1}{0} = +\infty.$ The {\em subdifferential} of $f$ at $x$, denoted by $\partial f(x)$, is the set of all vectors $\xi\in\EEE$ satisfying
\begin{equation}\label{eqn:subgrad_defn}
f(y)\geq f(x)+\langle \xi,y-x\rangle +o(\|y-x\|)\quad \textrm{as }y\to x,
\end{equation}
where $o(r)$ denotes any function satisfying $o(r)/r\to 0$ as $r\to 0$.
Standard results show that for a convex function $f$ the subdifferential $\partial f(x)$ reduces to the subdifferential in the sense of convex analysis, while for a differentiable function, it consists only of the gradient: $\partial f(x)=\{\nabla f(x)\}$. For any closed convex functions $h\colon\YY\to\R$ and
$C^1$-smooth map $F\colon\EEE\to\YY$, the chain rule holds \cite[Theorem 10.6]{rockafellar2009variational}:
$$\partial (h\circ F)(x)=\nabla F(x)^*\partial h(F(x)).
$$

\changelocaltocdepth{2}
\section{Algorithm and assumptions}\label{sec:algorithm}
In this section, we formally introduce the problem class we study, the different assumptions we make, and the Levenberg-Morrison-Marquardt subgradient method we use for the rest of the paper. As mentioned in the introduction, we consider
\begin{align}\label{eqn:problem}
\min_{x\in \EEE} f(x)  \quad \text{with} \quad f := \h \circ \c.
\end{align}
Here $\EEE$ and $\YY$ are finite-dimensional Euclidean spaces, $\h \colon \YY \rightarrow \R$ is a convex function, and $\c:\EEE \rightarrow \YY$ is a continuously differentiable map. Instantiations of this template include: low-rank symmetric matrix problems where $\EEE = \RR^{d \times r}, \YY = \cS^{d}_{+},$ and $\c(U) = UU^{\top}$, and asymmetric matrix problems where $\EEE = \R^{d_{0} \times r} \times \R^{d_{2} \times r}, \YY = \RR^{d_{1} \times d_{2}},$ and $\c(U, V) = UV^{\top}.$
\ifbool{showSquare}{From now on, the symbol $\cX^{\star}$ denotes the set of minimizers of \eqref{eqn:problem}.}{}

We say that problem \eqref{eqn:problem} is \emph{overparameterized} if for some $x^{\star} \in \cX^{\star},$ there is a sequence $\left(x_{j}\right)_{j} \subseteq \EEE$ converging to $x^{\star}$ such that the rank of $\nabla \c(x_{j})$ exceeds the rank of $\nabla \c(x^{\star}).$ Intuitively, this means that there are points arbitrarily close to a minimizer $x^{\star}$ for which a linear approximation of $\c$ requires more parameters than at $x^{\star}$ itself. This definition matches the natural notion of overparameterization for low-rank problems; we defer additional details to Section~\ref{sec:examples}.

\subsection{Algorithmic description}
\begin{algorithm}[t]
	\caption{Levenberg-Morrison-Marquardt Subgradient Method (LMM) \hfill} \label{alg:LM}
	
	{\bf Input}: Initial $x_0\in\EEE$, stepsizes $(\gamma_k)_{k\geq 0}\subset (0,\infty),$ and damping coefficients $(\lambda_{k})_{k \geq 0} \subset (0, \infty)$
	
	\vspace{.1cm}
	{\bf Step} $k\geq 0$:
	\begin{equation*}
	\begin{aligned}
	&\textrm{Pick~}  v_{k} \in \partial \h(\c(x_{k})).   \\
	&\textrm{Set~} x_{k+1} \leftarrow x_k- \gamma_{k}  \left(\nabla \c(x_k)^\top \nabla \c(x_k) + \lambda_k I\right)^{-1} \nabla \c(x_k)^\top v_k.
	\end{aligned}
	\end{equation*}
\end{algorithm}

The Levenberg-Morrison-Marquardt subgradient method is summarized in Algorithm~\ref{alg:LM}.
For structured problems, the linear system at each iteration can often be solved efficiently. For instance, for low-rank matrix recovery, it reduces to solving a much smaller linear system\ifbool{showSquare}{; we defer the details to Appendix~\ref{efficient}}{}. This method builds upon the GNP method introduced in \cite{davis2022linearly}, which sets $\lambda_{k} =0.$ The inspiration for GNP stems from a simple observation: when $\nabla \c$ has constant rank near $x^{\star}$---i.e., under exact parameterization---the image of $\c$ forms a manifold $\cM$ around $F(x^\star)$. An elegant argument shows that in this regime, the mapped iterates of GNP $ z_k = \c(x_{k})$ are akin to the iterates of a Riemannian subgradient method on $\cM$ with objective $\h$ \cite{davis2022linearly}. Consequently, the iterates $\c(x_k)$ are unaffected by the ill-conditioning of $\c.$ In contrast, overparameterization leads to problems where the image of $\c$ fails to form a manifold. In such cases, the Gauss-Newton preconditioner is not even well-defined since $\nabla \c(x_k)$ might not have constant rank. Our method bypasses this issue by adding a damping term to the preconditioner, 
thereby ensuring its invertibility and stability. 

We propose two ways to set the hyperparameters $\gamma_{k}$ and $\lambda_{k}$ of Algorithm~\ref{alg:LM}, which we dub ``configurations.'' {The first configuration} is based on the Polyak stepsize \cite{polyak1969minimization} and the damping parameter for the preconditioned gradient descent method in~\cite{zhang2021preconditioned}. From now on, we will use $h^\star$ and $\cZ^\star$ to denote the minimum value and set of minimizers of $\min_{z \in \Ima \c} h(z).$ Further, we use $\Pi^x $ as a shorthand for the projection matrix $\proj_{\Ima \nabla \c(x)}.$
\begin{pconfig}[Polyak]\label{assum: dampingparameter} %
	Set the stepsize to $\gamma_k = \gamma \frac{\h(z_k) - \h^{\star}}{\|\Pi^{x_k} v_k\|^2}$ %
    where  $\gamma >0$ is a tuning parameter. %
    Additionally, set $\lambda_{k}$ such that there exist constants $0 < \clb \leq \cub$ satisfying
	$$
	\clb \dist(z_k, \cZ^\star) \le \lambda_k \le \cub \dist(z_k, \cZ^\star).
	$$
\end{pconfig}
\noindent This parameter choice relies on detailed information about the loss function and the distance to the solution set—quantities that may not be readily available in practice. To address this limitation, we introduce \ifbool{showSquare}{two alternative configurations tailored for nonsmooth and smooth problems, which do not}{an alternative configuration, which does not} depend as heavily on such information.

\ifbool{showSquare}{\begin{pconfig}[Nonsmooth]\label{assum:geometric} Set $\gamma_k = \gamma q^k$ and $\lambda_k = \lambda q^k$, with $\gamma, \lambda > 0$ and $q \in (0,1)$. \end{pconfig}

\begin{pconfig}[Smooth]\label{assum:constant_stepsize} Set $\gamma_k = \gamma$ and $\lambda_k = \lambda q^k$, with $\gamma, \lambda > 0$ and $q \in (0,1)$. \end{pconfig}}{\begin{pconfig}\label{assum:geometric} Set $\gamma_k = \gamma q^k$ and $\lambda_k = \lambda q^k$, with $\gamma, \lambda > 0$ and $q \in (0,1)$. \end{pconfig}}

\ifbool{showSquare}{\noindent The primary difference between these two configurations is that the stepsize decreases geometrically for nonsmooth functions while it remains constant for smooth functions. These strategies are designed to emulate the behavior of the Polyak stepsize \cite{goffin1977convergence, davis2018subgradient}. In the following section, we demonstrate that these configurations achieve convergence rates comparable to those of the Polyak-based approach.}{}

\subsection{Regularity of the parameterization}
Overparameterization and the addition of the damping term complicate the analysis and require more nuanced regularity conditions on the parameterization $\c$. Conveniently, these conditions are independent of the conditioning assumptions for the outer function $\h$, allowing us to pair any sufficiently regular parameterization $\c$ with any well-conditioned $\h$\ifbool{showSquare}{, whether smooth or not}{}. The remainder of this section introduces these assumptions\ifbool{showSquare}{ and collects some algorithmic consequences. We begin with a standard assumption on $\c$}{}.

\begin{assumption}[Smooth parameterization]\label{assum: assumptiononc} \label{assum:mapc:smoothness}The map $\c$ is continuously differentiable, and there is a constant $\lc \geq 0$ %
such that
	\begin{align*}
	\norm{\nabla \c(x) - \nabla \c(y)}{\mathrm{op}} \le \lc \norm{x-y}{}, \quad \text{for all } x, y \in \EEE.
	\end{align*}
\end{assumption}

\noindent Henceforth, we use the following notation 
\begin{equation}\label{eq:matrix-P}
    P(x, \lambda) := \nabla \c(x) (\nabla \c(x)^\top \nabla \c(x) + \lambda I)^{-1} \nabla \c(x)^\top.
\end{equation}
The matrix $P(x, \lambda)$ \ifbool{showSquare}{will play a crucial role in our analysis.}{plays a crucial role in the analysis of the method provided in our technical report.} It can be seen as a regularized projection matrix; indeed, when $\lambda = 0$ and $\nabla \c(x)$ has full column rank, it reduces to the orthogonal projection onto the range of $\nabla \c(x)$, which corresponds to $\Pi^x.$
In what follows, we use the placeholders $z_k := \c(x_k)$\ifbool{showSquare}{ and
$$
P_k := P(x_k, \lambda_k).
$$}{.}
\ifbool{showSquare}{The next result collects a few properties of $P_{k};$
its proof is deferred to Appendix~\ref{app:proof-lem-boundontaylor}.  We introduce a bit of simplifying notation. Given any $x \in \EEE,$ let $U^{x}$ and $\sigma^{x}$ denote respectively the matrix of left singular vectors and the vector of singular values of $\nabla \c(x),$ moreover, we let $U^x_{1:j}$ be the matrix with the top $j$ singular vectors of $\nabla F(x)$.
\begin{lemma}\label{lem:boundontaylor}
	Let $x_k$ and $x_{k+1}$ be iterates from Algorithm~\ref{alg:LM}. Let $z_{k} = \c(x_{k})$ and $z_{k+1} = \c(x_{k+1})$. The following three hold true. 
    \begin{enumerate}
		\item{{\upshape ({\bfseries Nonexpansiveness})}} The operator norm of both $P_{k}$ and $I - P_{k}$ are less than or equal to one. Moreover, $\norm{P_k v}{} \le \norm{\P^{x_k}v}{}$ for any $v \in \YY$. 
		\item{\upshape ({\bfseries Restricted eigenvalues})}  \label{item:restricted_eigen}The the eigenvalues of $I - P_{k}$ restricted to the span generated by the top $j$ left singular vectors of $\nabla \c(x_{k})$ are bounded by $\frac{\lambda_{k}}{(\sigma_{j}^{x_k})^{2} + \lambda_{k}},$ i.e.,
        $$
        \|(I-P_k) v\|\leq \frac{\lambda_{k}}{(\sigma_{j}^{x_k})^{2} + \lambda_{k}} \|v\| \qquad \text{for }v \in \spann{U^{x_k}_{1:j}}.
        $$
		\item{\upshape ({\bfseries Approximation error})}  If in addition $\nabla F$ is $\lc$-Lipschitz when restricted onto the line segment connecting $x_k$ and $x_{k+1}$, then
		$\norm{z_{k+1} - (z_k - \gamma_k P_k v_k)}{} \le \frac{\lc}{8\lambda_k} \gamma_k^2 \norm{\P^{x_k}v_k}{}^2$.
		
	\end{enumerate}
\end{lemma}
\noindent In particular, the last item follows for any pair of consecutive iterates whenever Assumption~\ref{assum: assumptiononc} holds. Intuitively, this item states that an updated mapped iterate $z_{k+1} =\c(x_{k+1})$ can be approximated by a linear step in $\YY$ space. This approximation will play a critical role in our analysis.}{}

\ifbool{showSquare}{
 The following two assumptions are crucial to obtain linear convergence for ill-conditioned and overparameterized problems, respectively. These assumptions are stated near a point $\zs,$ which the reader can deem as a minimizer of $\h$ over the image of $\c.$ We let $\P_{j}^{x} = U^x_{1:j} (U^x_{1:j})^\top$ be the projection onto the subspace spanned by the top $j$ left singular vectors of $\nabla \c(x)$.
\begin{assumption}[Strong alignment]\label{ass:strong-alignment}
	For a given $\zs \in \Ima \c$, there exist a function $\rloc\colon \RR_{+} \rightarrow \RR_{+}$ and a constant $\sig > 0$ such that for any $\rho >0$
	and $z = \c(x) \in \B_{\rloc(\rho)}(z^{\star})$ there is an index $j$ for which
	\begin{align*}
	\big\|(I- \P^{x}_{j})(z - \zs)\big\|_{}  \le  \rho \norm{z - \zs}{} \qquad \text{and} \qquad {(\sigma_{j}^{x})^2} \ge \sig.
	\end{align*}
\end{assumption}
\noindent Intuitively, this assumption amounts to the alignment between the image of $\c$ and a low-dimensional linear approximation around $\zs$.
As alluded to earlier, in the exact parameterization regime when the rank of $\nabla \c (\cdot)$ is locally constant, the image of $\c$ forms a manifold $\cM$ of dimension $j = \rank (\nabla \c(\xs)).$
As such, $\P^{x}_{j}$ corresponds to the projection onto the tangent of $\cM$ at $\c(x).$ In turn, the error between centered manifold elements $z - \zs$ and the tangent increases at most quadratically in the norm of $z - \zs$, and one can easily derive that this assumption holds; see Lemma~\ref{constant_rank_implies_strong_alignment} in the appendix.

However, for overparameterized problems, there is no manifold structure, and Assumption~\ref{ass:strong-alignment} cannot hold. To overcome this issue, we introduce a weaker condition, allowing the singular values of the linear approximation to decrease gracefully as we approach $\zs.$}{The following assumption is crucial to obtain linear convergence for ill-conditioned and overparameterized problems. The assumption is stated near a point $\zs,$ which the reader can deem as a minimizer of $\h$ over the image of $\c.$ We let $\P_{j}^{x}$
be the projection onto the subspace spanned by the top $j$ left singular vectors of $\nabla \c(x)$.}

\begin{assumption}[Weak alignment]\label{ass:weak-alignment}
	For a given $\zs \in \Ima \c$, there exist functions $\rloc \colon \RR_{+} \rightarrow \RR_{+}$ and $\sig\colon \RR_{+} \rightarrow \RR_{+}$ such that for any $\rho >0$ and any $z = \c(x) \in \B_{\rloc(\rho)}(z^{\star})$ there is an index $j$ for which 
	\begin{align*}
	\big\|(I- \P^{x}_{j})(z - \zs)\big\|  \le  \rho \norm{z - \zs}{} \qquad \text{and} \qquad \left(\sigma_{j}^{x}\right)^{2} \ge \sig(\rho) \norm{z - \zs}{}.
	\end{align*}
\end{assumption}

\ifbool{showSquare}{\noindent In Section~\ref{sec:examples}, we will see that this assumption holds for a variety of overparameterized problems.}{\noindent As we shall see, this assumption provably holds a number of parameterizations, including the Burer-Monteiro factorization and the canonical polyadic (CP) tensor rank decomposition.} \ifbool{showSquare}{It is immediately clear that Assumption~\ref{ass:strong-alignment} implies Assumption~\ref{ass:weak-alignment}. %
}{}We emphasize that both assumptions on $\c$ are independent of the outer function $\h$.

\subsection{Regularity \ifbool{showSquare}{for nonsmooth outer functions}{of the outer function}}
Next, we introduce the regularity conditions on the outer function $\h$. Intuitively, regularity ensures that the function $\h$ is well-conditioned when restricted to the image of $\c.$ \ifbool{showSquare}{We present two different notions of conditions depending on the smoothness of the problem. We start with conditions for nonsmooth losses. {Recall that we use $\Pi^x $ as a shorthand for the projection matrix $\proj_{\Ima \nabla \c(x)}.$}}{}   

\begin{assumption}\label{assum:nonsmoothpenalty}
	The function $\h\colon \YY \rightarrow \RR$ satisfies the following properties.
	\begin{enumerate}
		\item{(\textbf{Unique minimizer})}\label{item:nonsmoothpenalty:uniqueminimizer} The function $\h$ has a unique minimizer $\zs$ over $\Ima\c.$
		\item{(\textbf{Convexity})}\label{item:nonsmoothpenalty:convexity} The function $\h$ is convex. 
	\item{(\textbf{Restricted sharpness})}\label{item:nonsmoothpenalty:sharpness} The function $\h$ is $\mu$-sharp on $\Ima \c.$ That is,
$$ \h(z) - \h^{\star} \ge \mu \cdot \norm{z - \zs}{} \qquad \text{for all }z \in \Ima \c,$$
where $\h^\star = \h(\zs)$ is the minimum of $\h\mid_{\Ima \c}$.
 \item{(\textbf{Restricted Lipschitzness})}\label{item:nonsmoothpenalty:lip} There exists a constant $L \ge 0$ such that 
\begin{enumerate}
	\item \label{item:nonsmoothpenalty:lip_a}For any $x \in \dom F$ and  $v \in \partial h( F(x)),$
	\begin{align*}
	\|\P^x v\| \le L.
	\end{align*}
	\item \label{item:nonsmoothpenalty:lip_b} For any $x \in \EEE$, $z = F(x),$  $v \in \partial h(F(x))$, and $\lambda >0$ %
	we have
	\begin{align*}
	|\dotp{v, (I-P(x,\lambda))(z - z^\star)}| \le L \norm{(I-P(x,\lambda))(z - z^\star)}{},
	\end{align*}
    where the matrix $P(x, \lambda)$ is given by \eqref{eq:matrix-P}.
\end{enumerate}

	\end{enumerate}
\end{assumption}
\noindent 
We point out that for our results, one can drop the uniqueness of the minimizer in Assumption~\ref{assum:nonsmoothpenalty}, but we assume it for simplicity. Although the fourth condition might seem complicated, it is satisfied by any globally Lipschitz convex function $h$. \ifbool{showSquare}{Moreover, a simple argument shows that it ensures that $h(z) - \h^\star \le 2L\|z - \zs\|$ for all $z \in \Ima \c.$ These regularity conditions are well-understood in the unconstrained setting where the parameterization is an identity, $\c = I.$ The seminal work~\cite{polyak1969minimization} showed that the subgradient method coupled with the Polyak stepsize converges linearly to minimizers in this setting.}{} %
\ifbool{showSquare}{The following is a direct consequence of Assumption~\ref{assum:nonsmoothpenalty}.
\begin{lemma}[{Aiming towards solution}]\label{lem:aiming}
	Suppose that Assumption~\ref{assum:nonsmoothpenalty} holds. Then,
	$$
	\dotp{v, z - \zs} \ge \h(z) - \h(\zs)  \ge \mu \norm{z-\zs}{} \qquad \text{for all } z \in \Ima \c\text{ and }v \in \partial h(z).	$$
\end{lemma}
\noindent Thus, negative subgradients point towards the solution.}{}

\ifbool{showSquare}{
\subsection{Regularity for smooth outer functions}
Paralleling regularity for nonsmooth functions $\h$, we now introduce analog regularity conditions for the smooth setting. Intuitively, they amount to quadratic lower and upper bounds.
\begin{assumption}\label{assum:smoothpenalty}
	The function $\h \colon \YY \rightarrow \R$ satisfies the following properties.
	\begin{enumerate}
		\item{(\textbf{Unique minimizer})}\label{item:smoothpenalty:uniqueminimizer} The function $\h$ has a unique minimizer $\zs$ over $\Ima\c.$
		\item{(\textbf{Convexity})}\label{item:smoothpenalty:convexity} The function $\h$ is convex.
	{\item{(\textbf{Restricted quadratic growth})}\label{item:smoothpenalty:sharpness} The function $\h$ exhibits $\mus$-quadratic growth on $\Ima \c;$ i.e., 
		$$ \h(z) - \h^{\star} \ge \frac{\mus}{2}  \|z - \zs\|^{2}  \qquad \text{for all }z \in \Ima \c,$$
		where $\h^\star = \h(\zs)$ is the minimum of $\h\mid_{\Ima\c}$. 
		\item\label{item:smoothpenalty:lip}{(\textbf{Restricted Smoothness})}  There exists a constant $\Lhs \ge 0$ such that for any $x \in \EE$ and $z = F(x)$ the following hold true.
		\begin{enumerate}
			\item We have
			\begin{align*}
				\norm{\Pi^x \nabla h(z)}{} \le \Lhs \norm{z - z^\star}{} \qquad  \text{and} \qquad \h(z) - h^\star \ge \frac{1}{2\Lhs}\left\|\Pi^x\nabla \h(z)\right\|^2.
			\end{align*}
			\item For any  $\lambda >0$
			\begin{align*}
			|\dotp{\nabla h(z), (I-P(x,\lambda))(z - z^\star)}| \le \Lhs \left\|z- z^\star\right\| \left\|(I-P(x,\lambda))(z - z^\star)\right\|,
			\end{align*}
            where $P(x, \lambda)$ is given in \eqref{eq:matrix-P}.
		\end{enumerate}
	}
	\end{enumerate}
\end{assumption}
\noindent The first two conditions are exactly the same as in the nonsmooth setting. 
The third condition is satisfied by any globally $\mus$-strongly convex function, while the latter one holds for any globally $\Lhs$-smooth function \cite{nesterov2013introductory}; e.g., both are trivially satisfied by $h(\cdot) = \frac{1}{2}\|\cdot \|_{2}^{2}.$
We collect an analog to Lemma~\ref{lem:aiming}; the proof follows from convexity and quadratic growth.
\begin{lemma}[{Aiming towards solution}]\label{lem:aiming_smooth}
	Suppose that Assumption~\ref {assum:smoothpenalty} holds. Then,
	$$
	\dotp{\nabla \h(z), z - \zs} \ge \h(z) - \h(\zs)  \ge \frac{\alpha}{2} \norm{z-\zs}{}^2 \qquad \text{for all }z \in \EEE.
	$$
\end{lemma}

We close this section with a bound on the progress made by the approximation from Lemma~\ref{lem:boundontaylor}. We highlight that the next result applies regardless of the smoothness of $\h.$

\begin{lemma}[{Linearization progress}]\label{lem:progressfromlineaization}
	Let $x_k$ be an iterate from Algorithm~\ref{alg:LM} with Configuration~\ref{assum: dampingparameter} (Polyak stepsizes) where we set the hyperparameter $\gamma \le 1$. Suppose that $h$ is convex and $F$ is continuously differentiable. Letting $z_k = \c(x_k)$, we have 
	\begin{align*}
	\norm{z_k - \gamma_k P_k v_k- \zs}{}^2 \le  \norm{z_k - \zs}{}^2- \gamma\frac{(\h(z_k) - h^\star)^2}{\norm{\Pi^{x_k} v_k}{}^2} + 2\gamma_k \dotp{(I-P_k)v_k, z_k -  \zs}.
	\end{align*}
\end{lemma}
\begin{proof} Expanding
	\begin{align}
	\norm{z_k - \gamma_k P_k v_k- \zs}{}^2 \notag
	&= \norm{z_k - \zs}{}^2- 2\gamma_k \dotp{P_k v_k, z_k -  \zs} + \gamma_k^2 \norm{P_k v_k}{}^2\notag\\
	&=  \norm{z_k -  \zs}{}^2- 2\gamma_k \dotp{v_k, z_k -  \zs} + 2\gamma_k \dotp{(I-P_k)v_k, z_k -  \zs} + \gamma_k^2 \norm{P_k v_k}{}^2\notag\\
	&\le  \norm{z_k -  \zs}{}^2- \gamma \frac{(\h(z_k) - h^\star)^2}{\norm{\P^{x_k} v_k}{}^2} +  2\gamma_k \dotp{(I-P_k)v_k, z_k -  \zs},
	\end{align}
	where the last inequality eliminates the term $\dotp{v_{k}, z_{k} - \zs}$ via convexity of $h$, and upper bounds the last using Lemma~\ref{lem:boundontaylor}, Configuration~\ref{assum: dampingparameter}, and $\gamma \le 1$. This completes the proof.
\end{proof}}{}

\section{General convergence guarantees}\label{sec:guarantees}

In this section, we present our general-purpose guarantees for Algorithm~\ref{alg:LM} \ifbool{showSquare}{under the various parameter choices and smoothness assumptions introduced in Section~\ref{sec:algorithm}. We also extend these guarantees to cases where the regularity assumptions for the parameterization $\c$ hold only locally, an extension that will be particularly useful for tensor problems in later sections. Since most proofs share the same structure, we provide only the simplest versions to illustrate the core ideas and defer technical details to Appendix~\ref{app:proofs-guarantees}}{under the assumptions and hyperparameter configuration presented in Section~\ref{sec:algorithm}}. %

\ifbool{showSquare}{\subsection{Guarantees for nonsmooth losses}}{}

\ifbool{showSquare}{We provide guarantees for both the exactly parameterized and overparameterized regimes. We start with the latter due to its novelty.}{}
{ \begin{theorem}[\textbf{Convergence under weak alignment and nonsmoothness}]\label{thm: onestepimprovement}
		Suppose that Assumptions~\ref{assum: assumptiononc},~\ref{ass:weak-alignment} and~\ref{assum:nonsmoothpenalty} hold.
		Further assume that $z_0 = \c(x_0)$ satisfies
		$\norm{z_0 - \zs}{} \le \rloc\left(\frac{\mu}{8\lh}\right).$ The following two hold.
		\begin{enumerate}
			\item{(\textbf{Polyak stepsize})}\label{item:nonsmooth-weak-polyak} Suppose we ran Algorithm~\ref{alg:LM} initialized at $x_{0}$ using Configuration~\ref{assum: dampingparameter} with $\gamma \le \min\left\{1, \frac{\clb}{\lc }\right\}$ and $\cub \le \frac{\mu }{8\lh} \sig\left(\frac{\mu}{8\lh}\right)$.
			Then, the iterates $z_{k }= \c(x_{k})$ satisfy
			$$
			\norm{z_{k} - \zs}{} \le \left(1 - \frac{\gamma \mu^2}{8 \lh^2}\right)^{k/2} \norm{z_0 - \zs}{} \quad \text{for all } k\geq 0.
			$$

			\item{(\textbf{Geometrically decaying stepsize})}\label{item:nonsmooth-weak-geom}  Suppose we ran Algorithm~\ref{alg:LM} initialized at $x_{0}$ using Configuration~\ref{assum:geometric} with 
			\begin{align*}
			&
			\lambda \le \frac{ M\sig\left(\frac{\mu}{8\lh}\right)}{128} \frac{\mu}{L},
			\quad \gamma \le \frac{1}{\lh} \cdot \min \left\{ \frac{M\mu}{64\lh},%
			\sqrt{\frac{2\lambda M}{\lc}}, \frac{\lambda \mu}{2\lc\lh}\right\}
			\quad \text{and} \quad 	q \geq \max\left\{1- \frac{\gamma \mu}{4M}, \frac{1}{\sqrt{2}}\right\}
			\end{align*}
			where $M= \rloc\left(\frac{\mu}{8\lh}\right)$. Then, the iterates $z_{k} = \c(x_{k})$ satisfy
			$$\|z_{k} - \zs\|\leq M q^{k} \quad \text{for all }k \geq 0.$$
			
		\end{enumerate}
	\end{theorem}
}

\ifbool{showSquare}{Here, we only prove the statement concerning the Polyak stepsize and defer the proof for the geometrically decaying stepsize to Appendix~\ref{app:proof-nonsmooth-geometrically}. The proofs of all our results follow the same template. Before proving Theorem~\ref{thm: onestepimprovement}, we introduce a proposition that provides the shared machinery underlying our argument for the Polyak stepsize.}{}

\ifbool{showSquare}{
	\begin{proposition}[{{One-step progress}}]\label{prop:master_polyak}
		Suppose that $x_k$ and $x_{k+1}$ are iterates of Algorithm~\ref{alg:LM} with Configuration~\ref{assum: dampingparameter}.  Assume in addition that $h$ is convex, $F$ is continuously differentiable, and $\nabla F$ is $\lc$-Lipschitz when restricted onto the line segment connecting $x_k$ and $x_{k+1}$. Define $z_k = \c(x_k)$ and $z_{k+1} = \c(x_{k+1})$. If the  stepsize hyperparameter satisfies $\gamma \le \min\left\{1, \frac{\clb }{\lc}\right\}$ and 	\begin{align}\label{eqn:key_zk_inner_prod_bound}
		|\dotp{(I-P_k) v_k, z_k - z^\star}| \le \frac{1}{4} (h(z_k) - h^\star),
		\end{align}
		then, we have
		\begin{align*}
		\|z_{k+1} - z^\star\|^2 \le \|z_k - z^\star\|^2 - \frac{\gamma}{8} \frac{(h(z_k) - h^\star)^2}{\|\P^{x_k}v_k\|^2}.
		\end{align*}
	\end{proposition}
	\begin{proof}[Proof of Proposition~\ref{prop:master_polyak}]
		To derive this bound, we apply the triangle inequality with a one-step linear approximation
		\begin{align}\label{eq:first-of-many-trinagles}
		\norm{z_{k+1} - \zs}{} &\leq \underbrace{\norm{z_{k+1} - \left(z_{k} - \gamma_{k} P_{k} v_{k}\right)}{}}_{T_{1}} + 	\underbrace{\norm{\left(z_{k} - \gamma_{k} P_{k} v_{k}\right) - \zs}{}}_{T_{2}}.
		\end{align}
		We focus on bounding each of these terms separately. To bound $T_{1}$ we apply Lemma~\ref{lem:boundontaylor} and obtain
		\begin{align}
		T_{1} & \leq \frac{\gamma_k^2 \lc}{8\lambda_k} \norm{\P^{x_k} v_k}{}^2 \notag \\
		&\le  \frac{\gamma_k^2 \lc}{8\clb\norm{z_{k} - \zs}{}} \norm{\P^{x_k} v_k}{}^2.	\label{eq:boundTone}
		\end{align}
		To bound $T_{2}$, we compute
		\begin{align}
		T_{2}^{2} &\leq
		\norm{z_k - \zs}{}^2- \gamma\frac{(\h(z_k) - h^\star)^2}{\norm{\P^{x_k} v_k}{}^2}  + 2\gamma_k \left|\dotp{(I-P_k)v_k, z_k -  \zs}\right|\notag\\
		&\leq
		\norm{z_k - \zs}{}^2- \gamma\frac{(\h(z_k) - h^\star)^2}{\norm{\P^{x_k} v_k}{}^2} + \frac{\gamma_k}{2}\left(\h(z_k) - h^\star\right)\notag \\
		& = \norm{z_k - \zs}{}^2- \frac{\gamma}{2}\frac{(\h(z_k) - h^\star)^2}{\norm{\P^{x_k} v_k}{}^2}, \label{eqn:T_2_bound_tighter}
		\end{align}
		where the first inequality follows from Lemma~\ref{lem:progressfromlineaization} and the second inequality follows from the bound~\eqref{eqn:key_zk_inner_prod_bound}. Next, we state a claim that we will use recurrently.
\begin{claim}\label{claim:aligned-projected-subgradients} If we have that $\left|\dotp{(I-P_k) v_k, z_k-\zs}\right| \leq \frac{1}{4}(h(z_k) - h^\star),$ then
    $$
    \frac{3}{4} \left(h(z_k) - \h(\zs)\right) \leq \dotp{P_k v_k, z_k-\zs} \leq \|\Pi^{x_k} v_k\| \|z_k -\zs\|.
    $$
\end{claim}
\begin{proof}[Proof of the Claim~\ref{claim:aligned-projected-subgradients}]
By the subgradient inequality
\begin{align*}
    h(z_k) -h^\star\leq 
    \dotp{P_k v_k, z-\zs} + \dotp{(I -P_k) v_k, z_k-\zs} \leq \dotp{P_k v_k, z_k-\zs} + \frac{1}{4} \left(h(z_k) -h^\star\right),
\end{align*}
rearranging the terms establishes the first inequality. The second inequality follows from Cauchy-Schwarz and Lemma~\ref{lem:boundontaylor}.
\end{proof}
        In particular, we trivially obtain $T_{2} \leq \|z_{k} - \zs\|_{}.$
		Invoking \eqref{eq:first-of-many-trinagles} gives
		\begin{align}
		&\|z_{k+1}-\zs\|^{2}\\
		& \leq T_{1}^{2} + T_{2}^{2} + 2T_{1}T_{2} \notag\\
		& \leq   \norm{z_k -\zs}{}^2 - \frac{\gamma}{2} \frac{(\h(z_k) - h^\star)^2}{\norm{\P^{x_k} v_k}{}^2} +\frac{\gamma^4\lc^2  }{64\clb^2  \|z_k - z^\star\|^2} \frac{(\h(z_k) -  h^\star)^4}{\norm{\P^{x_k} v_k}{}^4}	+  \frac{\gamma^2 \lc }{4\clb } \frac{(\h(z_k) - h^\star)^2}{\norm{\P^{x_k} v_k}{}^2} \notag\\
		& \leq   \norm{z_k -\zs}{}^2 - \frac{\gamma}{4} \frac{(\h(z_k) - h^\star)^2}{\norm{\P^{x_k} v_k}{}^2} +\frac{\gamma^4\lc^2 }{64\clb^2  \|z_k - z^\star\|^2} \frac{(\h(z_k) -  h^\star)^4}{\norm{\P^{x_k} v_k}{}^4}\notag\\
		& \leq   \norm{z_k -\zs}{}^2 - \frac{\gamma}{8} \frac{(\h(z_k) - h^\star)^2}{\norm{\P^{x_k}v_k}{}^2}, \label{eqn:one_step_zk_nonsmooth}
		\end{align}
		where the second inequality uses \eqref{eq:boundTone}, \eqref{eqn:T_2_bound_tighter}, and the definition of $\gamma_k$, while the third and the final inequalities follow from $\gamma \leq \min \left\{1, \frac{\clb }{\lc }\right\}$ together with Claim~\ref{claim:aligned-projected-subgradients}.
		The proof of Proposition~\ref{prop:master_polyak} is complete.
	\end{proof}

Armed with this proposition, we prove the main result of this section.
{\begin{proof}[Proof of Theorem~\ref{thm: onestepimprovement}] By induction, it suffices to show that for $z_k$ satisfying $\|z_k - z^\star\| \le \delta \left(\frac{\mu}{8\lh}\right)$,
		\begin{equation}\label{claim:nonsmooth_weak}
		\norm{z_{k+1} - \zs}{}^2 \le \left(1 - \frac{\gamma \mu^2}{8 \lh^2}\right) \norm{z_k - \zs}{}^2.
		\end{equation}
		Let $j$ be the index provided by Assumption~\ref{ass:weak-alignment} when applied to $\rho = \frac{\mu}{8\lh}$ and $z = z_{k} = \c(x_{k})$, i.e.,
		\begin{equation}\label{eq:applying-weak-aligntment}
		\norm{(I- \P_{j}^{x_{k}})(z_k - \zs)}{} \le  \frac{\mu}{8\lh}\norm{z_k - \zs}{} \quad \text{and} \quad   (\sigma_{i}^x)^2 \ge \sig\left(\frac{\mu}{8\lh}\right) \norm{z_k - \zs}{}.
		\end{equation}
		Invoking Item~\ref{item:nonsmoothpenalty:lip} of Assumption~\ref{assum:nonsmoothpenalty} and the triangle inequality, we derive
		\begin{align}
		\begin{split}\label{eq:projected-angle}
		\left|\dotp{(I-P_k)v_k, z_k -  \zs}\right| & \leq L \left\|(I - {P_{k}})(z_{k} - \zs)\right\| \\
		&\leq \lh \left( \left \|(I - {P_{k}}) \Pi^{{x_{k}}}_{j}(z_{k} - \zs) \right\| + \left\|\left(I - {P_{k}}\right) \left(I - \Pi^{{x_{k}}}_{j}\right)(z_{k} - \zs) \right\| \right). 
		\end{split}
		\end{align}
		Lemma~\ref{lem:boundontaylor} ensures that the eigenvalues of $I - P_{k}$ restricted to the span generated by the top $j$ left singular vectors of $\nabla \c(x_{k})$ are bounded by $\frac{\lambda_{k}}{(\sigma_{j}^x)^{2} + \lambda_{k}}.$ Using this fact in tandem with \eqref{eq:applying-weak-aligntment} gives
		\begin{align}\label{eqn:inner_prod_orthogonal}
		\begin{split}
		\left|\dotp{(I-P_k)v_k, z_k -  \zs}\right| &\le \lh\left(\frac{\lambda_k}{(\sigma_{j}^x)^{2}  + \lambda_k} \norm{\Pi_j^{x_k}(z_k - \zs)}{} + \frac{\mu}{8\lh} \norm{z_k - \zs}{}\right)\\
		&\le \lh  \left( \frac{\cub \norm{z_k - \zs}{}}{\sig\left(\frac{\mu}{8\lh}\right) \norm{z_k - \zs}{2} + \cub\norm{z_k - \zs}{}} + \frac{\mu}{8\lh}\right) \norm{z_k - \zs}{} \\
		&= \lh  \left( \frac{\cub }{\sig\left(\frac{\mu}{8\lh}\right)  + \cub} + \frac{\mu}{8\lh}\right) \norm{z_k - \zs}{} \\
		& \le \frac{\mu}{4}\norm{z_k - \zs}{} \\
		&\le \frac{1}{4}\left(\h(z_k) - h^\star\right).
		\end{split}
		\end{align}
		The second and third inequalities use the fact that the function $b \mapsto \frac{b}{a + b}$ is strictly increasing on $\RR_{+}$ for any given $a > 0$ together with the bounds $\lambda_{k}/\|z_{k} - \zs\|_{2} \leq \cub \leq \sig\left(\frac{\mu}{8\lh}\right) \frac{\mu}{8\lh}$ given by assumption. 
		Invoking Proposition~\ref{prop:master_polyak} and Assumption~\ref{assum:nonsmoothpenalty} gives
		\begin{align*}
		\|z_{k+1}-\zs\|^{2}
		& \leq   \norm{z_k -\zs}{}^2 - \frac{\gamma}{8} \frac{(\h(z_k) - h^\star)^2}{\norm{\P^{x_k} v_k}{}^2} 
		 \leq  \left(1 - \frac{\gamma \mu^2}{8{\lh^2}}\right)  \norm{z_k -\zs}{}^2,
		\end{align*}
		completing the proof of Theorem~\ref{thm: onestepimprovement}.
	\end{proof}
}}{} 
\ifbool{showSquare}{Let us make a few remarks about the statement for the Polyak stepsize. While our conclusions also apply to geometrically decaying stepsizes, they are more transparent in the Polyak case. This result ensures that the iterates $z_k$ will be within distance $\varepsilon > 0$ of the minimizer after $O\left( \frac{1}{\gamma} \frac{\lh^2}{\mu^2} \cdot \log\left(\frac{1}{\epsilon}\right)\right)$ iterations, which notably depends only on the conditioning of the outer function $\h.$ The parameter constraints enforce $\gamma \leq \frac{\cub}{\lc}$ and $\cub\leq \frac{\mu}{\lh} \sig\left(\frac{\mu}{\lh}\right)$. In the context of low-rank matrix recovery problems, this simplifies to $\gamma \lesssim \frac{\mu^2}{\lh^2}$ and so the best rate one can get is $O\left( \frac{\lh^4}{\mu^4} \cdot \log\left(\epsilon^{-1}\right)\right).$} {This result establishes that the iterates $z_k$ will be within distance $\varepsilon > 0$ of the minimizer after $O\left( \frac{1}{\gamma} \frac{\lh^2}{\mu^2} \cdot \log\left(\frac{1}{\epsilon}\right)\right)$ iterations, which notably depends only on the conditioning of the function $\h.$ 
Crucially, with our algorithm and parameter configuration, we obtain linear convergence in all regimes---smooth or nonsmooth, overparameterized or exactly parameterized---at a rate governed solely by the conditioning of $h$. Full proofs for every case are provided in the extended version of this paper \cite{diaz2025preconditioned}.} 

\ifbool{showSquare}{The bound on $\cub$ is likely an artifact of our proof. Indeed, in our numerical experiments, we take $\gamma$ and $\cub$ to be constants independent of $\mu/L$ and still observe linear convergence; see Section~\ref{sec:experiments}. Further, as we show in the next result, under strong alignment (Assumption~\ref{ass:strong-alignment})---which only holds for exactly parameterized problems---we can bypass the spurious bound on $\cub$ and prove a faster local rate.}{}
\ifbool{showSquare}{
{\begin{theorem}
		[\textbf{Convergence under strong alignment and nonsmoothness}]\label{thm: onestepimprovement_exact}
		Suppose that Assumptions~\ref{assum: assumptiononc},~\ref{ass:strong-alignment},  and~\ref{assum:nonsmoothpenalty} hold.
		Further assume that $z_0 = \c(x_0)$ satisfies
		$\norm{z_0 - \zs}{} \le \rloc\left(\frac{\mu}{8\lh}\right).$ The following two hold.
		\begin{enumerate}
			\item{(\textbf{Polyak stepsize})}\label{item:nonsmooth-strong-polyak} Suppose we ran Algorithm~\ref{alg:LM} initialized at $x_{0}$ using Configuration~\ref{assum: dampingparameter} with $\gamma \le \min\left\{1, \frac{\clb}{\lc}\right\}$. Further, assume  $\norm{z_0 - z^\star}{} \le \frac{\sig \mu}{8\cub \lh}$. 
			Then, the iterates $z_{k }= \c(x_{k})$ satisfy
			$$
			\norm{z_{k} - \zs}{}^2 \le \left(1 - \frac{\gamma \mu^2}{8 \lh^2}\right)^k \norm{z_0 - \zs}{}^2 \quad \text{for all } k\geq 0.
			$$
			\item{(\textbf{Geometrically decaying stepsize})}  Suppose we ran Algorithm~\ref{alg:LM} initialized at $x_{0}$ using Configuration~\ref{assum:geometric} with
			\begin{align*}
			&
			\lambda \le \frac{\sig}{32}\frac{\mu}{\lh},
			\quad \gamma \le \frac{1}{\lh} \cdot \min \left\{ \frac{M\mu}{64\lh}, %
			\sqrt{\frac{2\lambda M}{\lc}}, \frac{\lambda \mu}{2\lc\lh }\right\}
			\quad  \text{and} \quad 	q \geq \max\left\{1- \frac{\gamma \mu}{4M}, \frac{1}{\sqrt{2}}\right\},
			\end{align*}
			where $M= \rloc\left(\frac{\mu}{8\lh}\right).$ Then, the iterates $z_{k} = \c(x_{k})$ satisfy
			$$\|z_{k} - \zs\|\leq M q^{k} \quad \text{for all }k \geq 0.$$
			
		\end{enumerate}
	\end{theorem}
}
We defer the proof of this result to Appendix~\ref{sec:proof_exact_nonsmooth}. There are two main differences between this result and Theorem~\ref{thm: onestepimprovement} regarding the Polyak stepsize: $(i)$ we replace Assumption~\ref{ass:weak-alignment} with Assumption~\ref{ass:strong-alignment} and $(ii)$ we substitute the bound on $\cub$ with an additional constraint on the initial distance to optimum. By setting $\lambda_{k}$ to ensure $\frac{\clb}{\lc} = \Theta(1)$, we derive a local rate of $O\left( \frac{\lh^2}{\mu^{2}} \cdot \log\left(\epsilon^{-1}\right)\right).$
In turn, this shows that a properly tuned Levenberg-Morrison-Marquardt method matches the rates of the Gauss-Newton method in the absence of overparameterization~\cite[Theorem 3.1]{davis2022linearly}.

\subsection{Guarantees for smooth losses}
Analogous arguments to the ones used for nonsmooth losses can be applied to derive linear convergence for composite losses where the outer function $\h$ is smooth and has quadratic growth.
We state these guarantees here and defer the proof of the next result to Appendix~\ref{sec:proof_over_smooth}.
{\begin{theorem}[\textbf{Convergence under weak alignment and smoothness}]\label{thm: onestepimprovement_smooth}
		Suppose that Assumptions~\ref{assum: assumptiononc},~\ref{ass:weak-alignment}, and~\ref{assum:smoothpenalty} hold.
		Further assume that $z_0 = \c(x_0)$ satisfies
		$\norm{z_0 - \zs}{} \le \rloc\left(\frac{\mus}{16\Lhs}\right).$ The following two hold.
		\begin{enumerate}
			\item{(\textbf{Polyak stepsize})}\label{item:smooth-weak-polyak} Suppose we ran Algorithm~\ref{alg:LM} initialized at $x_{0}$ using Configuration~\ref{assum: dampingparameter} with $\gamma \le \min\left\{1, \frac{\clb }{\lc }\right\}$ and $\cub \le\frac{\mus}{16\Lhs} \sig\left(\frac{\mus}{16\Lhs}\right)$.
			Then, the iterates $z_{k }= \c(x_{k})$ satisfy
			$$
			\norm{z_{k} - \zs}{2}^2 \le \left(1 - \frac{\gamma \mus}{32 \Lhs}\right)^k \norm{z_0 - \zs}{2}^2 \quad \text{for all } k\geq 0.
			$$

			\item{(\textbf{Constant stepsize})}  Suppose we ran Algorithm~\ref{alg:LM} initialized at $x_{0}$ using Configuration~\ref{assum:constant_stepsize} with 
			\begin{align*}
			&
			\lambda \le \frac{M \sig\left(\frac{\mus}{16\Lhs}\right)}{64}\frac{\mus }{\Lhs},
			\quad \gamma \le \frac{1}{\Lhs}  \cdot \min \left\{ \frac{1}{8}, \sqrt{\frac{32\lambda }{ \lc M}}, \frac{\lambda }{2\lc M}\right\}
			\quad 
			\text{and} \quad 	q \geq \max\left\{\sqrt{1-\frac{ \gamma\mus}{2}}, \frac{1}{\sqrt{2}}\right\},
			\end{align*}
			where $M= \rloc\left(\frac{\mus}{16\Lhs}\right).$ Then, the iterates $z_{k} = \c(x_{k})$ satisfy
			$$\|z_{k} - \zs\|\leq M q^{k} \quad \text{for all }k \geq 0.$$
			
		\end{enumerate}
	\end{theorem}
}

Again, the convergence rate depends solely on the conditioning of the outer function $h$. For matrix recovery problems, this results in a convergence rate of $O\left(\frac{\Lhs^{3}}{\mus^{3}} \log(\varepsilon^{-1})\right)$, which exhibits an undesirable cubic dependence on the condition number. %
As in the nonsmooth setting, we can further improve this convergence rate under conditions of strong alignment. The proof of the next result appears in Appendix~\ref{sec:onestepimprovement_smooth_exact_proof}.
{\begin{theorem}[\textbf{Convergence under strong alignment and smoothness}]\label{thm: onestepimprovement_smooth_exact}
		Suppose that Assumptions~\ref{assum: assumptiononc},~\ref{ass:strong-alignment}, and~\ref{assum:smoothpenalty} hold.
		Further assume that $z_0 = \c(x_0)$ satisfies
		$\norm{z_0 - \zs}{} \le \rloc\left(\frac{\mus}{16\Lhs}\right).$ The following two hold.
		\begin{enumerate}
			\item{(\textbf{Polyak stepsize})}\label{item:smooth-strong-polyak} Suppose we ran Algorithm~\ref{alg:LM} initialized at $x_{0}$ using Configuration~\ref{assum: dampingparameter} with $\gamma \le \min\left\{1, \frac{\clb}{\lc}\right\}$. Additionally, suppose that $\norm{z_0 - z^\star}{2} \le \frac{\sig \mus}{16\cub \Lhs}$.
			Then, the iterates $z_{k }= \c(x_{k})$ satisfy
			$$
			\norm{z_{k} - \zs}{2}^2 \le \left(1 - \frac{\gamma \mus}{32 \Lhs}\right)^k \norm{z_0 - \zs}{2}^2 \quad \text{for all } k\geq 0.
			$$	
			\item{(\textbf{Constant stepsize})}  Suppose we ran Algorithm~\ref{alg:LM} initialized at $x_{0}$ using Configuration~\ref{assum:constant_stepsize} with 
			\begin{align*}
			&
			\lambda \le \frac{\sig \mus}{16\Lhs},
			\quad \gamma \le  \frac{1}{\Lhs} \cdot \min \left\{ \frac{1}{8}, \sqrt{\frac{32\lambda }{ \lc M}}, \frac{\lambda }{2 \lc M}\right\}
			\quad \text{and} \quad 	q \geq \max\left\{\sqrt{1- \frac{\gamma \mus}{2}}, \frac{1}{\sqrt{2}}\right\},
			\end{align*}
			where $M= \rloc\left(\frac{\mus}{16\Lhs}\right).$ Then, the iterates $z_{k} = \c(x_{k})$ satisfy
			$$\|z_{k} - \zs\|\leq M q^{k} \quad \text{for all }k \geq 0.$$
			
		\end{enumerate}
	\end{theorem}
}

\subsection{Guarantees under local regularity}
\label{sec:local-guarantees}
In some applications, such as tensor factorization, Assumptions~\ref{assum: assumptiononc} and~\ref{ass:weak-alignment} do not hold, i.e., there is no global Lipchitzness of $\nabla \c$ or alignment. Instead, these two only hold locally. Notably, even under these weaker local conditions, all our previous rates still hold. In this section, we extend our guarantees to such a local regime.
We start with local alternatives of Assumptions~\ref{assum: assumptiononc} and~\ref{ass:weak-alignment}. Recall that $\cX^{\star} = \argmin_{x} \h \circ \c(x)$ is the set of minimizers.

\begin{assumption}[Locally Lipschitz Jacobian]\label{assum:local_lip_jacob}
	The map $\c$ is continuously differentiable, and for a fixed $\xs\in \RR^{d}$, there exists $\jacvarepsilon>0$ and $\lc \ge 0$ such that %
	\begin{align*}
	\norm{\nabla \c(x) - \nabla \c(y)}{\rm{op}} \le \lc\norm{x-y}{2} \quad \text{for all } x,y \in B_\jacvarepsilon (x^{\star}).
	\end{align*}
\end{assumption}}{}

\ifbool{showSquare}{
\begin{assumption}[Local weak alignment]\label{ass:local-weak-alignment}
	For a fixed $\xs \in \cX^\star$ and $\zs=F(\xs)$ there exist functions $\rloc \colon \RR_{+} \rightarrow \RR_{+}, \sig\colon \RR_{+} \rightarrow \RR_{+}$ and a scalar $\alignvarepsilon>0$ such that for all $\rho >0$ we have that if $x \in \B_{\alignvarepsilon}(x^\star)$, and $z = \c(x) \in \B_{\rloc(\rho)}(z^{\star})$ then there is an index $j$ for which 
	\begin{align*}
	\norm{(I- \P^{x}_{j})(z - \zs)}{2}  \le  \rho \norm{z - \zs}{2} \qquad \text{and} \qquad \left(\sigma_{j}^{x}\right)^{2} \ge \sig(\rho) \norm{z - \zs}{2}.
	\end{align*}
\end{assumption} 

Next, we state a local guarantee under these local regularity assumptions. Although we only state it for weakly aligned and nonsmooth problems using the Polyak stepsize, there are similar guarantees for the other scenarios considered in this section. We defer those and the proof of the following result to Appendix~\ref{app:local-regularity}. 
The key idea to establish this result is to show that the iterates $x_{k}$ stay in the region where the previous two assumptions hold, after which the argument follows precisely as it did for the global assumptions. }{}

\ifbool{showSquare}{
\begin{theorem}[\textbf{Convergence under local weak alignment and nonsmoothness}]\label{thm:local-onestepimprovement}
	Suppose Assumptions~\ref{assum:nonsmoothpenalty},~\ref{assum:local_lip_jacob} and~\ref{ass:local-weak-alignment} hold.  Define $ \tilde q :=  \sqrt{1 - \frac{\gamma \mu^2}{8 \lh^2}}$, 
	and let $x_{0}$ and $z_0 =\c(x_0)$ be points satisfying
	$$\|x_{0} - \xs\|_{2} \leq \varepsilon /2
	\quad\text{and}\quad
	\norm{z_0 - \zs}{2} \le \min\left\{\rloc\left(\frac{\mu}{8\lh}\right), \frac{\left(1- \sqrt{\tilde q}\right)^2\varepsilon^2 \clb}{2\gamma^2 }\right\},$$ 
	where $\varepsilon = \min\left\{ \jacvarepsilon, \alignvarepsilon \right\}$. Suppose we ran Algorithm~\ref{alg:LM} initialized at $x_{0}$ using Configuration~\ref{assum: dampingparameter} with $\gamma \le \min\left\{1, \frac{\clb}{\lc}\right\}$ and $\cub \leq \frac{\mu}{8\lh}\sig( \frac{\mu}{8\lh})  $. Then, the iterates $x_{k}$ satisfy
	$$\|x_{k} -  x^{\star}\|_{2} < \varepsilon \quad \text{for all }k \geq 0,$$
	and, moreover, the mapped iterates $z_{k }= \c(x_{k})$ satisfy
	$$
	\|z_{k} - \zs\|^2 \le \left(1 - \frac{\gamma \mu^2}{8 \lh^2}\right)^k \|z_0 - \zs\|^2 \quad \text{for all } k\geq 0.
	$$
	
\end{theorem}}{}

\section{Consequences for statistical recovery problems}\label{sec:examples}
In this section, we instantiate the general convergence guarantees from Section~\ref{sec:guarantees} for concrete recovery problems in signal processing and data science. \ifbool{showSquare}{To this end, we show that our alignment assumptions hold for three families of parameterizations: squared-variable formulations, low-rank matrix factorizations, and CP tensor factorizations. Further, we establish that our restricted conditioning assumptions on the outer convex function are satisfied whenever well-established notions of strong identifiability, e.g., restricted isometry property, hold. Armed with these results, we}{We}  derive local convergence rates \ifbool{showSquare}{for nonnegative least squares, robust matrix sensing, and tensor factorization under standard assumptions from the literature. All proofs are deferred to Appendix~\ref{app:examples}}{for PSD matrix sensing and tensor factorization}.

\ifbool{showSquare}{

\subsection{Squared-variable formulations} \label{sec:square}

Scientists dealing with unmixing problems often wish to minimize a convex function $h\colon \RR^{r}\rightarrow \RR$ over the positive orthant $\RR^r_{+}$. A prominent example of this type of problem is nonnegative least squares \cite{LH95}. These problems arise naturally across several domains, including acoustics, imaging, and genomics \cite{lin2004nonnegative, szlam2010split, li2000parametric}. This type of problem can be reformulated as a composite optimization problem via the squared-variable map $c\colon x \mapsto x\odot x$ (where $\odot$ denotes the component-wise product) \cite{ding2023squared, levin2024effect}.
Although other algorithmic solutions might be preferable for this particular problem, e.g., the projected subgradient method, we cover this example as it provides a clear and simple illustration of our framework.
\paragraph{Regularity of the parameterization.} Throughout we assume $h$ has a unique minimizer $\zs$ over $\RR^{r}_{+}$ 
and let $\xs\in\RR^{r}$ be any vector such that $\zs = \xs \odot \xs.$ 
It is immediate that $\nabla \c(x) = 2 \diag(x)$ and, 
the problem is ill-conditioned when
$\max_{i\in [r]} |\xs_{i}| \gg \min_{i\in[r]} |\xs_{i}|$, and overparameterized whenever $r^\star:= \# \supp(\xs) < r$. The next result establishes regularity for the squared-variable formulation with potential overparameterization; its proof appears in Appendix \ref{proof: hadamard-weak-alignment}. 

\begin{theorem}[Weak alignment of squared-variable map]
\label{thm: weak-alignment-hadamard}
The map $\c\colon \RR^r \rightarrow \RR^r$ given by $x \mapsto x \odot x$ satisfies Assumption~\ref{assum:mapc:smoothness} with %
$\lc = 2$ and Assumption~\ref{ass:weak-alignment} with
   $$\sig(\rho) = \frac{\rho}{\max\{\sqrt{r - \rs},1\}} \quad \text{and} \quad \rloc(\rho) =
   \min\left\{ \min_{\substack{i,j\in[r] \\ \zs_i \neq \zs_j}}\dfrac{|\zs_i - \zs_j|}{2}  , \,\, \min_{ \substack{ i\in [r] \\ \zs_i \neq 0}}\frac{\zs_i
}{1 + \sig(\rho)} , \,\, \min_{\substack{i\in [r] \\ \zs_i \neq 0}}\frac{\zs_i
}{2}  \right\}
$$
for any given $\zs\in \RR^r_+$ with $ r^\star = \#\supp(\zs)$. 
\end{theorem}

\paragraph{Nonnegative least squares.}
We leverage the regularity of the squared-variable formulation to derive guarantees for
 nonnegative least squares
\cite{VBK04, KSD13, MFLS17, DLP+22}. For a matrix $A \in \RR^{m \times r}$ with $m \geq r$ and $b = A z^\star \in \RR^m$, \ifbool{showSquare}{define the smooth and nonsmooth formulations}{we define the problem} %
\begin{equation}
    \label{nnls-nonsmooth}
\min_{x \in \mathbb{R}^r } \frac{1}{2}\left\| A\left( x\odot x \right) - b\right\|_2^2, \quad
\text{ and }
\quad \min_{x \in \mathbb{R}^r } \left\| A\left( x\odot x \right) - b\right\|_2.
\end{equation}
 The following lemma is immediate, and so we omit its proof.
\begin{lemma}
\label{growth-assumption-nonsmooth-nonegative-least-square}
Suppose that $A$ has full rank. Then, the function $z \mapsto  \frac{1}{2} \|Az -  b\|^{2}_{2}$
satisfies Assumption \ref{assum:smoothpenalty} with constants $\mus =\sigma_{\min}(A)^2$ and $\Lhs = \sigma_{\max}(A)^2$. Similarly, the function $z \mapsto  \|Az -  b\|_{2}$
satisfies Assumption \ref{assum:nonsmoothpenalty} with constants $\mu = \sigma_{\min}(A)$ and $\lh=\sigma_{\max}(A)$.
\end{lemma}
Recall that $\kappa(A)$ denotes the condition number of $A$, i.e., $\kappa(A) = \sigma_{\max(A)}/\sigma_{\min}(A).$
Equipped with these results, we are in good shape to derive a local convergence rate.
\begin{corollary}[\textbf{Smooth nonegative least squares}]\label{cor:nnls-smooth}
Suppose Algorithm \ref{alg:LM} is applied to the first nonnegative least squares objective \eqref{nnls-nonsmooth} (squared), initialized at some $x_{0} \in \RR^r$ using Configuration~\ref{assum: dampingparameter} with $\gamma \leq \min\left\{1 , \frac{\clb}{2} \right\}$, $\cub \leq (2^{8} \max\{\sqrt{r-r^\star},1\})^{-1} \kappa^{-4}\left(A \right)$ and
$$
\norm{ x_0 \odot x_0 - \zs}{2} \leq \frac{1}{2}\min \left\{ \min_{\substack{i,j\in[r] \\ \zs_i \neq \zs_j}}\left\{{\zs_i - \zs_j}\right\} , \min_{i\in [r] | \zs_i \neq 0} \zs_i\right\}.
$$
Then, the iterates $x_{k}$ satisfy
$$
\norm{x_k \odot x_k - \zs }{2}^2 \leq \left( 1 - \frac{\gamma}{32} \kappa^{-2}\left(A \right) \right)^k \norm{x_0 \odot x_0 - \zs}{2}^2 \quad \text{for all $k\geq 0$}.
$$
\end{corollary}
\begin{corollary}[\textbf{Nonsmooth nonegative least squares}]\label{cor:nnls-nonsmooth}
Suppose Algorithm \ref{alg:LM} is applied to the second nonnegative least squares objective in \eqref{nnls-nonsmooth} (not squared), initialized at some $x_{0} \in \RR^r$ using Configuration~\ref{assum: dampingparameter} with $\gamma \leq \min\left\{1 , \frac{\clb}{2} \right\}$, $\cub \leq \frac{1}{64 \max\{\sqrt{r-r^\star},1\}} \kappa^{-2}\left(A \right)$ and
$$
\norm{ x_0 \odot x_0 - \zs}{2} \leq \frac{1}{2}\min \left\{ \min_{\substack{i,j\in[r] \\ \zs_i \neq \zs_j}}\left\{\zs_i - \zs_j\right\} , \min_{\substack{i\in [r] \\ \zs_i \neq 0}}\zs_i\right\}.
$$
Then, the iterates $x_{k}$ satisfy
$$
\norm{x_k \odot x_k - \zs }{2}^2 \leq \left( 1 - \frac{\gamma}{8} \kappa^{-2}\left(A \right) \right)^k \norm{x_0 \odot x_0 - \zs}{2}^2 \quad \text{for all $k\geq 0$}.
$$
\end{corollary}

     These corollaries follow directly from Theorem~\ref{thm: onestepimprovement} and Theorem~\ref{thm: onestepimprovement_smooth}, respectively. While similar rates hold for geometrically decaying step sizes and better rates hold in the exactly parameterized case $r^\star = r$, we omit them for brevity. The condition number of $A$ appears in the convergence rate because we incorporated $A$ into the definition of $h$; this rate is standard for gradient descent applied to least squares. Interestingly, the convergence rate we derive for the nonsmooth formulation of the problem is faster than the convergence rate for its smooth counterpart. In particular, Corollary~\ref{cor:nnls-nonsmooth} only allows for $\gamma \asymp  \kappa\left(A \right)^{-2}$ which translates to a rate of $O(\kappa(A)^{4}\log(1/\varepsilon))$, while Corollary~\ref{cor:nnls-smooth} imposes $\gamma \asymp  \kappa\left(A \right)^{-4}$ translating to a rate of $O(\kappa(A)^{6}\log(1/\varepsilon))$. We observe experimentally that our method is slightly faster when applied to the nonsmooth formulation; see Section~\ref{sec:nonnegative-exp}.
 
}{}
\subsection{Matrix recovery problems}\label{sec:matrices}
 
Several modern data science tasks can be formulated as the problem of recovering a rank-$r^\star$ matrix $Z^\star$ from a small set of noisy measurements $b = \cA(Z^\star) + \varepsilon \in \RR^m$ where $\cA$ is a known linear map and $\varepsilon$ models noise. Applications arise in imaging, recommendation systems, control theory, and communications \cite{charisopoulos2021low, chi2019nonconvex, davenport2016overview, wright2022high, diaz2019nonsmooth}.
\ifbool{showSquare}{
Remarkably, even though $Z^\star$ may be a large $d_{1} \times d_{2}$ matrix, the number of measurements $m$ required for recovery is often much lower, typically on the order of $O(r^\star(d_{1}+d_{2}))$.
}{For simplicity, we will focus on the case where $Z^\star$ is a PSD matrix. Remarkably, even though $Z^\star$ may be a large $d \times d$ matrix, the number of measurements $m$ required for recovery is often much lower, typically on the order of $O(d r^\star )$.}
A popular approach to tackle this problem leverages low-rankness by solving \ifbool{showSquare}{}{the composite optimization problem}
\ifbool{showSquare}{ one of two formulations:
\begin{equation}\label{eq:matrix-problems}
    \min_{X \in \mathbb{R}^{d \times r}} \ell\left( \mathcal{A}( {XX^\top} ) - b \right) \qquad \text{or} \qquad \min_{X \in \mathbb{R}^{d_1 \times r}, \, Y \in \mathbb{R}^{d_2 \times r}} \ell\left( \mathcal{A}({XY^\top} ) - b \right),
\end{equation}
}{
\begin{equation}\label{eq:matrix-problems}
    \min_{X \in \mathbb{R}^{d \times r}} \left\| \mathcal{A}( XX^\top ) - b \right\|_1,
\end{equation}}
\ifbool{showSquare}{
depending on whether the matrix is symmetric positive semidefinite $Z^\star \in \SSS^{d}_+$ or asymmetric $Z^\star \in \R^{d_1 \times d_2}$.
Here $r$ is an upper bound on the true rank $r^\star$ and $\ell(\cdot)$ is a measure of discrepancy. %
Common choices for $\ell(\cdot)$ include the $\ell_{2}$ norm squared, which is effective against small unbiased noise \cite{chi2019nonconvex}, and the $\ell_1$ norm, which is robust against gross outliers \cite{charisopoulos2021low}. 
}{
where $r$ is an upper bound on the true rank $r^\star$. Other common choices for losses are the $\ell_2$ norm squared; however, we chose the $\ell_1$ norm since it is robust to gross outliers of the form
\begin{equation}\label{eq:outliers}
b = \begin{cases} \cA(\Zs)_{i} & \text{if }i \in \cI^{c} \\ \eta_{i} & \text{otherwise,} \end{cases}
\end{equation}
where $\mathcal{I} \subseteq [m]$ is a subset of the entries and $\eta_i$ is arbitrary.} Iterative methods for these formulations are appealing since they do not need to project onto the set of low-rank matrices, which involves costly matrix factorizations that are prohibitively costly in large-scale settings. 

\ifbool{showSquare}{
In this section, we develop rates for Algorithm~\ref{alg:LM} applied to composite problems where the parameterization can be either $\Fsym \colon \RR^{d\times r} \rightarrow \SSS^{d}$ or $\Fasym\colon \RR^{d_1 \times r} \times \RR^{d_2 \times r} \rightarrow \RR^{d_1 \times d_2}$ given by
\begin{equation}\label{eq:matrix-param}
    X \mapsto XX^\top \qquad \text{and} \qquad (X, Y) \mapsto XY^\top, \quad\text{respectively.}
\end{equation}
We consider two concrete losses: $h(\cdot) = \frac{1}{2}\|\cA(\cdot) - b\|_2^2$ and $h(\cdot) = \|\cA(\cdot) - b\|_1$. In what follows, we develop theory for linear maps satisfying the standard restricted isometry property (RIP) or a modified version involving the $\ell_1$-norm---both of which hold for linear maps with appropriately normalized iid Gaussian entries. %
We leverage these results to derive guarantees for the $\ell_1$ loss that hold even when gross outliers corrupt a constant fraction of the measurements.

\subsubsection{Regularity of the parameterization}

As a first step, we show that both parameterizations in \eqref{eq:matrix-param} are smooth (Assumption~\ref{assum: assumptiononc}) and establish weak alignment (Assumption~\ref{ass:weak-alignment}) for the PSD factorization and its local analogue  (Assumption~\ref{ass:local-weak-alignment}) for the asymmetric factorization. 
The proofs of the following two results are rather technical and require carefully characterizing the spectrum of the Jacobians of these parametrizations; we defer these arguments to Appendices~\ref{section:proof-thm-regularity-matrix-sym} and~\ref{section:proof-thm-regularity-matrix-asym}, respectively.

\begin{theorem}[\textbf{Weak alignment of PSD factorization}]\label{thm:regularity-matrix-sym}
The map $\Fsym \colon \RR^{d\times r} \rightarrow \SSS^{d}$ given by $X \mapsto XX^\top$ satisfies Assumption~\ref{assum: assumptiononc} with $\lc = 2$ and Assumption~\ref{ass:weak-alignment} with
$$
\sig(\rho) 
\;=\; 
\frac{4\rho}{\,\sqrt{2(\,r - r^\star + 1)\,}},
\quad \text{and} \quad
\rloc(\rho) 
\;=\; 
\min\left\{
  \frac{\rho}{\sqrt{2
    }},
  \frac{1}{1 + \sig(\rho)}
, \frac{1}{3}\right\}
\lambda_{\rs}\left(Z^\star\right).
$$
for any $Z^\star \in \SSS^{d}_{+}$ with $\rank(Z^\star) = \rs.$
\end{theorem}

\begin{theorem}[\textbf{Weak alignment of asymmetric factorization}]  \label{thm:regularity-matrix-asym}  The map $\Fasym \colon \RR^{d_1\times r} \times \RR^{d_2 \times r} \rightarrow \RR^{d_{1} \times d_{2}}$ given by $(X, Y) \mapsto XY^\top$ satisfies Assumption~\ref{assum:local_lip_jacob} with $\lc = \sqrt{2}$ and $\jacvarepsilon = \infty$, and Assumption \ref{ass:local-weak-alignment} with $$\alignvarepsilon = \frac{1}{16\sqrt{2}}\frac{\min\left\{ \sigma^2_{\rs}\left(\Xs\right), \sigma_{\rs}^2\left(\Ys\right) \right\}}{\max\left\{ \sigma_1\left(\Xs\right), \sigma_1\left(\Ys\right) \right\}},\qquad
\sig(\rho) 
= 
\frac{\rho}{10\sqrt{2}\left(r - r^\star + 1\right)^2}\frac{\min\left\{ \sigma^2_{\rs}\left(\Xs\right), \sigma_{\rs}^2\left(\Ys\right) \right\}}{ \sigma^2_{\rs}\left(\Xs\right) + \sigma_{\rs}^2\left(\Ys\right)}, $$
$$ \text{and} \quad
\rloc(\rho) 
= \min\left\{
  \frac{\rho}{4}, \frac{1}{4\sig(\rho)}
\right\} \min\left\{ \sigma^2_{\rs}\left(\Xs\right), \sigma_{\rs}^2\left(\Ys\right) \right\}
$$
for any factorization $\Zs=\Xs\Ys^\top$ satisfying $\rankk{\Xs}=\rankk{\Ys} = \rs$ and the right singular vectors of the two factors match $V^\Xs= V^\Ys$.
\end{theorem}}{}

\ifbool{showSquare}{
These two regularity guarantees combined with the immediate fact that $h(Z) = \|Z - Z^{\star}\|_{F}$ satisfies Assumption~\ref{assum:nonsmoothpenalty}, can be used to derive fast convergence guarantees for Algorithm~\ref{alg:LM} applied to matrix factorization problems as the one we covered in the introduction; see Figure~\ref{fig:intro-example}. We observe that for the asymmetric setting, the alignment is only local, and we require the right singular vectors of the two factors to be the same. Although this might sound restrictive, spectral initialization procedures guarantee closedness to such balanced factors~\cite{tu2016low,chi2019nonconvex}.
We focus instead on more general matrix sensing problems where the input $\cA$ is not simply the identity.}{}

\ifbool{showSquare}{
\subsubsection{Noiseless matrix sensing}
In this section, we consider noiseless measurements $b = \cA(Z^{\star})$ and the smooth objective $\h(\cdot) = \frac{1}{2}\|\cA(\cdot) - b\|^{2}_{2}$. Although we could instead work with the nonsquared loss, we restrict attention to the smooth objective since $(i)$ the next section will explore an arguably more interesting nonsmooth loss, and $(ii)$ most existing theory pertains to this setting. We will state definitions and some results only for the asymmetric case, since the extension to the positive semidefinite case follows immediately.}{}

\ifbool{showSquare}{Our guarantees apply to maps satisfying the restricted isometry property (RIP)---a popular notion of strong identifiability that
underpins most existing guarantees for linear inverse problems. %
  A linear map $\mathcal{A}\colon \RR^{d_1 \times d_2} \rightarrow \RR^{m} $ satisfies RIP if there exists $\delta \in (0,1)$ such that
  \begin{equation}\label{def:RIP}
(1-\delta) \norm{Z}{F}^2 \leq \norm{\mathcal{A}(Z)}{2}^2 \leq (1+
\delta) \norm{Z}{F}^2%
\end{equation}
for all matrices $Z$ of rank at most $r.$ In short, this property ensures that distances between low-rank matrices are approximately preserved after mapping by $\cA$. While the identity map trivially satisfies this property, more interesting random maps with low-dimensional images also exhibit this behavior.
We say that $\cA$ has i.i.d. entries if $\cA(Z)_{i} = \dotp{A_{i}, Z}$ where the entries of $A_{i} \in \RR^{d_{1} \times d_{2}}$ are drawn i.i.d. and the matrices $A_{i}$ are independent of each other.}{}

\ifbool{showSquare}{
\begin{lemma}[Theorem 2.3 in \cite{candes2011tight}] Fix $r \leq \min(d_1, d_2)$ and $\delta \in (0,1)$. Assume that  $\mathcal{A}$ has i.i.d. entries with distribution $N(0,1/m)$. There exist universal constants $c_1, c_2, c_{3} > 0$ such that if $m \geq c_1r(d_1 + d_2)$, then $\mathcal{A}$ satisfies \eqref{def:RIP} for all matrices $Z$ of rank at most $r$ with probability at least $1 - c_2\exp(-c_3m)$.
\end{lemma}

In turn, RIP suffices for good conditioning. The proof of the next lemma is in Appendix~\ref{proof:lem-RIP-implies-assumptions}.

\begin{lemma}\label{lem: smooth-sensing-satisfies-assum}
  Suppose the map $\mathcal{A}$ satisfies~\eqref{def:RIP} for all matrices of rank at most $6r$, and that $b = \mathcal{A}(\Zs)$, with $\Zs\in \RR^{d_1 \times d_2}$ a rank $r$ matrix. Then, the function $h(\cdot) = \frac{1}{2} \norm{\mathcal{A}(\cdot) - b }{2}^2$ satisfies Assumption \ref{assum:smoothpenalty} with $\mus =(1-\delta)$ and $\Lhs = \frac{(1+\delta)^2}{(1-\delta)}$.
\end{lemma}}{}

\ifbool{showSquare}{
Therefore, applying Theorem~\ref{thm:regularity-matrix-sym} (resp. Theorem~\ref{thm:regularity-matrix-asym}) in tandem with the preceding lemma shows that the assumptions of the general convergence guarantee Theorem~\ref {thm: onestepimprovement_smooth} (resp. Theorem~\ref{thm:local-onestepimprovement-weak-smooth} in Appendix~\ref{app:local-regularity}) are satisfied in the symmetric (resp. asymmetric) case.\footnote{To derive the corollary in the asymmetric case we used that $1-(1-x)^\alpha \geq \alpha x$ for all $x \in [0,1]$ and $\alpha \in (0,1)$} %
\begin{corollary}[\textbf{Convergence for PSD matrix sensing}]
  Suppose that the measuring map $\mathcal{A}\colon \SSS^{d} \rightarrow \RR^{m}$ satisfies \eqref{def:RIP} for all matrices $Z$ of rank at most $6r$ and $b = \cA(\Zs)$. Algorithm \ref{alg:LM} is applied to the first objective in \eqref{eq:matrix-problems} with $\ell(z) = \|z\|^2_2$, initialized at $X_{0}$ using Configuration \ref{assum: dampingparameter} with $\gamma \leq \min\{1,\tfrac{\clb}{2}\} $, $\cub \leq  \frac{1}{64\sqrt{2(r-r^\star + 1)}}\frac{(1-\delta)^4}{(1+\delta)^4}$, and
  \begin{equation*} \norm{ X_0 X_0^\top - Z^\star}{F} \leq \frac{1}{16 \sqrt{2}}  \frac{(1-\delta)^2}{(1+\delta)^2} \lambda_{\rs} \left(Z^\star\right).
  \end{equation*}
  Then, the iterates satisfy
$$
\norm{X_k X_k^\top  - Z^\star }{F}^2 \leq \left( 1 - \frac{\gamma}{32}\frac{(1-\delta)^{2}}{(1+\delta)^2}\right)^k \norm{X_0 X_0^\top  - Z^\star }{F}^2 \quad \text{for all $k\geq 0$}.
$$
\end{corollary}}{}

\ifbool{showSquare}{
\begin{corollary}[\textbf{Convergence for asymmetric matrix sensing}] 
  Suppose that the measuring map $\mathcal{A}\colon \RR^{d_{1}} \times \RR^{d_{2}} \rightarrow \RR^{m}$ satisfies \eqref{def:RIP} for all matrices $Z$ of rank at most $6r$ and $b = \mathcal{A}(\Zs)$. Let $\Xs\Ys^\top=\Zs$ be a factorization satisfying $\rankk{\Xs} = \rankk{\Ys} = \rs$ and the right singular vectors of the two factors match $V^\Xs= V^\Ys$. Assume Algorithm \ref{alg:LM} is applied to the second objective in \eqref{eq:matrix-problems} with $\ell(z) = \|z\|_2^2$, initialized at $\left(X_{0}, Y_{0}\right)$ using Configuration \ref{assum: dampingparameter} with $\gamma \leq \min\{1,\tfrac{\clb}{\sqrt{2}}\}$, $\cub \leq \frac{1}{2^9 \cdot 5\sqrt{2} \left(r - r^\star + 1\right)^2} \frac{(1-\delta)^4}{(1+\delta)^4} \frac{\min\left\{ \sigma^2_{\rs}\left(\Xs\right), \sigma_{\rs}^2\left(\Ys\right) \right\}}{ \sigma^2_{\rs}\left(\Xs\right) + \sigma_{\rs}^2\left(\Ys\right)} $, 
    \begin{align*}
    \norm{(X_0,Y_0) - (\Xs,\Ys)}{F} &\leq \epsassymmatrixhalf, \qquad \text{and}\\
    \norm{ X_0 Y_0^\top - Z^\star}{F} &\leq \frac{1}{2^{23}}\frac{(1-\delta)^4}{(1+\delta)^4} \frac{\min\left\{ \sigma^4_{\rs}\left(\Xs\right), \sigma_{\rs}^4\left(\Ys\right) \right\}}{\min\left\{ \sigma^2_{1}\left(\Xs\right), \sigma_{1}^2\left(\Ys\right) \right\}} \clb .
    \end{align*}
 Then, the iterates satisfy
$$
\norm{X_k Y_k^\top  - Z^\star }{F}^2 \leq \left( 1 - \frac{\gamma}{32}\frac{(1-\delta)^2}{(1+\delta)^2} \right)^k \norm{X_0 Y_0^\top  - Z^\star }{F}^2 \quad \text{for all $k\geq 0$}.
$$
\end{corollary}

The only dependency on the conditioning of $\Zs$ appears in the size of the neighborhood where the algorithm exhibits linear convergence. The general guarantees under strong alignment can be used to derive faster rates in the exactly parameterized case; we omit such results for brevity.}{}

\ifbool{showSquare}{
\subsubsection{Robust matrix sensing}
In this section, we will study matrix sensing problems with gross outliers. That is, we consider corrupted measurements of the form 
\begin{equation}\label{eq:outliers}
b = \begin{cases} \cA(\Zs)_{i} & \text{if }i \in \cI^{c} \\ \eta_{i} & \text{otherwise,} \end{cases}
\end{equation}
where $\mathcal{I} \subseteq [m]$ is a subset of the entries and $\eta_i$ is arbitrary. Inspired by \cite{charisopoulos2021low}, we consider \eqref{eq:matrix-problems} with $\ell(z) = \|z\|_1$. Before stating our results for this loss, we take a small detour to show that for nonsmooth matrix problems, the rather complicated Assumption~\ref{assum:nonsmoothpenalty} is implied by a more standard form of restricted conditioning. This matches the assumptions for \texttt{ScaledGD} \cite{tong2021scaledsubgradient}. We defer the proof of the next lemma to Appendix~\ref{app:proof-restricted-regularity-implies-ugly-regularity}. 
\begin{lemma}\label{lem:restricted-regularity-implies-ugly-regularity}
    Let $h\colon\RR^{d_1\times d_2} \to \RR$ be a convex function and $Z^\star \in \RR^{d_1 \times d_2}$ satisfying the following two conditions.
    \begin{enumerate}
        \item (\textbf{Restricted sharpness}) For any $Z \in \RR^{d_1 \times d_2}$ with $\rank Z \leq r$ we have 
        \begin{equation}\label{eq:restricted-sharpness}
        \mu \|Z-\Zs\|_F \leq \left| h(Z) - h(\Zs)\right|.
        \end{equation}
        \item (\textbf{Restricted Lipschitzness}) For any pair $Z, \widetilde Z \in \RR^{d_1 \times d_2}$ with $\rank (Z-\widetilde Z) \leq 2r$ we have 
         \begin{equation}\label{eq:restricted-lipschitzness}
         \left| h(Z) - h(\widetilde Z)\right|\leq L \|Z-\widetilde Z\|_F.
        \end{equation}
    \end{enumerate}
    Then, $h$ satisfies Assumption~\ref{assum:nonsmoothpenalty} with $F = \Fasym$.
\end{lemma}
A completely analogous result holds for the symmetric parameterization; we omit it to avoid repetition. Next, we establish these notions of restricted Lipschitzness and sharpness. 
}{}
\ifbool{showSquare}{Just as in the noiseless case, we will enforce a restricted isometry property, but in this case, a mixed version with the $\ell_1$ norm.}{Our results here hold for measurement maps $\cA$ that do not distort the distances between low-rank matrices.}
In particular, we say that a linear map $\ifbool{showSquare}{\mathcal{A}\colon \RR^{d_1 \times d_2}}{\mathcal{A}\colon \RR^{d \times d}} \to \RR^{m}$ satisfies $\ell_{1}/\ell_{2}$-RIP if there exist constants $\dl, \du > 0$ such that
\begin{equation}\label{eq:loRIP}
  \dl \norm{Z}{F} \leq \norm{\mathcal{A}(Z)}{1} \le \du \norm{Z}{F} %
  \end{equation}
for all matrices $Z$ of rank at most $r.$ In turn, $\ell_{1}/\ell_{2}$-RIP does not suffice to handle outliers.
Instead, we require a slightly more restrictive condition. The map $\mathcal{A}$ satisfies the $\mathcal{I}$-outlier bound if there exist a constant $\constI > 0$ such that
\begin{equation}\label{eq:outlier-bound}
      \constI \norm{Z}{F} \leq \left( \norm{\mathcal{A}_{\mathcal{I}^c}(Z)}{1}  - \norm{\mathcal{A}_{\mathcal{I}}(Z)}{1} \right) %
\end{equation}
for matrices $Z$ of rank at most $r,$ where $\mathcal{A}_{\mathcal{I}}(Z)$ and $\mathcal{A}_{\cI^{c}}(Z)$ are the subvectors of $\mathcal{A}(Z)$ indexed by $\cI$ and $\cI^{c}$. In turn, random Gaussian mappings also satisfy these properties.
\ifbool{showSquare}{\begin{lemma}[Theorem 6.4 in \cite{charisopoulos2021low}]
  Fix $r \leq \min(d_1,d_2)$ and $\mathcal{I} \subseteq [m]$ with $\#\cI < m/2$. Define $\pfail = \#\cI/m$ and suppose that $\cA\colon \RR^{d_{1} \times d_{2}} \rightarrow \RR^{m}$ has i.i.d. Gaussian entries with distribution $N(0, 1/m^{2})$. There exist universal constants $c_{1}, c_{2}, c_{3} > 0$ such that if $m \geq \frac{c_{1}}{(1-2\pfail)^{2}}\ln\left(c_{2}+ \tfrac{c_{2}}{(1-2\pfail)^{2}}\right) r(d_{1} + d_{2} + 1)$, then $\cA$ satisfies \eqref{eq:loRIP} and \eqref{eq:outlier-bound} for matrices $Z$ of rank at most $r$ with probability at least $1 - 4 \exp(-c_{3}(1-2\pfail) m).$  %
\end{lemma} Several other random mappings satisfy \eqref{eq:loRIP} and \eqref{eq:outlier-bound}, including those used for phase retrieval and blind deconvolution \cite[Theorem 6.4]{charisopoulos2021low}.}{It is well-known that many random maps $\mathcal{A}$ satisfy these assumptions with high probability when $m \gtrsim (1-2p_{\textrm{fail}})^{-2}dr$ where $p_{\textrm{fail}}$ is the fraction of outliers in~\eqref{eq:outliers} \cite[Theorem 6.4]{charisopoulos2021low}. }
\ifbool{showSquare}{The following lemma shows that whenever the measurement map satisfies RIP and the outlier bound, the loss function $h(\cdot) = \norm{\mathcal{A}(\cdot) - b }{1}$ satisfies the restricted Lipschitz continuity and sharpness. The proof of this lemma appeared in a slightly different form in \cite{charisopoulos2021low}; we include it here for completeness. 
\begin{lemma}\label{matrix: regularity_assum} Suppose that $\mathcal{A}$ satisfies \eqref{eq:loRIP} and \eqref{eq:outlier-bound} for all matrices of rank at most $2r$ and that $b$ is taken as in \eqref{eq:outliers}.
Take the constants $\mu = \constI$ and $\lh=\du$. Then, the function $h(\cdot) = \norm{\mathcal{A}(\cdot) - b }{1}$ satisfies~\eqref{eq:restricted-sharpness} for all $Z$ with $rank Z \leq r$ and~\eqref{eq:restricted-lipschitzness} for all $Z, \widetilde Z$ with $\rank(Z - \widetilde Z)\leq 2r$.
\end{lemma}
\begin{proof}
 We start by establishing restricted sharpness. Label $\Delta =\left( \mathcal{A}(\Zs) - b \right)$, and let $Z$ be an arbitrary matrix with rank at most $r$. Applying the reverse triangle inequality yields
\begin{align*}
    \abs{h(Z) - h(\Zs) } &= \abs{ \norm{\mathcal{A}(Z - \Zs) + \Delta }{1}  -  \norm{\Delta }{1} }  \\
    & = \left(   \norm{\mathcal{A}_{\mathcal{I}^c}(Z - \Zs)}{1} + \sum_{i \in \mathcal{I}} \left( \abs{ \left[ \mathcal{A}(Z - \Zs) \right]_i +  \left[ \Delta \right]_i  } - \abs{\left[ \Delta \right]_i } \right) \right)\\
    &\geq \left( \norm{\mathcal{A}_{\mathcal{I}^c}(Z - \Zs)}{1} - \norm{\mathcal{A}_{\mathcal{I}}(Z - \Zs)}{1} \right) \\ 
    & \geq \constI \norm{Z - \Zs}{F},
\end{align*}
where the second inequality follows from \eqref{eq:outlier-bound}.    

 Next, we demonstrate that the function $h$ satisfies restricted Lipschitz continuity. Let $Z$ and $\Zprime$ be two matrices such that $\rankk{Z-\Zprime} \leq 2r$. 
  Once more, the reverse triangle inequality yields
\begin{align*}
    \abs{h(Z) - h(\Zprime) } &= \abs{ \norm{\mathcal{A}(Z) - b }{1} - \norm{\mathcal{A}(\Zprime) - b }{1} }  \\
    &\leq  \norm{\mathcal{A}(Z - \Zprime) }{1} \\
                        &\leq \du \norm{Z -\Zprime}{F},
\end{align*}
where the second inequality uses \eqref{eq:loRIP}. This concludes the proof.
\end{proof}}{}
\ifbool{showSquare}{These results allow us to invoke Theorems~\ref{thm: onestepimprovement} and~\ref{thm:local-onestepimprovement} to derive the following two corollaries.}{We leverage Theorem~\ref{thm: onestepimprovement} to derive the following corollary.}
\begin{corollary}[\textbf{Convergence for robust PSD matrix sensing}]
Suppose that the measurement map $\mathcal{A}\colon \SSS^{d} \rightarrow \RR^{m}$ satisfies~\eqref{eq:loRIP} and~\eqref{eq:outlier-bound} for all matrices $Z$ of rank at most $2r$, and that the vector $b\in \RR^m$ is taken as in~\eqref{eq:outliers}. Assume Algorithm \ref{alg:LM} is applied to the first objective in \eqref{eq:matrix-problems} with $\ell(z) = \|z\|_1$, initialized at $X_{0}$ using Configuration \ref{assum: dampingparameter} with $\gamma \leq \min\left\{1,\tfrac{\clb}{2}\right\}$, $\cub \leq  \frac{1}{16\sqrt{2(r-r^\star + 1)}}\frac{\omega_0^2}{\omega_2^2}$, and
  \begin{equation*} \norm{ X_0 X_0^\top - Z^\star}{F} \leq  \frac{1}{8\sqrt{2}}\frac{\omega_0}{\omega_2}\lambda_{\rs} \left(Z^\star\right).
  \end{equation*}
  Then, the iterates must satisfy
$$
\norm{X_k X_k^\top  - Z^\star }{F}^2 \leq \left( 1 - \frac{\gamma}{8}\frac{\omega_0^{2}}{\omega_2^2}\right)^k \norm{X_0 X_0^\top  - Z^\star }{F}^2 \quad \text{for all $k\geq 0$}.
$$
\end{corollary}
\ifbool{showSquare}{\begin{corollary}[\textbf{Convergence for robust asymmetric matrix sensing}]
  Suppose that $\mathcal{A}\colon \RR^{d_{1}} \times \RR^{d_{2}} \rightarrow \RR^{m}$ satisfies~\eqref{eq:loRIP} and~\eqref{eq:outlier-bound} for all matrices $Z$ of rank at most $2r$ and that the vector $b\in \RR^m$ is taken as in~\eqref{eq:outliers}. Let $\Xs\Ys^\top=\Zs$ be a factorization satisfying $\rankk{\Xs} = \rankk{\Ys} = \rs$ and the right singular vectors of the two factors match $V^\Xs= V^\Ys$. Assume Algorithm \ref{alg:LM}  is applied to the second objective in \eqref{eq:matrix-problems} with $\ell(z) = \|z\|_1$, initialized at $\left(X_{0}, Y_{0}\right)$ using Configuration \ref{assum: dampingparameter} with $\gamma \leq \min\left\{1, \frac{\clb}{\sqrt{2}} \right\}$, $\cub \leq \frac{1}{2^7\cdot 5\sqrt{2}\left(r - r^\star + 1\right)^2} \frac{\omega_0^2}{\omega_2^2} \frac{\min\left\{ \sigma^2_{\rs}\left(\Xs\right), \sigma_{\rs}^2\left(\Ys\right) \right\}}{ \sigma^2_{\rs}\left(\Xs\right) + \sigma_{\rs}^2\left(\Ys\right)}$,
  \begin{align*}
  \norm{(X_0,Y_0) - (\Xs,\Ys)}{F} &\leq \epsassymmatrixhalf, \qquad \text{and} \\
    \norm{ X_0 Y_0^\top - Z^\star}{F} &\leq \frac{1}{2^{19}}\frac{\omega_0^4}{\omega_2^4} \frac{\min\left\{ \sigma^4_{\rs}\left(\Xs\right), \sigma_{\rs}^4\left(\Ys\right) \right\}}{\min\left\{ \sigma^2_{1}\left(\Xs\right), \sigma_{1}^2\left(\Ys\right) \right\}} \clb.
  \end{align*}
Then, the iterates satisfy
$$
\norm{X_k Y_k^\top  - Z^\star }{F}^2 \leq \left( 1 - \frac{\gamma}{8}\frac{\omega_0^{2}}{\omega_2^2} \right)^k \norm{X_0 Y_0^\top  - Z^\star }{F}^2 \quad \text{for all $k\geq 0$}.
$$
\end{corollary}}{}

\subsection{Tensor factorization}\label{sec: tensor}

Tensors are generalizations of matrices that store information in $n$ modes as opposed to only two. They have numerous applications in recommender systems, biomedical imaging, quantum many-body simulations, and numerical linear algebra \cite{kolda2009tensor, landsberg2011tensors, Sidiropoulos00CPArray, cpimage, Morup06EEG, Miranda24CPBayesian, Acar11Netflix}. A major challenge in large-scale tensor analysis is the growth in storage and computation with increasing modes. To address this issue, practitioners typically employ low-rank tensor decompositions. Unlike the matrix SVD, tensor factorization admits no single canonical form; instead, a variety of models---such as canonical polyadic (CP), Tucker, and tensor train decompositions---are used, each with its properties and algorithmic trade-offs. 

In this section, we focus on finding a CP tensor factorization. Although we will work only with third-order tensors, many results here likely extend to arbitrary tensors. We start by introducing some notation. Intuitively, a third-order tensor $T$ can be viewed as a three-dimensional array of scalars. 
Given vectors $w \in \RR^{d_1}, x \in \RR^{d_2}$ and $y \in \RR^{d_3}$ we use $w \otimes x \otimes y$ to denote a tensor with components given by $(w \otimes x \otimes y)_{ijk} = w_{i} x_j y_k.$ A general tensor $T$ has a CP decomposition of rank $r$ if it can be written as $T = \sum_{i= 1}^r w^{(i)} \otimes x^{(i)} \otimes y^{(i)};$ further the decomposition is symmetric if $w^{(i)} = x^{(i)} = y^{(i)}$ for all $i$. The \emph{CP-rank} of $T$ is the minimum $r$ for which a CP decomposition exists; the \emph{symmetric CP-rank} is defined analogously. We refer the interested reader to~\cite{kolda2009tensor} for additional details. 
Our goal is, then, to factorize a three-dimensional tensor $T^\star$ with CP rank $r^\star$. To do so, we aim to fit the entries to one of the two explicit factorizations
\begin{equation}\label{eq:tensor-problems}
\min_{X\in \RR^{d\times r}} \left\|\sum_{j=1}^r X_j \otimes X_j \otimes X_j - \Ts\right\|_F \quad \text{or} \quad  \min_{\substack{
    W\in\RR^{d_1\times r},\;
    X\in\RR^{d_2\times r},\\
    Y\in\RR^{d_3\times r}
}} \left\|\sum_{j=1}^r W_j \otimes X_j \otimes Y_j - \Ts\right\|_F,
\end{equation}
depending on whether the tensor $T^\star$ is symmetric or not. Here, $X_{j}$ denotes the $j$th column of $X$ and the Frobenius norm is equal to the $\ell_2$ norm of the vectorized tensor. These are instances of composite optimization with  $h(T) = \| T - T^\star\|_F$ and parameterizations
\begin{equation}
    \label{eq:tensor-facotrizations} 
    \Fsym(X) := \sum_{j=1}^r X_j \otimes X_j \otimes X_j \qquad \text{and} \qquad \Fasym(W,X,Y):= \sum_{j=1}^r W_j \otimes X_j \otimes Y_j.
\end{equation}
Throughout, we assume the tensor of interest has CP-rank $\rs$.

\paragraph{Regularity of the parameterization.} We show that the symmetric and asymmetric factorization maps satisfy local strong alignment (Assumption~\ref{ass: local-strong-alignment}).
Recall that to streamline the exposition, we present this assumption and its implications (Theorem~\ref{thm:local-onestepimprovement-strong-nonsmooth}) only in Appendix~\ref{app:additional-local-guarantees}.
We defer the proof of these theorems to Appendices~\ref{proof-tensor-symmetric-reg-conditions} and \ref{proof-tensor-asymmetric-reg-conditions}, respectively.
\ifbool{showSquare}{
\begin{theorem}[\textbf{Strong alignment of the symmetric CP map}]
\label{thm-tensor-symmetric-reg-conditions}
   Let $\Xs \in \RR^{d\times r^\star}$ be a full-rank matrix and set $T^\star = \Fsym(\Xs)$. Then, the map $\Fsym$ with $r = \rs$ satisfies Assumption \ref{assum:local_lip_jacob} at $\Xs$ with $\jacvarepsilon=\|\Xs\|_F$ and $\lc = 12\norm{\Xs}{F}$, and Assumption~\ref{ass: local-strong-alignment} at $\Xs$ with
     \begin{align*}\alignvarepsilon = \epssymtensor, \quad \rloc(\rho) = \frac{\rho}{C}, \quad \text{ and } \quad
    \sig = \frac{1}{2}\sigma_{dr}\left( \nabla \Fsym(\Xs) \right)\end{align*}
  for some constants $R, C>0$ that depend only on $\Xs$.
\end{theorem} 

\begin{theorem}[\textbf{Strong alignment of the asymmetric CP map}]
\label{thm-tensor-asymmetric-reg-conditions}
   Let $(\Ws,\Xs,\Zs) \in \RR^{d_1\times \rs} \times \RR^{d_2\times \rs} \times \RR^{d_3\times \rs}$ be full-rank matrices. Then, the map $\Fasym$ with $r = \rs$ satisfies Assumption \ref{assum:local_lip_jacob} at $(\Ws,\Xs, \Ys)$ with $\jacvarepsilon= \norm{({\Ws},{\Xs},{\Ys})}{F}$ and $\lc = 4\sqrt{3}\norm{({\Ws},{\Xs},{\Ys})}{F}$, and Assumption~\ref{ass: local-strong-alignment} at $(\Ws,\Xs, \Ys)$  with
    \begin{align*}&\alignvarepsilon = \epsAsymtensor, \qquad \rloc(\rho) = \frac{\rho}{C},\\&  \qquad \qquad \qquad \qquad\text{ and }\qquad
    \sig = \frac{1}{2}\sigma_{(d_1+d_2+d_3-2)r}\left( \nabla \Fasym(\Ws,\Xs,\Ys) \right)\end{align*}
  for some constants $R, C>0$ that depend only on $(\Ws,\Xs,\Ys)$.
\end{theorem} }{}

We observe that, unlike our results for matrix factorization, here, we only handle ill-conditioning and fail to capture the overparameterized settings. Nonetheless, our numerical experiments (Section~\ref{sec:experiments}) suggest that Algorithm~\ref{alg:LM} converges linearly even for overparameterized problems.   
\ifbool{showSquare}{ 
\paragraph{Convergence rates.} The outer function for tensor factorization $h(T) = \norm{T - \Ts}{F}$ is trivially well-conditioned, in particular, it satisfies Assumption~\ref{assum:smoothpenalty} with $\mu= \lh=1$. Thus, applying Theorem~\ref{thm:local-onestepimprovement-strong-nonsmooth} yields the following two corollaries.}{}

\begin{corollary}[\textbf{Convergence rate for symmetric CP tensor factorization}]
    Let $T^\star \in \RR^{d} \otimes \RR^{d}\otimes \RR^{d}$ be a symmetric tensor with symmetric CP rank $r^\star$ and let $\Xs \in \RR^{d\times \rs}$ be such that $T^{\star} = \Fsym(\Xs).$ Consider the first problem \eqref{eq:tensor-problems} with $r = r^\star$ and suppose that we ran Algorithm~\ref{alg:LM} initialized at $X_0$ using Configuration~\ref{assum: dampingparameter} with $\gamma \leq \min \left\{1, \frac{\clb}{12\norm{\Xs}{F}} \right\}$ and
    \begin{align*}
            \norm{X_0 - \Xs}{F} &\leq\frac{1}{2} \min\left\{R, \frac{\sigma_{dr}( \nabla \Fsym(\Xs))}{24\norm{\Xs}{F}},\norm{\Xs}{F}  \right\},\\ 
        \norm{\Fsym(X_0) - \Ts}{F} &\leq \min  \left\{ \frac{1}{8 C}, \frac{\sigma_{dr}\left( \nabla \Fsym(\Xs) \right)}{16\cub},  \frac{1}{2^{10}} {\clb\min\left\{R, \frac{\sigma_{dr}\left( \nabla \Fsym(\Xs)  \right)}{24\norm{\Xs}{F}},\norm{\Xs}{F}  \right\}^2} \right\},
    \end{align*}
    where $C, R > 0$ are constants depending only on $X^\star$.
    Then, the iterates satisfy
    $$
    \norm{\Fsym(X_k) - \Ts}{F}^2 \leq \left( 1 - \frac{\gamma}{8} \right)^k\norm{\Fsym(X_0) - \Ts}{F}^2 \quad \text{for all $k\geq0$.}
    $$
\end{corollary}
\ifbool{showSquare}{
\begin{corollary}[\textbf{Convergence for asymmetric CP tensor factorization}]
  Let $T^{\star} \in \RR^{d_{1}} \otimes \RR^{d_{2}}\otimes \RR^{d_{3}}$ be a tensor with CP rank $\rs$ and let $(\Ws, \Xs,\Ys) \in \RR^{d_{1}\times \rs} \times \RR^{d_{2}\times \rs} \times \RR^{d_{3}\times \rs} $ be such that $T^{\star} = \Fasym(\Ws, \Xs, \Ys)$.
    Consider the second problem in \eqref{eq:tensor-problems} with $r = \rs$ and suppose that we ran Algorithm~\ref{alg:LM} initialized at $W_0,X_0,Y_0$ using Configuration~\ref{assum: dampingparameter} with $\gamma \leq \min \left\{1, \frac{\clb}{4\sqrt{3} \norm{({\Ws},{\Xs},{\Ys}) }{F}} \right\}$, $$\norm{({W_0},{X_0},{Y_0}) - ({\Ws},{\Xs},{\Ys})}{F} \leq \epsAsymtensorhalf,$$ and 
    \begin{align*}
    &\norm{\Fasym(W_0,X_0,Z_0) - \Ts}{F} \\ 
    &\leq\min \Bigg\{ \frac{1}{8 C}, \frac{\sigma_{(d_1+d_2+d_3-2)r}\left( \nabla \Fasym(\Ws,\Xs,\Ys) \right)}{16\cub}, \\ &\hspace{1.5cm}\frac{1}{2^{10}} {\clb } \min\bigg\{R, \frac{\sigma_{(d_1+d_2+d_3-2)r}\left( \nabla \Fasym(\Ws,\Xs,\Zs)  \right)}{8\sqrt{3}\norm{({\Ws},{\Xs},{\Ys})}{F}}, \norm{({\Ws},{\Xs},{\Ys})}{F}  \bigg\}^2 \Bigg\},
    \end{align*}
    where $C, R > 0$ are constants depending only on $(W^{\star},X^\star, Y^{\star})$.
Then, the iterates satisfy
    $$
    \norm{\Fasym(W_k,X_k,Y_k) - \Ts}{F}^2 \leq \left( 1 - \frac{\gamma}{8} \right)^k\norm{\Fasym(W_0,X_0,Y_0) - \Ts}{F}^2 \quad \text{for all $k\geq0$.}
    $$
\end{corollary}

}{} %

\noindent Unlike our results for matrices, our tensor guarantees only handle exactly parameterized problems. Moreover, by invoking Theorem~\ref{thm:local-onestepimprovement-strong-smooth}, we can derive similar rates with worse constants for the smooth loss $h(T) = \frac{1}{2} \norm{T-\Ts}{2}^2$.

\changelocaltocdepth{1}
\section{Numerical experiments}\label{sec:experiments}
In this section, we present numerical results that support our theoretical guarantees. Sections\ifbool{showSquare}{~\ref{sec:nonnegative-exp},}{} \ref{sec:matrix-exp}\ifbool{showSquare}{,}{} and \ref{sec:tensor-exp} include experiments for \ifbool{showSquare}{nonnegative least squares,} matrix sensing\ifbool{showSquare}{,}{} and tensor factorization, respectively. 
The code for reproducing these experiments is available 
 \ifbool{showSquare}{at \begin{quote} \centering \href{https://github.com/aglabassi/preconditioned\_composite\_opti}{https://github.com/aglabassi/preconditioned\_composite\_opti}.
 \end{quote}}{in \href{https://github.com/aglabassi/preconditioned\_composite\_opti}{this link}.}
\ifbool{showSquare}{}{We do not compare against other methods since, to our knowledge, no other methods can provably handle the settings we consider.}

\paragraph{Implementation details.} We run all methods on a Google Colab compute unit with 12GB of system RAM, and a \texttt{T4} GPU with 16GB of RAM (A100 for tensor experiments). We use \texttt{Python 3.11.11} and \texttt{Pytorch 2.5.1} paired with \texttt{Cuda 12.4}. Further, we use Pytorch's double-precision floating-point format.  For almost all experiments, we solve the linear systems via the Conjugate Gradient method %
with a maximum of 100 iterations and a tolerance level of $10^{-25}$. We use implicit evaluations of the matrix-vector products $\nabla F(x)^\top \nabla F(x) v $ using \ifbool{showSquare}{the derivations given in the Appendix~\ref{efficient}}{efficient action derivations}. \ifbool{showSquare}{The only expectation is nonnegative least squares, for which we use $( \nabla F(x)^\top \nabla F(x) + \lambda I )^{-1} v = v \odot \frac{1}{x\odot x + \lambda \mathbf{1} }$ directly.}{}

\ifbool{showSquare}{
\subsection{Nonnegative least squares with smooth and nonsmooth losses} \label{sec:nonnegative-exp}
For our first experiment, we consider the \ifbool{showSquare}{two}{} nonnegative least squares formulation\ifbool{showSquare}{s}{}\eqref{nnls-nonsmooth} from Section~\ref{sec:square}\ifbool{showSquare}{}{ and its squared counterpart (classical $\ell_2$-norm squared)}. We generate the ground truth via
\(
z^\star = \begin{bmatrix}
1, \ldots, \frac{1}{\tau} & 
\mathbf{0}_{r - r^\star}
\end{bmatrix}^\top \in \mathbb{R}^r,
\)
where $\tau$ and $r-r^\star$ respectvely control the ill-conditionedness and overparameterization of the map $F(x) = x\odot x$. Ill-conditioning of the map $F(x) = x\odot x$ at $z^\star$ occurs when $\max_{i \mid z_i^\star \neq 0 }{ \abs{z_i^\star}} \gg\min_{i \mid z_i^\star \neq 0 }{ \abs{z_i^\star}}$, and overparameterization when $ \dim(z^\star) > \norm{z^\star}{0} $. We vary $\tau \in \{1, 100\}$ and $r\in \{ 10,100 \},$ and take $ r^\star = 10.$ We generate matrices \( A \in \mathbb{R}^{m \times r} \) with \( m = 2 r \), and $\kappa(A)=10$, and set $b = A z^\star$. We initialize all methods at the same random $x_0$ satisfying $\norm{x_0 \odot x_0 - z^\star}{2} =  10^{-2}\norm{z^\star}{2}$.

\paragraph{Baselines.} We compare the performance of iterative methods applied to \ifbool{showSquare}{the smooth and nonsmooth formulations in \eqref{nnls-nonsmooth}}{the problem in \eqref{nnls-nonsmooth} and its squared counterpart}. For both losses, we test Algorithm~\ref{alg:LM} against the standard subgradient method and the Gauss-Newton subgradient method from \cite{davis2022linearly}. 
All methods use the Polyak-type stepsizes. For the subgradient method it is exactly the Polyak stepsize $(f(x_k) - \min f)/\|g_k\|^2$ with $g_k \in \partial f(x_k).$ For the other two methods, we use the stepsize from Configuration~\ref{assum: dampingparameter}. For Algorithm~\ref{alg:LM} we use $\lambda_k =  10^{-2} \norm{Ax-b}{2}$ as an estimator of the quantity $\norm{x_k\odot x_k - z^\star}{2}$, which emulates Configuration~\ref{assum: dampingparameter} without requiring access to $z^\star.$
\begin{figure}[t]
  \centering

  \begin{subfigure}[b]{0.48\linewidth}
    \centering
    \includegraphics[width=\linewidth]{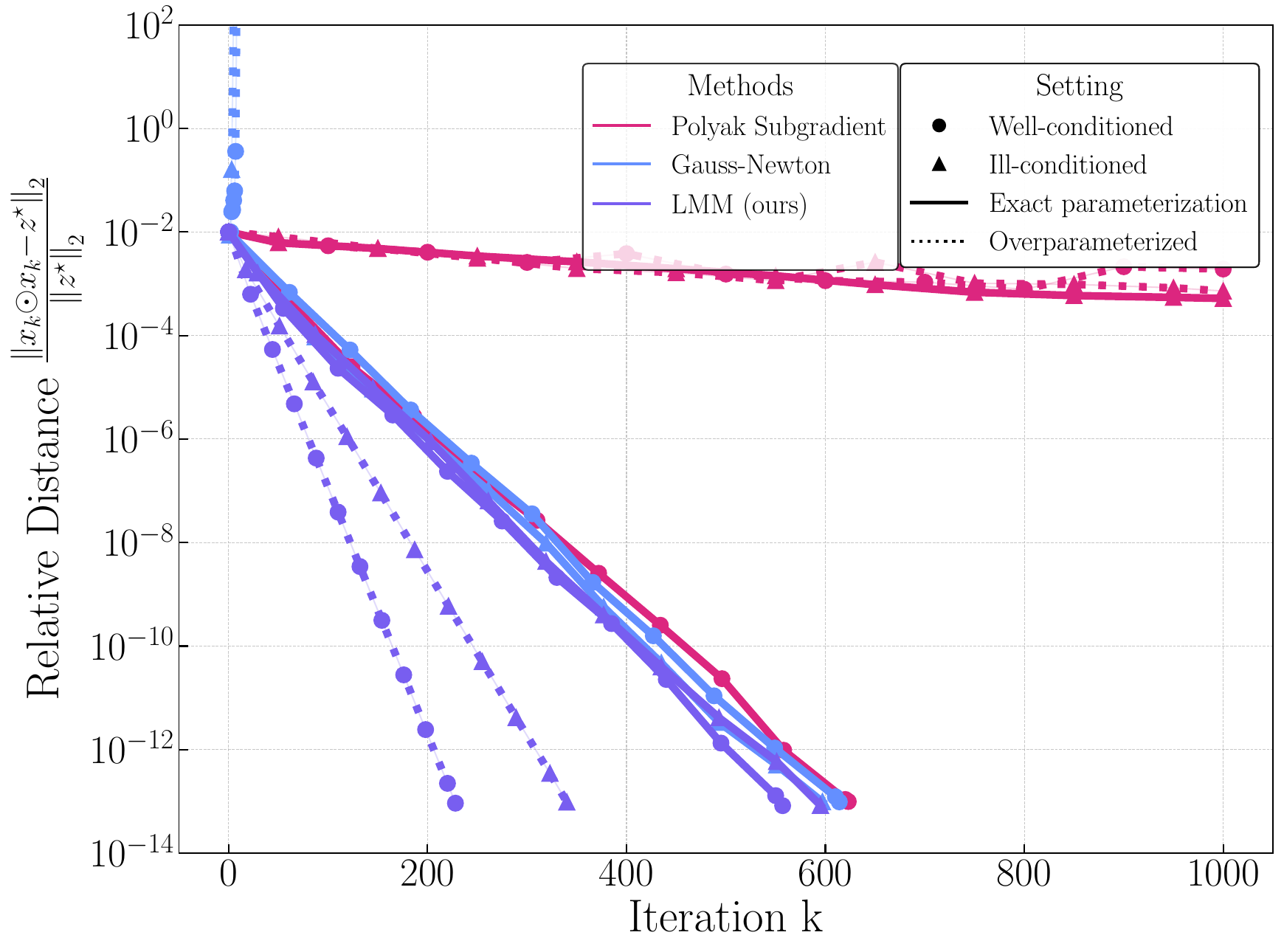}  %
    \caption{Objective $\frac{1}{2}\norm{A(x\odot x) - b}{2}^2$}                                   %
  \end{subfigure}
  \hfill
  \begin{subfigure}[b]{0.48\linewidth}
    \centering
    \includegraphics[width=\linewidth]{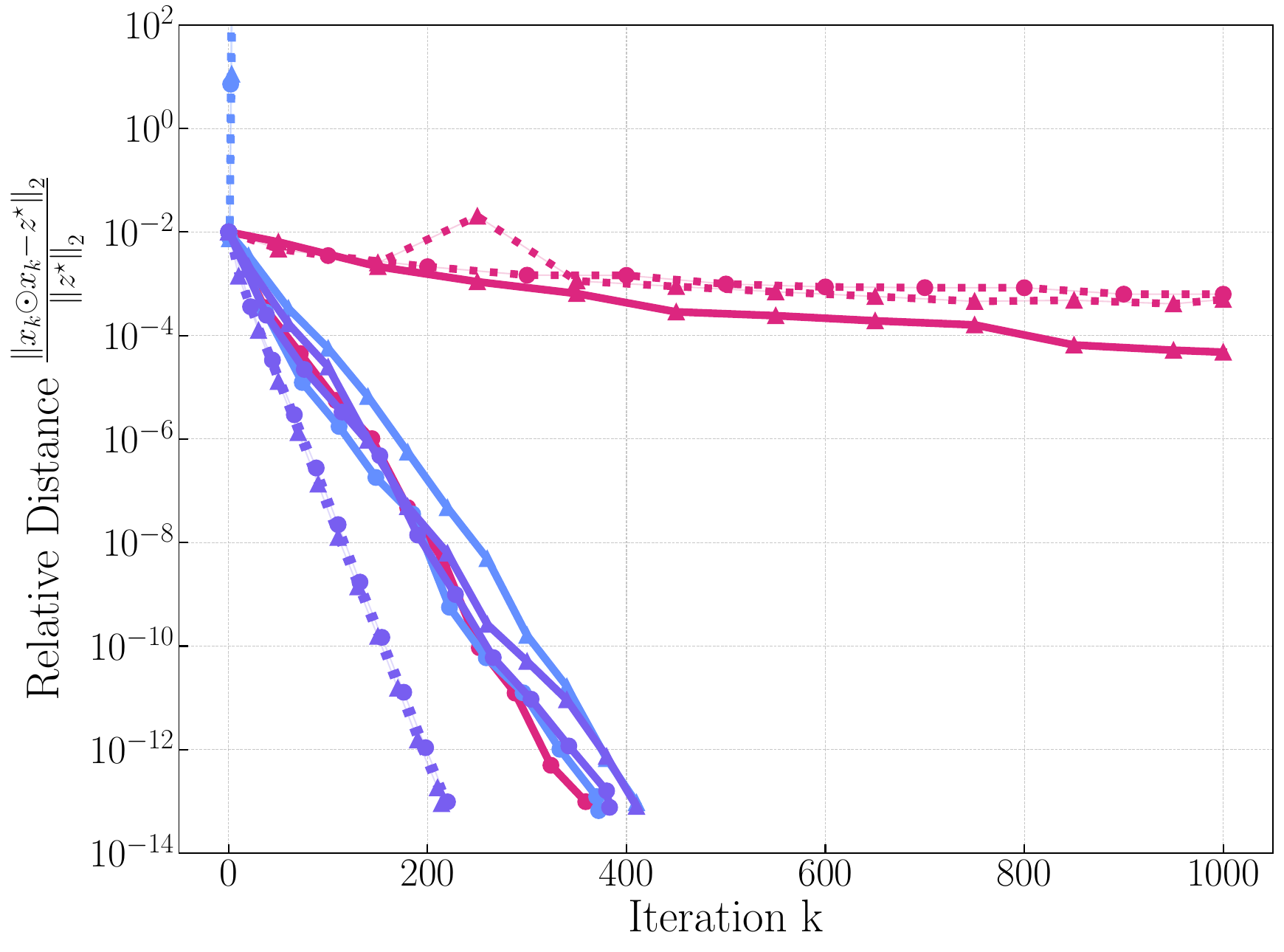}  %
    \caption{Objective $\norm{A(x\odot x) - b}{2}$}                                   %
  \end{subfigure}

  \caption{Relative distance against iteration count for nonnegative least squares losses \eqref{nnls-nonsmooth}. 
  }
  \label{fig:nonegative _least_squares}
\end{figure}

\paragraph{Discussion.}Figure~\ref{fig:nonegative _least_squares} displays the results. On the one hand, the Polyak subgradient method fails to converge linearly, whether there is ill-conditioning or overparameterization, and, further,  Gauss-Newton diverges in the overparameterization case, which is expected as the precondition is ill-defined. On the other hand, Algorithm~\ref{alg:LM} is robust and converges linearly in all settings.
Notably, the methods exhibit faster convergence when applied to the nonsmooth formulations, highlighting the benefit of using a nonsmooth loss for regression.
}{}

\ifbool{showSquare}{
\begin{figure}[ht]
    \centering
    
    \begin{subfigure}[b]{0.48\textwidth}
        \centering
        \includegraphics[width=\textwidth]{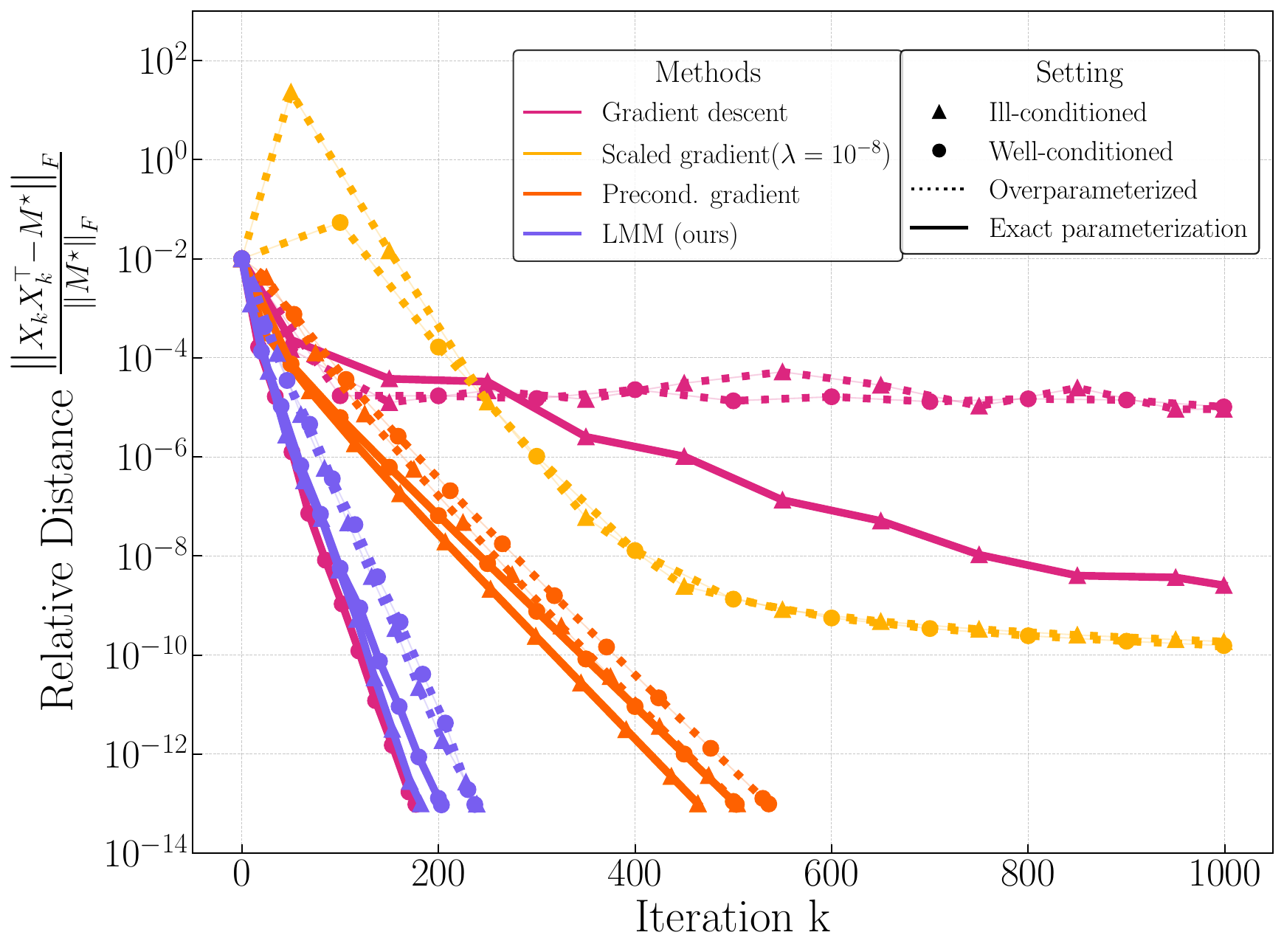}
        \caption{Symmetric Matrix ($d=100$)}
    \end{subfigure}
    \hfill
    \begin{subfigure}[b]{0.48\textwidth}
        \centering
        \includegraphics[width=\textwidth]{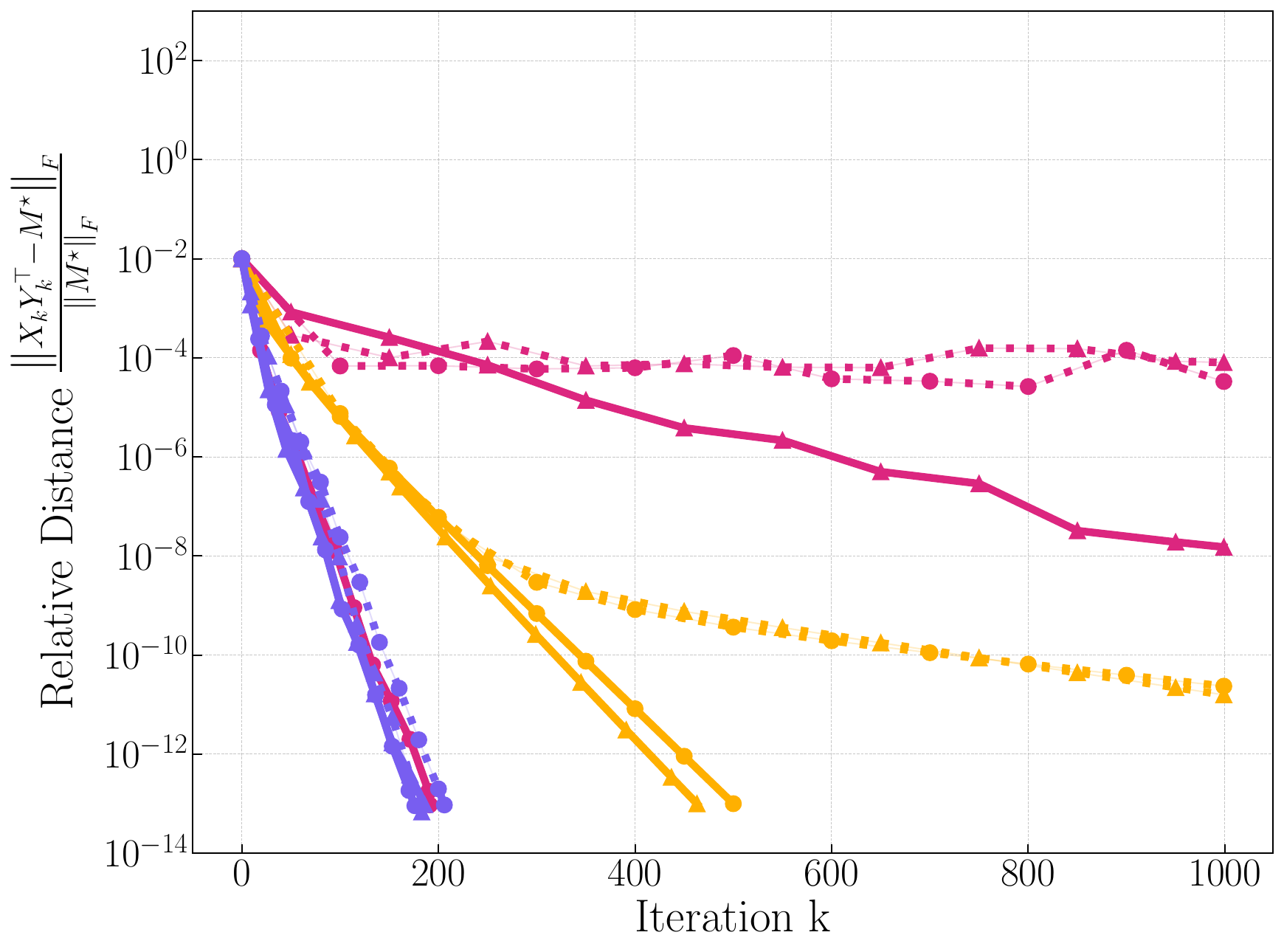}
        \caption{Asymmetric Matrix ($d=100$)}
    \end{subfigure}
    
    \vskip\baselineskip
    
    \begin{subfigure}[b]{0.48\textwidth}
        \centering
        \includegraphics[width=\textwidth]{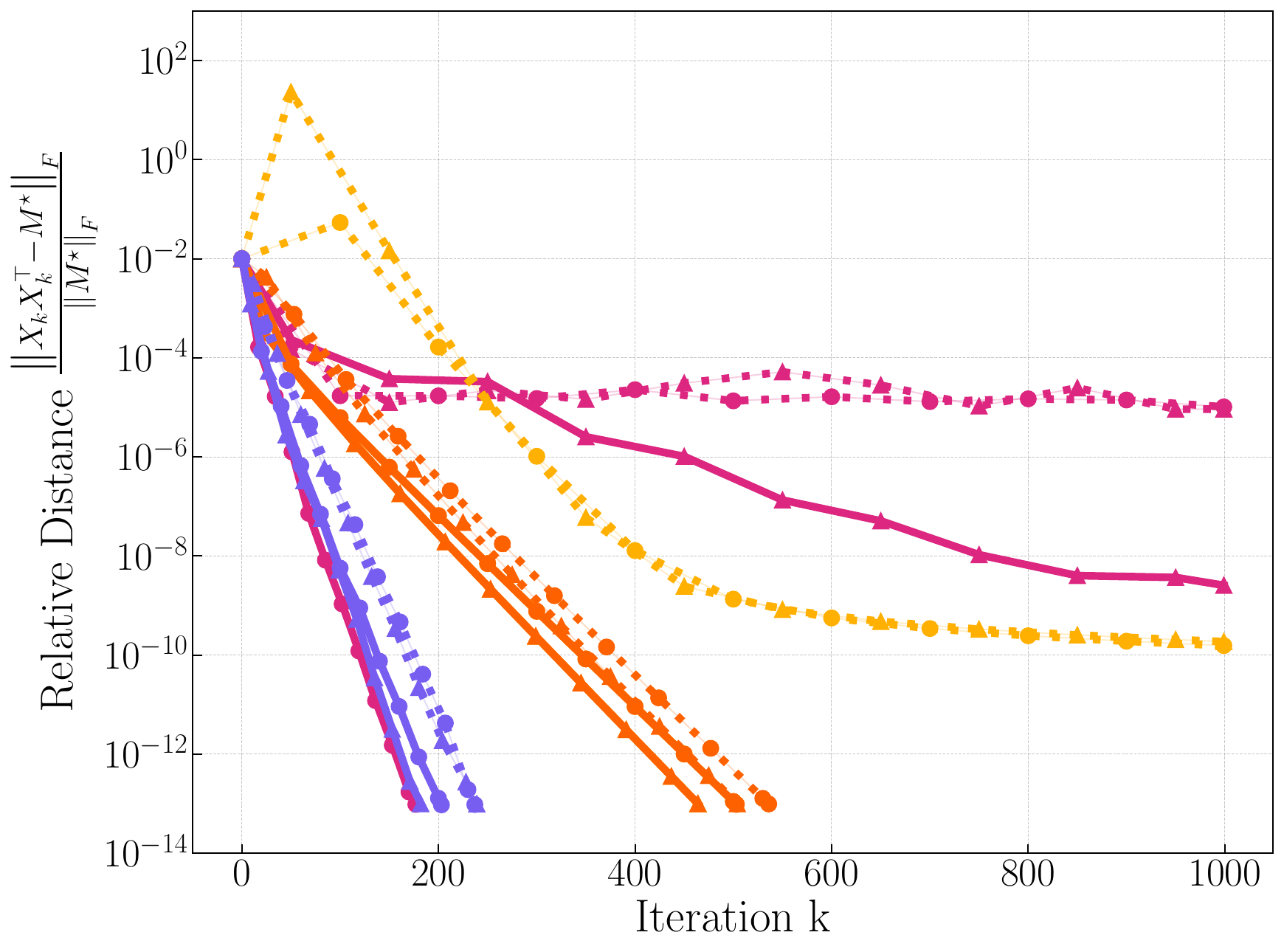}
        \caption{Symmetric Matrix ($d=200$)}
    \end{subfigure}
    \hfill
    \begin{subfigure}[b]{0.48\textwidth}
        \centering
        \includegraphics[width=\textwidth]{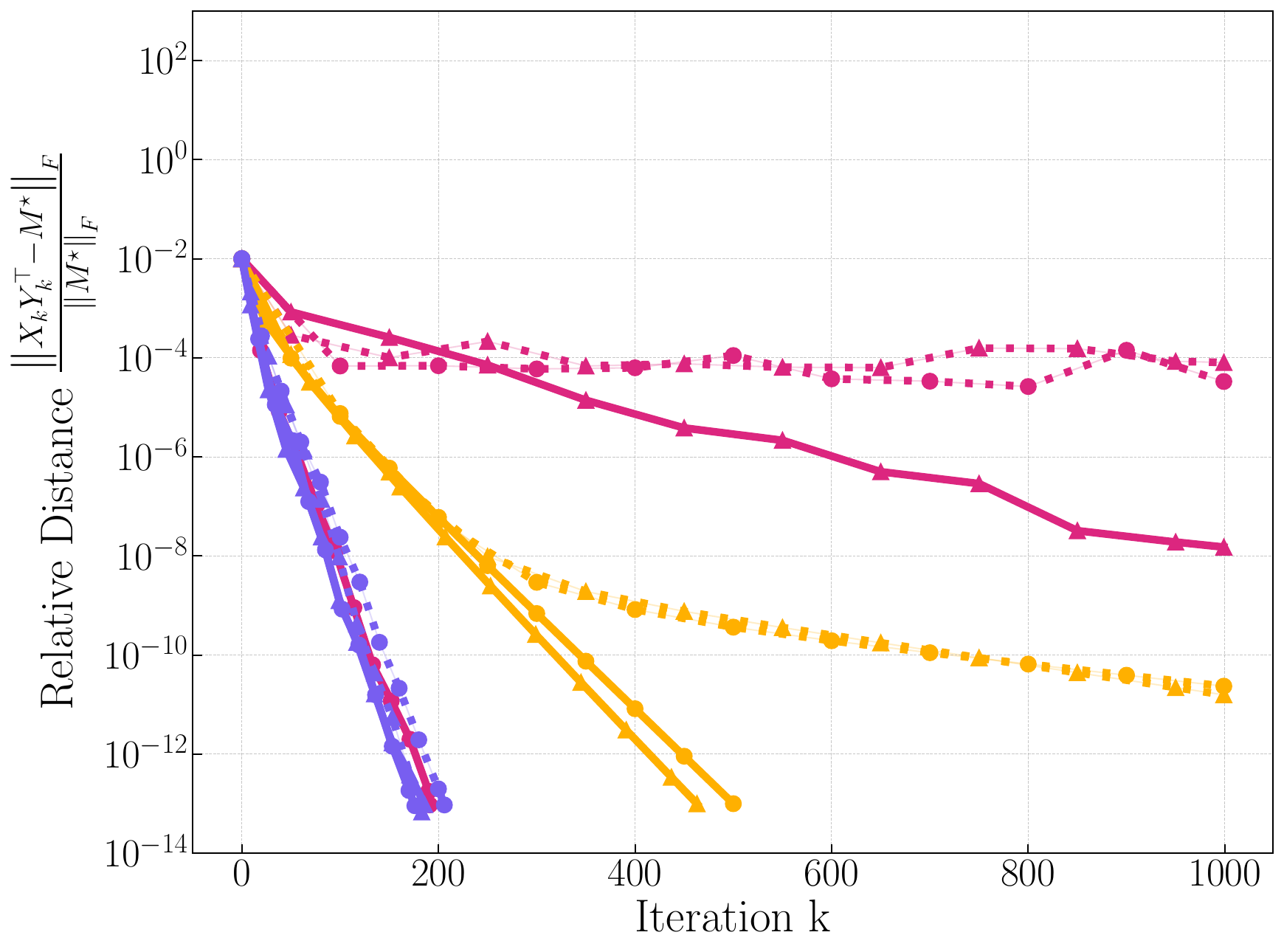}
        \caption{Asymmetric Matrix ($d=200$)}
    \end{subfigure}
    
    \caption{Smooth matrix sensing with the $\ell_2$-norm squared. We use $m =4dr$ ($m=2dr$ for symmetric), with $r^\star = 2$, $r\in\{2,5\}$.}
    \label{fig: matrix_smooth}
\end{figure}}{}

\ifbool{showSquare}{
\begin{figure}[ht]
    \centering
    
    \begin{subfigure}[b]{0.48\textwidth}
        \centering
        \includegraphics[width=\textwidth]{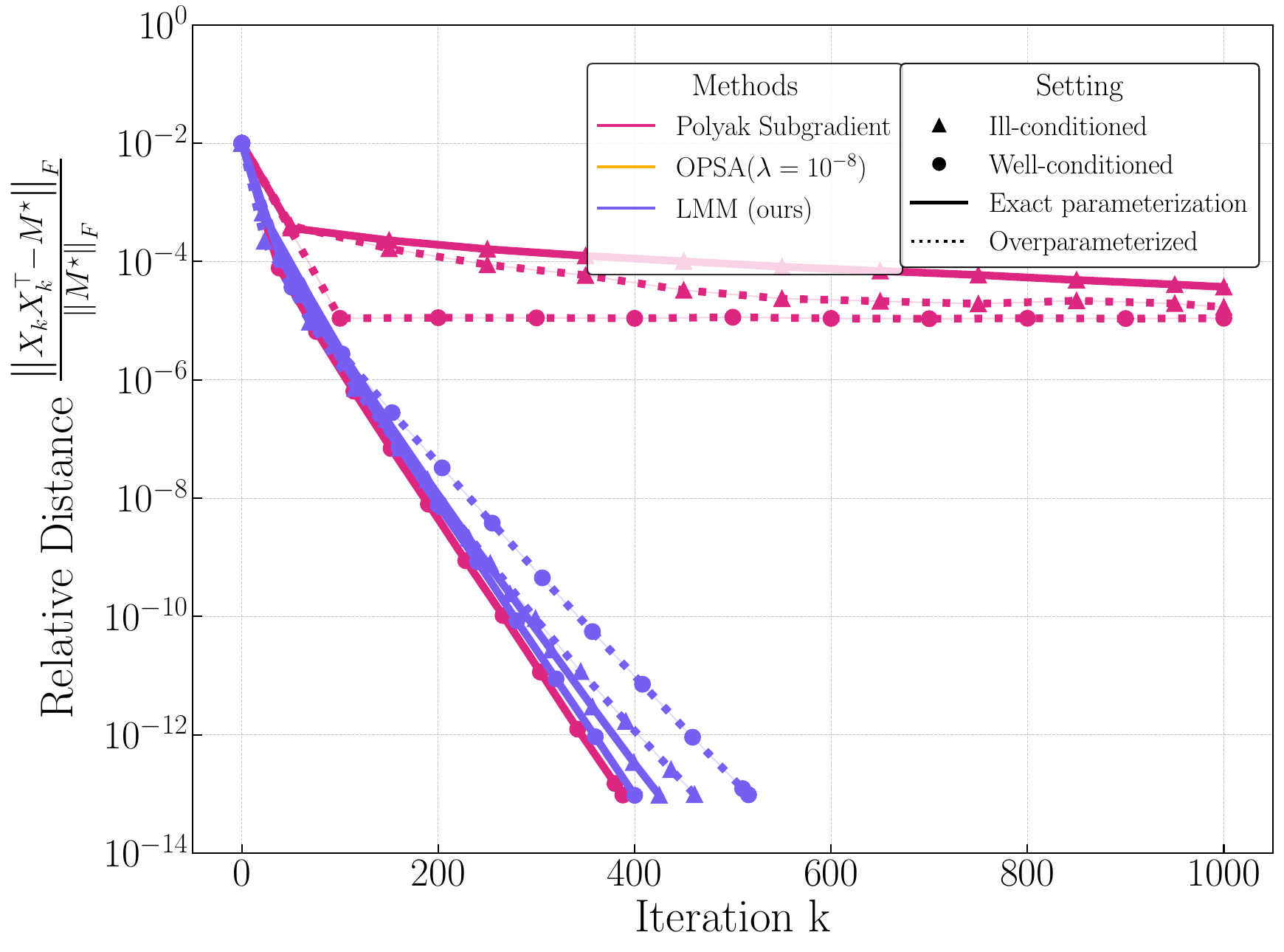}
        \caption{Symmetric Matrix ($d=100$)}
    \end{subfigure}
    \hfill
    \begin{subfigure}[b]{0.48\textwidth}
        \centering
        \includegraphics[width=\textwidth]{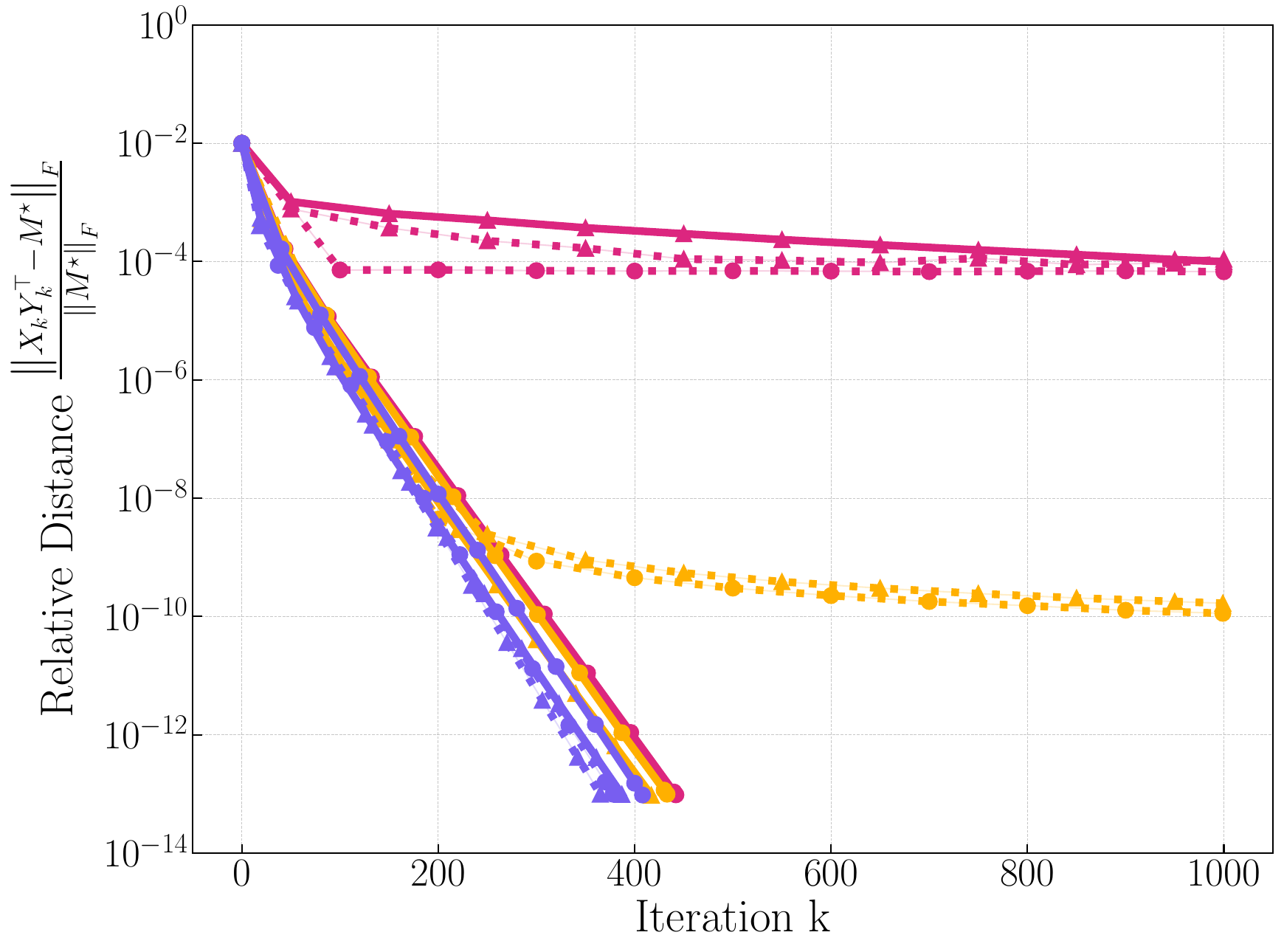}
        \caption{Asymmetric Matrix ($d=100$)}
    \end{subfigure}
    
    \vskip\baselineskip
    
    \begin{subfigure}[b]{0.48\textwidth}
        \centering
        \includegraphics[width=\textwidth]{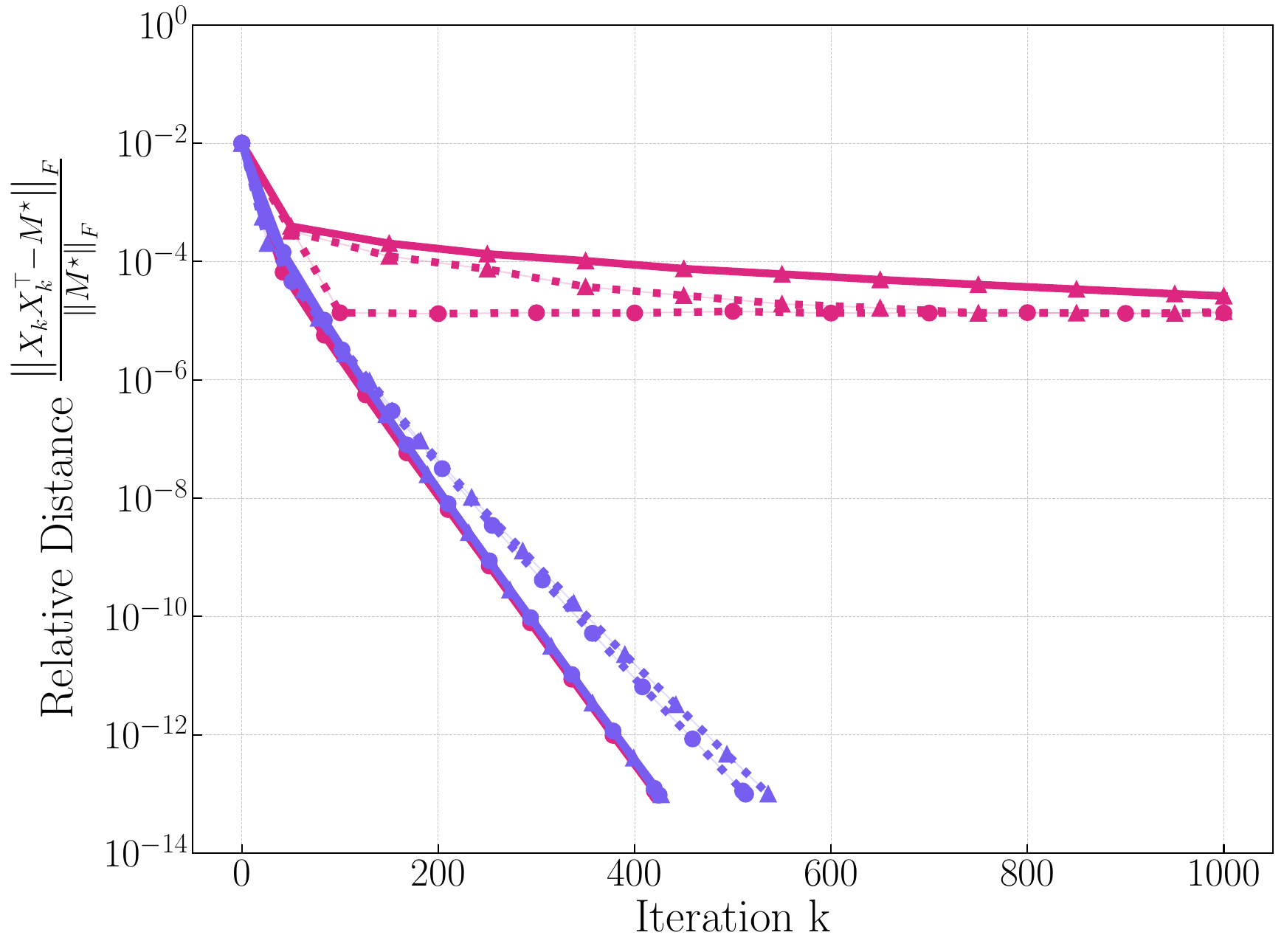}
        \caption{Symmetric Matrix ($d=200$)}
    \end{subfigure}
    \hfill
    \begin{subfigure}[b]{0.48\textwidth}
        \centering
        \includegraphics[width=\textwidth]{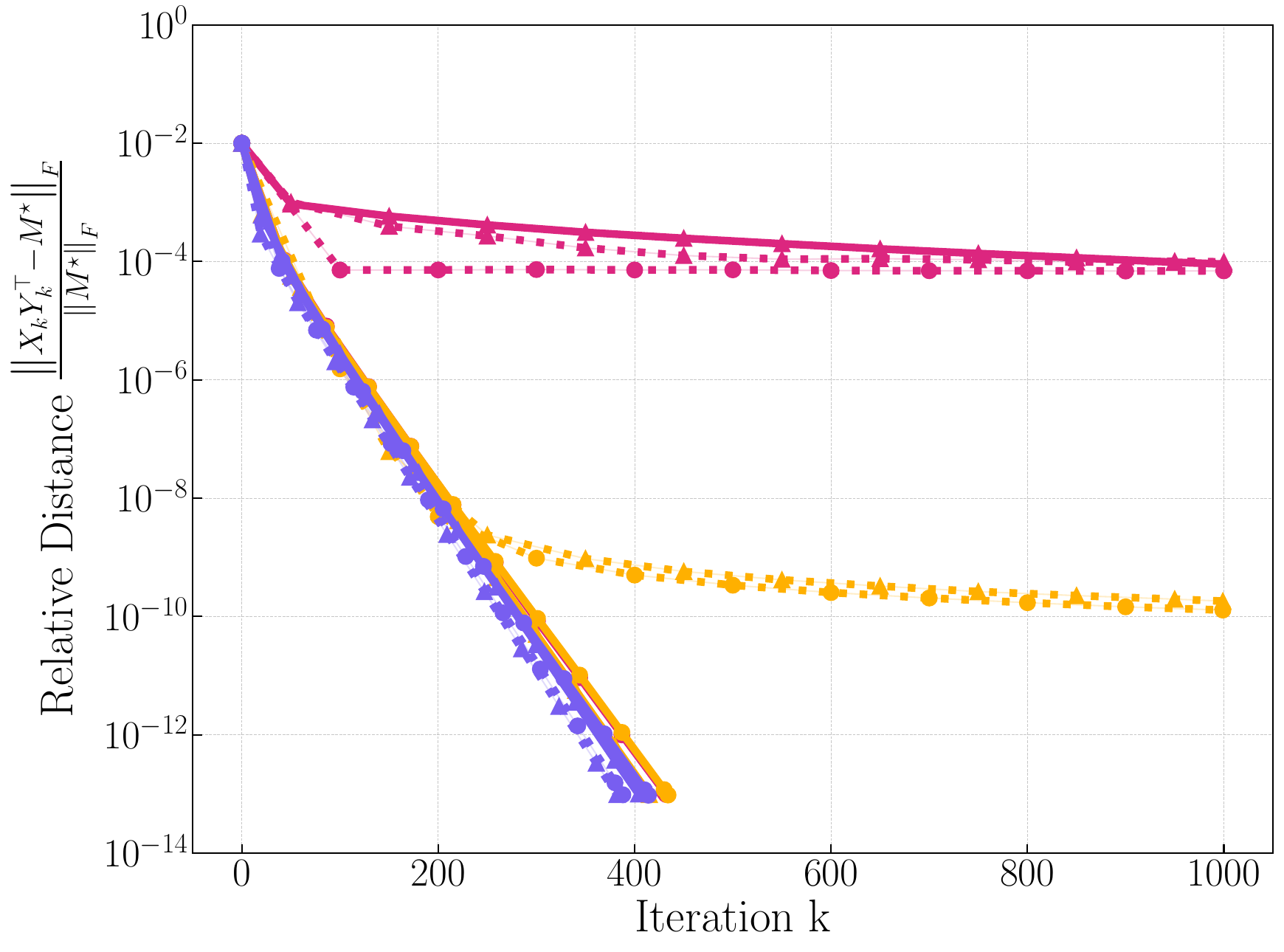}
        \caption{Asymmetric Matrix ($d=200$)}
    \end{subfigure}
    
    \caption{Matrix sensing with the $\ell_1$-norm. We use $m =4dr$ ($m=2dr$ for symmetric) with $r^\star = 2$, $r\in \{2,5\}$. \texttt{OPSA} \cite{giampouras2024guarantees} only applies to the asymmetric setting.}
    \label{fig: matrix_nonsmooth}
\end{figure}}{\begin{figure}[t]
    \centering
    
    \begin{subfigure}[b]{0.48\textwidth}
        \centering
        \includegraphics[width=\textwidth]{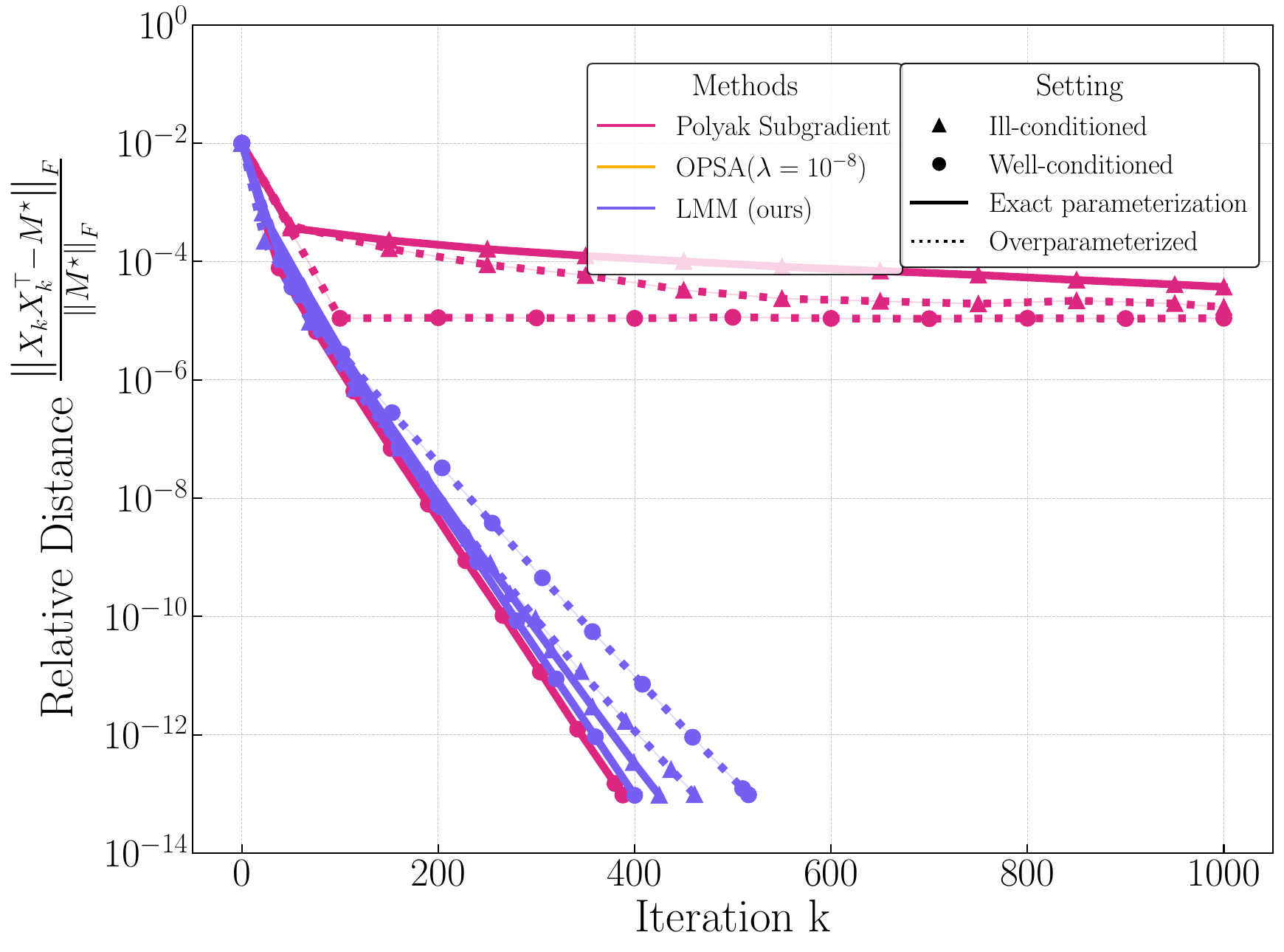}
        \caption{Matrix sensing ($d=100$)}
    \end{subfigure}
    \hfill
  \begin{subfigure}[b]{0.48\textwidth}
        \centering
        \includegraphics[width=\textwidth]{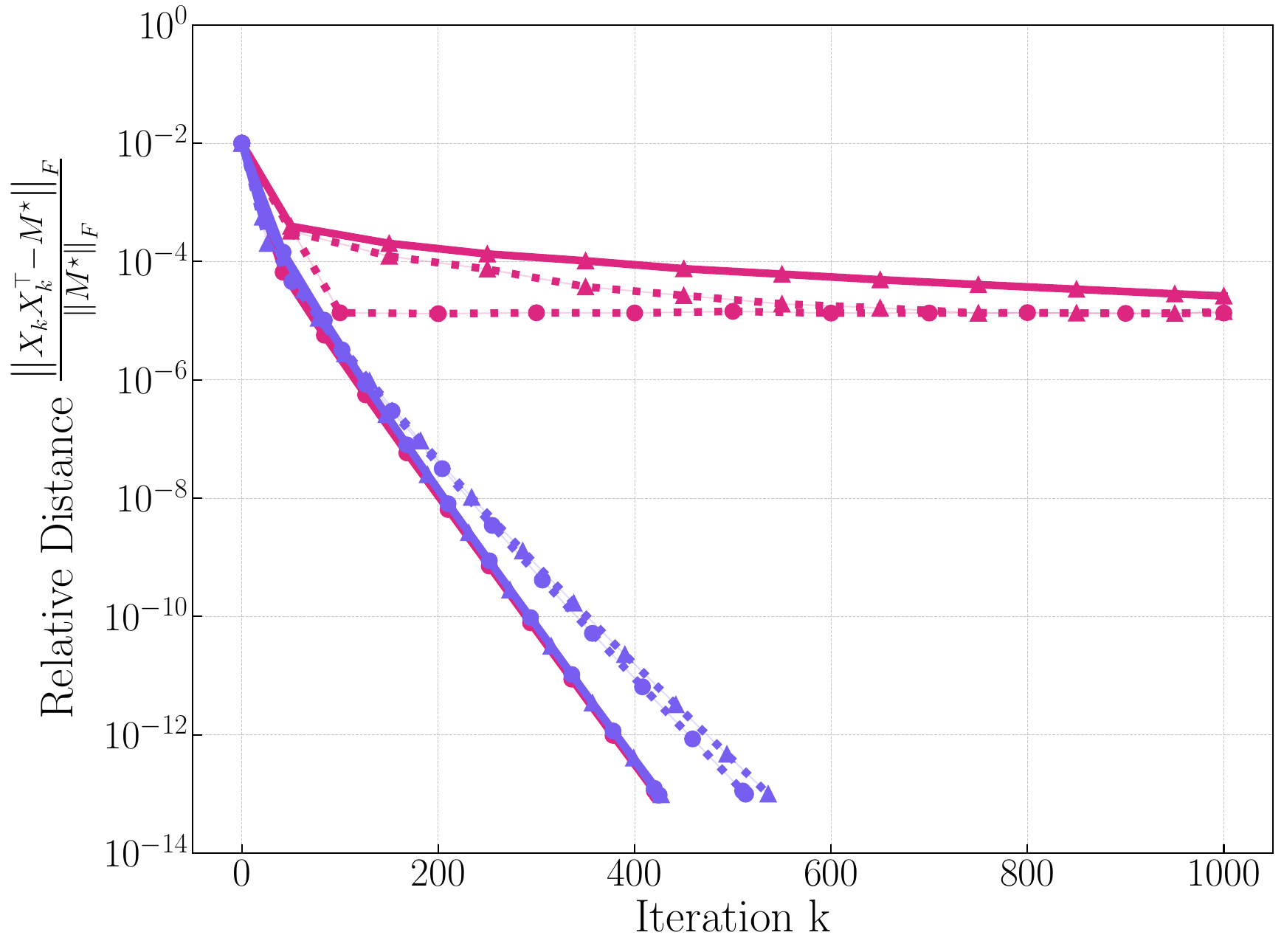}
        \caption{Matrix sensing ($d=200$)}
    \end{subfigure}
    \caption{Matrix sensing with the $\ell_1$-norm with different dimensions. OPSA \cite{giampouras2024guarantees} only applies to the asymmetric setting. We use $m =2dr$ measurements and set $r^\star = 2$, $r\in \{2,5\}$.
    }
    \label{fig: matrix_nonsmooth}
\end{figure}}

\subsection{Matrix Sensing}\label{sec:matrix-exp}

\ifbool{showSquare}{For our second batch of experiments, we}{We} consider the matrix problems introduced in~(\ref{eq:matrix-problems}). 
We run \ifbool{showSquare}{three}{two} experiments to evaluate $(i)$ convergence\ifbool{showSquare}{, $(ii)$ hyperparameter sensitivity,}{} and \ifbool{showSquare}{$(iii)$}{$(ii)$} robustness to outliers. \ifbool{showSquare}{All three}{The two} types of experiments use similar losses and parameter configurations, which we describe next.

\paragraph{Setup.} 
We solve \ifbool{showSquare}{}{positive definite} matrix sensing using \ifbool{showSquare}{the squared \(\ell_2\)-loss \(h(M)=\|\mathcal{A}(M)-b\|_2^2\) and}{}the \(\ell_1\)-loss \(h(M)=\|\mathcal{A}(M)-b\|_1\). \ifbool{showSquare}{We consider both PSD and general asymmetric ground truths: for PSD sensing we}{We} set \(M^\star = X^\star{X^\star}^\top\) where \(X^\star = U\,D^{1/2}\) with \(U\in\mathbb{R}^{d\times r^\star}\) drawn at random satisfying \(U^\top U=I\) and \(D=\mathrm{diag}(\xi_1,\dots,\xi_{r^\star})\) with $\xi_i$ linearly spaced in \([1/\tau,1]\) for \(\tau\in\{1,100\}\)\ifbool{showSquare}{; for asymmetric sensing we similarly draw \(Y^\star=V\,D^{1/2}\) and set \(M^\star=X^\star(Y^\star)^\top\) to ensure \(\kappa(M^\star)=\tau\)}{}. To test dimension‐independent convergence we vary \(d\in\{100,200\}\), and to probe overparameterization we fix \(r^\star=2\) while varying \(r\in\{2,5\}\). The map \(\mathcal{A}:\mathbb{R}^{d\times d}\to\mathbb{R}^m\) has i.i.d.\ \({\rm N}(0,\frac{1}{m})\) entries with \(m=2dr\) \ifbool{showSquare}{(or \(m=4dr\) for asymmetric)}{}, and all methods are initialized identically with relative error \(10^{-2}\).

\paragraph{Baselines.}
\ifbool{showSquare}{For the smooth problems, we compare with gradient descent, \texttt{PrecGD} \cite{zhang2021preconditioned} (symmetric only), \texttt{ScaledGD}$(\lambda)$ \cite{xu2023power}. In the nonsmooth setting,}{To the best of our knowledge, there is no preconditioned subgradient method for Burer-Monteiro matrix recovery. Thus,} we compare our method against the Polyak subgradient method\ifbool{showSquare}{, and \texttt{OPSA} \cite{giampouras2024guarantees} (asymmetric only)}{}. 
Unless otherwise stated, our method uses Polyak stepsizes (Configuration \ref{assum: dampingparameter}), where we set $\gamma=1$. \ifbool{showSquare}{For constant-stepsize methods, \texttt{PrecGD} and \texttt{ScaledGD}$(\lambda)$, we tune to select the largest parameter that leads to convergence, i.e., $\gamma_k = 1/2$. For \texttt{ScaledGD}$(\lambda)$ we set $\lambda = 10^{-8}.$ For \texttt{PrecGD} and Algorithm~\ref{alg:LM}}{Moreover,} we use a damping parameter of \ifbool{showSquare}{$\lambda_k = 2.5\cdot 10^{-3}\sqrt{f(x_k)}$ in the smooth setting, or}{} $\lambda_k =   10^{-5} \cdot f(x_k)$ \ifbool{showSquare}{in the nonsmooth one,}{} as an estimator for the quantity $\norm{z_k- z^\star}{2}$.

\paragraph{Experiment 1: convergence rates.} We generate noiseless observations \(b = \mathcal{A}(M^\star)\) and solve the recovery problem using\ifbool{showSquare}{ both the squared \(\ell_2\)-norm and}{} the \(\ell_1\)-norm. \ifbool{showSquare}{Figures~\ref{fig: matrix_smooth} and~\ref{fig: matrix_nonsmooth} report results for the smooth and nonsmooth formulations, respectively, benchmarked against established competitor methods. Algorithm~\ref{alg:LM} consistently matches or outperforms existing approaches across both problem classes. \texttt{ScaledGD} and \texttt{OPSA} employ a fixed damping parameter, which restricts their linear convergence to a neighborhood around the optimum; by contrast, \texttt{PrecGD} and our method sustain linear convergence all the way to the exact solution.}{Figure~\ref{fig: matrix_nonsmooth} reports the results} \ifbool{showSquare}{Furthermore,}{and} corroborate our theoretical finding that Algorithm~\ref{alg:LM} achieves a convergence rate independent of the problem dimension.

\ifbool{showSquare}{\paragraph{Experiment 2: hyperparameter sensitivity.}
In this experiment, we probe the robustness of Configuration \ref{assum:geometric} for the $\ell_1$ norm loss.  We set stepsizes to $\gamma_k = \gamma q^k$ and damping parameters to $\lambda_k = 10^{-5} q^k$ and vary $q\in\{0.95, 0.96, 0.97\}$ and $\gamma \in \{ 10^{-j} \mid j = 1, \dots, 8\}$. We cap the total number of iterations at $10^3$. Compared to the previous experiment, we take a smaller dimension $d =30$ and consider more aggressive ill-conditioning by varying $\tau \in\{1, 10^4\}$. Figure~\ref{fig: sensitivity} shows the median number of iterations needed to achieve a given relative error of $10^{-8}$ over $100$ trials.
This experiment suggests that Algorithm~\ref{alg:LM} converges efficiently across a broad spectrum of hyperparameter settings.

\begin{figure}[t]
    \centering
\begin{tabular}{@{}c@{\hspace{-0.05em}}c@{\hspace{-0.05em}}c@{\hspace{-0.05em}}c@{}}
        \stepcounter{imagerow}\raisebox{45pt}{%
          \rotatebox[origin=c]{90}{\strut \vspace{0cm}$q=0.95$}} &
        {\hspace{0.2cm}\includegraphics[width=0.31\textwidth]{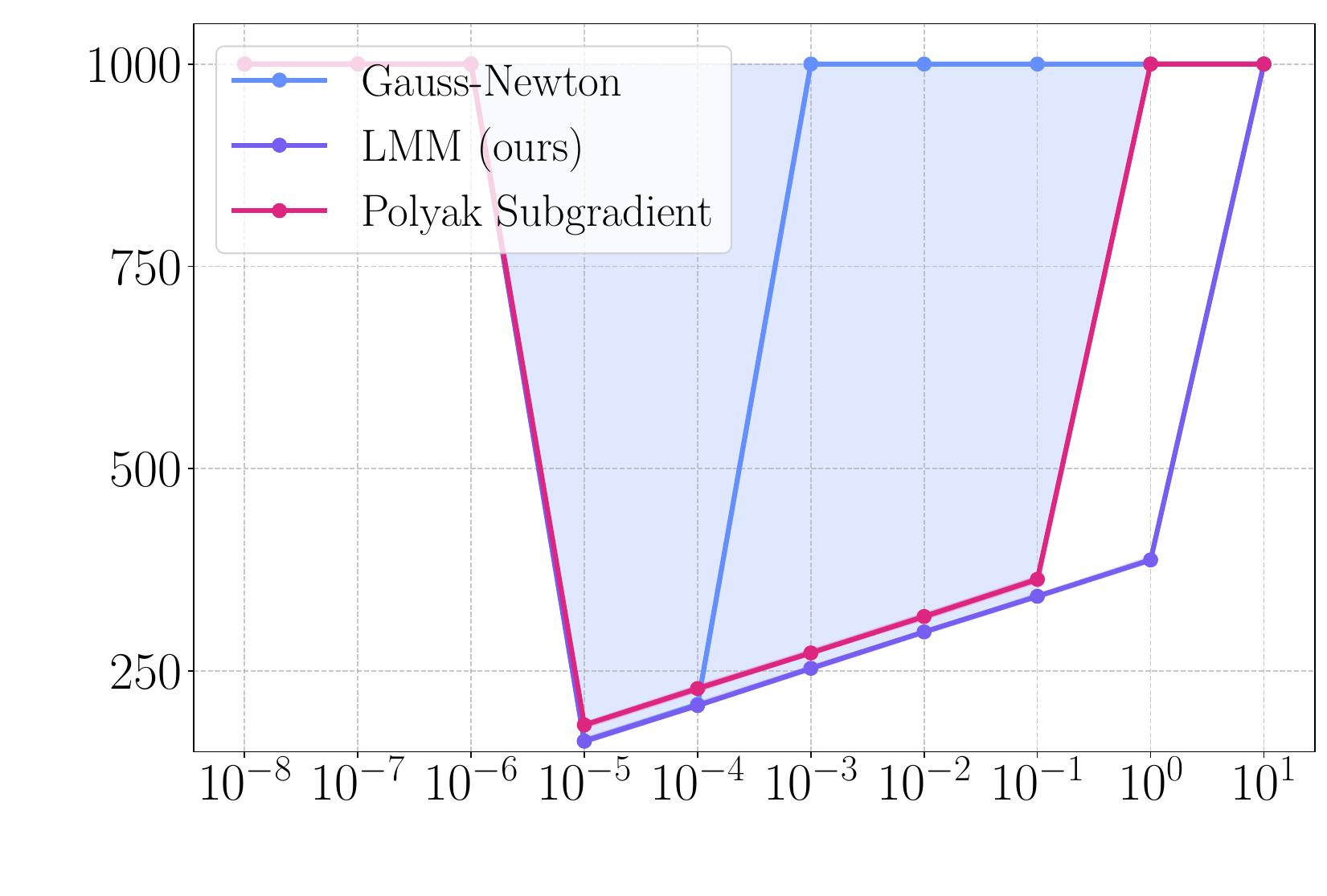}} &
        {\includegraphics[width=0.31\textwidth]{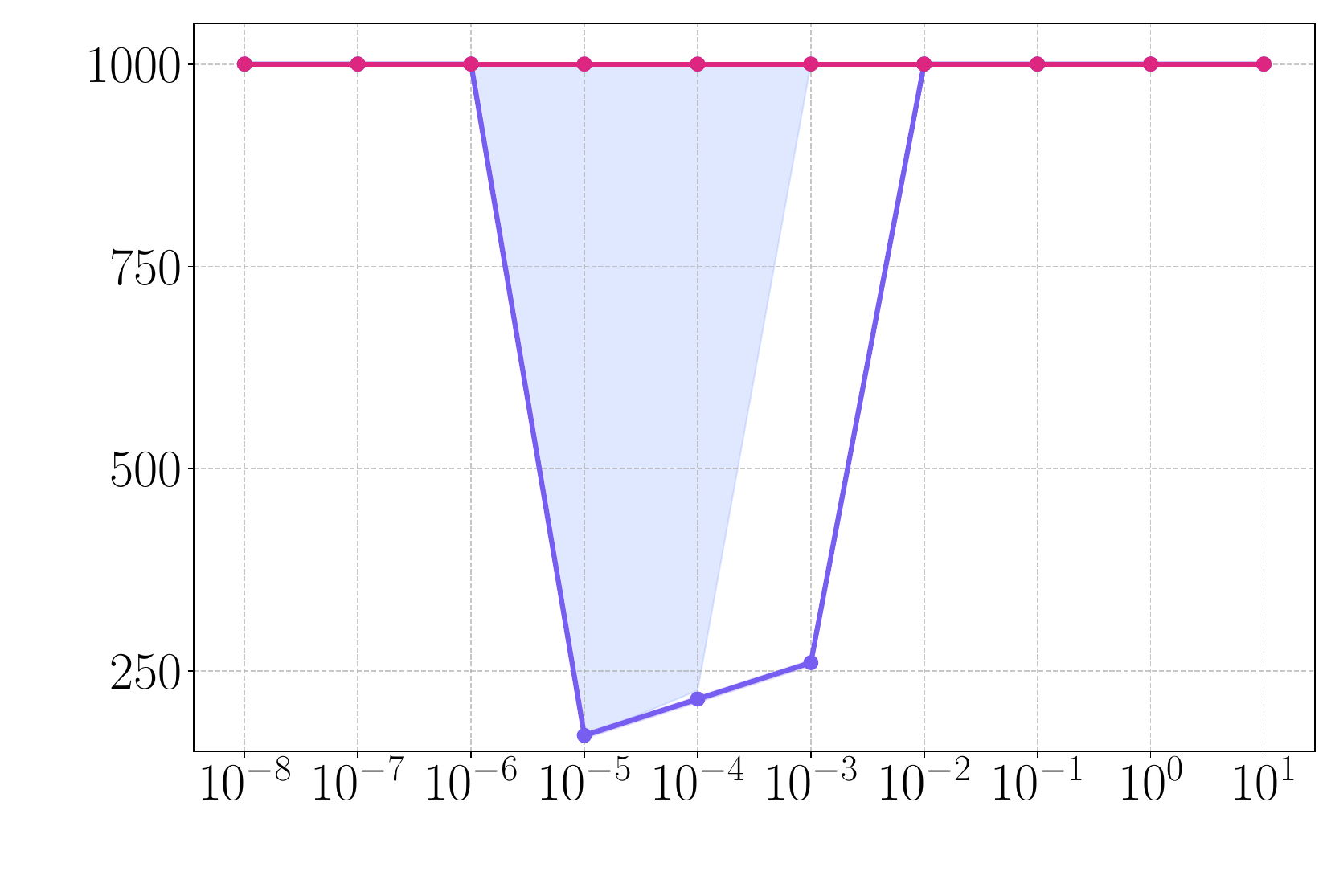}} &
        {\includegraphics[width=0.31\textwidth]{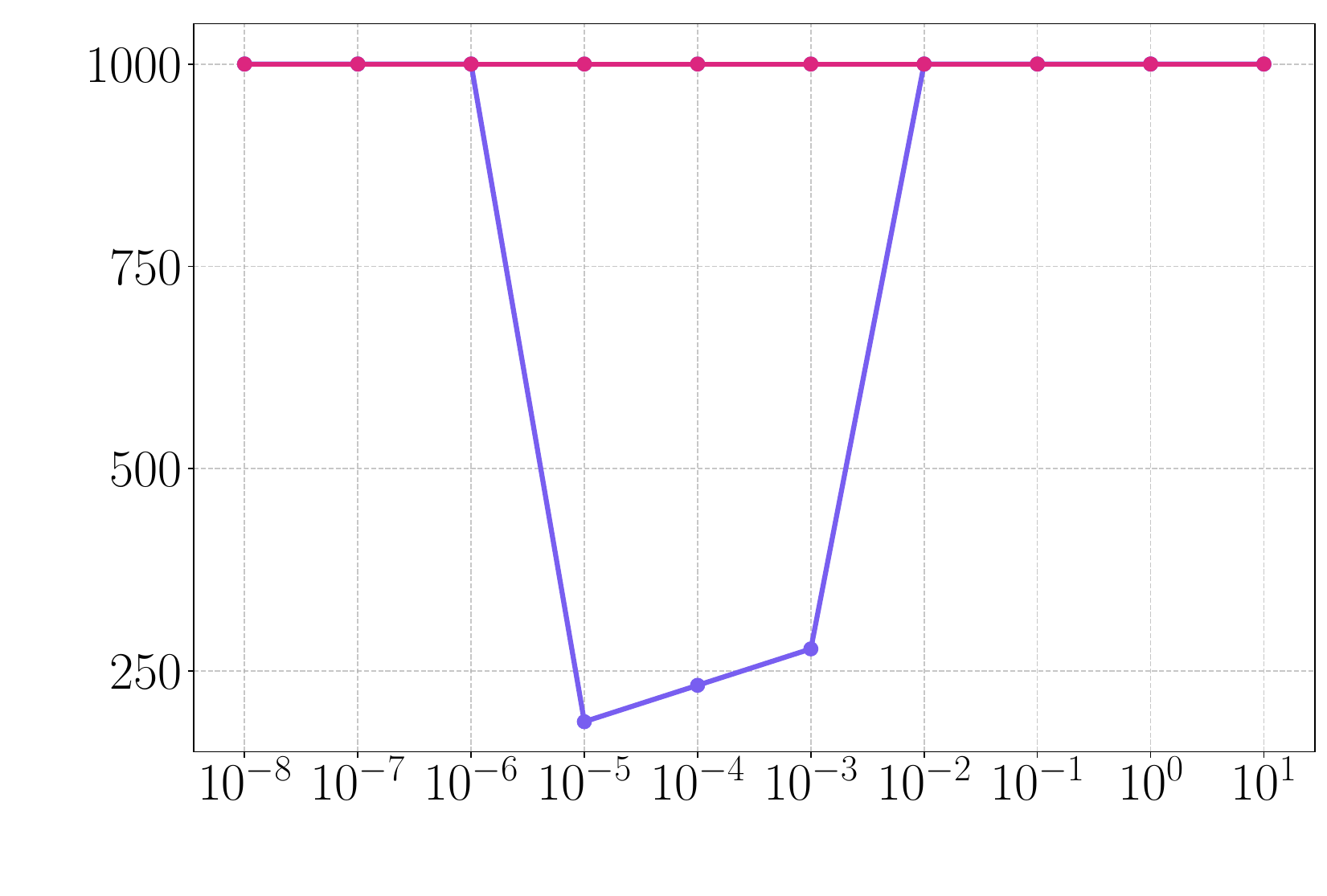}} \\
        \stepcounter{imagerow}\raisebox{45pt}{%
          \rotatebox[origin=c]{90}{\strut \vspace{0cm}$q=0.96$}} &
        {\hspace{0.2cm}\includegraphics[width=0.31\textwidth]{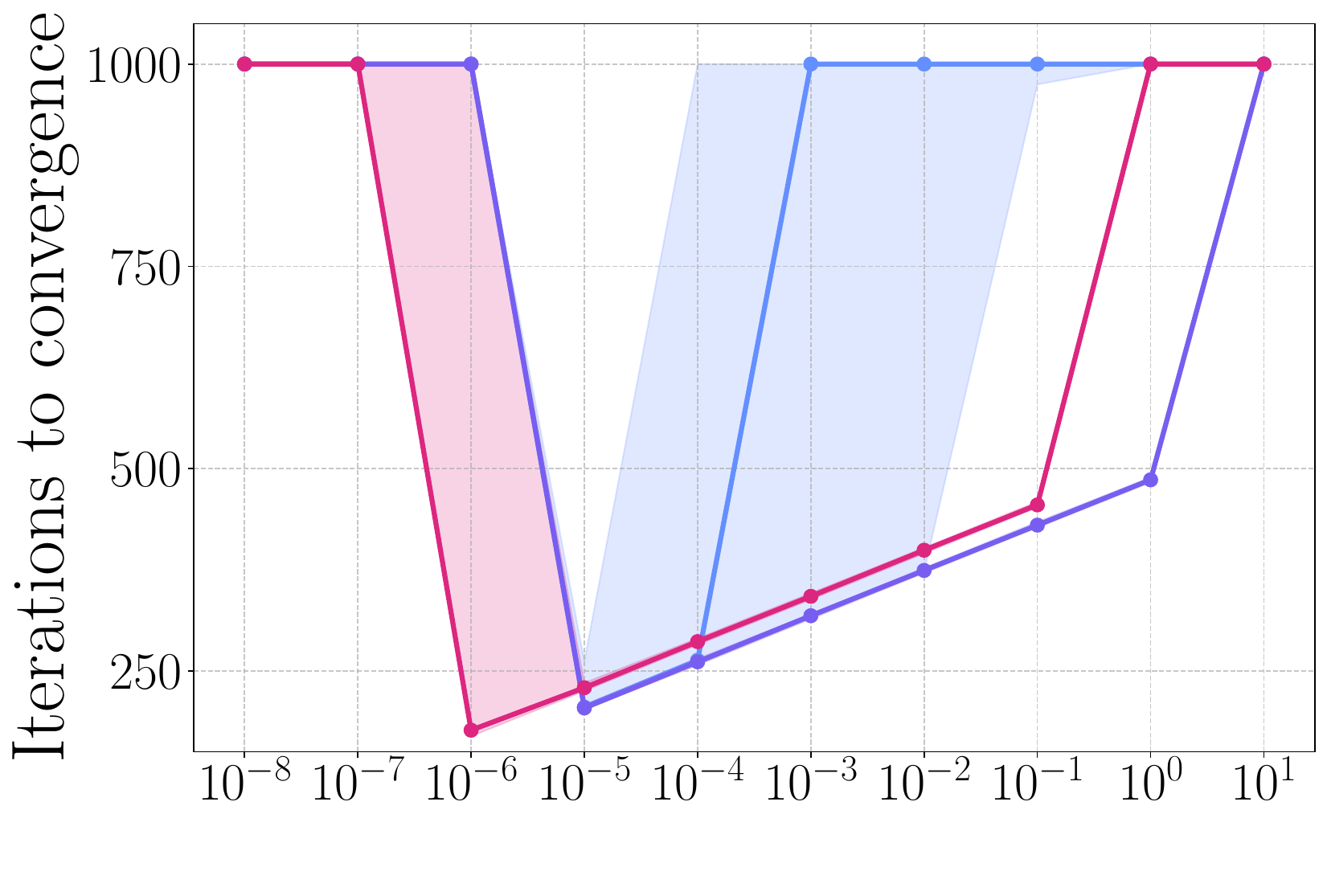}} &
        {\includegraphics[width=0.31\textwidth]{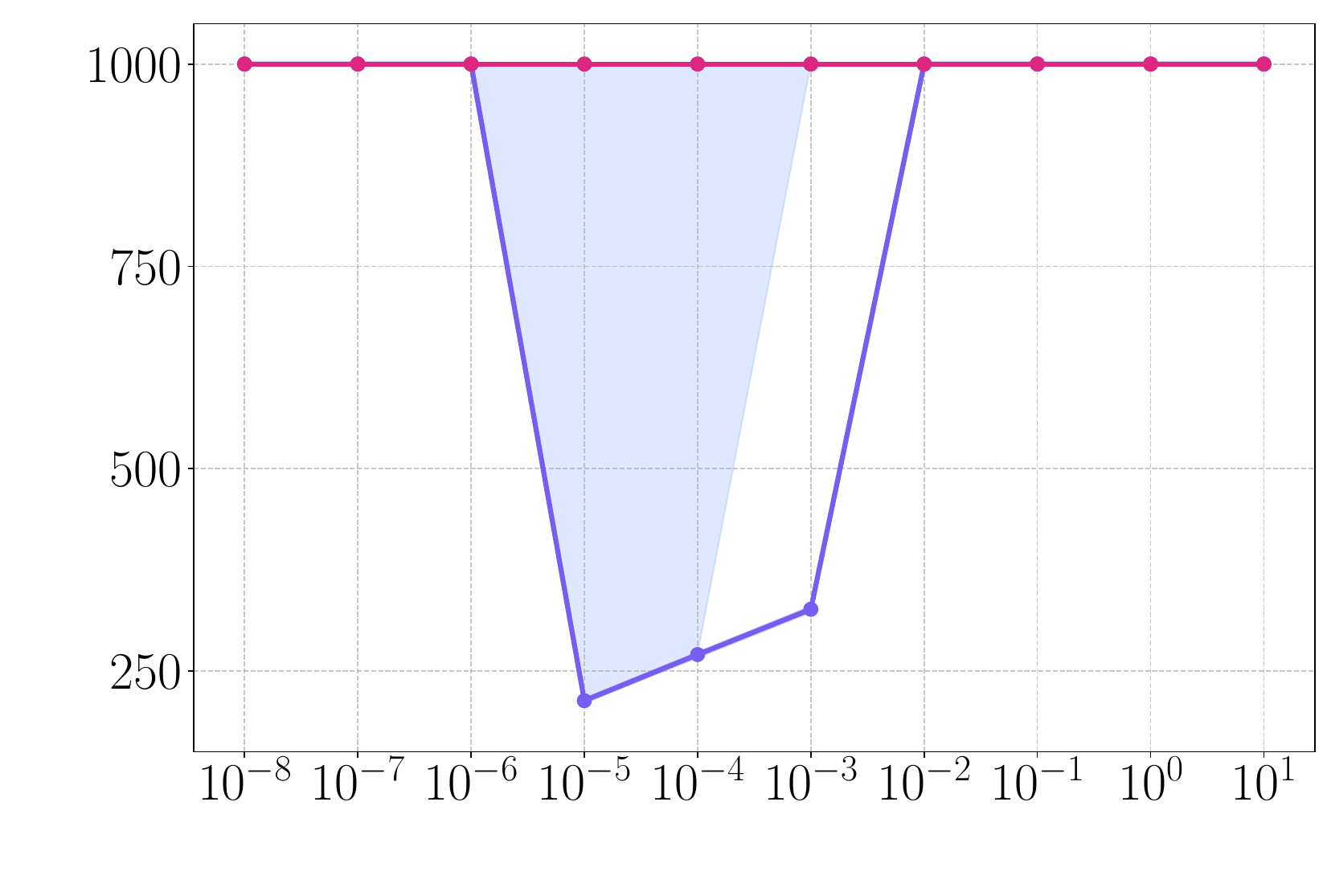}} &
        {\includegraphics[width=0.31\textwidth]{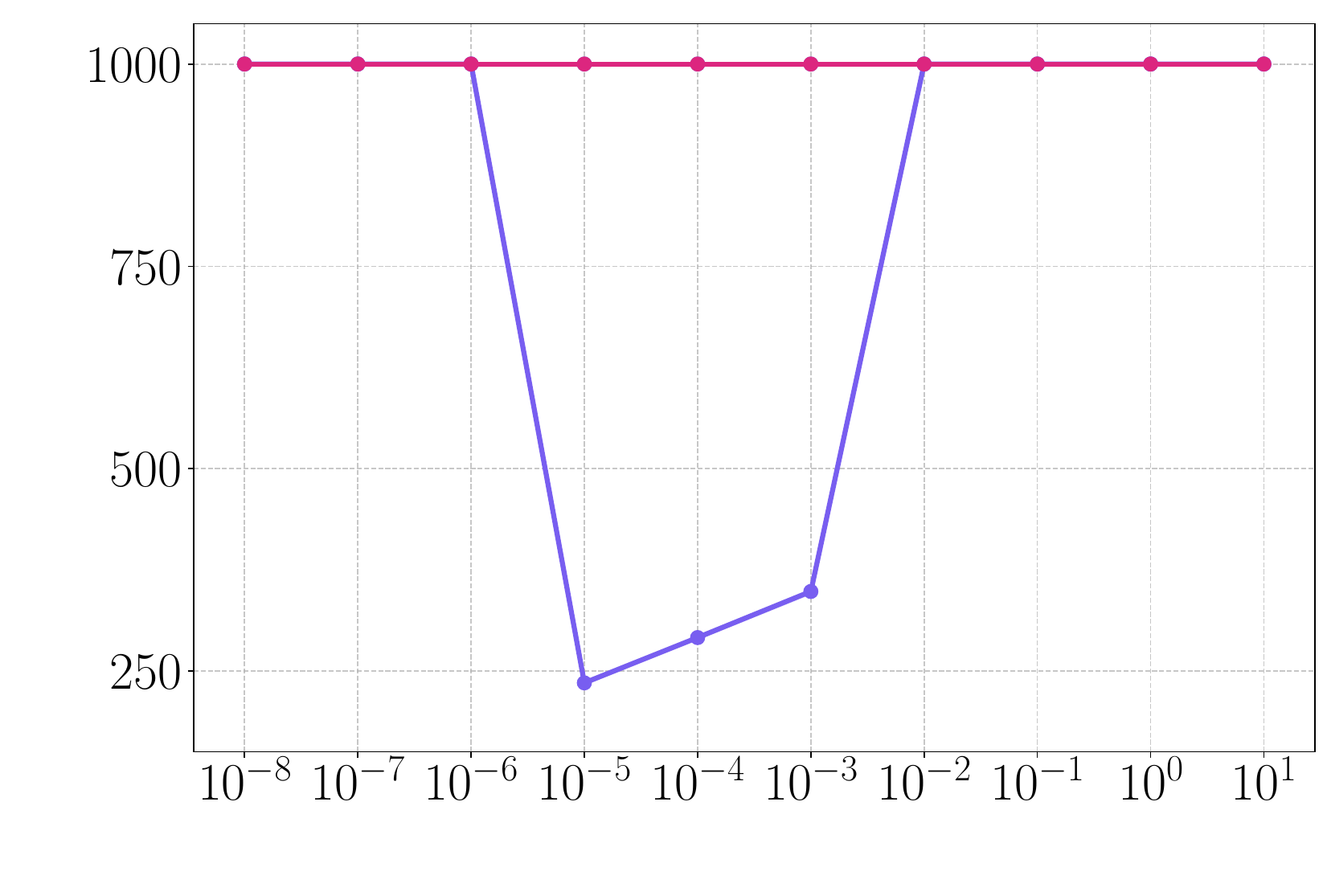}} \\
        \stepcounter{imagerow}\raisebox{45pt}{%
          \rotatebox[origin=c]{90}{\strut \vspace{0cm}\texttt{$q=0.97$}}} &
        \hspace{0.2cm}\subfloat[Well-conditioned]{\includegraphics[width=0.31\textwidth]{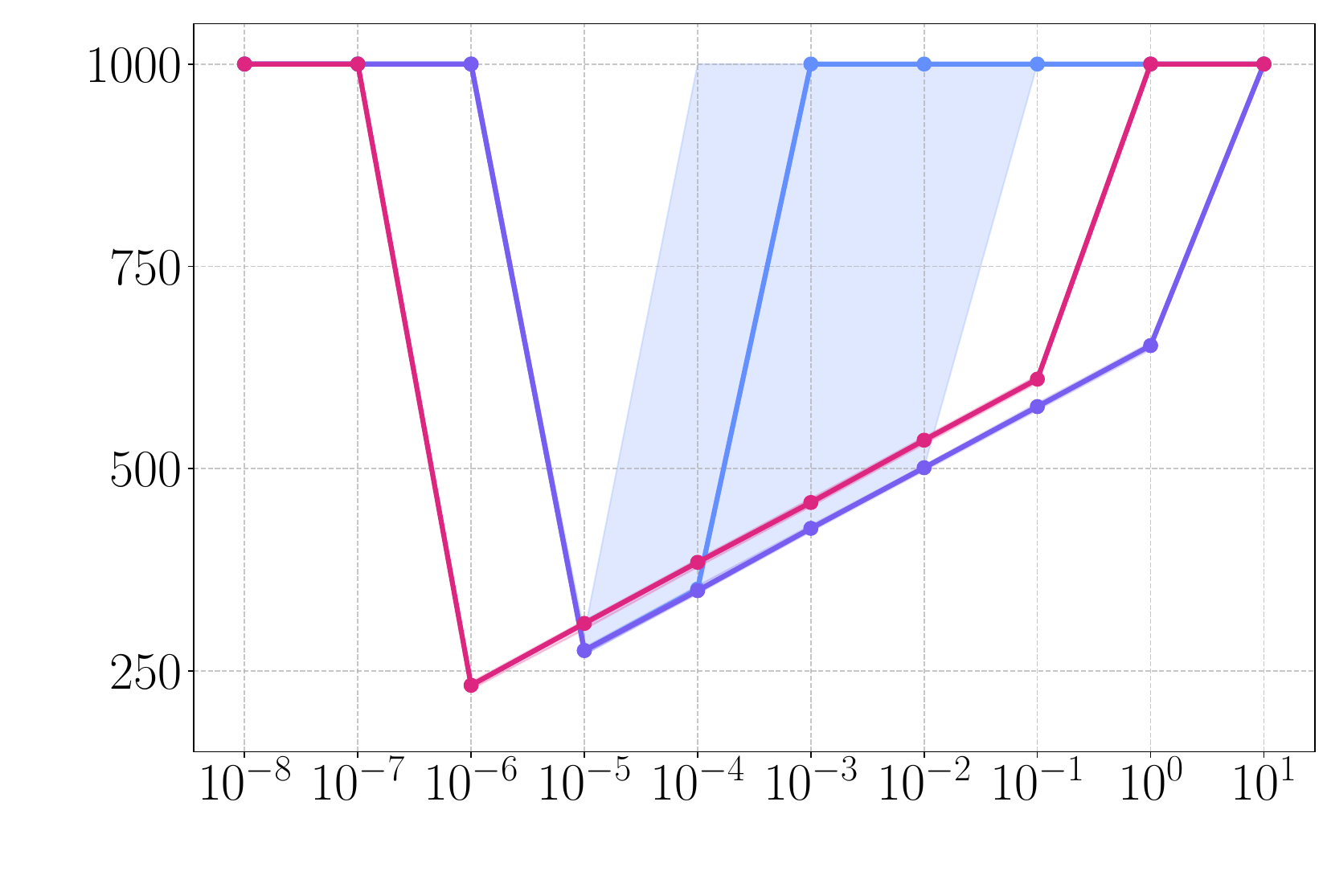}} &
        \subfloat[Ill-conditioned]{\includegraphics[width=0.31\textwidth]{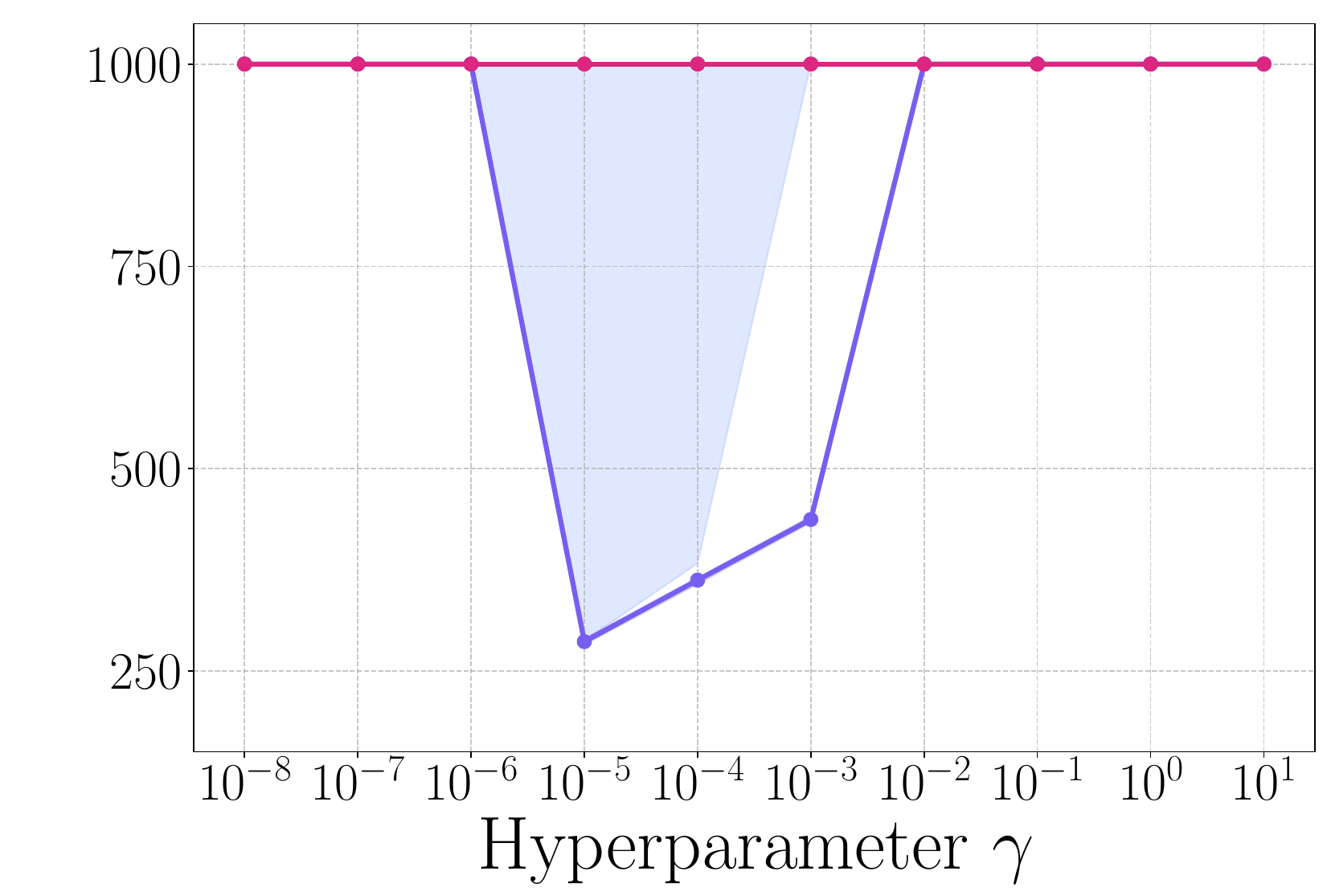}} &
        \subfloat[Ill-conditioned, Overparameterized]{\includegraphics[width=0.31\textwidth]{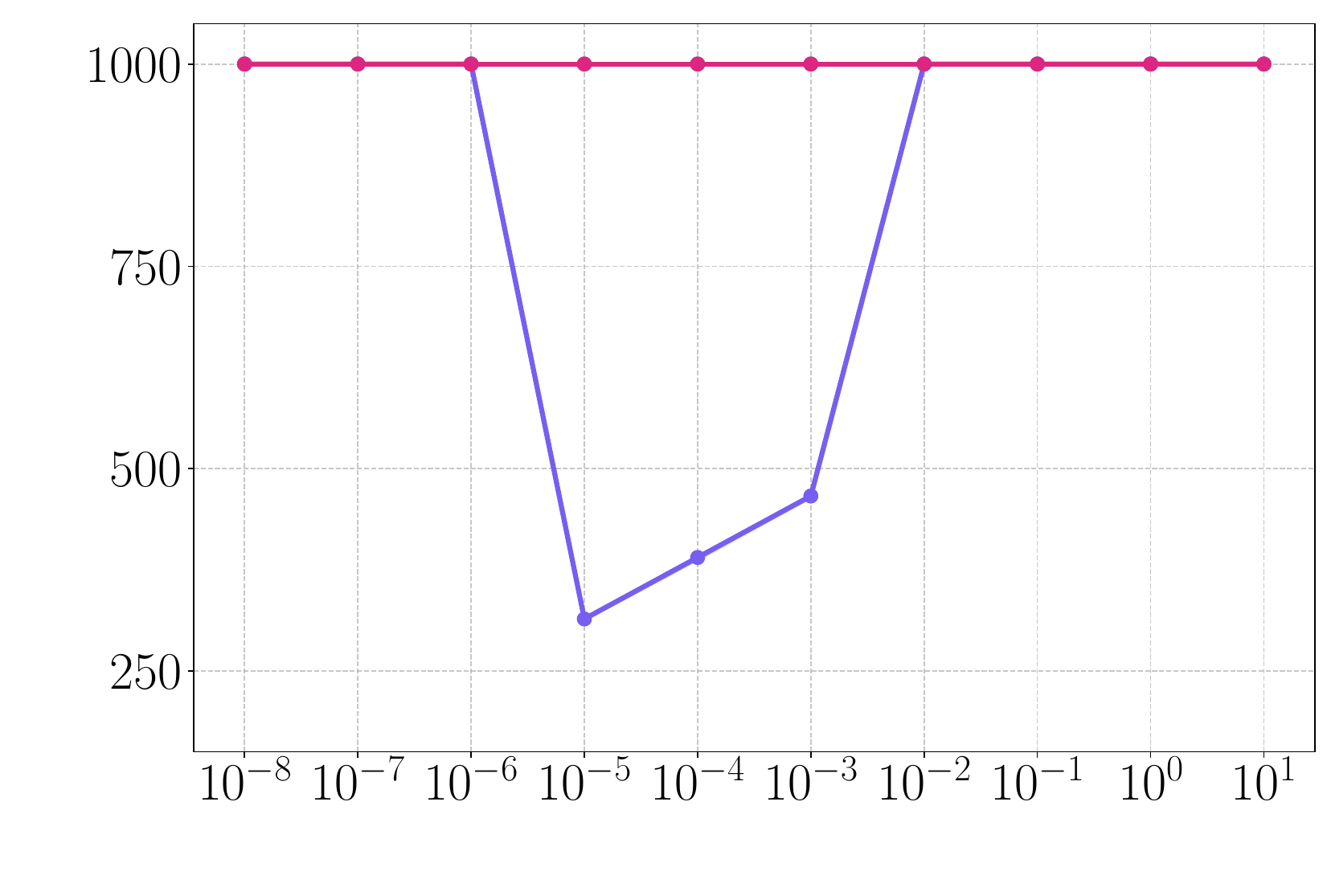}} \\
    \end{tabular}
    \caption{Median number of iterations to achieve convergence ($100$ draws) versus hyperparameter $\gamma$. We declare that a method converged when it reaches a relative error of $10^{-8}$ and cap the maximum number of iterations to $1000.$ The shaded area represents the $5$th and $95$th percentiles, respectively. 
    }
    \label{fig: sensitivity}
\end{figure}}{}

\paragraph{Experiment \ifbool{showSquare}{3}{2}: robustness to gross outliers.}
\ifbool{showSquare}{For our last batch of experiments,}{Next,} we test the ability of our method to solve the \ifbool{showSquare}{PSD, $\ell_1$ norm formulation,}{problem} with different levels of gross outliers. We set the dimension to $d =30$ and consider more aggressive ill-conditioning by varying $\tau \in\{1, 10^4\}$. We corrupt the vector \(b\) via \eqref{eq:outliers} where the outliers are set to $\eta_i = \mathcal{A}(\overline{M})_i$ for some other random matrix $\overline{M} \in \SSS^d_+.$
We vary the corruption level $p_{\mathrm{fail}} = \#\mathcal{I}/m$ between $0$ and $1/2.$ Notice that when $p_{\mathrm{fail}} > 1/2$, the solution switches to $\overline{M}.$
We compare against the standard \ifbool{showSquare}{and Gauss-Newton subgradient methods \cite{davis2022linearly}}{}. The Polyak stepsize is not applicable because the true minimum value \(\min f\) is unknown. Instead, we use Configuration~\ref{assum:geometric} with $
  \lambda = 10^{-5}, \gamma = 10^{-4},  \text{ and } q=0.97$. Figure~\ref{fig:transition_plot_matrix} displays the results with phase transition plots. For each pair \((m, p_{\mathrm{fail}})\), we run $20$ problem instances and report the success ratio. A run is successful if, within $500$ iterations, it achieves a relative error to fall below \(\epsilon = 10^{-8}\).\ifbool{showSquare}{ The Gauss-Newton preconditioned method exhibits unpredictable behavior when employing these geometrically decaying stepsizes; indeed, the guarantees in~\cite{davis2022linearly} do not cover this stepsize strategy.}{} \ifbool{showSquare}{On the other hand, Algorithm~\ref{alg:LM} displays more stable performance, supporting our theory.}{We observe that Algorithm~\ref{alg:LM} is able to recover the ground truth even with ill-conditioning and overparameterization.}

\ifbool{showSquare}{
\begin{figure}[t!]
  \def\arraystretch{0}%
  \setlength{\imagewidth}{\dimexpr \textwidth - 16\tabcolsep}%
  \divide \imagewidth by 3

  \begin{tabular}{@{}cIIIc@{}}                 %
    \stepcounter{imagerow}\raisebox{0.23\imagewidth}{\rotatebox[origin=c]{90}{\strut \shortstack[c]{\texttt{Polyak}\\\texttt{Subgradient}}}} &
    \includegraphics[width=0.95\imagewidth, trim=40 820 200 100, clip]{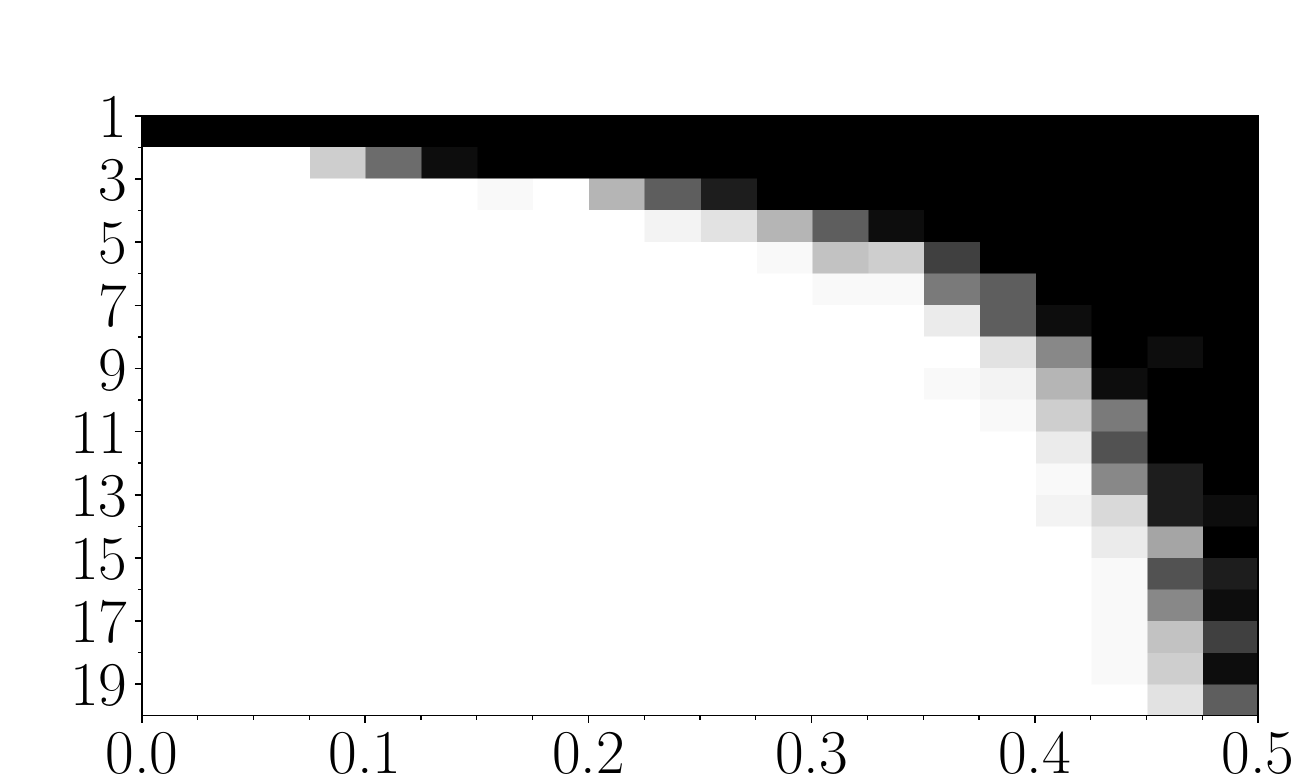} &
    \includegraphics[width=0.95\imagewidth, trim=40 820 200 100, clip]{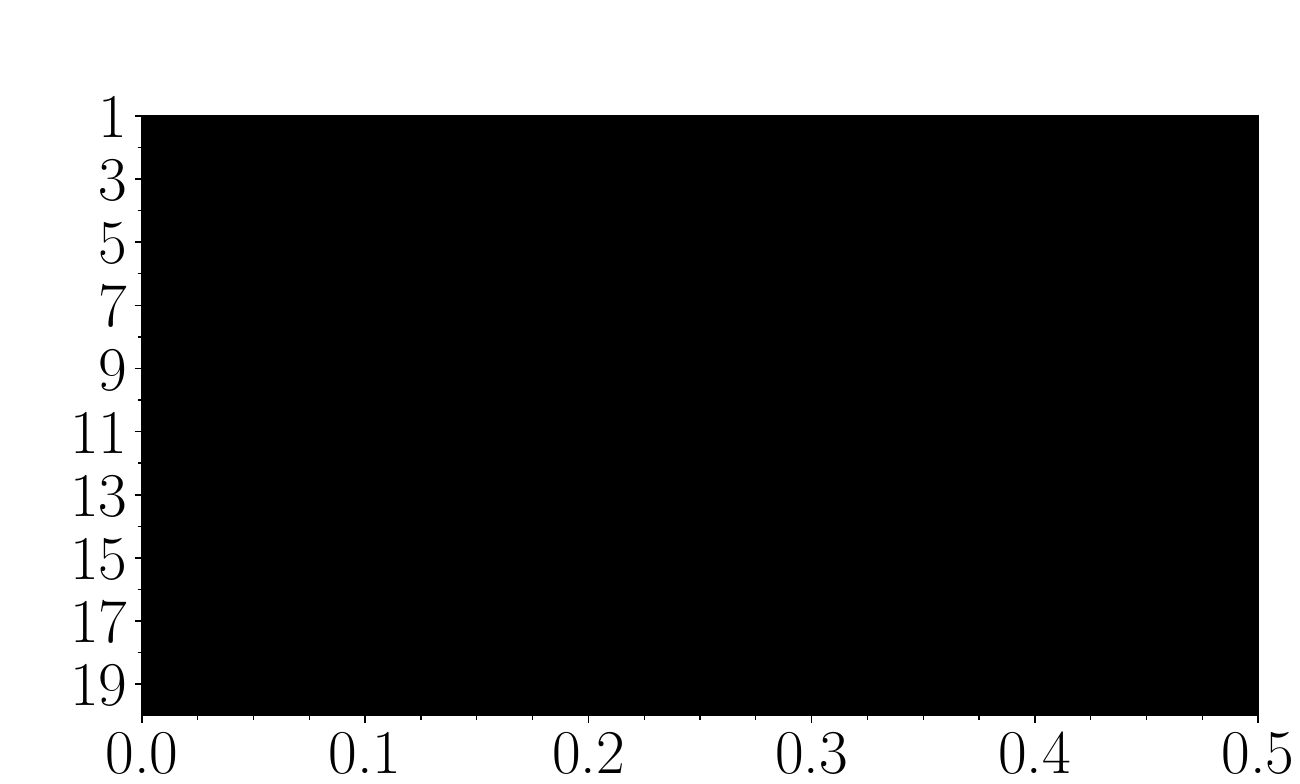} &
    \includegraphics[width=0.95\imagewidth, trim=40 820 200 100, clip]{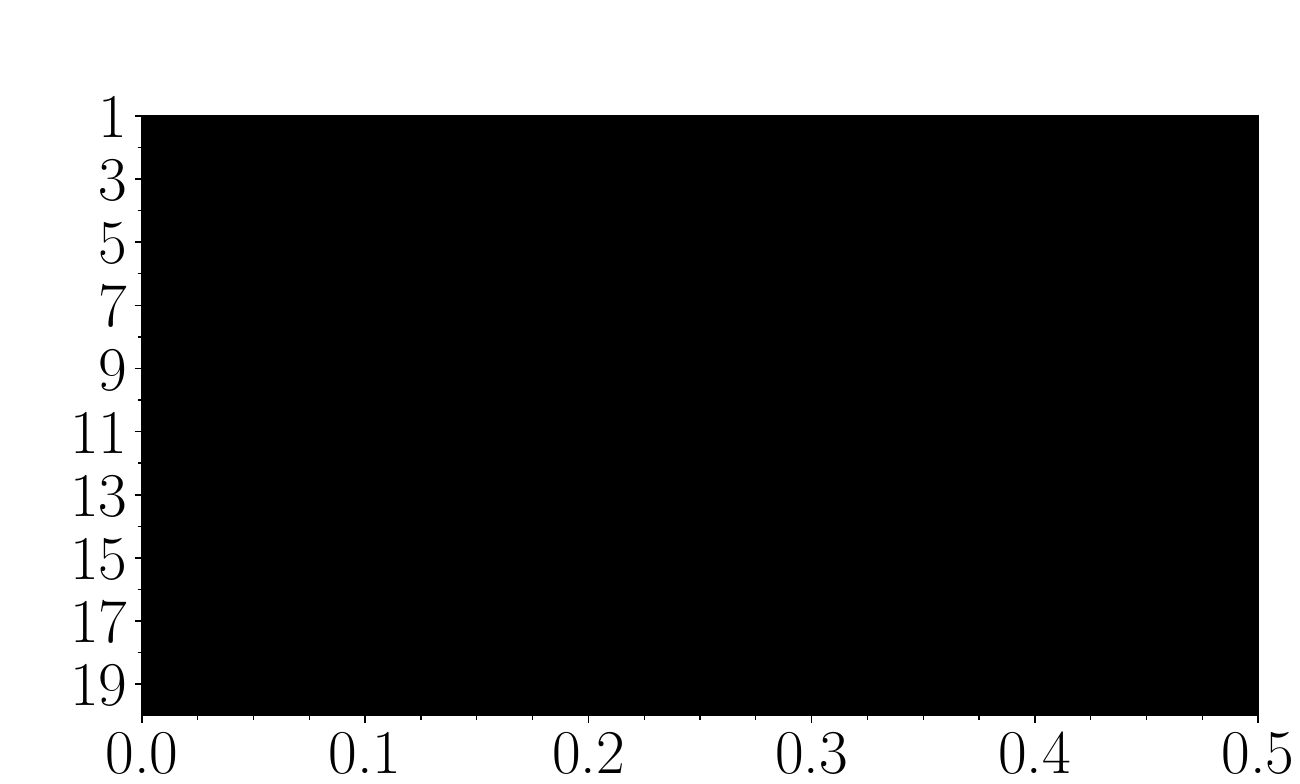} &
    \raisebox{65pt}{\multirow{3}{*}{{\includegraphics[height=220pt]{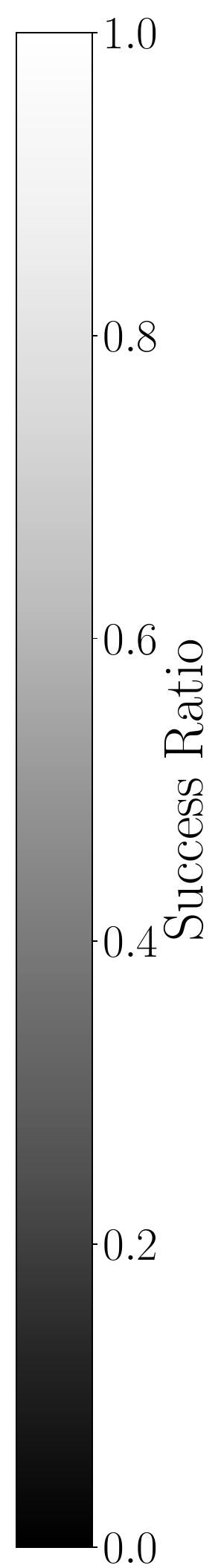}}}}\\[2\tabcolsep]
    \stepcounter{imagerow}\raisebox{30pt}{\rotatebox[origin=c]{90}{\strut  \shortstack[c]{\texttt{Gauss-}\\\texttt{Newton}}}} &
    \raisebox{66pt}{\includegraphics[width=0.95\imagewidth, trim=40 820 200 820, clip]{experiments/exp1_outliers_vs_measurements/_2,1_transition_plot.pdf}} &
    \raisebox{66pt}{\includegraphics[width=0.95\imagewidth, trim=40 820 200 820, clip]{experiments/exp1_outliers_vs_measurements/_2,100_transition_plot.pdf}} &
    \raisebox{66pt}{\includegraphics[width=0.95\imagewidth, trim=40 820 200 820, clip]{experiments/exp1_outliers_vs_measurements/_5,100_transition_plot.pdf}} &\\[2\tabcolsep]
    \stepcounter{imagerow}\raisebox{45pt}{\rotatebox[origin=c]{90}{\strut  \shortstack[c]{\texttt{LMM}\\\texttt{ (ours)}}}} &
    \raisebox{22pt}{\includegraphics[width=0.95\imagewidth, trim=40 130 200 820, clip]{experiments/exp1_outliers_vs_measurements/_2,1_transition_plot.pdf}} &
    \raisebox{15pt}{\includegraphics[width=0.95\imagewidth, trim=40 100 200 820, clip]{experiments/exp1_outliers_vs_measurements/_2,100_transition_plot.pdf}} &
    \raisebox{22pt}{\includegraphics[width=0.95\imagewidth, trim=40 130 200 820, clip]{experiments/exp1_outliers_vs_measurements/_5,100_transition_plot.pdf}} &\\[2\tabcolsep]
    \setcounter{imagecolumn}{0} & (a) Well‑conditioned &
    (b) Ill‑conditioned &
    (c) Ill‑conditioned,\\ & & &  \qquad overparameterized &
  \end{tabular}
  \caption{Matrix sensing transition plots of success rates (in \%) over 20 trials for each \((m, p_{\mathrm{fail}})\). Success is declared when the relative error is below \(\epsilon=10^{-8}\) with an iteration budget of 500.}
  \label{fig:transition_plot_matrix}
\end{figure}}{
\begin{figure}[t!]
\centering
\begin{tikzpicture}
    \node (tab) at (0,0) {%
      \begin{adjustbox}{valign=m}
      \begin{tabular}{@{}cIIc@{}}
        \raisebox{35pt}{\rotatebox[origin=c]{90}{\strut \shortstack[c]{\texttt{Polyak}\\\texttt{Subgradient}}}} &
        \includegraphics[width=0.3\textwidth, trim=40 820 200 100, clip]{experiments/exp1_outliers_vs_measurements/_2,1_transition_plot.pdf} &
        \includegraphics[width=0.3\textwidth, trim=40 820 200 100, clip]{experiments/exp1_outliers_vs_measurements/_5,100_transition_plot.pdf} &
        \raisebox{66.5pt}{\multirow{2}{*}{\includegraphics[height=156pt]{experiments/exp1_outliers_vs_measurements/ee.pdf}}} \\
        \raisebox{42pt}{\rotatebox[origin=c]{90}{\strut \shortstack[c]{\texttt{LMM (ours)}}}} &
        \raisebox{0pt}{\includegraphics[width=0.3\textwidth, trim=40 130 200 820, clip]{experiments/exp1_outliers_vs_measurements/_2,1_transition_plot.pdf}} &
        \raisebox{0pt}{\includegraphics[width=0.3\textwidth, trim=40 130 200 820, clip]{experiments/exp1_outliers_vs_measurements/_5,100_transition_plot.pdf}} & \\
        & \multicolumn{2}{c}{\raisebox{10pt}{\scriptsize Corruption level}} & \\[0pt]
        & (a) Well‑conditioned & (b) Ill‑conditioned, \\ 
        & & \qquad overparameterized & 
      \end{tabular}
      \end{adjustbox}
    };
    \node[rotate=90, anchor=south] at ([xshift=50pt, yshift=20pt]tab.west) {\scriptsize Measurement ratio $m/\!2d$};
\end{tikzpicture}

\caption{
  Matrix sensing transition plots of success rates (\%) over 20 trials for each $(m, p_{\mathrm{fail}})$. Success is declared when the relative error is below $\epsilon=10^{-8}$ with an iteration budget of 500. 
}
\label{fig:transition_plot_matrix}
\vspace{-10pt}
\end{figure}
}

\subsection{Tensor factorization and sensing}\label{sec:tensor-exp}
\ifbool{showSquare}{
\begin{figure}[t]
    \centering

    \begin{subfigure}[b]{0.48\textwidth}
        \centering
        \includegraphics[width=\textwidth]{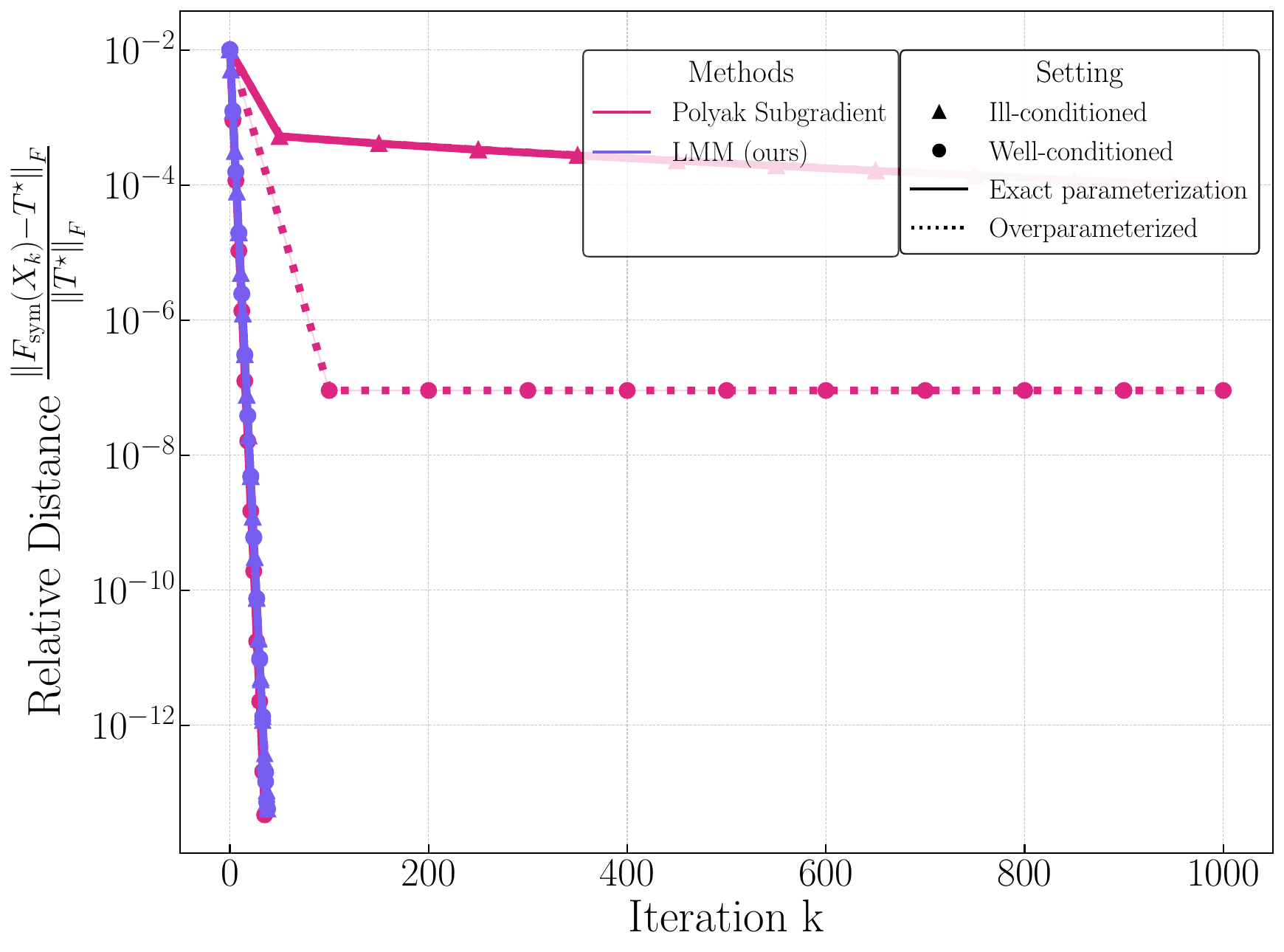}
        \caption{Symmetric Tensor ($d=500$)}
    \end{subfigure}
    \hfill
    \begin{subfigure}[b]{0.48\textwidth}
        \centering
        \includegraphics[width=\textwidth]{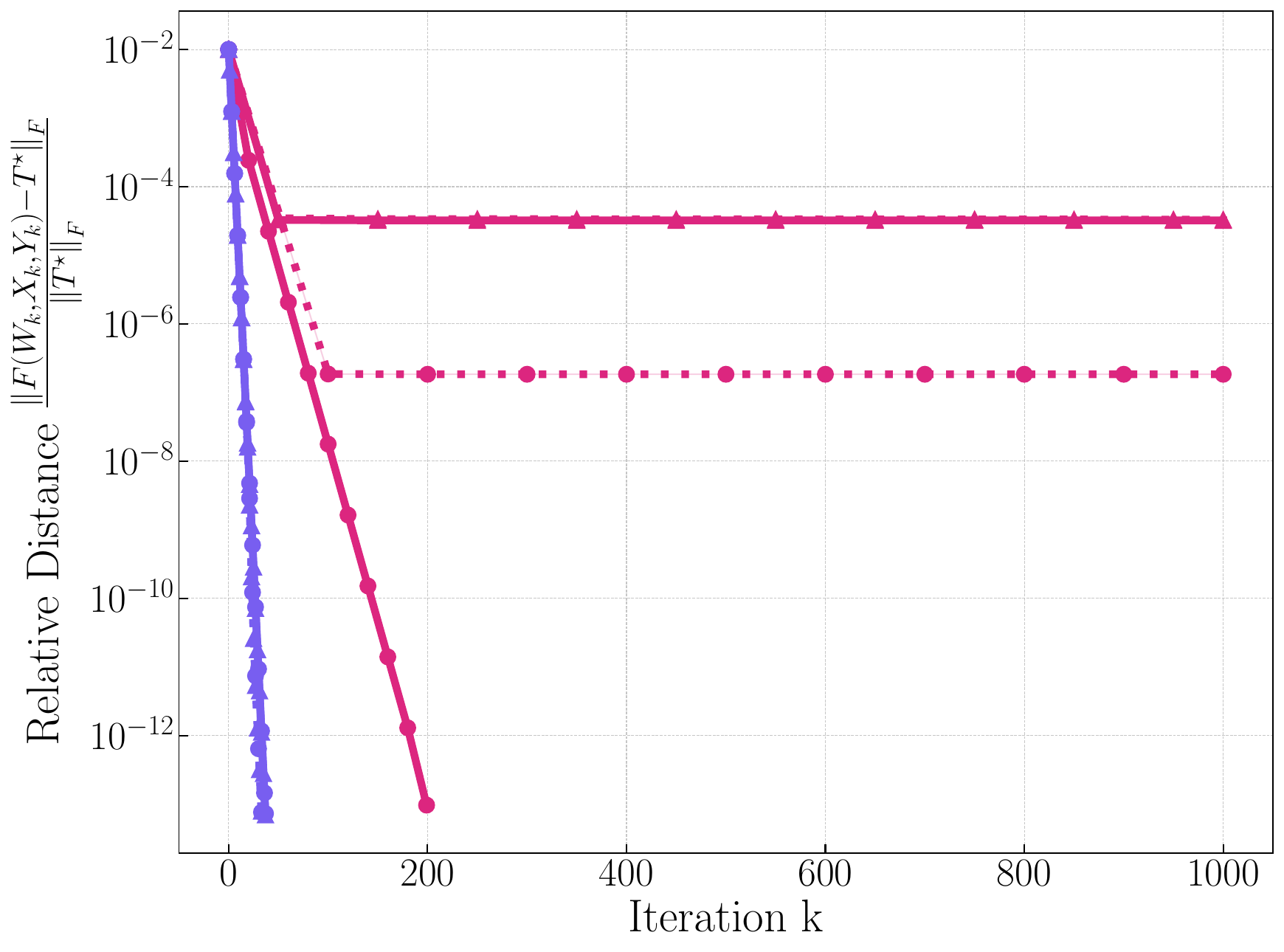}
        \caption{Asymmetric Tensor ($d=500$)}
    \end{subfigure}
    
    \caption{Tensor factorization with the $\ell_2$-norm. We use $r^\star = 2$, $r\in \{2,5\}$.}
    \label{fig: tensor_smooth}
\end{figure}

\begin{figure}[t]
    \centering
    
    \begin{subfigure}[b]{0.48\textwidth}
        \centering
        \includegraphics[width=\textwidth]{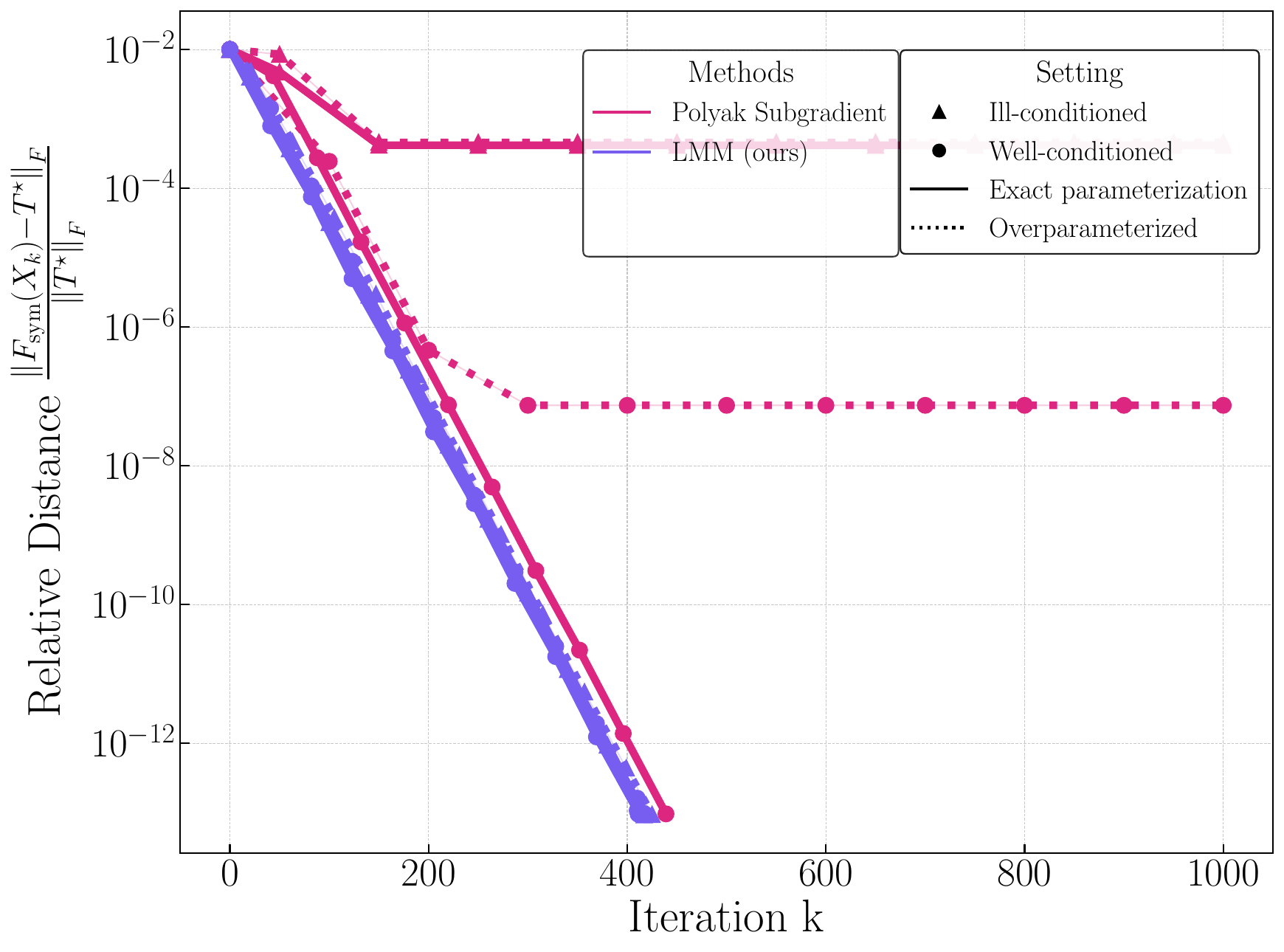}
        \caption{Symmetric Tensor ($d=50$)}
    \end{subfigure}
    \hfill
    \begin{subfigure}[b]{0.48\textwidth}
        \centering
        \includegraphics[width=\textwidth]{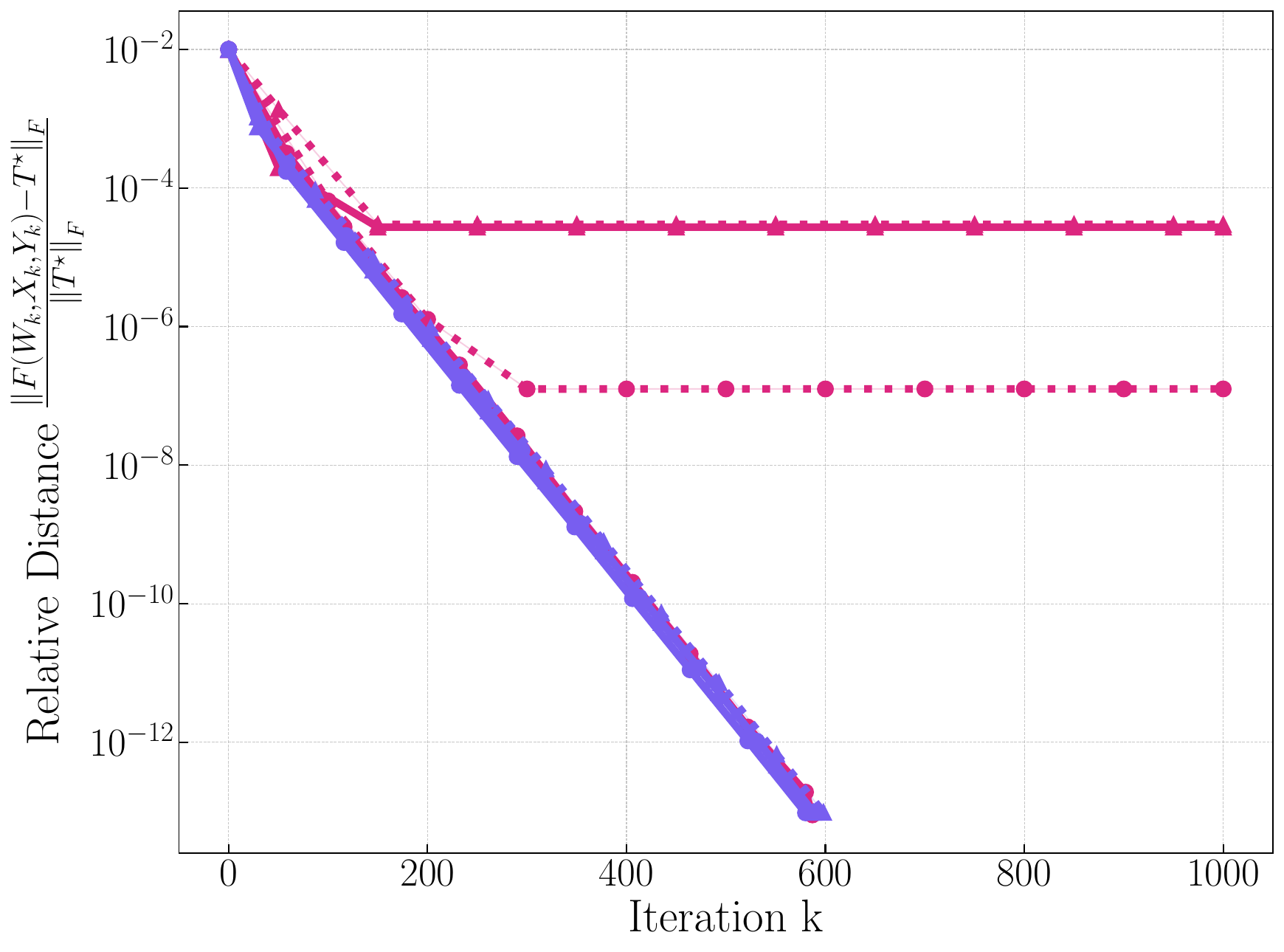}
        \caption{Asymmetric Tensor ($d=50$)}
    \end{subfigure}
    
    \caption{Robust tensor sensing with the $\ell_1$-norm. We use $m =5dr$ \ifbool{showSquare}{($30dr$ for asymmetric)}{} with $r^\star = 2$, $r\in \{2,5\}$ and $10\%$ of gross outliers.}
    \label{fig: tensor_nonsmooth}
\end{figure}}{

\begin{figure}[t]
    \centering

    \begin{subfigure}[b]{0.48\textwidth}
        \centering
        \includegraphics[width=\textwidth]{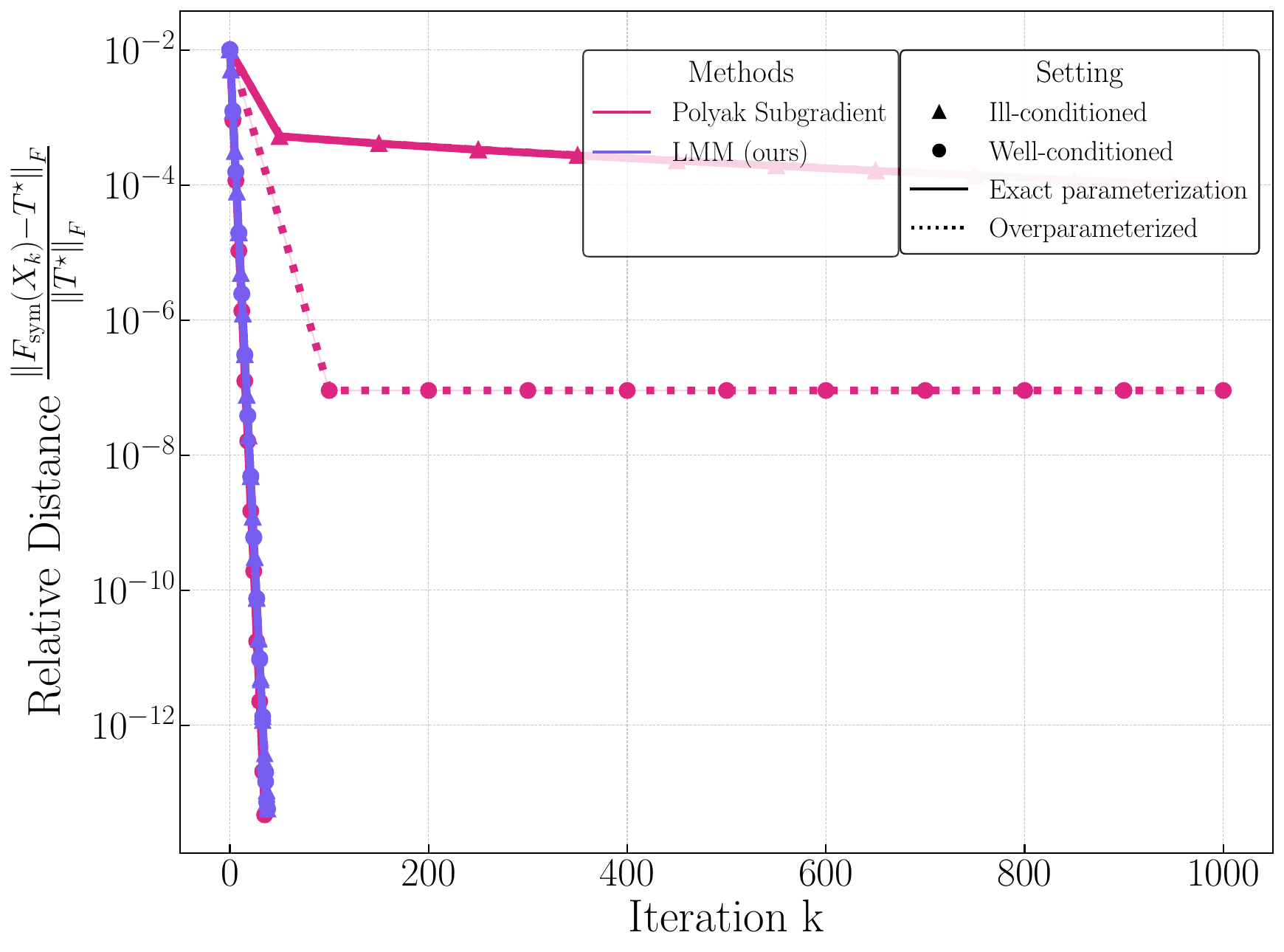}
        \caption{Tensor factorization ($d=500$)}
        \label{fig: tensor factorization}
    \end{subfigure}
    \hfill
    \begin{subfigure}[b]{0.48\textwidth}
        \centering
        \includegraphics[width=\textwidth]{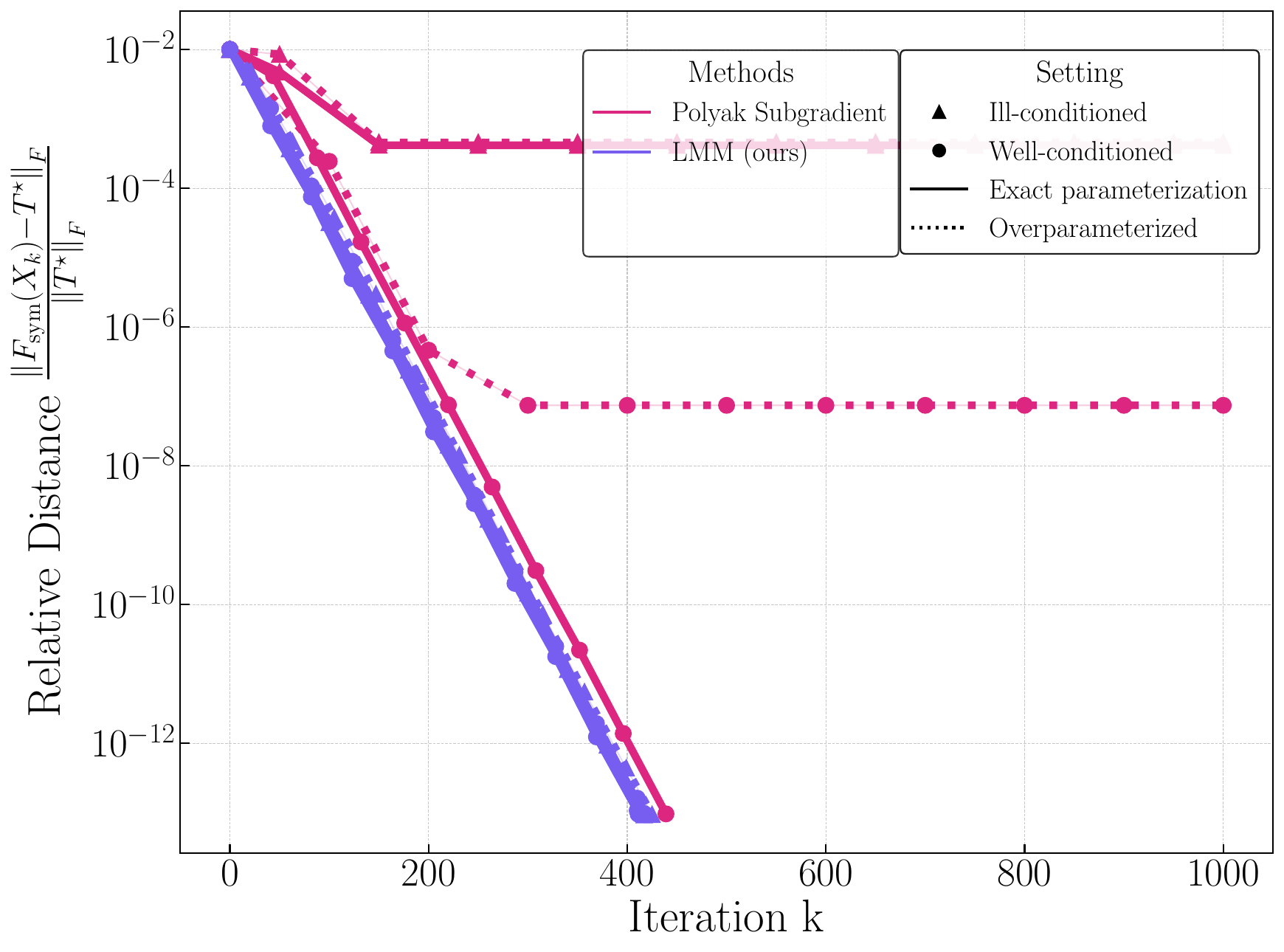}
        \caption{Robust tensor sensing $(d=50)$}
        \label{fig: tensor sensing}
    \end{subfigure}
    
    \caption{Tensor factorization with the $\ell_2$-norm (left), and tensor sensing under the $\ell_1$-loss with $m=5dr$ measurements and $10\%$ of gross outliers (right).\MD{Subfigure (a) should have the legend, the other subfigure shouldn't include it. You included 2 in Figure 7 and none in Figure 8. It should be one and one. Fix all the other plots as well.}}
    \label{fig: tensor}
\end{figure}

}

Finally, in our last batch of experiments, we evaluate Algorithm~\ref{alg:LM} on both\ifbool{showSquare}{}{ symmetric} tensor factorization~\eqref{eq:tensor-problems} and\ifbool{showSquare}{}{ symmetric} robust tensor sensing. While our theoretical guarantees address only the factorization setting, the empirical results suggest that Algorithm~\ref{alg:LM} also works for tensor sensing. We leave the formal analysis of this case as an open question for future work.

    \paragraph{Setup.} For the factorization problem, we use the $\ell_2$-norm $h(T) = \norm{T-T^\star}{2}$. For the sensing problem, we use the \(\ell_1\)-loss \(h(T)=\|\mathcal{A}(T)-b\|_1\) with \(\mathcal{A}\) a linear measurement map\ifbool{showSquare}{ and consider both symmetric and asymmetric CP‐factorizations \eqref{eq:tensor-facotrizations}}{}. We generate \ifbool{showSquare}{factor matrices \(W^\star,X^\star,Y^\star\in\mathbb{R}^{d\times r^\star}\)}{the factor matrix $X^\star \in \RR^{d\times r}$} by drawing \(U\in\mathbb{R}^{d\times r^\star}\) uniformly with \(U^\top U=I\) and setting \(X^\star=U\,D^{1/3}\), where the diagonal matrix \(D\) has entries spanning \([1/\tau,1]\); the ground‐truth tensor is then \ifbool{showSquare}{\(T^\star=\Fsym(X^\star)\) or \(T^\star=\Fasym(W^\star,X^\star,Y^\star)\) depending on the experiment}{$T^\star=\Fsym(X^\star)$}. All methods are initialized at random with relative error \(10^{-2}\). We set \(d=500\) for factorization and \(d=50\) for sensing, vary \(\tau\in\{1,100\}\) and \(r\in\{2,5\}\), and draw \(\mathcal{A}\) with i.i.d.\ \(\mathcal{N}(0,1/m)\) entries with \(m=5dr\) \ifbool{showSquare}{(or \(30dr\) for asymmetric)}{}, taking observations \(b=\mathcal{A}(T^\star)\). 

\paragraph{Baselines.} 
To our knowledge, no preconditioned first‐order method offers convergence guarantees for CP tensor factorization.\footnote{A provably convergent version of \texttt{ScaledGD} exists for the Tucker asymmetric factorization \cite{dong2022fast}.} Thus, we only test against the subgradient method. 
For the factorization experiment, we use Configuration~\ref{assum: dampingparameter} with $\gamma = \frac{1}{2}$ and $\lambda_k = 10^{-3}f(x_k).$ For the robust sensing experiment, we use  Configuration~\ref{assum:geometric} with $\gamma = 10^{-3}$ \ifbool{showSquare}{($10^{-5}$ for asymmetric)}{}, $\lambda=10^{-5}$ and $q=0.94$\ifbool{showSquare}{($0.96$ for asymmetric)}{}.

\paragraph{Discussion.} \ifbool{showSquare}{Figure~\ref{fig: tensor_smooth}}{Figure~\hyperref[fig: tensor factorization]{4(a)}} shows the output for large tensor factorization using the $\ell_2$-norm. Algorithm~\ref{alg:LM} consistently displays fast convergence. Although we do not include a plot here, we observe that the convergence is much faster when using the unscared $\ell_2$ compared to its squared counterpart.\ifbool{showSquare}{ This observation is consistent with the nonnegative least squares experiment.}{} \ifbool{showSquare}{Figure~\ref{fig: tensor_nonsmooth}}{Figure~\hyperref[fig: tensor sensing]{4(b)}} shows the convergence for tensor sensing using the $\ell_1$-norm with $10\%$ gross outliers of the form $\eta_i = \mathcal{A}(\overline{T})_i$ for a spurious signal $\overline{T}\in \RR^{d\times d \times d}$. Consistently, Algorithm~\ref{alg:LM} outperforms the subgradient method while remaining robust to ill-conditioning and overparameterization.

 \section*{Acknowledgements}
 We thank Philippe Toint for pointing us to relevant literature, particularly for pointing us to the work of Morrison~\cite{morrison1960methods}, whose contribution to the development of Algorithm~\ref{alg:LM} for nonlinear least‐squares has too often gone unrecognized.
 
\begingroup
  \let\clearpage\relax
  \bibliographystyle{abbrvnat}
  \bibliography{biblio}

\begin{thebibliography}{109}
\providecommand{\natexlab}[1]{#1}
\providecommand{\url}[1]{\texttt{#1}}
\expandafter\ifx\csname urlstyle\endcsname\relax
  \providecommand{\doi}[1]{doi: #1}\else
  \providecommand{\doi}{doi: \begingroup \urlstyle{rm}\Url}\fi

\bibitem[Acar et~al.(2011{\natexlab{a}})Acar, Dunlavy, and Kolda]{Acar2011}
E.~Acar, D.~M. Dunlavy, and T.~G. Kolda.
\newblock A scalable optimization approach for fitting canonical tensor decompositions.
\newblock \emph{Journal of Chemometrics}, 25\penalty0 (2):\penalty0 67--86, 2011{\natexlab{a}}.
\newblock \doi{10.1002/cem.1335}.

\bibitem[Acar et~al.(2011{\natexlab{b}})Acar, Dunlavy, Kolda, and Mørup]{Acar11Netflix}
E.~Acar, D.~M. Dunlavy, T.~G. Kolda, and M.~Mørup.
\newblock Scalable tensor factorizations for incomplete data.
\newblock In \emph{Proceedings of the 2011 SIAM International Conference on Data Mining (SDM)}, pages 701--712, Mesa, AZ, 2011{\natexlab{b}}.

\bibitem[Agarwal et~al.(2010)Agarwal, Snavely, Seitz, and Szeliski]{lm5a}
S.~Agarwal, N.~Snavely, S.~M. Seitz, and R.~Szeliski.
\newblock Bundle adjustment in the large.
\newblock In \emph{Proceedings of the 11th European Conference on Computer Vision (ECCV)}, volume 6312 of \emph{Lecture Notes in Computer Science}, pages 29--42. Springer, 2010.

\bibitem[Argyros and Hilout(2011)]{gn2}
I.~K. Argyros and S.~Hilout.
\newblock On the gauss--newton method.
\newblock \emph{Journal of Applied Mathematics and Computing}, 35:\penalty0 537--550, 2011.
\newblock \doi{10.1007/s12190-010-0377-8}.

\bibitem[Beck and Teboulle(2009)]{beck2009fast}
A.~Beck and M.~Teboulle.
\newblock A fast iterative shrinkage-thresholding algorithm for linear inverse problems.
\newblock \emph{SIAM journal on imaging sciences}, 2\penalty0 (1):\penalty0 183--202, 2009.

\bibitem[Benthem and Keenan(2004)]{VBK04}
M.~H.~V. Benthem and M.~R. Keenan.
\newblock Fast algorithm for the solution of large-scale non-negativity-constrained least squares problems.
\newblock \emph{Journal of Chemometrics: A Journal of the Chemometrics Society}, 18\penalty0 (10):\penalty0 441--450, 2004.

\bibitem[Bhatia(2013)]{bhatia2013matrix}
R.~Bhatia.
\newblock \emph{Matrix analysis}, volume 169.
\newblock Springer Science \& Business Media, 2013.

\bibitem[Bhojanapalli et~al.(2016)Bhojanapalli, Neyshabur, and Srebro]{bhojanapalli2016global}
S.~Bhojanapalli, B.~Neyshabur, and N.~Srebro.
\newblock Global optimality of local search for low rank matrix recovery.
\newblock \emph{Advances in Neural Information Processing Systems}, 29, 2016.

\bibitem[Bj{\"o}rck(2024)]{bjorck2024numerical}
{\AA}.~Bj{\"o}rck.
\newblock \emph{Numerical methods for least squares problems}.
\newblock SIAM, 2024.

\bibitem[Borwein and Lewis(2000)]{cov_lift}
J.~Borwein and A.~Lewis.
\newblock \emph{Convex analysis and nonlinear optimization}.
\newblock CMS Books in Mathematics/Ouvrages de Math\'ematiques de la SMC,~3. Springer-Verlag, New York, 2000.
\newblock ISBN 0-387-98940-4.
\newblock Theory and examples.

\bibitem[Burer and Monteiro(2003)]{burer2003nonlinear}
S.~Burer and R.~D. Monteiro.
\newblock A nonlinear programming algorithm for solving semidefinite programs via low-rank factorization.
\newblock \emph{Mathematical programming}, 95\penalty0 (2):\penalty0 329--357, 2003.

\bibitem[Burer and Monteiro(2005)]{burer2005local}
S.~Burer and R.~D. Monteiro.
\newblock Local minima and convergence in low-rank semidefinite programming.
\newblock \emph{Mathematical programming}, 103\penalty0 (3):\penalty0 427--444, 2005.

\bibitem[Burke(1985)]{burke1985descent}
J.~V. Burke.
\newblock Descent methods for composite nondifferentiable optimization problems.
\newblock \emph{Mathematical Programming}, 33:\penalty0 260--279, 1985.

\bibitem[Burke and Ferris(1995)]{burke1995gauss}
J.~V. Burke and M.~C. Ferris.
\newblock A gauss--newton method for convex composite optimization.
\newblock \emph{Mathematical Programming}, 71:\penalty0 179--194, 1995.

\bibitem[Candes and Plan(2011)]{candes2011tight}
E.~J. Candes and Y.~Plan.
\newblock Tight oracle inequalities for low-rank matrix recovery from a minimal number of noisy random measurements.
\newblock \emph{IEEE Transactions on Information Theory}, 57\penalty0 (4):\penalty0 2342--2359, 2011.

\bibitem[Cartis et~al.(2011)Cartis, Gould, and Toint]{cartis2011evaluation}
C.~Cartis, N.~I.~M. Gould, and P.~L. Toint.
\newblock On the evaluation complexity of composite function minimization with applications to nonconvex nonlinear programming.
\newblock \emph{SIAM Journal on Optimization}, 21\penalty0 (4):\penalty0 1721--1739, 2011.

\bibitem[Charisopoulos et~al.(2021)Charisopoulos, Chen, Davis, D{\'\i}az, Ding, and Drusvyatskiy]{charisopoulos2021low}
V.~Charisopoulos, Y.~Chen, D.~Davis, M.~D{\'\i}az, L.~Ding, and D.~Drusvyatskiy.
\newblock Low-rank matrix recovery with composite optimization: good conditioning and rapid convergence.
\newblock \emph{Foundations of Computational Mathematics}, 21\penalty0 (6):\penalty0 1505--1593, 2021.

\bibitem[Chen and Rockafellar(1997)]{chen1997convergence}
G.~H.-G. Chen and R.~T. Rockafellar.
\newblock Convergence rates in forward-backward splitting.
\newblock \emph{SIAM Journal on Optimization}, 7\penalty0 (2):\penalty0 421--444, 1997.

\bibitem[Chen(2011)]{gn1}
P.~Chen.
\newblock Hessian matrix vs. gauss–newton hessian matrix.
\newblock \emph{SIAM Journal on Numerical Analysis}, 49\penalty0 (4):\penalty0 1417--1435, 2011.
\newblock \doi{10.1137/100799988}.
\newblock URL \url{https://doi.org/10.1137/100799988}.

\bibitem[Chen and Wainwright(2015)]{chen2015fast}
Y.~Chen and M.~J. Wainwright.
\newblock Fast low-rank estimation by projected gradient descent: General statistical and algorithmic guarantees.
\newblock \emph{arXiv preprint arXiv:1509.03025}, 2015.

\bibitem[Chen et~al.(2021)Chen, Chi, Fan, and Ma]{Chen_2021}
Y.~Chen, Y.~Chi, J.~Fan, and C.~Ma.
\newblock Spectral methods for data science: A statistical perspective.
\newblock \emph{Foundations and Trends® in Machine Learning}, 14\penalty0 (5):\penalty0 566–806, 2021.
\newblock ISSN 1935-8245.
\newblock \doi{10.1561/2200000079}.
\newblock URL \url{http://dx.doi.org/10.1561/2200000079}.

\bibitem[Cheng and Zhao(2024)]{cheng2024accelerating}
C.~Cheng and Z.~Zhao.
\newblock Accelerating gradient descent for over-parameterized asymmetric low-rank matrix sensing via preconditioning.
\newblock In \emph{ICASSP 2024-2024 IEEE International Conference on Acoustics, Speech and Signal Processing (ICASSP)}, pages 7705--7709. IEEE, 2024.

\bibitem[Chi et~al.(2019)Chi, Lu, and Chen]{chi2019nonconvex}
Y.~Chi, Y.~M. Lu, and Y.~Chen.
\newblock Nonconvex optimization meets low-rank matrix factorization: An overview.
\newblock \emph{IEEE Transactions on Signal Processing}, 67\penalty0 (20):\penalty0 5239--5269, 2019.

\bibitem[Clarke et~al.(2008)Clarke, Ledyaev, Stern, and Wolenski]{clarke2008nonsmooth}
F.~H. Clarke, Y.~S. Ledyaev, R.~J. Stern, and P.~R. Wolenski.
\newblock \emph{Nonsmooth analysis and control theory}, volume 178.
\newblock Springer Science \& Business Media, 2008.

\bibitem[Davenport and Romberg(2016)]{davenport2016overview}
M.~A. Davenport and J.~Romberg.
\newblock An overview of low-rank matrix recovery from incomplete observations.
\newblock \emph{IEEE Journal of Selected Topics in Signal Processing}, 10\penalty0 (4):\penalty0 608--622, 2016.

\bibitem[Davis and Jiang(2022)]{davis2022linearly}
D.~Davis and T.~Jiang.
\newblock A linearly convergent gauss-newton subgradient method for ill-conditioned problems.
\newblock \emph{arXiv preprint arXiv:2212.13278}, 2022.

\bibitem[Davis et~al.(2018)Davis, Drusvyatskiy, MacPhee, and Paquette]{davis2018subgradient}
D.~Davis, D.~Drusvyatskiy, K.~J. MacPhee, and C.~Paquette.
\newblock Subgradient methods for sharp weakly convex functions.
\newblock \emph{Journal of Optimization Theory and Applications}, 179:\penalty0 962--982, 2018.

\bibitem[Davis et~al.(2020)Davis, Drusvyatskiy, Kakade, and Lee]{davis2018stochastic}
D.~Davis, D.~Drusvyatskiy, S.~Kakade, and J.~D. Lee.
\newblock Stochastic subgradient method converges on tame functions.
\newblock \emph{Foundations of computational mathematics}, 20\penalty0 (1):\penalty0 119--154, 2020.

\bibitem[Davis et~al.(2024)Davis, Drusvyatskiy, and Jiang]{davis2024gradient}
D.~Davis, D.~Drusvyatskiy, and L.~Jiang.
\newblock Gradient descent with adaptive stepsize converges (nearly) linearly under fourth-order growth.
\newblock \emph{arXiv preprint arXiv:2409.19791}, 2024.

\bibitem[Diakonikolas et~al.(2022)Diakonikolas, Li, Padmanabhan, and Song]{DLP+22}
J.~Diakonikolas, C.~Li, S.~Padmanabhan, and C.~Song.
\newblock A fast scale-invariant algorithm for non-negative least squares with non-negative data.
\newblock In \emph{NeurIPS}, 2022.

\bibitem[D{\'\i}az(2019)]{diaz2019nonsmooth}
M.~D{\'\i}az.
\newblock The nonsmooth landscape of blind deconvolution.
\newblock \emph{NeurIPS Workshop: Optimization for Machine Learning}, 2019.

\bibitem[D{\'\i}az et~al.(2019)D{\'\i}az, Quiroz, and Velasco]{diaz2019local}
M.~D{\'\i}az, A.~J. Quiroz, and M.~Velasco.
\newblock Local angles and dimension estimation from data on manifolds.
\newblock \emph{Journal of Multivariate Analysis}, 173:\penalty0 229--247, 2019.

\bibitem[Ding and Wright(2023)]{ding2023squared}
L.~Ding and S.~J. Wright.
\newblock On squared-variable formulations.
\newblock \emph{arXiv preprint arXiv:2310.01784}, 2023.
\newblock Available at \url{https://arxiv.org/abs/2310.01784}.

\bibitem[Ding et~al.(2021)Ding, Jiang, Chen, Qu, and Zhu]{ding2021rank}
L.~Ding, L.~Jiang, Y.~Chen, Q.~Qu, and Z.~Zhu.
\newblock Rank overspecified robust matrix recovery: Subgradient method and exact recovery.
\newblock \emph{Advances in Neural Information Processing Systems}, 34:\penalty0 26767--26778, 2021.

\bibitem[Ding et~al.(2022)Ding, Qin, Jiang, Zhou, and Zhu]{ding2022validation}
L.~Ding, Z.~Qin, L.~Jiang, J.~Zhou, and Z.~Zhu.
\newblock A validation approach to over-parameterized matrix and image recovery.
\newblock \emph{arXiv preprint arXiv:2209.10675}, 2022.

\bibitem[Dong et~al.(2022)Dong, Tong, Ma, and Chi]{dong2022fast}
H.~Dong, T.~Tong, C.~Ma, and Y.~Chi.
\newblock Fast and provable tensor robust principal component analysis via scaled gradient descent.
\newblock \emph{arXiv preprint arXiv:2206.09109}, 2022.

\bibitem[Drusvyatskiy and Lewis(2016)]{drusvyatskiy2016error}
D.~Drusvyatskiy and A.~S. Lewis.
\newblock Error bounds, quadratic growth, and linear convergence of proximal methods.
\newblock \emph{Math. Oper. Res.}, 2016.
\newblock To appear; \texttt{arXiv:1602.06661}.

\bibitem[Drusvyatskiy and Paquette(2019)]{drusvyatskiy2017composite}
D.~Drusvyatskiy and C.~Paquette.
\newblock Efficiency of minimizing compositions of convex functions and smooth maps.
\newblock \emph{Mathematical Programming}, 178\penalty0 (1-2):\penalty0 503--558, 2019.

\bibitem[Duchi and Ruan(2017)]{duchi2017solving}
J.~C. Duchi and F.~Ruan.
\newblock Solving (most) of a set of quadratic equalities: composite optimization for robust phase retrieval.
\newblock Preprint, 2017.

\bibitem[Fan and Yuan(2005)]{lm2}
J.-y. Fan and Y.-x. Yuan.
\newblock On the quadratic convergence of the levenberg--marquardt method without nonsingularity assumption.
\newblock \emph{Computing}, 74\penalty0 (1):\penalty0 23--39, 2005.

\bibitem[Fischer et~al.(2024)Fischer, Izmailov, and Solodov]{fischer2024levenberg}
A.~Fischer, A.~F. Izmailov, and M.~V. Solodov.
\newblock The levenberg--marquardt method: an overview of modern convergence theories and more.
\newblock \emph{Computational Optimization and Applications}, 89\penalty0 (1):\penalty0 33--67, 2024.

\bibitem[Fletcher(2009)]{fletcher2009model}
R.~Fletcher.
\newblock A model algorithm for composite nondifferentiable optimization problems.
\newblock In \emph{Nondifferential and variational techniques in optimization}, pages 67--76. Springer, 2009.

\bibitem[Ge et~al.(2017)Ge, Jin, and Zheng]{ge2017no}
R.~Ge, C.~Jin, and Y.~Zheng.
\newblock No spurious local minima in nonconvex low rank problems: A unified geometric analysis.
\newblock In \emph{International Conference on Machine Learning}, pages 1233--1242. PMLR, 2017.

\bibitem[Giampouras et~al.(2024)Giampouras, Cai, and Vidal]{giampouras2024guarantees}
P.~Giampouras, H.~Cai, and R.~Vidal.
\newblock Guarantees of a preconditioned subgradient algorithm for overparameterized asymmetric low-rank matrix recovery.
\newblock \emph{Preprint}, 2024.

\bibitem[Goffin(1977)]{goffin1977convergence}
J.-L. Goffin.
\newblock On convergence rates of subgradient optimization methods.
\newblock \emph{Mathematical programming}, 13:\penalty0 329--347, 1977.

\bibitem[Hagan and Menhaj(1994)]{lm1a}
M.~T. Hagan and M.~B. Menhaj.
\newblock Training feedforward networks with the marquardt algorithm.
\newblock \emph{IEEE Transactions on Neural Networks}, 5\penalty0 (6):\penalty0 989--993, 1994.

\bibitem[Izmailov et~al.(2019)Izmailov, Solodov, and Uskov]{lm3}
A.~F. Izmailov, M.~V. Solodov, and E.~I. Uskov.
\newblock A globally convergent levenberg--marquardt method for equality-constrained optimization.
\newblock \emph{Computational Optimization and Applications}, 2019.
\newblock \doi{10.1007/s10589-019-00123-0}.
\newblock Early access.

\bibitem[Jia et~al.(2024)Jia, Fangchen, Meng, and Sun]{jiaglobally}
X.~Jia, F.~Fangchen, D.~Meng, and D.~Sun.
\newblock Globally q-linear gauss-newton method for overparameterized non-convex matrix sensing.
\newblock In \emph{The Thirty-eighth Annual Conference on Neural Information Processing Systems}, 2024.

\bibitem[Jiang et~al.(2023)Jiang, Chen, and Ding]{jiang2023algorithmic}
L.~Jiang, Y.~Chen, and L.~Ding.
\newblock Algorithmic regularization in model-free overparametrized asymmetric matrix factorization.
\newblock \emph{SIAM Journal on Mathematics of Data Science}, 5\penalty0 (3):\penalty0 723--744, 2023.

\bibitem[Jin et~al.(2023)Jin, Li, Lyu, Du, and Lee]{jin2023understanding}
J.~Jin, Z.~Li, K.~Lyu, S.~S. Du, and J.~D. Lee.
\newblock Understanding incremental learning of gradient descent: A fine-grained analysis of matrix sensing.
\newblock In \emph{International Conference on Machine Learning}, pages 15200--15238. PMLR, 2023.

\bibitem[Karim et~al.(2024)Karim, Dulal, and Navasca]{Karim_2024}
R.~G. Karim, D.~Dulal, and C.~Navasca.
\newblock A modified levenberg-marquardt algorithm for tensor cp decomposition in image compression.
\newblock In \emph{2024 Data Compression Conference (DCC)}, page 563–563. IEEE, Mar. 2024.
\newblock \doi{10.1109/dcc58796.2024.00080}.
\newblock URL \url{http://dx.doi.org/10.1109/DCC58796.2024.00080}.

\bibitem[Kim et~al.(2013)Kim, Sra, and Dhillon]{KSD13}
D.~Kim, S.~Sra, and I.~S. Dhillon.
\newblock A non-monotonic method for large-scale non-negative least squares.
\newblock \emph{Optimization Methods and Software}, 28\penalty0 (5):\penalty0 1012--1039, 2013.

\bibitem[Kolda and Bader(2009)]{kolda2009tensor}
T.~G. Kolda and B.~W. Bader.
\newblock Tensor decompositions and applications.
\newblock \emph{SIAM Review}, 51\penalty0 (3):\penalty0 455--500, 2009.
\newblock \doi{10.1137/07070111X}.

\bibitem[Krishnan and Soman(2022)]{cpimage}
K.~K. Krishnan and K.~P. Soman.
\newblock Canonical polyadic decomposition of eeg image tensor for bci applications.
\newblock In M.~Tuba, S.~Akashe, and A.~Joshi, editors, \emph{ICT Systems and Sustainability}, pages 819--826, Singapore, 2022. Springer Nature Singapore.
\newblock ISBN 978-981-16-5987-4.

\bibitem[Lan(2012)]{lan2012optimal}
G.~Lan.
\newblock An optimal method for stochastic composite optimization.
\newblock \emph{Mathematical Programming}, 133\penalty0 (1):\penalty0 365--397, 2012.

\bibitem[Lan(2016)]{lan2016gradient}
G.~Lan.
\newblock Gradient sliding for composite optimization.
\newblock \emph{Mathematical Programming}, 159\penalty0 (1):\penalty0 201--235, 2016.

\bibitem[Landsberg(2011)]{landsberg2011tensors}
J.~M. Landsberg.
\newblock \emph{Tensors: geometry and applications: geometry and applications}, volume 128.
\newblock American Mathematical Soc., 2011.

\bibitem[Laufer and Nadler(2025)]{laufer2025rgnmr}
E.~V. Laufer and B.~Nadler.
\newblock Rgnmr: A gauss-newton method for robust matrix completion with theoretical guarantees.
\newblock \emph{arXiv preprint arXiv:2505.12919}, 2025.

\bibitem[Lawson and Hanson(1995)]{LH95}
C.~L. Lawson and R.~J. Hanson.
\newblock \emph{Solving Least Squares Problems}, volume~15 of \emph{Classics in Applied Mathematics}.
\newblock Society for Industrial and Applied Mathematics (SIAM), Philadelphia, PA, 1995.
\newblock Revised reprint of the 1974 original.

\bibitem[Lee(2003)]{lee2003introduction}
J.~M. Lee.
\newblock \emph{Introduction to Smooth Manifolds}, volume 218 of \emph{Graduate Texts in Mathematics}.
\newblock Springer, New York, 2003.
\newblock ISBN 0-387-95495-3.

\bibitem[Levenberg(1944)]{levenberg1944method}
K.~Levenberg.
\newblock A method for the solution of certain non-linear problems in least squares.
\newblock \emph{Quarterly of applied mathematics}, 2\penalty0 (2):\penalty0 164--168, 1944.

\bibitem[Levin et~al.(2024)Levin, Kileel, and Boumal]{levin2024effect}
E.~Levin, J.~Kileel, and N.~Boumal.
\newblock The effect of smooth parametrizations on nonconvex optimization landscapes.
\newblock \emph{Mathematical Programming}, pages 1--49, 2024.

\bibitem[Lewis and Wright(2015)]{lewis2015proximal}
A.~S. Lewis and S.~J. Wright.
\newblock A proximal method for composite minimization.
\newblock \emph{Mathematical Programming}, pages 1--46, 2015.

\bibitem[Li and Speed(2000)]{li2000parametric}
L.~Li and T.~P. Speed.
\newblock Parametric deconvolution of positive spike trains.
\newblock \emph{The Annals of Statistics}, 28\penalty0 (5):\penalty0 1279--1301, 2000.

\bibitem[Li et~al.(2020)Li, Zhu, Man-Cho~So, and Vidal]{li2020nonconvex}
X.~Li, Z.~Zhu, A.~Man-Cho~So, and R.~Vidal.
\newblock Nonconvex robust low-rank matrix recovery.
\newblock \emph{SIAM Journal on Optimization}, 30\penalty0 (1):\penalty0 660--686, 2020.

\bibitem[Li et~al.(2018)Li, Ma, and Zhang]{li2018algorithmic}
Y.~Li, T.~Ma, and H.~Zhang.
\newblock Algorithmic regularization in over-parameterized matrix sensing and neural networks with quadratic activations.
\newblock In \emph{Conference On Learning Theory}, pages 2--47. PMLR, 2018.

\bibitem[Lin et~al.(2004)Lin, Lee, and Saul]{lin2004nonnegative}
Y.~Lin, D.~D. Lee, and L.~K. Saul.
\newblock Nonnegative deconvolution for time of arrival estimation.
\newblock In \emph{2004 IEEE International Conference on Acoustics, Speech, and Signal Processing}, volume~2, pages ii--377. IEEE, 2004.

\bibitem[Lions and Mercier(1979)]{lions1979splitting}
P.-L. Lions and B.~Mercier.
\newblock Splitting algorithms for the sum of two nonlinear operators.
\newblock \emph{SIAM Journal on Numerical Analysis}, 16\penalty0 (6):\penalty0 964--979, 1979.
\newblock \doi{10.1137/0716071}.

\bibitem[Liu et~al.(2025)Liu, Han, Tang, Tang, and Wang]{liu2025efficient}
Z.~Liu, Z.~Han, Y.~Tang, S.~Tang, and Y.~Wang.
\newblock Efficient over-parameterized matrix sensing from noisy measurements via alternating preconditioned gradient descent.
\newblock \emph{arXiv preprint arXiv:2502.00463}, 2025.

\bibitem[Luo and Zhang(2023)]{luo2023low}
Y.~Luo and A.~R. Zhang.
\newblock Low-rank tensor estimation via riemannian gauss-newton: Statistical optimality and second-order convergence.
\newblock \emph{The Journal of Machine Learning Research}, 24\penalty0 (1):\penalty0 18274--18321, 2023.

\bibitem[Ma et~al.(2019)Ma, Wang, Chi, and Chen]{ma2019implicit}
C.~Ma, K.~Wang, Y.~Chi, and Y.~Chen.
\newblock Implicit regularization in nonconvex statistical estimation: Gradient descent converges linearly for phase retrieval, matrix completion, and blind deconvolution.
\newblock \emph{Foundations of Computational Mathematics}, 2019.

\bibitem[Ma and Fattahi(2023)]{ma2023global}
J.~Ma and S.~Fattahi.
\newblock Global convergence of sub-gradient method for robust matrix recovery: Small initialization, noisy measurements, and over-parameterization.
\newblock \emph{Journal of Machine Learning Research}, 24\penalty0 (96):\penalty0 1--84, 2023.

\bibitem[Ma et~al.(2023)Ma, Bi, Lavaei, and Sojoudi]{ma2023geometric}
Z.~Ma, Y.~Bi, J.~Lavaei, and S.~Sojoudi.
\newblock Geometric analysis of noisy low-rank matrix recovery in the exact parametrized and the overparametrized regimes.
\newblock \emph{INFORMS Journal on Optimization}, 2023.

\bibitem[Marquardt(1963)]{marquardt1963algorithm}
D.~W. Marquardt.
\newblock An algorithm for least-squares estimation of nonlinear parameters.
\newblock \emph{Journal of the society for Industrial and Applied Mathematics}, 11\penalty0 (2):\penalty0 431--441, 1963.

\bibitem[Miranda(2024)]{Miranda24CPBayesian}
M.~Miranda.
\newblock A canonical polyadic tensor basis for fast bayesian estimation of multi‑subject brain activation patterns.
\newblock \emph{Frontiers in Neuroinformatics}, 18:\penalty0 1399391, Aug. 2024.
\newblock \doi{10.3389/fninf.2024.1399391}.
\newblock Published Aug 12 2024.

\bibitem[Mordukhovich(2006)]{Mord_1}
B.~S. Mordukhovich.
\newblock \emph{Variational Analysis and Generalized Differentiation I: Basic Theory}.
\newblock Grundlehren der mathematischen Wissenschaften, Vol 330, Springer, Berlin, 2006.
\newblock ISBN 3540254374.

\bibitem[Mor{\'e}(2006)]{more2006levenberg}
J.~J. Mor{\'e}.
\newblock The levenberg-marquardt algorithm: implementation and theory.
\newblock In \emph{Numerical analysis: proceedings of the biennial Conference held at Dundee, June 28--July 1, 1977}, pages 105--116. Springer, 2006.

\bibitem[Morrison(1960)]{morrison1960methods}
D.~D. Morrison.
\newblock Methods for nonlinear least squares problems and convergence proofs.
\newblock In J.~Lorell and F.~Yagi, editors, \emph{Proceedings of the Seminar on Tracking Programs and Orbit Determination}, pages 1--9, Pasadena, USA, 1960. Jet Propulsion Laboratory.

\bibitem[Myre et~al.(2017)Myre, Frahm, Lilja, and Saar]{MFLS17}
J.~M. Myre, E.~Frahm, D.~J. Lilja, and M.~O. Saar.
\newblock Tnt-nn: A fast active set method for solving large non-negative least squares problems.
\newblock In \emph{Procedia Computer Science}, volume 108, pages 755--764, 2017.

\bibitem[Mørup et~al.(2006)Mørup, Madsen, and Hansen]{Morup06EEG}
M.~Mørup, K.~H. Madsen, and L.~K. Hansen.
\newblock Decomposing neuroimaging data sets with parallel factor analysis.
\newblock \emph{NeuroImage}, 33\penalty0 (3):\penalty0 1094--1109, 2006.

\bibitem[Nesterov(2007)]{nesterov2007modified}
Y.~Nesterov.
\newblock Modified gauss--newton scheme with worst case guarantees for global performance.
\newblock \emph{Optimisation methods and software}, 22\penalty0 (3):\penalty0 469--483, 2007.

\bibitem[Nesterov(2018)]{nesterov2013introductory}
Y.~Nesterov.
\newblock \emph{Lectures on convex optimization}.
\newblock Springer, 2nd edition, 2018.

\bibitem[Nocedal and Wright(1999)]{nocedal1999numerical}
J.~Nocedal and S.~J. Wright.
\newblock \emph{Numerical optimization}.
\newblock Springer, 1999.

\bibitem[Polyak(1969)]{polyak1969minimization}
B.~T. Polyak.
\newblock Minimization of unsmooth functionals.
\newblock \emph{USSR Computational Mathematics and Mathematical Physics}, 9\penalty0 (3):\penalty0 14--29, 1969.

\bibitem[Pourbagher et~al.(2016)Pourbagher, Derakhshandeh, and Golshan]{lm3a}
R.~Pourbagher, S.~Y. Derakhshandeh, and M.~E.~H. Golshan.
\newblock Application of high-order levenberg--marquardt method for solving the power flow problem in ill-conditioned systems.
\newblock \emph{IET Generation, Transmission \& Distribution}, 10\penalty0 (12):\penalty0 3017--3022, 2016.

\bibitem[Pratt et~al.(1998)Pratt, Shin, and Hicks]{gn3}
R.~G. Pratt, C.~Shin, and G.~J. Hicks.
\newblock Gauss--newton and full newton methods in frequency--space seismic waveform inversion.
\newblock \emph{Geophysical Journal International}, 133\penalty0 (2):\penalty0 341--362, 1998.
\newblock \doi{10.1046/j.1365-246X.1998.00558.x}.

\bibitem[Pujol(2007)]{lm6a}
J.~Pujol.
\newblock The solution of nonlinear inverse problems and the levenberg--marquardt method.
\newblock \emph{Geophysics}, 72\penalty0 (4):\penalty0 W1--W16, 2007.

\bibitem[Rockafellar and Wets(2009)]{rockafellar2009variational}
R.~T. Rockafellar and R.~J.-B. Wets.
\newblock \emph{Variational analysis}, volume 317.
\newblock Springer Science \& Business Media, 2009.

\bibitem[Schacke(2004)]{schacke2004kronecker}
K.~Schacke.
\newblock On the kronecker product.
\newblock \emph{Master's thesis, University of Waterloo}, 2004.

\bibitem[Sidiropoulos et~al.(2000)Sidiropoulos, Bro, and Giannakis]{Sidiropoulos00CPArray}
N.~D. Sidiropoulos, R.~Bro, and G.~B. Giannakis.
\newblock Parallel factor analysis in sensor array processing.
\newblock \emph{IEEE Transactions on Signal Processing}, 48\penalty0 (8):\penalty0 2377--2388, 2000.

\bibitem[Soltanolkotabi et~al.(2025)Soltanolkotabi, St{\"o}ger, and Xie]{soltanolkotabi2025implicit}
M.~Soltanolkotabi, D.~St{\"o}ger, and C.~Xie.
\newblock Implicit balancing and regularization: Generalization and convergence guarantees for overparameterized asymmetric matrix sensing.
\newblock \emph{IEEE Transactions on Information Theory}, 2025.

\bibitem[St{\"o}ger and Soltanolkotabi(2021)]{stoger2021small}
D.~St{\"o}ger and M.~Soltanolkotabi.
\newblock Small random initialization is akin to spectral learning: Optimization and generalization guarantees for overparameterized low-rank matrix reconstruction.
\newblock \emph{Advances in Neural Information Processing Systems}, 34:\penalty0 23831--23843, 2021.

\bibitem[Szlam et~al.(2010)Szlam, Guo, and Osher]{szlam2010split}
A.~Szlam, Z.~Guo, and S.~Osher.
\newblock A split bregman method for non-negative sparsity penalized least squares with applications to hyperspectral demixing.
\newblock In \emph{2010 IEEE International Conference on Image Processing}, pages 1917--1920. IEEE, 2010.

\bibitem[Tong et~al.(2021{\natexlab{a}})Tong, Ma, and Chi]{tong2021accelerating}
T.~Tong, C.~Ma, and Y.~Chi.
\newblock Accelerating ill-conditioned low-rank matrix estimation via scaled gradient descent.
\newblock \emph{Journal of Machine Learning Research}, 22:\penalty0 1--63, 2021{\natexlab{a}}.

\bibitem[Tong et~al.(2021{\natexlab{b}})Tong, Ma, and Chi]{tong2021scaledsubgradient}
T.~Tong, C.~Ma, and Y.~Chi.
\newblock Low-rank matrix recovery with scaled subgradient methods: Fast and robust convergence without the condition number.
\newblock In \emph{2021 IEEE Data Science and Learning Workshop (DSLW)}, pages 1--6, 2021{\natexlab{b}}.
\newblock \doi{10.1109/DSLW51110.2021.9523407}.

\bibitem[Tong et~al.(2022{\natexlab{a}})Tong, Ma, and Chi]{tong2022accelerating}
T.~Tong, C.~Ma, and Y.~Chi.
\newblock Accelerating ill-conditioned robust low-rank tensor regression.
\newblock In \emph{ICASSP 2022-2022 IEEE International Conference on Acoustics, Speech and Signal Processing (ICASSP)}, pages 9072--9076. IEEE, 2022{\natexlab{a}}.

\bibitem[Tong et~al.(2022{\natexlab{b}})Tong, Ma, Prater-Bennette, Tripp, and Chi]{tong2022scaling}
T.~Tong, C.~Ma, A.~Prater-Bennette, E.~Tripp, and Y.~Chi.
\newblock Scaling and scalability: Provable nonconvex low-rank tensor estimation from incomplete measurements.
\newblock \emph{Journal of Machine Learning Research}, 23\penalty0 (163):\penalty0 1--77, 2022{\natexlab{b}}.

\bibitem[Tu et~al.(2016)Tu, Boczar, Simchowitz, Soltanolkotabi, and Recht]{tu2016low}
S.~Tu, R.~Boczar, M.~Simchowitz, M.~Soltanolkotabi, and B.~Recht.
\newblock Low-rank solutions of linear matrix equations via procrustes flow.
\newblock In \emph{International Conference on Machine Learning}, pages 964--973. PMLR, 2016.

\bibitem[Wind(2023)]{wind2023asymmetric}
J.~S. Wind.
\newblock Asymmetric matrix sensing by gradient descent with small random initialization.
\newblock \emph{arXiv preprint arXiv:2309.01796}, 2023.

\bibitem[Wright and Ma(2022)]{wright2022high}
J.~Wright and Y.~Ma.
\newblock \emph{High-dimensional data analysis with low-dimensional models: Principles, computation, and applications}.
\newblock Cambridge University Press, 2022.

\bibitem[Wright(1990)]{wright1990convergence}
S.~J. Wright.
\newblock Convergence of an inexact algorithm for composite nonsmooth optimization.
\newblock \emph{IMA Journal of Numerical Analysis}, 10\penalty0 (3):\penalty0 299--321, 1990.

\bibitem[Wu(2025)]{wu2025guaranteed}
T.~Wu.
\newblock Guaranteed nonconvex low-rank tensor estimation via scaled gradient descent.
\newblock \emph{arXiv preprint arXiv:2501.01696}, 2025.

\bibitem[Xiong et~al.(2023)Xiong, Ding, and Du]{xiong2023over}
N.~Xiong, L.~Ding, and S.~S. Du.
\newblock How over-parameterization slows down gradient descent in matrix sensing: The curses of symmetry and initialization.
\newblock \emph{arXiv preprint arXiv:2310.01769}, 2023.

\bibitem[Xu et~al.(2023)Xu, Shen, Chi, and Ma]{xu2023power}
X.~Xu, Y.~Shen, Y.~Chi, and C.~Ma.
\newblock The power of preconditioning in overparameterized low-rank matrix sensing.
\newblock In \emph{Proceedings of the 40th International Conference on Machine Learning (ICML)}, 2023.

\bibitem[Ye and Du(2021)]{ye2021global}
T.~Ye and S.~S. Du.
\newblock Global convergence of gradient descent for asymmetric low-rank matrix factorization.
\newblock \emph{Advances in Neural Information Processing Systems}, 34:\penalty0 1429--1439, 2021.

\bibitem[Zhang et~al.(2023)Zhang, Fattahi, and Zhang]{zhang2021preconditioned}
G.~Zhang, S.~Fattahi, and R.~Y. Zhang.
\newblock Preconditioned gradient descent for overparameterized nonconvex burer–monteiro factorization with global optimality certification.
\newblock \emph{Journal of Machine Learning Research}, 24:\penalty0 1--55, 2023.

\bibitem[Zhang et~al.(2021)Zhang, Bi, and Lavaei]{zhang2021general}
H.~Zhang, Y.~Bi, and J.~Lavaei.
\newblock General low-rank matrix optimization: Geometric analysis and sharper bounds.
\newblock \emph{Advances in Neural Information Processing Systems}, 34:\penalty0 27369--27380, 2021.

\bibitem[Zhu et~al.(2018)Zhu, Li, Tang, and Wakin]{zhu2018global}
Z.~Zhu, Q.~Li, G.~Tang, and M.~B. Wakin.
\newblock Global optimality in low-rank matrix optimization.
\newblock \emph{IEEE Transactions on Signal Processing}, 66\penalty0 (13):\penalty0 3614--3628, 2018.

\bibitem[Zhuo et~al.(2024)Zhuo, Kwon, Ho, and Caramanis]{zhuo2024computational}
J.~Zhuo, J.~Kwon, N.~Ho, and C.~Caramanis.
\newblock On the computational and statistical complexity of over-parameterized matrix sensing.
\newblock \emph{Journal of Machine Learning Research}, 25\penalty0 (169):\penalty0 1--47, 2024.

\end{thebibliography}

\appendix
\section{Missing proofs from Section~\ref{sec:algorithm}}
{ \subsection{Proof of Lemma~\ref{lem:boundontaylor}}\label{app:proof-lem-boundontaylor}
 First, we show the bound on the norm of $P_{k}$ and $I-P_k$.
  Set $d = \dim(\EEE),$ $m = \dim(\YY)$, and $r = \rank(\nabla F(x_{k}))$. Let $\nabla F(x_{k}) = U \Sigma V^\top$ be the economy SVD of matrix $\nabla F(x_{k})$, where $U \in O(m, r), V \in O(d, r), \Sigma = \diagg{\sigma},$ and $\sigma = \sigma\left(\nabla \c(x)\right)$, for notational convience we do not index these matrices with $x_k.$ Then, $$ P_k = U \Sigma (\Sigma^\top \Sigma + \lambda_k I)^{-1} \Sigma^{\top} U^\top.$$
  Let $w = \left( \frac{(\sigma_{1}^{x_k})^{2}}{(\sigma_{1}^{x_k})^{2} + \lambda_{k}}, \dots, \frac{(\sigma_{r}^{x_k})^{2}}{(\sigma_{r}^{x_k})^{2} + \lambda_{k}} \right).$ We know from linear algebra that $P_{k} = U\diag\left(w\right)U^{\top}$. The eigenvalues of $P_{k}$ are bounded by one since $\lambda_{k} > 0$, so $\norm{P_k}{\textrm{op}} \le 1.$ Similarly, let $v =  \left( \frac{\lambda_{k}}{(\sigma_{1}^{x_k})^{2} + \lambda_{k}}, \dots, \frac{\lambda_{k}}{(\sigma_{r}^{x_k})^{2} + \lambda_{k}} \right)$ we can write
  \begin{align}\label{eqn:I-Pk_SVD}
      I - P_k = U \diag(v) U^\top.
  \end{align}
  It's clear that $\norm{I-P_k}{\textrm{op}} \le 1$. Moreover, for any $v \in Y$, we have 
  \begin{align*}
      \norm{P_k v}{}  = \norm{U \diag(w) U^\top v}{} \le \norm{U^\top v}{} = \norm{\Pi^{x_k} v}{},
  \end{align*}
  where the last equality follows since $\Pi^{x_k} = UU^\top$ Therefore, the first item holds.

Next, we note that the second item holds immediately from~\eqref{eqn:I-Pk_SVD} and the monotonicity of singular values $\{\sigma_i^{x_k}\}_{i=1}^r$. Lastly, observe that $
	\gamma_k P_k v_k = \nabla F(x_k) ( x_k- x_{k+1})$, thus
	\begin{align}
	\norm{z_{k+1} - (z_k - \gamma_k P_k v_k)}{} &= \norm{F(x_{k+1}) - F(x_k) - \nabla F(x_k)(x_{k+1} - x_k))}{} \notag \\
	&\le  \frac{\lc}{2} \norm{x_{k+1} - x_k}{}^2\notag\\
	&=\frac{\lc}{2}\gamma_k^2 \norm{(\nabla F(x_k)^\top \nabla F(x_k) + \lambda_k I)^{-1} \nabla F(x_k)^\top v_k}{}^2, \label{eqn: taylor} 	\end{align}
where the inequality follows from Taylor's theorem. Just as before, we have that
$$ (\nabla F(x_k)^\top \nabla F(x_k) + \lambda_k I)^{-1} \nabla F(x_k)^\top = V(\Sigma^\top \Sigma + \lambda_k I)^{-1} \Sigma^{\top} U^\top. $$
Once again, the nonzero singular values of this matrix correspond to $\frac{\sigma_{i}}{\sigma_{i}^{2} + \lambda_{k}}.$
	By Young's inequality, we have $\sigma_i^2 + \lambda_k \ge 2 \sigma_i\sqrt{\lambda_k}$, so
	$\frac{\sigma_i}{\sigma_i^2 + \lambda_k}  \le \frac{1}{2\sqrt{\lambda_k}},$ which implies
    \begin{align} \label{eqn:nablaF_transpose_bound}
        \norm{(\nabla F(x_k)^\top \nabla F(x_k) + \lambda_k I)^{-1} \nabla F(x_k)^\top v_k }{} \le \frac{1}{2\sqrt{\lambda_k}} \norm{U^\top v_k}{} = \frac{1}{2\sqrt{\lambda_k}} \norm{\Pi^{x_k} v_k}{}.
    \end{align}
    The conclusion follows directly from the estimates~\eqref{eqn: taylor} and~\eqref{eqn:nablaF_transpose_bound}. This concludes the proof of Lemma~\ref{lem:boundontaylor}. \qed
}
\section{Missing proofs from Section~\ref{sec:guarantees}}\label{app:proofs-guarantees}
 \subsection{Proof of Theorem~\ref{thm: onestepimprovement} (Geometric decaying stepsizes)}\label{app:proof-nonsmooth-geometrically}
 We start by establishing a couple of auxiliary lemmas.

\begin{lemma}\label{lem: taylorerrorgeometric}
	Suppose that Assumptions~\ref{assum: assumptiononc} and~\ref{assum:nonsmoothpenalty} hold. For any $x_k, x_{k+1}$ generated by Algorithm~\ref{alg:LM}, let  $z_k = F(x_k)$ and $ z_{k+1} = F(x_{k+1})$. Then we have
	$$
	\norm{z_{k+1} - (z_k - \gamma_k P_k v_k)}{} \le \frac{\lc\lh^2 \gamma_k^2 }{8 \lambda_k}.
	$$
\end{lemma}
\begin{proof}
	A combination of Lemma~\ref{lem:boundontaylor} and Assumption~\ref{assum:nonsmoothpenalty} yields the desired bound.
\end{proof}
{
\begin{lemma}\label{lem:onestepgeometric}
	Suppose that Assumptions~\ref{assum: assumptiononc},~\ref{ass:weak-alignment}, and \ref{assum:nonsmoothpenalty} hold. For any $z_k$ such that $\|z_k - \zs\| \le  \delta\left(\tfrac{\mu}{8\lh}\right)$, we have
	\begin{align*}
	\norm{z_k - \gamma_k P_k v_k - \zs}{}^2	&\le \norm{z_k - \zs}{}^2- \frac{3\mu \gamma_k}{2}\norm{z_k -  \zs}{}\\
	&\quad \quad +2\lh \gamma_k \frac{\lambda_k}{\sig\left(\tfrac{\mu}{8\lh}\right)\norm{z_k -  \zs}{} + \lambda_k} \norm{z_k -  \zs}{} + \gamma_k^2 \lh^2.
	\end{align*}
\end{lemma}
\begin{proof}
	Note that
	\begin{align}
	\norm{z_k - \gamma_k P_k v_k -  \zs}{}^2 &= \norm{z_k -\zs}{}^2- 2\gamma_k \dotp{v_k, z_k -\zs} + 2\gamma_k \dotp{(I-P_k)v_k, z_k - \zs} + \gamma_k^2 \norm{P_k v_k}{}^2 \notag\\
	&\le \norm{z_k - \zs}{}^2 - 2 \gamma_k \mu \norm{z_k -\zs}{} + 2\gamma_k \dotp{v_k, (I-P_k)(z_k - \zs)} + \gamma_k^2 \lh^2, \label{eqn: onestepbound}
	\end{align}
	where the  inequality follows from Lemma~\ref{lem:boundontaylor},  Lemma~\ref{lem:aiming}, and  Item~\ref{item:nonsmoothpenalty:lip} of Assumption~\ref{assum:nonsmoothpenalty}. On the other hand, by the same argument as in \eqref{eq:projected-angle} and~\eqref{eqn:inner_prod_orthogonal}, we have
	\begin{align*}
	\norm{(I-P_k) (z_k - \zs)}{} \le \left(\frac{\lambda_k}{\sig\left(\tfrac{\mu}{8\lh}\right)\norm{z_k -  \zs}{} + \lambda_k}+ \frac{\mu}{8\lh} \right)\norm{z_k -  \zs}{}.
	\end{align*}
	By Item~\ref{item:nonsmoothpenalty:lip} of Assumption~\ref{assum:nonsmoothpenalty}, we have
	\begin{align}\label{eqn: onestepboundCauchy}
	2\gamma_k \dotp{v_k, (I-P_k)(z_k -  \zs)} \le 2\lh \gamma_k \left(\frac{\lambda_k}{\sig\left(\tfrac{\mu}{8\lh}\right)\norm{z_k -  \zs}{}+ \lambda_k}+ \frac{\mu}{8\lh} \right)\norm{z_k -  \zs}{}.
	\end{align}
	The desired inequality follows from a combination of~\eqref{eqn: onestepbound} and~\eqref{eqn: onestepboundCauchy}.
\end{proof}
}
	We prove the theorem by induction. The conclusion holds for $k=0$ by assumption. Next, suppose that the conclusion holds for some $k \ge 0$. We consider two cases:
	\begin{enumerate}[leftmargin=0.2cm]
		\item[] \textbf{Case 1.} Suppose first that $\norm{z_k - \zs}{} \le \frac{M}{4}q^{k}.$ We have
		\begin{equation}\label{eqn: geometric_decay_first}
		\begin{aligned}
		\norm{z_k - \gamma_k P_k v_k - \zs}{}^2
		&\le \norm{z_k -\zs}{}^2 +2\lh \gamma_k \frac{\lambda_k}{\sig\left(\tfrac{\mu}{8\lh}\right)\norm{z_k -  \zs}{}+ \lambda_k}\norm{z_k -\zs}{} + \gamma_k^2 \lh^2\\
		&\le \norm{z_k - \zs}{}^2 +2\lh \gamma_k \norm{z_k - \zs}{} + \gamma_k^2 \lh^2\\
		&\le \left(\frac{M^2}{16} + \frac{\gamma \lh  M}{2} + \gamma^2 \lh^2\right) q^{2k}\\
		&\le \frac{M^2}{4} q^{2k+2},
		\end{aligned}
		\end{equation}
		where the first inequality follows from Lemma~\ref{lem:onestepgeometric}, the second inequality follows from  the fact that $$\frac{\lambda_k}{\sig\left(\tfrac{\mu}{8\lh}\right)\norm{z_k -  \zs}{} + \lambda_k} \le 1,$$ 
        the third inequality follows from the assumption that $\norm{z_k - \zs}{} \le \frac{M}{4}q^k$, and the last inequality follows from $q\ge \frac{1}{\sqrt{2}}$ and our assumption on $\gamma\leq M\mu/(64L^2) \leq M/(64L)$. As a result, $\norm{z_k - \gamma_k P_k v_k - \zs}{} \le \frac{M}{2}q^{k+1}$. Moreover, by Lemma~\ref{lem: taylorerrorgeometric}, the triangle inequality, and our assumption $\gamma^2\leq 2\lambda M/(\lc \lh^2)$, we have
		\begin{align*}
		\norm{z_{k+1} - \zs}{} &\le \frac{M}{2}q^{k+1} + \frac{\gamma^2 \lc \lh^2 }{8 \lambda} q^k\\
		&\le  Mq^{k+1}.
		\end{align*}
		
		\item[] \textbf{Case 2}. Now, suppose $\frac{M}{4}q^k \le \norm{z_k - \zs}{} \le M q^k.$ We have
		\begin{align*}
		&\norm{z_k - \gamma_k P_k v_k- \zs}{}^2\\
		&\le \norm{z_k - \zs}{}^2  - \frac{3\gamma \mu }{2}q^k \norm{z_k - \zs}{} + 2\gamma \lh \frac{\lambda q^k}{ \sig\left(\tfrac{\mu}{8\lh}\right)\norm{z_k -  \zs}{} + \lambda q^k} q^k\norm{z_k - \zs}{} + \gamma^2 \lh^2 q^{2k}\\
		&\le \norm{z_k - \zs}{}^2 - \frac{3\gamma \mu  }{2M} \norm{z_k - \zs}{}^2 +  2\gamma \lh  \frac{4\lambda}{ \sig\left(\tfrac{\mu}{8\lh}\right) M + 4\lambda} q^k\norm{z_k - \zs}{} + \frac{16\gamma^2 \lh^2}{M^2} \norm{z_k - \zs}{}^2\\
		&\le \norm{z_k - \zs}{}^2 - \frac{3\gamma \mu }{2M} \norm{z_k - \zs}{}^2 +  2\gamma \lh \frac{16\lambda}{ \sig\left(\tfrac{\mu}{8\lh}\right) M^2 } \norm{z_k - \zs}{}^2 + \frac{16\gamma^2 \lh^2}{M^2} \norm{z_k - \zs}{}^2\\
		&\le \left(1- \frac{\gamma \mu}{M}\right) \norm{z_k -\zs}{}^2,
		\end{align*}
		where the first inequality follows from Lemma~\ref{lem:onestepgeometric}, the second and third inequalities follow from the assumed bound on  $\norm{z_k - \zs}{}$, and the last inequality follows from our assumption on $\lambda$ and $\gamma$. Taking the square root of both sides, we have
		\begin{equation}\label{eq:one-more}
		\norm{z_k - \gamma_k P_k v_k- \zs}{} \le  \sqrt{1 - \frac{\gamma \mu}{M}} \norm{z_k - z^\star}{} \le \left(1 - \frac{\gamma \mu}{2M}\right) \norm{z_k - z^\star}{}
		\end{equation}
        where the second inequality follows since $(1-x)^{1/2} \leq 1-x/2$ for all $x \leq 1,$ which holds due to our constraints on $\gamma.$
		Then,
		\begin{align*}
		\norm{z_{k+1} -\zs}{} &\le  \norm{z_k - \gamma_k P_k v_k- \zs}{}+ \norm{z_{k+1} - (z_k - \gamma_k P_k v_k)}{} \\
		&\le \left(1 - \frac{\gamma \mu}{2M}\right) \norm{z_k - \zs}{} + \frac{ \gamma^2\lc \lh^2 }{8 \lambda} q^k\\
		&\le \left(1 - \frac{\gamma \mu}{4M}\right) \norm{z_k - \zs}{}\\
		&\le Mq^{k+1},
		\end{align*}
		where the first inequality follows from Lemma~\ref{lem: taylorerrorgeometric} and the triangle inequality, the second inequality follows  from~\eqref{eq:one-more}, Lemma~\ref{lem: taylorerrorgeometric}, and our assumption $\gamma \leq \lambda \mu /(2\lc \lh^2)$, and the last inequality follows from the inductive hypothesis and the fact that $1 - \gamma\mu/(4M) \leq q$. 
	\end{enumerate}
	The induction is complete, finishing the proof of Theorem~\ref{thm: onestepimprovement}.
\qed
\subsection{Proof of Theorem~\ref{thm: onestepimprovement_exact}}\label{sec:proof_exact_nonsmooth}
The proof of the following lemma is essentially the same as that of Lemma~\ref{lem:onestepgeometric}. We omit the details.
\begin{lemma}\label{lem:onestepgeometric_strong}
	Suppose that Assumptions~\ref{assum: assumptiononc},~\ref{ass:strong-alignment}, and \ref{assum:nonsmoothpenalty} hold. For any $z_k$ such that $\|z_k - \zs\| \le  \delta\left(\tfrac{\mu}{8\lh}\right)$, we have
	\begin{align*}
	\norm{z_k - \gamma_k P_k v_k - \zs}{}^2	\le \norm{z_k - \zs}{}^2- \frac{3\mu \gamma_k}{2}\norm{z_k -  \zs}{}  +2\lh \gamma_k \frac{\lambda_k}{\sig + \lambda_k} \norm{z_k -  \zs}{} + \gamma_k^2 \lh^2.
	\end{align*}
\end{lemma}
{\paragraph{Proof for Polyak stepsize.} By induction, it suffices to prove the following claim:
	\begin{claim}\label{claim:nonsmooth_strong}
		For any $z_k= \c(x_k)$ with
		$\norm{z_k - \zs}{} \le \min\left\{\rloc\left(\frac{\mu}{8\lh}\right), \frac{\sig \mu}{8\cub \lh} \right\},$ we have
		$$
		\norm{z_{k+1} - \zs}{}^2 \le \left(1 - \frac{\gamma \mu^2}{8 \lh^2}\right) \norm{z_k - \zs}{}^2.
		$$
	\end{claim}
	
	To this end, we let  $j$ be the index provided by Assumption~\ref{ass:strong-alignment} when applied to $\rho = \frac{\mu}{8\lh}$ and $z = z_k = F(x_k)$, i.e.,
	$$\norm{(I- \P_{j}^{x})(z_k - \zs)}{} \le  \frac{\mu}{8\lh}\norm{z_k - \zs}{} \quad \text{and} \quad   (\sigma_{j}^x)^2 \ge \sig .$$
	Following the similar calculation as in~\eqref{eqn:inner_prod_orthogonal}, we have
	{\allowdisplaybreaks\begin{align*}
	|\dotp{(I - P_k)v_k, z_k -z^\star}| &\le  \lh \left(\norm{(I- P_k)\P_{j}^x(z_k - \zs)}{} + \norm{(I- P_k)(I - \P_{j}^x)(z_k - \zs)}{} \right) \\
	&\le \lh \left(\frac{\lambda_k}{(\sigma_i^x)^2 + \lambda_k} \norm{\P_{j}^x(z_k - \zs)}{} + \frac{\mu}{8\lh} \norm{z_k - \zs}{} \right)\\
	&\le  \lh \left( \frac{\cub \norm{z_k - \zs}{}}{\sig  + \cub\norm{z_k - \zs}{}} + \frac{\mu}{8\lh}\right) \norm{z_k - \zs}{} \\
	& \le \frac{\mu}{4}\norm{z_k - \zs}{} \\
	&\le \frac{\h(z_k) - h^\star}{4},
	\end{align*}}
	where the fourth inequality follows from the bound $\|z_k - z^\star\|  \le \frac{\sig \mu}{8\cub \lh}.$
	Invoking Proposition~\ref{prop:master_polyak} and Assumption~\ref{assum:nonsmoothpenalty} gives
	\begin{align*}
	\|z_{k+1}-\zs\|^{2}
	& \leq   \norm{z_k -\zs}{}^2 - \frac{\gamma}{8} \frac{(\h(z_k) - h^\star)^2}{\norm{\P^{x_k} v_k}{}^2} \\
	& \leq  \left(1 - \frac{\gamma \mu^2}{8{\lh^2}}\right)  \norm{z_k -\zs}{}^2,
	\end{align*}
	as desired.
}
\paragraph{Proof for geometrically decaying stepsize.} We prove it by induction. First, the conclusion holds for $k=0$. Now suppose that the conclusion holds for some $k\ge 0$. We consider two cases:
\begin{enumerate}[leftmargin=0.2cm]
	\item[] \textbf{Case 1.} Suppose $\norm{z_k - \zs}{} \le \frac{M}{4} q^k$. Using Lemma~\ref{lem:onestepgeometric_strong} and the same argument in~\eqref{eqn: geometric_decay_first}, we obtain $\norm{z_{k+1} - \zs}{} \le Mq^{k+1}$.
	\item[] \textbf{Case 2.} Suppose $\frac{M}{4} q^k \le \norm{z_k - \zs}{} \le {M} q^k$.  We have
	\begin{align*}
	\norm{z_k - \gamma_k P_k v_k - \zs}{}^2 &\le \norm{z_k - \zs}{}^2- \frac{3\mu \gamma_k}{2}\norm{z_k -  \zs}{}  +2\lh \gamma_k \frac{\lambda_k}{s + \lambda_k} \norm{z_k -  \zs}{} + \gamma_k^2 \lh^2\\
	&\le \norm{z_k - \zs}{}^2 - \frac{3 \gamma \mu  }{2M} \norm{z_k - \zs}{}^2 +  \frac{8 \gamma\lambda \lh}{ s M } \norm{z_k - \zs}{}^2 + \frac{16 \gamma^2\lh^2 }{M^2} \norm{z_k - \zs}{}^2\\
	&\le \left(1- \frac{\gamma \mu}{M}\right) \norm{z_k -\zs}{}^2,
	\end{align*}
	where the first inequality follows from Lemma~\ref{lem:onestepgeometric_strong}, the second inequality follows from the lower bound on $\norm{z_k - z^\star}{2}$, and the third inequality follows our bounds on $\lambda$ and $\gamma$. The rest of the proof follows from the same argument as the proof of Theorem~\ref{thm: onestepimprovement}.
\end{enumerate}
This completes the proof of Theorem~\ref{thm: onestepimprovement_exact}. \qed
\subsection{Proof of Theorem~\ref{thm: onestepimprovement_smooth}} \label{sec:proof_over_smooth}
We start by stating a couple of auxiliary lemmas.
{ \begin{lemma}\label{lem: taylorerrorgeometric_smooth}
	Suppose Assumptions~\ref{assum:mapc:smoothness} and~\ref{assum:smoothpenalty} hold, and let $x_k$ and $x_{k+1}$ be iterates generated by Algorithm~\ref{alg:LM} under Configuration~\ref{assum:constant_stepsize}. Define $z_k = \c(x_k)$ and $z_{k+1} = \c(x_{k+1})$. Then, we have
	$$
	\norm{z_{k+1} - (z_k - \gamma P_k \nabla \h(z_k))}{} \le \frac{ \lc\gamma^2 }{8 \lambda q^k} \norm{\P^{x_k}\nabla \h(z_k)}{}^2.
	$$
\end{lemma}
\begin{proof}
	A combination of Lemma~\ref{lem:boundontaylor} and the choice of $\lambda_k$ and $\gamma_k$ in Configuration~\ref{assum:constant_stepsize} yields the desired result.
\end{proof}

\begin{lemma}\label{lem:onestep_geometric_smooth_weak}
	Suppose that Assumptions~\ref{assum: assumptiononc},~\ref{ass:weak-alignment}, and~\ref{assum:smoothpenalty} hold. Assume that we are under Configuration~\ref{assum:constant_stepsize}  and that $\gamma \le \frac{1}{8\Lhs}$. For any $z_k$ such that $\norm{z_k - \zs}{}\le \rloc\left(\frac{\mus}{16\Lhs}\right)$, we have
	\begin{align*}
	\norm{z_k - \gamma P_k \nabla \h(z_k) - \zs}{}^2 &\le \norm{z_k - \zs}{}^2 - \frac{7 \gamma}{4}\dotp{\nabla \h(z_k), z_k - \zs}  + 2\Lhs\gamma \frac{\lambda_k}{\sig\left(\tfrac{\mus}{16 \Lhs}\right)\norm{z_k -  \zs}{} + \lambda_k} \norm{z_k -  \zs}{}^2.
	\end{align*}
\end{lemma}
\begin{proof}
	By expanding the square and adding and subtracting $2\gamma \dotp{\nabla \h(z_k), z_k - \zs}$ we get
	\begin{align*}
	&\norm{z_k - \gamma P_k \nabla \h(z_k)  - \zs}{}^2\\
	 &\le  \norm{z_k - \zs}{}^2 - 2\gamma \dotp{\nabla \h(z_k), z_k - \zs} + 2\gamma \dotp{(I- P_k) \nabla \h(z_k), z_k - \zs} + \gamma^2 \norm{\P^{x_k} \nabla \h(z_k)}{}^2\\
	&\le \norm{z_k - \zs}{}^2 - \frac{7 \gamma}{4}\dotp{\nabla \h(z_k), z_k - \zs} + 2\gamma \dotp{(I- P_k) \nabla \h(z_k), z_k - \zs},
	\end{align*}
	where the first inequality follows from Lemma~\ref{lem:boundontaylor}, and the second inequality follows from Lemma~\ref{lem:aiming}, Item~\ref{item:smoothpenalty:lip} of Assumption~\ref{assum:smoothpenalty}, and $\gamma \le \frac{1}{8\beta}$.  We focus on bounding the inner product in the last term
	\begin{align*}
	|\dotp{(I- P_k) \nabla \h(z_k), z_k - \zs}| &\le \Lhs \norm{z_k - z^\star}{} \norm{(I-P_k)(z_k -\zs)}{}\\
	&\le \Lhs  \left(\frac{\lambda_k}{\sig\left(\tfrac{\mus}{16 \Lhs}\right)\norm{z_k -  \zs}{} + \lambda_k} + \frac{ \mus}{16 \Lhs}\right) \norm{z_k -  \zs}{}^2,
	\end{align*}
	where the first inequality follows from Item~\ref{item:smoothpenalty:lip} of Assumption~\ref{assum:smoothpenalty} and last inequality follows from the same calculation as~\eqref{eqn:inner_prod_orthogonal} with $v_k$ replaced by $\nabla h(z_k)$.
	This concludes the proof of the lemma.
\end{proof}
}

\paragraph{Proof for Polyak stepsize.}
{ By induction, it suffices to prove the following claim.
	\begin{claim}\label{claim:smooth_weak}
		For any $z_k = \c(x_k)$ with $\norm{z_k - \zs}{2} \le \rloc\left(\frac{ \mus}{16 \Lhs}\right)$, we have
		$$
		\norm{z_{k+1} - \zs}{}^2 \le \left(1 - \frac{\gamma \mus}{32 \Lhs}\right) \norm{z_k - \zs}{}^2.
		$$
	\end{claim}
	Let $j$ be the index provided by Assumption~\ref{ass:weak-alignment} when applied to $\rho = \frac{ \mus}{16 \Lhs}$ and $z = z_{k} = \c(x_{k})$, i.e.,
	\begin{equation}
	\norm{(I- \P_{j}^{x_{k}})(z_k - \zs)}{} \le  \frac{ \mus}{16 \Lhs}\norm{z_k - \zs}{} \quad \text{and} \quad   (\sigma_{i}^x)^2 \ge \sig\left(\tfrac{\mus}{16 \Lhs}\right) \norm{z_k - \zs}{}.
	\end{equation}
	Thus, we have
	\begin{align}
        \label{eqn:smooth_inner_prod_orthognal}
	\begin{split}
	\left|\dotp{(I-P_k)\nabla h(z_k), z_k -  \zs}\right|
	&\le  \Lhs \left( \frac{\cub }{\sig\left(\tfrac{\mus}{16 \Lhs}\right)  + \cub} + \frac{ \mus}{16 \Lhs}\right) \norm{z_k - \zs}{}^2 \\
    & \le \Lhs \frac{2\mus}{16\Lhs}\norm{z_k - \zs}{}^2 \\
&\le \frac{1}{4}\left(\h(z_k) - h^\star\right),
	\end{split}
	\end{align}
	where the first inequality follows from the same calculation as~\eqref{eqn:inner_prod_orthogonal} and the second inequality is due to $\cub \leq \frac{\mus}{16\Lhs} \sig(\frac{\mus}{16\Lhs} )$. Applying Proposition~\ref{prop:master_polyak} and  Assumption~\ref{assum:smoothpenalty}, we have
	\begin{align*}
	\norm{z_{k+1}-\zs}{}^2 &\le \norm{z_k -\zs}{}^2  - \frac{\gamma}{8} \frac{(\h(z_k) - h^\star)^2}{\norm{\P^{x_k} \nabla \h(z_k)}{}^2}\\
    & \le   \norm{z_k -\zs}{}^2  - \frac{\gamma}{8} \frac{h(z_k) - h^\star}{\|z_k - z^\star\|^2}\frac{h(z_k) - h^\star }{\norm{\P^{x_k} \nabla \h(z_k)}{}^2} \|z_k - z^\star\|^2\\
	& \le \left(1 - \frac{  \gamma \mus }{32 \Lhs}\right) 
    \norm{z_k -\zs}{}^2,
	\end{align*}
	concluding the proof for the Polyak stepsize.
}
\paragraph{Proof for geometrically decaying stepsize.}
We prove the rate by induction. Based on our assumption, the conclusion holds for $k=0$. Now suppose that the conclusion holds for some $k \ge 0$. We consider two cases:
\begin{enumerate}[leftmargin=0.2cm]
	\item[] \textbf{Case 1.} $\norm{z_k - \zs}{} \le \frac{M}{4}q^{k}.$ We have
	\begin{align*}
	\norm{z_k - \gamma P_k \nabla h(z_k) - \zs}{}^2 &\le \norm{z_k -\zs}{}^2 +2\Lhs \gamma \frac{\lambda_k}{\sig\left(\frac{\mus}{16 \Lhs}\right)\norm{z_k -  \zs}{}+ \lambda_k}\norm{z_k -\zs}{}^2\\
	&\le \norm{z_k - \zs}{}^2 +2\Lhs \gamma \norm{z_k - \zs}{}^2\\
	&\le \left(\frac{M^2}{16} + \frac{\Lhs \gamma M^2}{8} \right) q^{2k}\\
	&\le \frac{M^2}{4} q^{2k+2},
	\end{align*}
	where the first inequality follows from Lemma~\ref{lem:onestep_geometric_smooth_weak}, the second inequality follows from $$\frac{\lambda_k}{\sig\left(\frac{\mus}{16 \Lhs}\right)\norm{z_k -  \zs}{} + \lambda_k} \le 1,$$ the third inequality follows from $\norm{z_k - \zs}{} \le \frac{M}{4}q^k$, and the last inequality follows from $q\ge \frac{1}{\sqrt{2}}$ and our assumption on $\gamma$. As a result, $\norm{z_k - \gamma_k P_k \nabla h(z_k) - \zs}{} \le \frac{M}{2}q^{k+1}$. Moreover, by Lemma~\ref{lem: taylorerrorgeometric_smooth}, Item~\ref{item:smoothpenalty:lip} of Assumption~\ref{assum:smoothpenalty} , and our bound on $\gamma$, we derive
	\begin{align*}
	\norm{z_{k+1} - \zs}{} &\le \norm{z_k - \gamma P_k  \nabla h(z_k)- \zs}{} + \frac{ \lc\gamma^2 }{8 \lambda q^k} \norm{\P^{x_k} \nabla \h(z_k)}{}^2\\
	&\le \frac{M}{2}q^{k+1} + \frac{\Lhs^2\lc \gamma^2 M^2}{128 \lambda} q^k\\
	&\le  Mq^{k+1}.
	\end{align*}
    
	\item[] \textbf{Case 2.} $\frac{M}{4}q^k \le \norm{z_k - \zs}{} \le M q^k.$ We have
	\begin{align*}
	\norm{z_k - \gamma P_k \nabla h(z_k)- \zs}{}^2 &\le \norm{z_k - \zs}{}^2  - \frac{7 \gamma}{4}\dotp{\nabla \h(z_k), z_k - \zs}+ 2\Lhs \gamma \frac{\lambda q^k \norm{z_k - \zs}{}^2}{ \sig\left(\tfrac{\mus}{16 \Lhs}\right)\norm{z_k -  \zs}{} + \lambda q^k} \\
	&\le \norm{z_k - \zs}{}^2 - \frac{7 \gamma}{4}\dotp{\nabla \h(z_k), z_k - \zs} +  2\Lhs \gamma \frac{4\lambda}{ \sig\left(\tfrac{\mus}{16 \Lhs}\right) M } \norm{z_k - \zs}{}^2 \\
	&\le \norm{z_k - \zs}{}^2 - \frac{3 \gamma}{2}\dotp{\nabla \h(z_k), z_k - \zs},
	\end{align*}
	where the first inequality follows from Lemma~\ref{lem:onestep_geometric_smooth_weak}, the second inequality follows from the assumed bound on  $\norm{z_k - \zs}{}$,  and the last inequality follows from the bound on $\lambda$ and Lemma~\ref{lem:aiming_smooth}.
	As a result of Lemma~\ref{lem: taylorerrorgeometric_smooth} and triangle inequality, we have
\begin{align*}
	\norm{z_{k+1} - \zs}{}^2 &\le \left(\norm{z_k - \gamma P_k \nabla h(z_k) - \zs}{} + \frac{ \lc\gamma^2 }{8 \lambda q^k} \norm{\P^{x_k} \nabla \h(z_k)}{}^2 \right)^2\\
	&\le \left(\sqrt{\norm{z_k - \zs}{}^2 - \frac{3 \gamma}{2}\dotp{\nabla \h(z_k), z_k - \zs}} + \frac{\lc \gamma^2}{8 \lambda q^k} \norm{\P^{x_k}  \nabla \h(z_k)}{}^2 \right)^2\\
	&= \norm{z_k - \zs}{}^2 - \frac{3 \gamma}{2}\dotp{\nabla \h(z_k), z_k - \zs}  + \frac{\lc^2 \gamma^4}{64 \lambda^2 q^{2k}} \norm{\P^{x_k} \nabla \h(z_k)}{}^4   \\ 
    &\qquad +\frac{\lc \gamma^2}{4 \lambda q^k} \norm{z_k -\zs}{2}  \norm{\P^{x_k}  \nabla \h(z_k)}{}^2.
	\end{align*}
	Note that by  Lemma~\ref{lem:aiming} and Item~\ref{item:smoothpenalty:lip} of Assumption~\ref{assum:smoothpenalty}, we have 
	$$\norm{\P^{x_k}  \nabla \h(z_k)}{}^2 \le 2\Lhs \dotp{\nabla \h(z_k), z_k - \zs}.$$ 
	Combining with the upper bound on $\norm{z_k - \zs}{}$ and our assumption on $\gamma$, we have
	\begin{align*}
	\norm{z_{k+1} - \zs}{}^2 &\le  \norm{z_k - \zs}{}^2 - \gamma \dotp{\nabla \h(z_k), z_k - \zs}\\
	&\le \left(1- \frac{\gamma \mus}{2} \right)\norm{z_k - \zs}{}^2\\
	&\le M^2q^{2k+2}.
	\end{align*}
\end{enumerate}
The induction is complete, and so is the proof of Theorem~\ref{thm: onestepimprovement_smooth}. \qed
\subsection{Proof of Theorem~\ref{thm: onestepimprovement_smooth_exact}}\label{sec:onestepimprovement_smooth_exact_proof}
We start with an auxiliary result. 
\begin{lemma}\label{lem:onestep_geometric_smooth_strong}
	Suppose that Assumptions~\ref{assum: assumptiononc},~\ref{ass:strong-alignment}, and~\ref{assum:smoothpenalty} hold. Suppose that we are under Configuration~\ref{assum:constant_stepsize} and that $\gamma \le \frac{1}{8\Lhs}$. For any $z_k$ such that $\norm{z_k - \zs}{}\le \rloc\left(\frac{\mus}{16\Lhs}\right)$, we have
	\begin{align*}
	\norm{z_k - \gamma P_k \nabla \h(z_k) - \zs}{}^2 &\le \norm{z_k - \zs}{}^2 - \frac{7 \gamma}{4}\dotp{\nabla \h(z_k), z_k - \zs}  + 2\Lhs\gamma \frac{\lambda_k}{\sig+ \lambda_k} \norm{z_k -  \zs}{}^2.
	\end{align*}
\end{lemma}
\begin{proof}
	The proof follows the same argument as the proof of Lemma~\ref{lem:onestep_geometric_smooth_weak}.
\end{proof}
{ \paragraph{Proof for Polyak stepsize.} By induction, it suffices to prove the following claim.
	\begin{claim}\label{claim:smooth_strong}
		For any $z_k = \c(x_k)$ with $\norm{z_k - \zs}{} \le \left\{r\left(\frac{ \mus}{16 \Lhs}\right), \frac{\sig \mus}{16 \cub \Lhs}\right\}$, we have
		$$
		\norm{z_{k+1} - \zs}{}^2 \le \left(1 - \frac{\gamma \mus}{32 \Lhs}\right) \norm{z_k - \zs}{}^2.
		$$
	\end{claim}
	Let  $j$ be the index provided by Assumption~\ref{ass:strong-alignment} when applied to $\rho = \frac{ \mus}{16 \Lhs}$ and $z = z_k = F(x_k)$, i.e.,
	$$\norm{(I- \P_{j}^{x})(z_k - \zs)}{} \le  \frac{ \mus}{16 \Lhs}\norm{z_k - \zs}{} \quad \text{and} \quad   (\sigma_{j}^x)^2 \ge \sig .$$
	Following the similar calculation as~\eqref{eqn:inner_prod_orthogonal} and~\eqref{eqn:smooth_inner_prod_orthognal}, we have
	\begin{align*}
	|\dotp{(I - P_k) \nabla h(z_k), z_k -z^\star}| &\le  \Lhs \norm{z_k - z^\star}{}\left(\norm{(I- P_k)\P_{j}^x(z_k - \zs)}{} + \norm{(I- P_k)(I - \P_{j}^x)(z_k - \zs)}{} \right) \\
	&\le \Lhs \norm{z_k - z^\star}{} \left(\frac{\lambda_k}{(\sigma_j^x)^2 + \lambda_k} \norm{\P_{j}^x(z_k - \zs)}{} + \frac{ \mus}{16 \Lhs} \norm{z_k - \zs}{} \right)\\
	&\le  \Lhs \left( \frac{\cub \norm{z_k - \zs}{}}{\sig  + \cub\norm{z_k - \zs}{}} + \frac{ \mus}{16 \Lhs}\right) \norm{z_k - \zs}{}^2 \\
    &\le  \Lhs \left( \frac{1}{\frac{16\Lhs}{\mus}+1} + \frac{ \mus}{16 \Lhs}\right) \norm{z_k - \zs}{}^2 \\
	& \le \frac{\mus}{8}\norm{z_k - \zs}{}^2 \\
	&\le \frac{\h(z_k) - h^\star}{4},
	\end{align*}
	where the fourth inequality follows from the bound $\|z_k - z^\star\|  \le \frac{\sig \mus}{16\cub \Lhs}$ and for $\frac{1}{1+x^{-1}} + x \leq 2x$ for any $x\geq0$. Applying Proposition~\ref{prop:master_polyak} and  Assumption~\ref{assum:smoothpenalty}, we get
	\begin{align*}
	\norm{z_{k+1}-\zs}{}^2 &\le \norm{z_k -\zs}{}^2  - \frac{\gamma}{8} \frac{(\h(z_k) - h^\star)^2}{\norm{(U^{x_k})^\top \nabla \h(z_k)}{}^2}\\
	& \le \left(1 - \frac{\gamma \mus}{32 \Lhs}\right) \norm{z_k -\zs}{}^2,
	\end{align*}
	proving the result for the Polyak stepsize.
} 

\paragraph{Proof for geometrically decaying stepsize.} We prove the theorem by induction. Based on our assumption, the conclusion holds for $k=0$. Now suppose that the conclusion holds for some $k \ge 0$. We consider two cases:
\begin{enumerate}[leftmargin=0.2cm]
	\item[] \textbf{Case 1.}  $\norm{z_k - \zs}{} \le \frac{M}{4}q^{k}.$ We have
	\begin{align*}
	\norm{z_k - \gamma P_k \nabla h(z_k)- \zs}{}^2 &\le \norm{z_k -\zs}{}^2 +2\Lhs \gamma \frac{\lambda_k}{\sig+ \lambda_k}\norm{z_k -\zs}{}^2\\
	&\le \norm{z_k - \zs}{}^2 +2\Lhs \gamma \norm{z_k - \zs}{}^2\\
	&\le \left(\frac{M^2}{16} + \frac{\Lhs \gamma M^2}{8} \right) q^{2k}\\
	&\le \frac{M^2}{4} q^{2k+2},
	\end{align*}
	where the first inequality follows from Lemma~\ref{lem:onestep_geometric_smooth_strong}, the second inequality follows from $\frac{\lambda_k}{\sig + \lambda_k} \le 1$, the third inequality follows from the assumption that $\norm{z_k - \zs}{} \le \frac{M}{4}q^k$, and the last inequality follows from $q\ge \frac{1}{\sqrt{2}}$ and the bound on $\gamma$. As a result, $\norm{z_k - \gamma_k P_k \nabla h(z_k) - \zs}{} \le \frac{M}{2}q^{k+1}$. Moreover, by Lemma~\ref{lem: taylorerrorgeometric_smooth}, Item~\ref{item:smoothpenalty:lip} of Assumption~\ref{assum:smoothpenalty} , and our assumption on $\gamma$, we have
	\begin{align*}
	\norm{z_{k+1} - \zs}{} &\le \norm{z_k - \gamma P_k \nabla h(z_k) - \zs}{} + \frac{ \lc\gamma^2 }{8 \lambda q^k} \norm{\P^{x_k} \nabla \h(z_k)}{}^2\\
	&\le \frac{M}{2}q^{k+1} + \frac{\Lhs^2\lc \gamma^2 M^2}{128 \lambda} q^k\\
	&\le  Mq^{k+1}.
	\end{align*}

	\item[] \textbf{Case 2.}  $\frac{M}{4}q^k \le \norm{z_k - \zs}{2} \le M q^k.$ We have
	\begin{align*}
	\norm{z_k - \gamma P_k \nabla h(z_k)- \zs}{}^2 &\le \norm{z_k - \zs}{}^2  - \frac{7 \gamma}{4}\dotp{\nabla \h(z_k), z_k - \zs}+  \frac{2\Lhs \gamma\lambda  }{ \sig }\norm{z_k - \zs}{}^2 \\
	&\le \norm{z_k - \zs}{}^2 - \frac{3 \gamma}{2}\dotp{\nabla \h(z_k), z_k - \zs},
	\end{align*}
	where the first inequality follows from Lemma~\ref{lem:onestep_geometric_smooth_strong} and $q\le 1$,  and the last inequality follows from $\lambda \leq \frac{\sig \mus}{16\Lhs}$ and Lemma~\ref{lem:aiming_smooth}.
	The rest of the proof follows the same as the proof of Theorem~\ref{thm: onestepimprovement_smooth}.
\end{enumerate}
The induction is complete, and so is the proof of Theorem~\ref{thm: onestepimprovement_smooth_exact}. \qed

\subsection{Additional results and proofs from Section~\ref{sec:local-guarantees}}\label{app:local-regularity}
In this section, we prove Theorem~\ref{thm:local-onestepimprovement} and state additional local guarantees we omitted in Section~\ref{sec:local-guarantees}.
\subsubsection{Proof of Theorem~\ref{thm:local-onestepimprovement}}
The proofs of our local guarantees rely on the following two auxiliary results.
\begin{lemma}\label{lem:one_step_x_bound}
	Let $x_k$ and $x_{k+1}$ be iterates of Algorithm~\ref{alg:LM} under Configuration~\ref{assum: dampingparameter} and write $z_k = \c(x_k)$ and $z_{k+1} = \c(x_{k+1})$. Suppose that $\left|\dotp{(I-P_k) v_k, z_k-\zs}\right| \leq \frac{1}{4}(h(z_k) - h^\star)$ holds. 
    Then, we have
	\begin{align*}
	\norm{x_k- x_{k+1}} \le \frac{2\gamma}{3} \frac{\|z_k - z^\star\|^{1/2}}{\sqrt{\clb}}.
	\end{align*}
\end{lemma}
\begin{proof}
Recall from Claim~\ref{claim:aligned-projected-subgradients} that
 \begin{align}\begin{split} \label{eq:wannabe-L}
        \frac{3}{4} (\h(z_k) - \h^\star) &\leq \|\Pi^{x_k} v_k\|\|z_k - \zs\|
        \end{split}.
    \end{align}
	By the definition of Algorithm~\ref{alg:LM}, we have
	\begin{align*}
	\|{x_k - x_{k+1}}\| &= \gamma_k \left\|(\nabla F(x_k)^\top \nabla F(x_k) + \lambda_k I)^{-1} \nabla F(x_k)^\top {v_k}\right\|\\
	&\le \frac{\gamma (h(z_k) - h^\star)}{2\sqrt{\lambda_k} \|\Pi^{x_k}v_k\|}\\
	&\leq \frac{2\gamma}{3} \frac{\norm{z_k - z^\star}{2}^{1/2}}{\sqrt{\clb}},
	\end{align*}
	where the second line follows from \eqref{eqn:nablaF_transpose_bound} and the last line follows from \eqref{eq:wannabe-L} together with the lower bound on $\lambda_k \ge \clb \|z_k - z^\star\|.$
\end{proof}
\begin{proposition}
	\label{prop: local-lipshshitz-c-conditions-for-linear-convergence}
	Let $\{x_k\}_{k\ge 0}\subseteq \EEE$ and $\{z_k\}_{k \ge 0}\subseteq \YY$ be two sequences and let $\xs \in \EEE$ and $\zs \in \YY$ be two given points. Suppose that there exist constants $C>0$, $\varepsilon > 0$, $r>0$,  and $q \in (0,1)$ be constants such that the following two hold.
	\begin{enumerate}
		\item For any $k \ge 0$, $\|x_k - x_{k+1}\|\le C\|z_k - \zs\|^{1/2} $. \label{item:small_step}
		\item If $\|x_k - \xs\| \le \varepsilon$, $\|{x_{k+1}-\xs}\| \le \varepsilon$, and $\|{z_k - \zs}\| \le r$, then $\|{z_{k+1} - \zs}\| \le q \|{z_k - \zs}\|$. \label{item:contract}
	\end{enumerate}
	Then, if $\norm{x_0  - \xs}{2} \le \frac{\varepsilon}{2}$ and $\norm{z_0 - \zs}{2}\le \min \left\{\left(\frac{(1-q^{1/2})\varepsilon}{2C}\right)^{2}, r \right\}$ hold, we have
	\begin{align*}
	\norm{x_{k} -\xs}{2} \leq \varepsilon \quad \text{and} \quad \norm{z_k - \zs}{2} \le \norm{z_0 - \zs}{2} q^k , \qquad \forall k \ge 0.
	\end{align*}
\end{proposition}
\begin{proof}
	We apply induction to prove that for any $k \ge 0$,
	\begin{align}
	\norm{x_k- x^\star}{2} \le \frac{\varepsilon}{2} + \sum_{i=0}^{k-1} C\norm{z_0 - \zs}{2}^{1/2} q^{i/2} \leq \varepsilon, \qquad \norm{z_k - \zs}{2} \le \norm{z_0 - \zs}{2} q^k. \label{eqn:induction}
	\end{align}
	The bound~\eqref{eqn:induction} holds for $k=0$ by our assumption. Suppose that~\eqref{eqn:induction} holds for $k$. Note that $\norm{z_0 - \zs}{2} \le r$, we have $\|z_k - \zs\| \le r$. As a result of Item~\ref{item:small_step} and the induction hypothesis, we have
	\begin{align*}
	\norm{x_{k+1} -\xs}{2} &\le \norm{x_k -\xs}{2} + \norm{x_k - x_{k+1}}{2}\\
	&\le \frac{\varepsilon}{2} + \sum_{i=0}^{k-1} C\norm{z_0 - \zs}{2}^{1/2}q^{i/2} + C \norm{z_k -\zs}{2}^{1/2}\\
	&\le \frac{\varepsilon}{2} + \sum_{i=0}^{k} C\norm{z_0 - \zs}{2}^{1/2}q^{i/2}\\
	&\le \varepsilon.
	\end{align*}
	Additionally, by Item~\ref{item:contract}, we have
	\begin{align*}
	\norm{z_{k+1} - \zs}{2} \le q \norm{z_k -\zs}{2} \le \norm{z_0 - \zs}{2} q^{k+1}.
	\end{align*}
	The induction is complete, and the proof is finished.
\end{proof}
Armed with these results, we can prove Theorem~\ref{thm:local-onestepimprovement}.
\begin{proof}[Proof of Theorem~\ref{thm:local-onestepimprovement}]
	Notice that the iterates satisfy the assumption in Lemma~\ref{lem:one_step_x_bound} by the same argument we used in \eqref{eqn:inner_prod_orthogonal}. With this, we will verify the two conditions required by Proposition~\ref{prop: local-lipshshitz-c-conditions-for-linear-convergence}. First, notice that Lemma~\ref{lem:one_step_x_bound} directly implies Item~\ref{item:small_step} of Proposition~\ref{prop: local-lipshshitz-c-conditions-for-linear-convergence}, with $C = \frac{2\gamma}{3\sqrt{\clb}}$. Next, we establish Item~\ref{item:contract} using the same arguments from the proof of Theorem~\ref{thm: onestepimprovement}. Note that the derivation of~\eqref{claim:nonsmooth_weak} in that proof only relies on the Lipschitz continuity of $\nabla F$ along the line segment between $x_k$ and $x_{k+1}$, rather than requiring the stronger global Lipschitz condition outlined in Assumption~\ref{assum: assumptiononc}. Furthermore, due to Assumption~\ref{ass:local-weak-alignment}, the inequalities~\eqref{eq:applying-weak-aligntment} remain valid whenever $\norm{x_k - \xs}{2} \leq \alignvarepsilon$. Thus, if both points $x_k$ and $x_{k+1}$ are in the ball $B_{\alignvarepsilon}(\xs)$, we have
	$$\|z_{k+1} - z^\star\|_2 \le \left(1-\frac{\gamma \mu^2}{8L^2}\right)^{1/2} \|z_k - z^\star\|_2.$$
	Consequently, Item~\ref{item:contract} holds with parameters $r = \delta\left(\frac{\mu}{8L}\right)$ and $q = \left(1 - \frac{\gamma \mu^2}{8L^2}\right)^{1/2}$.
Having established both conditions, the theorem follows directly from Proposition~\ref{prop: local-lipshshitz-c-conditions-for-linear-convergence}.
\end{proof}

\subsubsection{Local convergence guarantees}\label{app:additional-local-guarantees}
Next, we present extensions to the other settings we considered. 

\begin{assumption}[Local strong alignment]\label{ass: local-strong-alignment}
    	For fixed $\xs \in \cX^\star$ and $\zs=F(\xs)$ there exist functions $\rloc \colon \RR_{+} \rightarrow \RR_{+}$ and scalars $\alignvarepsilon, \sig>0$ such that for all $\rho >0$, if $x \in \B_{\alignvarepsilon}(x^\star)$ and $z = \c(x) \in \B_{\rloc(\rho)}(z^{\star})$, then there is an index $j$ for which 
	\begin{align*}
	\norm{(I- \P^{x}_{j})(z - \zs)}{2}  \le  \rho \norm{z - \zs}{2} \qquad \text{and} \qquad \left(\sigma_{j}^{x}\right)^{2} \ge \sig.
	\end{align*}
\end{assumption}

We omit the proofs of the following three theorems since they follow an analogous argument to that in the proof of Theorem~\ref{thm:local-onestepimprovement}, with Claim~\ref{claim:nonsmooth_weak} replaced by Claims~\ref{claim:nonsmooth_strong}, \ref{claim:smooth_weak}, and \ref{claim:smooth_strong}, respectively.
\begin{theorem}[\textbf{Convergence under local strong alignment and nonsmoothness}]\label{thm:local-onestepimprovement-strong-nonsmooth}
	Suppose  Assumptions~\ref{assum:nonsmoothpenalty},~\ref{assum:local_lip_jacob} and~\ref{ass: local-strong-alignment} hold.  Define $ \tilde q :=  \sqrt{1 - \tfrac{\gamma \mu^2}{8 \lh^2}}$, 
	and let $x_{0}$ and $z_0 =\c(x_0)$ be points satisfying
	$$\|x_{0} - \xs\|_{2} \leq \varepsilon /2
	\quad\text{and}\quad
	\norm{z_0 - \zs}{2} \le \min\left\{\rloc\left(\frac{\mu}{8\lh}\right), \frac{\sig \mu}{8\cub \lh}, \frac{\left(1- \sqrt{\tilde q}\right)^2\varepsilon^2 \clb}{2\gamma^2 }\right\},$$ 
	where $\varepsilon = \min\left\{ \jacvarepsilon, \alignvarepsilon \right\}$. Suppose we ran Algorithm~\ref{alg:LM} initialized at $x_{0}$ using Configuration~\ref{assum: dampingparameter} with $\gamma \le \min\left\{1, \frac{\clb}{\lc}\right\}$. Then, the iterates $x_{k}$ satisfy
	$$\|x_{k} -  x^{\star}\|_{2} < \varepsilon \quad \text{for all }k \geq 0,$$
	and, moreover, the mapped iterates $z_{k }= \c(x_{k})$ satisfy
	$$
	\|z_{k} - \zs\|^2 \le \left(1 - \frac{\gamma \mu^2}{8 \lh^2}\right)^k \|z_0 - \zs\|^2 \quad \text{for all } k\geq 0.
	$$
    
\end{theorem}
\begin{theorem}[\textbf{Convergence under local weak alignment and smoothness}]\label{thm:local-onestepimprovement-weak-smooth}
	Suppose Assumptions~\ref{assum:smoothpenalty},~\ref{assum:local_lip_jacob} and~\ref{ass:local-weak-alignment} hold. Define $ \tilde q :=  \sqrt{1 - \frac{\gamma \mus}{32 \Lhs}}$ and let $x_{0}$ and $z_0 =\c(x_0)$ be points satisfying
	$$\|x_{0} - \xs\|_{2} \leq \varepsilon /2
	\quad\text{and}\quad
	\norm{z_0 - \zs}{2} \le \min\left\{\rloc\left(\frac{\mus}{16\Lhs}\right), \frac{\left(1- \sqrt{\tilde q}\right)^2\varepsilon^2 \clb}{2\gamma^2 }\right\},$$ 
	where $\varepsilon = \min\left\{ \jacvarepsilon, \alignvarepsilon \right\}$. Suppose we ran Algorithm~\ref{alg:LM} initialized at $x_{0}$ using Configuration~\ref{assum: dampingparameter} with $\gamma \le \min\left\{1, \frac{\clb}{\lc}\right\}$ and $\cub \leq \frac{\mus}{16\Lhs}\sig( \frac{\mus}{16\Lhs})  $. Then, the iterates $x_{k}$ must satisfy
	$$\|x_{k} -  x^{\star}\|_{2} < \varepsilon \quad \text{for all }k \geq 0$$
	and, moreover, the mapped iterates $z_{k }= \c(x_{k})$ satisfy
	$$
	\|z_{k} - \zs\|^2 \le \left(1 - \frac{\gamma \mus}{32 \Lhs}\right)^k \|z_0 - \zs\|^2 \quad \text{for all } k\geq 0.
	$$

\end{theorem}
\begin{theorem}[\textbf{Convergence under local strong alignment and smoothness}]\label{thm:local-onestepimprovement-strong-smooth} 
Suppose Assumptions~\ref{assum:smoothpenalty},~\ref{assum:local_lip_jacob} and~\ref{ass: local-strong-alignment} hold. Define $ \tilde q :=  \left(1 - \frac{\gamma \mus}{32 \Lhs}\right)^{1/2}$ and let $x_{0}$ and $z_0 =\c(x_0)$ be points satisfying
	$$\|x_{0} - \xs\|_{2} \leq \varepsilon /2
	\quad\text{and}\quad
	\norm{z_0 - \zs}{2} \le \min\left\{\rloc\left(\frac{\mus}{16\Lhs}\right), \frac{\sig \mus}{16\cub \Lhs},  \frac{\left(1- \sqrt{\tilde q}\right)^2\varepsilon^2 \clb}{2\gamma^2 }\right\},$$ 
	where $\varepsilon = \min\left\{ \jacvarepsilon, \alignvarepsilon \right\}$. If one runs Algorithm~\ref{alg:LM} initialized at $x_{0}$ using Configuration~\ref{assum: dampingparameter} with $\gamma \le \min\left\{1, \frac{\clb}{\lc}\right\}$, then, the iterates $x_{k}$ must satisfy
	$$\|x_{k} -  x^{\star}\|_{2} < \varepsilon \quad \text{for all }k \geq 0,$$
	and, moreover, the mapped iterates $z_{k }= \c(x_{k})$ satisfy
	$$
	\|z_{k} - \zs\|^2 \le \left(1 - \frac{\gamma \mus}{32 \Lhs}\right)^k \|z_0 - \zs\|^2 \quad \text{for all } k\geq 0.
	$$
\end{theorem}

\section{Missing proofs from Section~\ref{sec:examples}}\label{app:examples}
In this section, we establish that weak and strong alignment hold for the parameterizations introduced in Section~\ref{sec:examples}. 

\subsection{Sufficient conditions for alignment}

Our proofs rely on establishing sufficient conditions for weak and strong alignment. In what follows, we present these conditions. We will use the symbols $\bar{r}^\star$ and $\bar{r}$ instead of $r^\star$ and $r$ to distinguish the rank of $\nabla F$ from the rank of its potential input, which will be particularly important when dealing with low-rank matrices.

\subsection{Strong alignment}
We start by showing that local strong alignment holds whenever the rank of the map $\c$ is constant near $\xs$, generalizing the assumptions for the Gauss-Newton subgradient method \cite{davis2022linearly}.
\begin{lemma}[Constant rank implies local strong alignment]\label{constant_rank_implies_strong_alignment}
    Let $\xs \in \RR^d$ and $\zs = F(\xs)$. Assume that the map $F$ satisfies Assumption~\ref{assum:local_lip_jacob} with $\jacvarepsilon > 0$ and $\lc\geq0$. Suppose there exists $\varepsilon>0$ with 
    $$
    \rankk{\nabla F(\xs)} = \rankk{\nabla \c(x)}=:\bar{r}^\star \qquad \text{for all $x \in \ball{\xs}{\varepsilon}$.}
    $$
    Then, there exist positive constants $R$ and $C$ such that $\c$ satisfies Assumption~\ref{ass: local-strong-alignment} with 
    \begin{align*}
        \delta(\rho) = \frac{\rho}{C},\quad j = \bar{r}^\star, \quad s = \frac{1}{2} \sigma_{\bar{r}^\star}\left( \nabla F(\xs) \right) \quad\text{ and }\quad
        \alignvarepsilon = \minn{R, \frac{\sigma_{\bar{r}^\star}\left( \nabla F(\xs) \right) }{2\lc},\jacvarepsilon}.
    \end{align*}
\end{lemma}

\begin{proof}
     We start with the first inequality in Assumption~\ref{ass: local-strong-alignment}. By the Constant Rank Theorem~\cite[Theorem 4.12]{lee2003introduction}, there exists a constant $ R'>0$ such that the set $\mathcal{M}:=F\left(\B_{R'}(\xs)\right)$ is a $C^1$-smooth manifold. It is well-known that near any point the distance between a manifold and its tangent grows quadratically \cite[Lemma 3.2]{diaz2019local}, that is, there are constants $C$ and $R''$ such that for any $z:=\c(x) \in \B_{R''}(\zs)$ we have
\begin{equation}\label{eq:tangeant-space-characterization}
\|(I-\proj_{\mathcal{T}_{\mathcal{M}}(z)})(\zs-z)\| \leq C \|\zs-z\|^2,
 \end{equation}  
 where $\mathcal{T}_{\mathcal{M}}(z)$ is the tangent space of $\mathcal{M}$ at $z.$ Moreover, since $ \mathcal{T}_{\mathcal{M}}(z)=\operatorname{range}\left( \nabla \c(x) \right)$  \cite[Chapter 5]{lee2003introduction}. Therefore,
 \begin{align*}
 \norm{\left(I - \P_{j}^x\right)\left(z-\zs\right)}{} & \leq C\norm{z-\zs}{}^2  \leq \rho \norm{z-\zs}{},
 \end{align*}
  where the first inequality follows from  $\P_{j}^x = \proj_{\operatorname{range}\left( \nabla \c(x) \right)} = \proj_{\mathcal{T}_{\mathcal{M}}(z)}$, and the second inequality follows from  $\norm{z-\zs}{}\leq \delta(\rho) \leq \frac{\rho}{C}$. 
 
  To establish the lower bound on the singular value, we leverage Weyl's inequality. Since $F$ satisfies Assumption~\ref{assum:local_lip_jacob} with $\jacvarepsilon$ and $\lc\geq0$, we have that $\abs{\sigma_{\bar{r}^\star}\left(\nabla \c(x)\right) - \sigma_{\bar{ r}^\star}\left(\nabla F(\xs) \right)} \leq \lc \norm{x-\xs}{}.$ Thus, \begin{align*} \sigma_{\bar{r}^\star}\left(\nabla \c(x)\right) &\geq \sigma_{\bar{r}^\star}\left(\nabla F(\xs)\right) - \alignvarepsilon\lc \geq \frac{1}{2}\sigma_{\bar{r}^\star}\left(\nabla     F(\xs)\right).\end{align*} 
Since $F$ is continuous in $\ball{\xs}{\jacvarepsilon}$ with a Lipschitz constant $\jacvarepsilon L_{\nabla F} + \norm{\nabla F(\xs)}{op}$, we conclude upon taking $R = \min\left\{R', \frac{R''}{\jacvarepsilon L_{\nabla F} + \norm{\nabla F(\xs)}{op}}\right\}$.
\end{proof}

\subsection{Weak alignment}
Recall that given a point $x$ and a smooth map $\c \colon \EEE \rightarrow \YY$, we let $\sigma_j^x$ and $\Pi_j^x$ denote the $j$th singular value of $\nabla \c(x)$ and the projection onto the span of its top $j$ left singular vectors. The following Proposition extends the proof idea in~\cite[Lemma 24]{zhang2021preconditioned}.
\begin{proposition}[Sufficient condition for weak alignment] \label{thm: sufficient-condition-weak-alignment}
	 Fix $\zs \in \Ima \c$. %
     Suppose there are functions $\sig\colon \RR_+\rightarrow \RR_+$ and $\delta\colon\RR_+ \rightarrow \RR_+$ satisfying that for any $\rho>0$ and any $z = \c(x)$ with $\norm{z - \zs}{} \leq \delta(\rho)$, there exists an integer $\bar r ^\star$ with $\bar r^\star \leq \bar r := \rank (\nabla F(x))$ such that the following statements hold.
	\begin{enumerate}
		\item Projected differences are bounded $\norm{\left( I - \Pi^x_{\bar r}  \right) \left(z - \zs\right)}{} \leq  \rho \norm{ z- \zs}{}$.
		\item For all $k \in \{\bar r^\star + 1, \dots, \bar r\}$ we have
        $$(\sigma_k^{x})^2  \leq \sig(\rho) \norm{z-\zs}{} \implies  \norm{\left( I - \Pi_{{k-1}}^x  \right) \left(z - \zs\right) }{} \leq \rho \norm{z-z^\star}{}.$$ %
		\item The $\bar r^\star$-th singular value is lower bounded $\left(\sigma_{\bar r^\star}^x\right)^2 \geq \sig(\rho) \norm{z - z^\star}{}$.
	\end{enumerate}
	Then, the map $\c$ satisfies Assumption \ref{ass:weak-alignment}.
\end{proposition}
\begin{proof}
	The result follows immediately from backward induction on $k$.
\end{proof}
We also introduce a local version of this proposition.

\begin{proposition}[Sufficient condition for local weak alignment] \label{thm: sufficient-condition-local-weak-alignment}
	 Let $\xs \in \EEE$ and $\zs \in \Ima \c$ with $\zs = \c(\xs)$. %
    Suppose there is a scalar $\alignvarepsilon > 0$, and functions $\sig\colon \RR_+\rightarrow \RR_+$ and $\delta\colon\RR_+ \rightarrow \RR_+$ satisfying that for any $\rho>0$ and any $z = \c(x)$ with $\norm{z - \zs}{} \leq \delta(\rho)$ and $\norm{x-\xs}{} \leq \alignvarepsilon$ there exists an integer $\bar r ^\star$ with $\bar r^\star \leq \bar r := \rank (\nabla F(x))$ such that the following statements hold. 
	\begin{enumerate}
		\item Projected differences are bounded $\norm{\left( I - \Pi^x_{\bar r}  \right) \left(z - \zs\right)}{} \leq  \rho \norm{ z- \zs}{}$.
		\item For all $k \in \{\bar r^\star + 1, \dots, \bar r\}$ we have
        $$(\sigma^x_k)^2  \leq \sig(\rho) \norm{z-\zs}{} \implies  \norm{\left( I - \Pi_{{k-1}}^x  \right) \left(z - \zs\right) }{} \leq \rho \norm{z-z^\star}{}.$$ %
		\item The $\bar r^\star$-th singular value is lower bounded $\left(\sigma_{\bar r^\star}^x\right)^2 \geq \sig(\rho) \norm{z - z^\star}{}$.
	\end{enumerate}
	Then, the map $\c$ satisfies Assumption \ref{ass:local-weak-alignment}.
\end{proposition}

\subsection{Proofs from Section~\ref{sec:square}}
\subsubsection{Proof of Theorem~\ref{thm: weak-alignment-hadamard}} \label{proof: hadamard-weak-alignment}
Let us prove that the Hadamard map $x \mapsto x \odot x$ is smooth with parameter $\lc = 2$.
 We have
\begin{align*}
 \norm{\nabla \c(x) - \nabla \c(y) }{\textrm{op}} = \norm{2\diagg{x - y}}{\textrm{op}} \le  2\norm{x-y}{2}.
\end{align*}

To prove that the map satisfies weak alignment Assumption~\ref{ass:weak-alignment}, we leverage Proposition~\ref{thm: sufficient-condition-weak-alignment}. Thus, we will establish the three conditions stated in that proposition. To do so, we first introduce some auxiliary lemmas. These require a bit of extra notation. For a vector $x\in \RR^r$, we let $S_x$ be the set of permutations of the indexes $[r]$ that orders the  entries of $x$ in non ascending order, i.e., $\pi\in S_x$
$$
x_{\pi(1)} \geq x_{\pi(2)} \geq \dots \geq  x_{\pi(r)}.
$$
Note that $S$ is not a singleton whenever there are ties.
We say that that two vectors $x, y \in \RR^r$ are similarly ordered if $S_x \cap S_y \neq \emptyset.$

\begin{lemma}
\label{lem:same-order}
    Let $\xs \in \RR^r$ be a fixed vector. Suppose $x \odot x \in \B_\varepsilon(\xs \odot \xs)$ with 
    $$\varepsilon =\min_{i,j | x_i \neq x_j } \frac{ | (\xs_{i})^2 - (\xs_{j})^2 | }{2} \bigwedge \min_{i\in [r] | \xs_i \neq 0 } \frac{(\xs_i)^2}{2} .$$ Then, $\# \supp(x) \geq \# \supp(\xs)$ and, moreover, the component-wise squares $x\odot x$ and $\xs\odot \xs$ are similarly ordered. The result holds trivially when $\varepsilon = +\infty$, i.e., all components of $\xs$ are equal.
\end{lemma}
\begin{proof}

    We defer the proof of $\# \supp(x) \geq \# \supp(\xs)$ to the end. Let us construct a permutation in the intersection of $S_{\xs \odot \xs}$ and $S_{x\odot x}.$ We start with a base permutation $\pi \in S_{\xs \odot \xs},$ which we will modify inductively. Cosider the partition $B_1, \dots, B_\ell$ of $ [r]$ such that for any $i, j \in B_k$ we have ${\xs_i}^2 = {\xs_j}^2$, and if $i \in B_n$ and $j \in B_m$ with $n<m$ then ${\xs_i}^2 > {\xs_j}^2.$ We claim that we also have $i \in B_n$ and $j \in B_m$ with $n<m$ then ${x}_i^2 > {x}_j^2.$ To see this, take $\ik = \argmin_{i \in [B_n]} x_{\pi(i)}^2$ and $\overline{j} = \argmax_{j \in [B_m]} x_{\pi(i)}^2$. Using the triangle inequality and the definition of the partition, we derive
    $$x_{\pi(\ik)}^2 - x_{\pi(\overline{j}))}^2 \geq {\xs_{\pi(\ik)}}^2 - {\xs_{\pi(\overline{j})}}^2 - 2\varepsilon > 0,$$
    where the strict inequality follows since $x\odot x \in \B_\varepsilon (\xs \odot \xs);$ which proves our claim.

    We construct $\pi'$ from $\pi$ as follows. Start with $\pi' = \pi$. We know that the indices in $B_1$ correspond to the top components in $x\odot x$ and $\xs \odot \xs$, so $\pi$ sends them to the first $\# B_1$ components. We can modify $\pi'$ to respect the ordering of the top $\# B_1$ entries of $x\odot x$. Since the new $\pi'$ only differs from $\pi$ in the $B_1$ block, it also belongs to $S_{\xs \odot \xs}.$ We can apply the same update with $B_2$, and so on until $B_\ell.$ After which, we have that $\pi' \in S_{\xs \odot \xs} \cap S_{x \odot x};$ proving that the squared vectors are similarly ordered.

    Take $\pi \in S_{\xs \odot \xs} \cap S_{x \odot x}$ and recall that $r^\star = \# \supp(\xs).$ Thus, $\left(\xs_{\pi(r^\star)}\right)^2 > 0.$ By the definition of $\varepsilon$,
    $$
    x_{\pi(r^\star)}^2 \geq \frac{(\xs_{\pi(r^\star)})^2}{2} > 0.
    $$
    Therefore, $\# \supp(x) \ge r^\star$. This finishes the proof.
\end{proof}

Recall that for $x \in \RR^r$, we let $\Pi_j^x$ and $\sigma^x_j$ correspond to the projection onto the subspace generated by the top $j$ singular vector of $\nabla \c(x)$ and its $j$th  top singular value, respectively.
\begin{lemma}
\label{lemma-had}
Let $\xs\in \RR^r$ with $r^\star $ nonzero entries and $x\in \RR^r$  such that $x \odot x \in B_\varepsilon(\xs \odot \xs)$ for  $$\varepsilon =\min_{i,j | x_i \neq x_j } \frac{ | (\xs_{i})^2 - (\xs_{j})^2 | }{2} \bigwedge \min_{i\in [r] | \xs_i \neq 0 } \frac{(\xs_i)^2}{2} .$$
Then, we have
\MDA{ If you want to use cases, use it after the $\leq$. But in this case, it is unnecessary.
$$\begin{cases}
\displaystyle
\| (I - \Pi^x_k) (x \odot x - x^* \odot x^*) \|_{2}
\;\le\; (r - r^\star)^{\tfrac12} \,\sigma_{k+1}^2,
& \quad \text{if }  k \in \{ r^\star \ldots r-1 \},\\[6pt]
\displaystyle
\| (I - \Pi^x_k) (x \odot x - x^* \odot x^*) \|_{2}
\;=\; 0,
& \quad \text{if $k=r$.}
\end{cases}
$$}
$$
\| (I - \Pi^x_k) (x \odot x - x^* \odot x^*) \|_{2} \leq \sqrt{r - k} \, (\sigma^x_{k+1})^2
\quad \text{for all }k \in \{r^\star, \dots, r\}.
$$
\end{lemma}

\begin{proof}
By Lemma~\ref{lem:same-order} there is a relabeling of the indexes \(({j_i})_{i\in [r]} \) that simultaneously sorts the magnitude of the entries of $x$ and $\xs$ in nonascending order, i.e.,
$
|x_{j_1}| \geq \ldots \geq |x_{j_r}|
$ and $|\xs_{j_1}| \geq \ldots \geq |\xs_{j_r}|$. Since
\(\nabla \c(x) = 2\,\diag(x),\) its SVD is given by
\begin{align*}
U = & (e_{j_1}, e_{j_2}, \ldots, e_{j_r}), \\
\Sigma = & 2\,\diag(|x_{j_1}|, |x_{j_2}|, \ldots, |x_{j_r}|), \\
V = & \diag(\operatorname{sign}(x_{j_1}), \operatorname{sign}(x_{j_2}), \ldots, \operatorname{sign}(x_{j_r})),
\end{align*}
where $e_k$ is the $k$-th standard vector basis.
Hence, we have
\begin{align}\label{eqn:ordering}
[(I - \Pi^x_k) v ]_i =
\begin{cases}
    0 & \text{if }  i \in \{j_1, \dots, j_k\}\\
    v_i & \text{otherwise},
\end{cases}
\end{align}
for an arbitrary vector $v \in \RR^r$. Therefore,  for any $k \in \{r^\star, \dots, r\},$ we have
{ \allowdisplaybreaks
\begin{align*}
 \norm{ (I - \Pi^x_k) (x\odot x - \xs \odot \xs) }{2}^2 &= \norm{ \left(I - \Pi^x_k\right)  x \odot x  -  \left(I - \Pi^x_k \right)  \xs\odot \xs}{2}^2 \\
 &= \norm{ (0, \ldots, 0, x_{j_{k+1}}^2, \ldots, x_{j_{r}}^2)^\top - (0, \ldots, 0)^\top }{2}^2 \\
 &= \norm{ (x_{j_{k+1}}^2, \ldots, x_{j_{r}}^2)^\top }{2}^2 \\
 &= \sum_{i=k+1}^r x_{j_{i}}^4 \\
 &\leq \left(r - k  \right) (\sigma^x_{k+1})^4,
\end{align*}
}
where the second equality follows from~\eqref{eqn:ordering}. This establishes the result.
\end{proof}

We can now finish the proof of Theorem \ref{thm: weak-alignment-hadamard}.
 Recall that $r^\star = \# \supp(\xs)$, $z = x \odot x$, and $\zs = \xs \odot \xs$.  We show that $x \mapsto x \odot x$ satisfies the three conditions from Proposition~\ref{thm: sufficient-condition-weak-alignment} and
$$\sig(\rho) = \frac{\rho}{\sqrt{r - r^\star} \vee 1} \text{ and } \delta(\rho) =  \left( \min_{i,j \mid  x^\star_i \neq  x^\star_j} \frac{ |{\xs_i}^2 - {\xs_j}^2|}{2} \right) \bigwedge \frac{ \lambda_{r^\star} ( \operatorname{diag}(\zs))}{1 + \sig(\rho)}\bigwedge \frac{ \lambda_{r^\star} ( \operatorname{diag}(\zs))}{2} .
$$
Suppose that $\|z - \zs\| \leq \delta(\rho).$
\begin{enumerate}
    \item By Lemma \ref{lemma-had},  setting $k=r$, we have that  $
    \norm{\left( I - \Pi^x_{r}  \right) (z - \zs)}{2}^2   = 0 \leq \rho \norm{z - \zs}{2}.$
    \item Take any $k \in \{ r^\star + 1, \ldots, r \}$, by Lemma \ref{lemma-had}, we have  $
    \norm{\left( I - \Pi^x_{k-1}  \right) (z - \zs)}{2}^2   \leq (r-k - 1) \sigma_k^4
    $. Thus if $ \sigma_k^2 \leq \frac{\rho }{ \sqrt{ r - r^\star} \vee 1} \norm{z -\zs}{2}$, then, $$\norm{\left( I - \Pi^x_{k-1}  \right) (z - \zs) }{2} \leq \rho\frac{\sqrt{r - k - 1}}{\sqrt{r - r^\star} \vee 1} \norm{z - \zs}{2} \leq \rho \norm{z - \zs}{2} .$$
    \item By Weyl's inequality, we have $\lambda_{r^\star} ( \operatorname{diag}(z)) \geq \lambda_{r^\star} (  \operatorname{diag}(\zs)) - \norm{z -  \zs}{2} $. By the choice of the neighborhood $\delta(\rho)$,
   we have $\lambda_{r^\star} (  \operatorname{diag}(\zs)) \geq ( 1 + \sig(\rho)) \norm{z - \zs}{2}$, we conclude that $\sigma_{r^\star}^2 = \lambda_{r^\star} \diagg{z} \geq \sig(\rho) \norm{z - \zs}{2}$.
\end{enumerate}
Therefore, by Proposition~\ref{thm: sufficient-condition-weak-alignment}, Assumption~\ref{ass:weak-alignment} holds; completing the proof.

\subsection{Proofs from Section~\ref{sec:matrices}}

\subsubsection{Proof of Theorem~\ref{thm:regularity-matrix-sym}} \label{section:proof-thm-regularity-matrix-sym}
We show that the symmetric Burer-Monteiro map $\Fsym(X) = XX^\top$ is smooth with parameter $\lc = 2$. A straightforward computation reveals that the Jacobian of this map and its adjoint act on $Y \in \RR^{d \times r}$ and $Z \in \SSS^d$ via
\begin{equation}\label{action: nablac-and-nablacT-bm} \nabla \Fsym(X)[Y] = YX^\top + XY^\top \quad \text{and} \quad \nabla \Fsym(X)^\top [Z] =  (Z + Z^\top)X.\end{equation}
Therefore,
\begin{align*}
\norm{\nabla \Fsym(X) - \nabla \Fsym(Y) }{\textrm{op}} &= \sup_{ \norm{W}{F}=1} \norm{\nabla \Fsym(X)[W] - \nabla \Fsym(Y)[W] }{F} \\
&%
=\sup_{\norm{W}{F}=1} \norm{\left( WX^\top + XW^\top \right)  -  \left(  WY^\top + YW^\top \right)  }{F} \\
&\leq \sup_{\norm{W}{F} = 1 } 2 \norm{W(X-Y)^\top}{F} \\
&\leq 2 \norm{X-Y}{F}.
\end{align*}
Thus, $\lc = 2$ as claimed.

Next, we prove weak alignment. Toward this goal, we state two auxiliary results. Consider any two $Z, Z^\star \in \SSS^d_+$ with $\rs = \rank Z^\star \leq \rank Z \leq r \le d$ and let $X, \Xs \in \RR^{d \times r}$ be any matrices such that $Z = XX^\top$ and $Z^\star = \Xs (\Xs)^\top.$ We  denote the SVD decompositions of $X$ and $\nabla \Fsym(X)$ as $\svd{X}$ and $U \Sigma V^\top$, %
respectively. We use $U_i^X, U_i$ to denote the $i$-th column of $U^x$ and $U$, respectively. With slight abuse of notation, we imagine completing the columns of $U^X$ and $U$ by choosing additional vectors in $\RR^d$ such that $\{U^X_i\}_{i=1}^{d}$ and $\{U_i\}_{i=1}^{\binom{d+1}{2}}$ forms an orthonormal basis of $\RR^d$ and $\RR^{\binom{d+1}{2}}$ (we identify $\SSS^d$ with $\RR^{\binom{d+1}{2}}$), respectively. Further, we let $\Pi^X_j$ be the orthogonal projection onto the span of the top $j$ left singular vectors of $\nabla \Fsym(X)$ and $\sigma_j = \Sigma_{jj}$ be its $j$-th singular value.

\begin{proposition}\label{lem: bm-alignment}
 Let $ \Xs \in \RR^{d\times r}$. %
 For any $\rho > 0$ and any $X \in \RR^{d \times r}$ with $$\norm{ X X^\top  - \Xs {\Xs}^\top  }{F} \leq \min\left\{ \frac{\rho }{\sqrt{2}}, \frac{1}{3} \right\} \sigma_{\rs}^2 \left( \Xs \right),$$ we have that
\begin{align*}
    \norm{ \left(I - \Pi^X_k \right)\left[{X X^\top  - \Xs {\Xs}^\top }\right]}{F}^2
    & \leq \frac{1}{16}\left(r- \rs + 1 \right) { \sigma_{k+1}^4}+ \frac{\rho^2}{2} \norm{ X X^\top  - \Xs {\Xs}^\top  }{F}^2
\end{align*}
for any %
$k\in \left\{ \rank(\nabla \Fsym (\Xs)), \ldots,  \rank(\nabla \Fsym(X)) \right\}$.
\end{proposition}
\begin{lemma}\label{lem:base_case_sym}  If $\norm{XX^\top  - \Xs {\Xs}^\top  }{\operatorname{op}} \leq \frac{1}{3} \sigma_{\rs}^2 \left( \Xs \right)$, then one has $\sigma_{\bar{r}^\star}^2 \geq \sigma_{\rs}^2 \left( X \right).$
\end{lemma}
  We will prove these two results soon. Before delving into their proof, let us use these results to derive weak alignment. To this end, we show that the Burer-Monteiro map satisfies the three conditions from Proposition \ref{thm: sufficient-condition-weak-alignment} with
   $$\sig(\rho) = \frac{4\rho}{\sqrt{2(r - \rs + 1)}} \text{ and } \delta(\rho) =\min\left\{ \frac{\rho}{\sqrt{2}}, \frac{1}{1 + \sig(\rho)}  , \frac{1}{3}   \right\}\lambda_{r^\star} \left(\Zs\right).
$$
Define $\bar r^\star = \rankk{\nabla \Fsym(\Xs)}$ and take any $Z \in \Ima \Fsym$ such that  $\norm{Z - \Zs}{F} \leq \delta(\rho)$.
\begin{enumerate}

\item Using Proposition~\ref{lem: bm-alignment} with $k=\rank \nabla \Fsym(X)$, we have  $$\norm{ \left(I - \Pi^X_{\bar r} \right) \left[Z  - \Zs\right] }{F}^2 \leq 0 + \frac{\rho^2}{2} \norm{Z-\Zs}{F}^2 \leq \rho^2 \norm{Z - \Zs}{F}^2.$$
\item Let $k \in \{ \bar{r}^\star +1, \ldots, \rank \nabla \Fsym(X) \}$ and assume $\sigma_{k}^2 \leq \sig(\rho) \norm{Z-\Zs}{F}$. Again by Proposition~\ref{lem: bm-alignment}, we have
\begin{align*}
\norm{ \left(I - \Pi^X_{k-1} \right)\left[Z-\Zs\right] }{F}^2 &  \leq  \frac{1}{16}(r - \rs +1 ) \sigma_{k}^4 + \frac{\rho^2}{2} \norm{Z-\Zs}{F}^2 \\ & \leq \left( \frac{1}{16}(r -\rs + 1)\sig(\rho)^2 + \frac{\rho^2}{2} \right) \norm{Z -\Zs}{F}^2 \\ &= \rho^2 \norm{Z -\Zs}{F}^2,
\end{align*}
where the last line follows from the definition of $\sig(\rho).$ %
\item By Lemma \ref{lem:base_case_sym}, we have that $\sigma_{\bar{r}^\star}^2 \geq \sigma_{\rs}^2 \left( X \right)$. By Weyl's inequality, we have
$ \sigma_{\rs}^2 \left( X \right) \geq \sigma_{\rs}^2 \left( \Xs \right) - \norm{Z - \Zs}{F} $. By the choice of $\delta(\rho)$, we get $\sigma_{\rs}^2 \left( \Xs \right)  \geq (\sig(\rho) + 1 - 1) \norm{Z - \Zs}{F}$ and thus $\sigma_{\bar{r}^\star}^2 \geq \sig(\rho) \norm{Z - \Zs}{F}$.
\end{enumerate}
Then, invoking Proposition~\ref{thm: sufficient-condition-weak-alignment} establishes Theorem~\ref{thm:regularity-matrix-sym}.
To complete the proof, we must still prove Proposition~\ref{lem: bm-alignment} and Lemma~\ref{lem:base_case_sym}. The following are auxiliary results for such a purpose. Lemma~\ref{lem:base_case_sym} follows directly from Lemma~\ref{lem:spectral-bm} and Lemma~\ref{prop:conv-ord}.

\begin{lemma}[{Spectral characterization for Burer-Monteiro}]
    \label{lem:spectral-bm}
    Let $\Fsym : \RR^{d\times r} \rightarrow \SSS^{d}$ be given by $\Fsym(X) = XX^\top.$ Then, the eigenpairs of $\nabla \Fsym(X) \nabla \Fsym(X)^\top$ are given by

      \[
    \left( 2\left(\sigma_{i}^2\left( X \right)+\sigma_{j}^2\left( X \right)\right),  \frac{1}{c_{ij}}\left(\ux_{i}  {\ux_{j}}^\top  + \ux_{j} {\ux_{i}}^\top \right) \right)
    \]
    for all $(i, j) \in [d] \times [d]$ with $i\le j$. Here the normalizing constants are  $c_{i,j} = 2$ if $i = j$ and $\sqrt{2}$ otherwise. Moreover, the eigenvectors form an orthonormal basis of $\SSS^d$.
\end{lemma}
\noindent Lemma~\ref{lem:spectral-bm} likely already exists in the literature. We include a proof in Appendix~\ref{sec:proof-spectral-bm} for completeness. In turn, we need to understand how to conveniently index the eigenvalues and eigenvectors of $\nabla\Fsym(X)^\top \nabla \Fsym(X).$ The next few results develop such an indexing. A direct result is Lemma~\ref{lem:base_case_sym}.

\begin{corollary}
\label{cor:eigen} Let $\Delta =  \{ (i,j) \mid 1\leq i\leq j\leq d \}.$ Then, there exists a bijection $\tau \colon \Delta \rightarrow \left[\binom{d+1}{2}\right]$ such that for $(i, j) \in \Delta$, we have
\begin{align}\begin{split}\label{eq:bijection}
\lambdaa{\nabla \Fsym(X) }{\tau(i,j)} &= 2\left(\sigma_{i}^2\left( X \right) + \sigma_{j}^2\left( X \right)\right),  \quad \text{and}\\ 
U_{\tau(i,j)} &= \frac{1}{c_{ij}}{\ux_{i} \otimeskron {\ux_{j}} + \ux_{j} \otimeskron {\ux_{i}}}.
\end{split}
\end{align}

\end{corollary}
\begin{corollary}
\label{cor:rank}
    Let $X \in \mathbb{R}^{d\times r}$ of rank $\rprime$. Then, the rank of $\nabla \Fsym(X)$ is $d\rprime - \binom{\rprime}{2}.$
\end{corollary}
\noindent These two are direct corollaries of Lemma~\ref{lem:spectral-bm}. In particular, $\bar r^\star = d\rs - \binom{\rs}{2}$ and the maximum rank of $\nabla \Fsym(X)$ is $dr - \binom{r}{2}.$ Consider the partition of $\Delta$ given by
$$\Delta_{\rs}=\{ (i,j) \mid 1\leq i \leq j \leq d \text{ and }  i \leq \rs \} \quad \text{ and } \quad \Delta^{c}_{\rs} = \Delta \setminus \Delta_{\rs}.$$

\begin{lemma}
 \label{prop:conv-ord}
 If $\norm{XX^\top  - \Xs {\Xs}^\top  }{\operatorname{op}} \leq \frac{1}{3} \sigma_{\rs}^2 \left( \Xs \right)$, then there is a bijection $\tau \colon \Delta \rightarrow \left[\binom{d+1}{2}\right]$ satisfying \eqref{eq:bijection} and
    \begin{equation}\label{eq:senior-bijection}
    \tau(i,j) < \tau(n,m)\quad \text{ for all $(i,j) \in \Delta_{\rs}$ and $(n, m) \in \Delta^c_{\rs}$ }.
    \end{equation}
 \end{lemma}
 \begin{proof}
   Consider the bijection $\tau$ furnished by Corollary~\ref{cor:eigen}.
   We claim that for any for any  $i\in \{1, \ldots, \rs\}$ and $n \in \{\rs+1, \ldots, d\}$,
   \begin{equation}\label{eq:tiny-eigenvalues}
 \lambdax{n} \leq \frac{\sigma_{i}^2\left( X \right)}{2}.
  \end{equation}
  To show this inequality, we repeatedly apply Weyl's inequality
  \begin{align*}
    \lambdax{n} \leq \lambdax{\rs + 1} &= \lambdax{\rs + 1} - \sigma_{\rs+1}^2\left( \Xs \right)\\
    & \leq \norm{XX^\top  - \Xs {\Xs}^\top  }{\operatorname{op}}  \\
    & \leq \frac{1}{2} \left(\sigma_{\rs}^2\left( \Xs \right) - \norm{XX^\top  - \Xs {\Xs}^\top  }{\operatorname{op}}  \right) \\
    & \leq \frac{1}{2} \lambdax{\rs} \leq \frac{1}{2}\sigma_{i}^2\left( X \right),
\end{align*}
where the second inequality follows since $\norm{XX^\top  - \Xs {\Xs}^\top  }{\operatorname{op}} \leq \tfrac{\sigma_{\rs}^2\left( \Xs \right)}{3}$.
     Then, for any $(i,j)\in \Delta_{\rs}$ and $(n, m) \in \Delta_{\rs}^{c}$ we have
     \begin{align*}
             \lambda_{\tau(n,m)}\left( \nabla \Fsym(X)\nabla \Fsym(X)^{\top} \right) &= 2\lambdax{n} + 2\lambdax{m} \\
                                                   &\leq 2 \sigma_{i}^2\left( X \right) \\
       &\leq 2\sigma_{i}^2\left( X \right) + 2 \sigma_{j}^2\left( X \right) = \lambda_{\tau(i,j)}\left( \nabla \Fsym(X)\nabla \Fsym(X)^{\top} \right),
     \end{align*}
     where the first inequality follows from \eqref{eq:tiny-eigenvalues}. Therefore, by~\eqref{eq:bijection} we derive $\tau(i,j)\leq \tau(n, m)$; if this inequality does not hold strictly, we could modify $\tau$ to enforce strictness without breaking bijectivity. This establishes the result.
\end{proof}

We are now ready to prove Proposition~\ref{lem: bm-alignment}.

\begin{proof}[Proof of Proposition~\ref{lem: bm-alignment}]
 Recall that $U_{i}$ and $\sigma_{i}$ denote the $i$th top left-singular vector and singular value of $\nabla \Fsym(X),$ respectively. We use ${\Xs}=\svd{{\Xs}}$ to denote the SVD of $\Xs$, further we use $\sigma^X_i = \Sigma_{ii}^{X}$ and $\sigma^{\Xs}_{i} = \Sigma_{ii}^{\Xs}$. Let $\tau$ be the bijection provided by Lemma~\ref{prop:conv-ord} and expand
 \begin{align*}
  &\left(I - \Pi^X_k \right) [X X^\top  - \Xs {\Xs}^\top] \\ &=  \left( \sum_{\ell=k+1}^{\binom{d+1}{2}} U_\ell U_\ell^\top \right)   \vect{X X^\top  - \Xs {\Xs}^\top}\\
  & = \sum_{\substack{ 1 \leq i\leq j\leq d, \\ \tau(i,j)\geq k+1 }} \frac{1}{ c_{i,j}} \left( \ux_{i} \otimeskron \ux_{j} +\ux_{j} \otimeskron \ux_{i}  \right) \left( \ux_{i} \otimeskron \ux_{j} +\ux_{j} \otimeskron \ux_{i}  \right)^\top
    \vect{X X^\top  - \Xs {\Xs}^\top },
\end{align*}
where the last equality follows from Corollary \ref{cor:eigen}, with $c_{i,j}=(2 + 2\mathbbm{1}_{i=j})$. %
The next claim will help us understand this sum.
\begin{claim}
\label{action_nablac_1}
 For any $1 \leq i \leq j \leq d$, we have
     \begin{align*}
&  \left( {\ux_{i}} \otimeskron {\ux_{j}} + {\ux_{j}} \otimeskron {\ux_{i}} \right) \left( {\ux_{i}} \otimeskron {\ux_{j}} + {\ux_{j}} \otimeskron {\ux_{i}} \right)^\top  \vect{ \Xs\Xs^\top}\\
&\qquad = 2 {\ux_i}^{\top} \Xs {\Xs}^\top \ux_j \left( {\ux_{i}} \otimeskron {\ux_{j}} + {\ux_{j}} \otimeskron {\ux_{i}} \right) 
\end{align*}
and
\begin{align*}
&  \left( {\ux_{i}} \otimeskron {\ux_{j}} + {\ux_{j}} \otimeskron {\ux_{i}} \right) \left( {\ux_{i}} \otimeskron {\ux_{j}} + {\ux_{j}} \otimeskron {\ux_{i}} \right)^\top  \vect{ XX^\top}\\
&\qquad = \begin{cases}
    4 \sigma_{i}^2\left( X \right)\cdot \ux_{i} \otimeskron {\ux_{i}}  & \text{if } i = j, \\
    0 & \text{otherwise.}
    \end{cases}
\end{align*}

\end{claim}
\noindent To prove this claim, for any $\tilde X \in \RR^{d\times r}$, we apply properties of the Kronecker product, \eqref{kr_p1} and \eqref{kr_p2} to $\Xprime {\Xprime}^\top$ to derive
  \begin{align*}
 \left( {\ux_{i}} \otimeskron {\ux_{j}} + {\ux_{j}} \otimeskron {\ux_{i}} \right)^\top  \vect{ \Xprime {\Xprime}^\top} & = 2{\ux_{i}}^{\top}\Xprime {\Xprime}^{\top}\ux_{j}  \\
 & = 2 \sum_{\ell = 1}^{r}\lambda_{\ell}(\Xprime {\Xprime}^{\top})\inner{\uxprime_{\ell}}{\ux_{i}}\inner{\uxprime_{\ell}}{\ux_{j}}.
  \end{align*}
 The first and second equations imply the first and second statements, respectively.
  
  Applying Claim~\ref{action_nablac_1} and properties of the Kronecker product yields
  \begin{align*}
    \left(I - \Pi^X_k \right) [X X^\top  - \Xs {\Xs}^\top] &= \underbrace{ 4 \sum_{\substack{ 1 \leq i \leq d, \\ \tau(i,i)\geq k+1 }}  \frac{1}{4}\left( \ux_{i} \otimeskron \ux_{i} \right) \sigma_{i}^2\left( X \right)}_{T_{1} := } \\  & \quad - \underbrace{2\sum_{\substack{ 1 \leq i\leq j\leq d \\ \tau(i,j)\geq k+1 }} \frac{1}{c_{i,j}}   {\ux_i}^{\top} \Xs {\Xs}^\top \ux_j \left( {\ux_{i}} \otimeskron {\ux_{j}} + {\ux_{j}} \otimeskron {\ux_{i}} \right)}_{T_{2} :=}.
    \end{align*}
with $c_{i,j}=(2 + 2\mathbbm{1}_{i=j})$. Taking the Frobenius norm and applying Young's inequality, the inequality yields
\begin{align*}
  \left\| \left(I - \Pi^X_k \right)  \left[X X^\top  - \Xs {\Xs}^\top \right] \right\|_{F}^{2} \leq 2\|T_{1}\|_{2}^{2} + 2\|T_{2}\|_{2}^{2}.
\end{align*}
We focus on bounding each term. 

Since $k \geq \rank \nabla \Fsym(\Xs) = \# \Delta_{\rs}$, and $\tau$ satisfies \eqref{eq:bijection} and \eqref{eq:senior-bijection} must have that \begin{equation}\label{eq:damn-you-k} \tau(i,i) \ge k+1 \quad \text{implies both} \quad i > r^\star \quad \text{and} \quad \lambda_i(XX^\top) \le \frac{1}{4} \sigma_{k+1}^2.
   \end{equation} 
Equipped with these facts, we use the triangle inequality to obtain
\begin{align*}
  \norm{T_{1}}{2}^2 &           = \norm{ \sum_{\substack{ 1 \leq i \leq d, \\ \tau(i,i)\geq k+1 }} \left( \ux_{i} \otimeskron \ux_{i} \right) \sigma_{i}^2\left( X \right) }{2}^2 \\
& =  \sum_{\substack{ 1 \leq i \leq d, \\ \tau(i,i)\geq k+1 }} \lambda_i^2(XX^\top) \cdot \norm{  \left( \ux_{i} \otimeskron \ux_{i} \right) }{2}^2 \\
& =   \sum_{\substack{ 1 \leq i \leq d, \\ \tau(i,i)\geq k+1 }}  \frac{1}{16}\sigma_{k+1}^4\\
& \leq \frac{1}{16}(r -\rs+1)\sigma_{k+1}^4,
\end{align*}
where the second line follows from the orthonormality of $\ux_{i} \otimeskron \ux_{i} $, and the last two lines follows from~\eqref{eq:damn-you-k}.

Finally, we turn to the bounding $T_{2}.$
Thus, expanding $T_{2},$ we get
{\allowdisplaybreaks
\begin{align*}
  \norm{T_{2}}{2}^{2}
    &= \norm{ 2\sum_{\substack{ 1 \leq i\leq j\leq d \\ \tau(i,j)\geq k+1 }} \frac{1}{(2 + 2\mathbbm{1}_{i=j})}  {\ux_i}^{\top} \Xs {\Xs}^\top \ux_j \left( {\ux_{i}} \otimeskron {\ux_{j}} + {\ux_{j}} \otimeskron {\ux_{i}} \right)}{2}^2\\
     & \leq  \norm{ \sum_{\substack{ 1 \leq i\leq j\leq d \\ \tau(i,j)\geq k+1 }}  {\ux_{i}}^\top \Xs {\Xs}^\top {\ux_{j}} \left( {\ux_{i}} \otimeskron {\ux_{j}} \right)  }{2}^2\\
     &=  \norm{ \sum_{\substack{ 1 \leq i\leq j\leq d \\ \tau(i,j)\geq k+1 }}  \ux_{i} {\ux_{i}}^\top \Xs {\Xs}^\top {\ux_{j}}  {\ux_{j}}^\top   }{F}^2,
  \end{align*}}
  where the inequality follows from Cauchy-Schwarz inequality and  $(2 + 2\mathbbm{1}_{i=j}) \geq 2$. By the argument as \eqref{eq:damn-you-k}, we  have that \begin{align*}
      \tau(i,j) \ge k +1 \quad  \text{implies} \quad \min\{i,j\} > r^\star.
  \end{align*} 
  Hence, we  have
{\allowdisplaybreaks \begin{align*}
 \norm{T_{2}}{2}^{2}&\leq
    \norm{  \sum_{i=r^\star + 1}^{d}  {\ux_{i}} {\ux_{i}}^\top  \Xs {\Xs}^\top   \sum_{\substack{1\leq j \leq d\\\tau(i,j) \geq k+1}} {\ux_{j}} {\ux_{j}}^\top }{F}^2 \\
 &\leq \norm{ \left( \sum_{i=\rs+1}^{d} {\ux_{i}} {\ux_{i}}^\top \right) \Xs{\Xs}^\top  \left( \sum_{j=\rs+1}^d {\ux_{j}} {\ux_{j}}^\top \right) }{F}^2 \\
                    &\leq \frac{\rho^2}{2} \norm{ X X^\top  - \Xs {\Xs}^\top  }{F}^2,
\end{align*}}
where the second and third inequalities follow from Lemma~\ref{lem:adding_stuff_to_orhogonal} and Lemma~\ref{second-term-assymetric}, respectively.

\end{proof}
\noindent This completes the proof of Theorem~\ref{thm:regularity-matrix-sym}.

\subsubsection{Proof of Theorem~\ref{thm:regularity-matrix-asym}} \label{section:proof-thm-regularity-matrix-asym}
We first show that the asymmetric matrix factorization map $\Fasym(X,Y) = XY^\top$ satisfies Assumption~\ref{assum:mapc:smoothness} with parameter $\lc = \sqrt{2}$.  A straight computation establishes that the Jacobian of this map and its adjoint acts on $(\Xprime, \Yprime) \in \RR^{d_1\times r} \times \RR^{d_2 \times r}$ and $Z \in \RR^{d_1\times d_2}$  via \begin{equation}\label{action: nablac-and-nablacT-asym} \nabla \Fasym(X,Y)[(\Xprime, \Yprime)] = X\Yprime^\top + \Xprime Y^\top \quad \text{and} \quad \nabla \Fasym(X,Y)^\top [Z] = (ZY, Z^\top X).\end{equation}
Therefore,
\begin{align*}
\norm{\nabla \Fasym(X,Y) - \nabla \Fasym(\Xprime, \Yprime) }{\textrm{op}} &= \sup_{\norm{(A_1,A_2)}{F} = 1} \norm{ \left(\nabla \Fasym(X,Y) - \nabla \Fasym(\Xprime, \Yprime) \right) [(A_1,A_2)] }{F} \\
&= \sup_{\norm{(A_1,A_2)}{F} = 1} \norm{\left(  (X-\Xprime) {A_2}^\top \right) + \left( A_1(Y-\Yprime)^\top  \right) }{F} \\
&\leq \sup_{\norm{(A_1,A_2)}{F} = 1}  \norm{\left(X - \Xprime\right)A_2^\top}{F} + \norm{A_1\left(Y-\Yprime \right)}{F } \\
&\leq  \norm{X - \Xprime}{F} + \norm{Y-\Yprime}{F } \\
& \leq \sqrt{2} \norm{(X,Y) - (\Xprime, \Yprime)}{F},
\end{align*}
where the last inequality comes from Young's inequality. Thus $\lc=\sqrt{2}$.

We turn to proving local weak alignment. Let us introduce some notation. Recall that we fixed a factorization $Z^\star = X^\star (Y^\star)^\top$ with $\rankk{X^\star} = \rankk{Y^\star} = \rankk{Z^\star} = \rs$. Consider any matrix $Z$ with $\rs \leq \rank (Z) \leq r$, and let $X \in \RR^{d_1 \times r}$, $Y \in \RR^{d_2 \times r}$ be any matrices such that $Z = XY^\top$. We  denote the SVD decompositions of $X,Y$ and $\nabla \Fasym(X,Y)$ as $\svd{X}$, $\svd{Y}$, and $U \Sigma V^\top$, %
respectively.  We use $U_i^X, U_i^Y,U_i$ to denote the $i$-th column of $U^X$ , $U^Y$, and $U$, respectively. With slight abuse of notation, we imagine completing the columns of $U^X$, $U^Y$, and $U$ and adding additional vectors such that $\{U^X_i\}_{i=1}^{d_1}$, $\{U^Y_i\}_{i=1}^{d_2}$,  and $\{U_i\}_{i=1}^{d_1d_2}$ form an orthonormal basis of $\RR^{d_1}$, $\RR^{d_2}$, and $\RR^{d_1d_2}$, respectively. Further, we let $\Pi^{(X,Y)}_j$ be the orthogonal projection onto the span of the top $j$ left singular vectors of $\nabla \Fasym(X,Y)$ and $\sigma_j$ be its $j$th singular value. Moreover, we denote by $\bar r$ the rank of $\nabla \Fasym(X,Y)$ and by $\bar r^\star$ the rank of $\nabla \Fasym(\Xs,\Ys)$. We state two key results that underline our arguments. 
\begin{proposition}
\label{lem: asym-alignment}
    Let $(\Xs,\Ys) \in \mathbb{R}^{d_1 \times r} \times \mathbb{R}^{d_2\times r}$ be a factorization of $Z^\star = X^\star (Y^\star)^\top$ satisfying $ \rankk{\Xs} = \rankk{\Ys} = \rs $ and $V^\Xs = V^\Ys$. Let $(X,Y) \in \mathbb{R}^{d_1 \times r} \times \mathbb{R}^{d_2\times r}$ be a pair of factors that satisfies $\norm{(X,Y)-(\Xs,\Ys)}{F} \leq \epsassymmatrix.$ Then, for any $\rho > 0$, if
$$\norm{XY^\top - \Xs\Ys^\top}{F} \leq \rho \frac{ \min\left\{\sigma^2_{\rs}\left( \Xs \right), \sigma^2_{\rs}\left( \Ys \right)\right\}}{4},
$$
we have that
    \begin{align*}
    &\norm{\left(I - \Pi^{(X,Y)}_{k} \right) [XY^\top - \Xs\Ys^\top ]}{F}^2 \\ 
& \qquad \leq \left(5\sqrt{2}\frac{ \sigma^2_{\rs}\left( \Xs \right) + \sigma^2_{\rs}\left( \Ys \right) }{\min\left\{\sigma^2_{\rs}\left( \Xs \right), \sigma^2_{\rs}\left( \Ys \right)\right\}} (r-\rs+1)^2 \right)^2 \sigma^4_{k+1} + \frac{\rho^2}{2} \norm{XY^\top - \Xs\Ys^\top}{F}^2
\end{align*}
    for any $k\in \left\{ \rank(\nabla \Fasym (\Xs,\Ys)), \ldots,  \rank(\nabla \Fasym(X,Y)) \right\}$.  %
\end{proposition} 
\begin{lemma}\label{lem:base_case_asym}
    Suppose $(\Xs,\Ys) \in \mathbb{R}^{d_1 \times r} \times \mathbb{R}^{d_2\times r}$ satisfies $\rankk{\Xs} = \rankk{\Ys} = \rs$. Further, let $(X,Y) \in \mathbb{R}^{d_1 \times r} \times \mathbb{R}^{d_2\times r}$ be matrices satisfying $\norm{(X,Y)-(\Xs,\Ys)}{F} \leq  \frac{1}{16\sqrt{2}} \frac{ \min  \left\{  \sigma_{\rs}^2\left(\Xs\right), \sigma_{\rs}^2\left(\Ys\right) \right\}}  {\max \left\{ {\sigma_1\left(\Xs\right)}, {\sigma_1\left(\Ys\right)}  \right\}}.$ Then,
     $$\sigma_{\bar r^\star}^2 \geq \min\left\{ \sigma_{\rs}^2 \left( X \right), \sigma_{\rs}^2\left(Y\right) \right\}.$$
\end{lemma}
 We will soon prove these two results. Before that, let us use them to derive the local weak alignment property. To this end, we show that the asymmetric map satisfies the three conditions from Proposition~\ref{thm: sufficient-condition-local-weak-alignment} with
\begin{align*}
\alignvarepsilon & =\epsassymmatrix, \\
\sig(\rho) &= \frac{\rho \min\left\{\sigma^2_{\rs}\left( \Xs \right), \sigma^2_{\rs}\left( \Ys \right)\right\}}{10\sqrt{2}\left( \sigma^2_{\rs}\left( \Xs \right) + \sigma^2_{\rs}\left( \Ys \right) \right) (r-\rs+1)^2}, \quad \text{and} \\
\delta(\rho) &= \min\left\{ \frac{\rho}{4}, \frac{1}{4 \sig(\rho)} \right\} \min\left\{ \sigma^2_{\rs}\left( \Xs \right), \sigma^2_{\rs}\left( \Ys \right) \right\}.
\end{align*}
Take any $XY^\top = Z \in \Ima F$ such that $\norm{Z-Z^\star}{F} \leq \delta(\rho)$ and $\norm{(X,Y) - (\Xs,\Ys)}{F} \le \alignvarepsilon$.
\begin{enumerate}
\item Applying Proposition~\ref{lem: asym-alignment} with $\bar r=\rank \nabla F(X,Y)$, we have  $$\norm{ \left(I - \Pi^{(X,Y)}_{\bar r} \right)  \left[ Z  - Z^\star\right] }{F}^2 \leq \rho^2 \norm{Z-Z^\star}{F}^2.$$
\item Let $k \in \{ \bar r^\star + 1, \ldots , \bar r \}$ and assume $\sigma_{k}^2 \leq \sig(\rho) \norm{Z-Z^\star}{F}$. Again by Proposition~\ref{lem: asym-alignment}, we have
\begin{align*}
&\norm{ \left(I - \Pi^{(X,Y)}_{k-1} \right) \left[Z  - Z^\star \right] }{F}^2  
\\
& \qquad \leq \left( 5\sqrt{2} \frac{ \sigma^2_{\rs}\left( \Xs \right) + \sigma^2_{\rs}\left( \Ys \right) }{\min\left\{\sigma^2_{\rs}\left( \Xs \right), \sigma^2_{\rs}\left( \Ys \right)\right\}} (r-\rs+1)^2 \right)^2 \sigma_{k}^4 + \frac{\rho^2}{2} \norm{Z-Z^\star}{F}^2 
\\ & \qquad \leq \left( \left( 5\sqrt{2} \frac{ \sigma^2_{\rs}\left( \Xs \right) + \sigma^2_{\rs}\left( \Ys \right) }{\min\left\{\sigma^2_{\rs}\left( \Xs \right), \sigma^2_{\rs}\left( \Ys \right)\right\}} (r-\rs+1)^2 \right)^2\sig(\rho)^2 + \frac{\rho^2}{2} \right) \norm{Z-Z^\star}{F}^2 
\\ & \qquad = \rho^2 \norm{Z-Z^\star}{F},
\end{align*}
where the last equality follows from the definition of $s(\rho)$.
\item  Assume without loss of generality that $\min\left\{ \sigma_{\rs}^2 \left( X \right), \sigma_{\rs}^2 \left( Y \right)\right\} = \sigma_{\rs}^2 \left( X \right)$.  We have:
\begin{align*}
 \sigma_{\bar r^\star}^2 &\geq \min\left\{ \sigma_{\rs}^2 \left( X \right), \sigma_{\rs}^2 \left( Y \right)\right\} \\
&\geq \left(\sigma_{\rs}\left( \Xs \right) - \norm{X - \Xs}{F}\right)^2 \\ 
&\geq \frac{1}{4}\min\left\{ \sigma^2_{\rs}\left( \Xs \right), \sigma^2_{\rs}\left( \Ys \right) \right\} \\
&\geq  \sig(\rho) \norm{XY^\top - \Xs\Ys^\top}{F},
\end{align*}
where the first line follows from Lemma~\ref{lem:base_case_asym}, the second line follows from Weyl's inequality and the choice of $\alignvarepsilon$, and the last line follows from the choice of $\delta(\rho)$. 
\end{enumerate} 
Then, the assumptions of Proposition~\ref{thm: sufficient-condition-local-weak-alignment} hold, which establishes Theorem~\ref{thm:regularity-matrix-asym}.

To complete the proof, we must still prove Proposition~\ref{lem: asym-alignment} and Lemma~\ref{lem:base_case_asym}. The following are auxiliary results for such a purpose. Lemma~\ref{lem:base_case_asym} follows directly from Lemma~\ref{lem: spectral_asymmetric} and Lemma~\ref{lem:fatahi-asym-convenable-ordering}.
\begin{lemma}[{Spectral Characterization}]
\label{lem: spectral_asymmetric}
    The eigenpairs of $\nabla F(X,Y) \nabla F(X,Y)^\top$ are given by
    $$
    \sigma_{i}^2\left( X \right) +  \sigma_{j}^2\left( Y \right)  \text{ with eigenvector }   \uy_{j} {\ux_{i}}^\top 
    $$
    for all $i\in [d_1]$ and $j \in [d_2]$. Moreover, these eigenvectors form an orthonormal basis.
\end{lemma}
\noindent Lemma~\ref{lem: spectral_asymmetric} likely already exists in the literature. We include a proof in Appendix~\ref{sec:proof-spectral-asym}. %

\begin{lemma}\label{lem:distance_factors_implies_distance_op_asymmetric}
Suppose that
\(
\norm{(X,Y)-(\Xs,\Ys)}{F} \leq \epsassymmatrix.
\)
Then, 
\[
\norm{XX^\top- \Xs\Xs^\top}{F}+\norm{YY^\top- \Ys\Ys^\top}{F} \leq \frac{1}{2\sqrt{2}}\min\left\{\sigma^2_{\rs}\left( \Xs \right), \sigma^2_{\rs}\left( \Ys \right)\right\}.
\]
\end{lemma}
\noindent We defer the proof of this lemma to Appendix~\ref{proof: lem-distance_factors_implies_distance_op_asymmetric}.
\begin{corollary}
\label{cor: ass-eig}
    There exists a bijection $\tau\colon [d_1] \times [d_2] \mapsto [d_1d_2]$ such that
    \begin{equation}\label{eq: ass-eig}
        \sigma^2_{\tau(i,j)} = \sigma_{i}^2 \left( X \right) + \sigma_{j}^2 \left( Y \right) \text{ and } U_{\tau(i,j)} = U^Y_j {U^X_i}^{\top}.
    \end{equation}
\end{corollary}
\begin{corollary}
 Let $X,Y \in \mathbb{R}^{d_1\times r} \times  \mathbb{R}^{d_2\times r} $ be of the ranks $r_1, r_2$. Then, the rank of $\nabla F(X,Y)$ is $d_1r_2 + d_2 r_1 - r_1 r_2$.
\end{corollary}
\noindent These two corollaries are direct consequences of Lemma~\ref{lem: spectral_asymmetric}. In particular, $\bar r^\star = (d_1+d_2 - \rs)\rs$, and the maximum rank of $\nabla F(X,Y)$ is $(d_1+d_2 - r)r$. Consider the partition of $[d_1]\times [d_2]$ given by
$$\Delta_{\rs}=\{ (i,j) \mid i\leq \rs \text{ or }  j \leq \rs \} \quad \text{ and } \quad \Delta^{c}_{\rs} = \Delta \setminus \Delta_{\rs},$$and observe that $\card{\Delta_{\rs}}= \rankk{\nabla F(\Xs,\Ys)}$. We derive a useful lemma for alignment.

\begin{lemma}\label{lem:fatahi-asym-convenable-ordering}
If $
    \norm{ XX^\top- \Xs\Xs^\top }{\operatorname{op}} + \norm{ YY^\top- \Ys\Ys^\top }{\operatorname{op}} \leq \frac{1}{2} \min\left\{\sigma^2_{\rs}\left( \Xs \right), \sigma^2_{\rs}\left( \Ys \right) \right\},$ then there exists a bijection $\tau$ satisfying~(\ref{eq: ass-eig}) such that 
    $$
        \tau(i,j)  < \tau(l,m)  \quad \text{for all } (i,j) \in \Delta_{\rs}\text{ and }(l,m) \in \Delta_{\rs}^c.
    $$
Consequently, if $\tau(l,m) > \rankk{ \nabla F(\Xs, \Ys)}$, then $l > \rs$ and $m > \rs$. 
\end{lemma}
\begin{proof}[Proof of Lemma~\ref{lem:fatahi-asym-convenable-ordering}]
To establish that $\sigma_{i}^2 \left( X \right) + \sigma_{j}^2 \left( Y \right) \geq \sigma_{l}^2 \left( X \right) + \sigma_{m}^2 \left( Y \right)$ for any $(i,j) \in \Delta_{\rs}\text{ and }(l,m) \in \Delta_{\rs}^c$, it is sufficient to show 
\begin{equation}\label{eq: main}
\sigma_{\rs+1}^2 \left( X \right) + \sigma_{\rs+1}^2 \left( Y \right) \leq  \min\left\{ \sigma_{\rs}^2 \left( X \right) ,  \sigma_{\rs}^2 \left( Y \right)\right\}.
\end{equation} 
The bound on $\norm{ XX^\top- \Xs\Xs^\top }{\operatorname{op}} + \norm{ YY^\top- \Ys\Ys^\top }{\operatorname{op}}$ %
implies that 
\begin{align}\label{ineq:intermed}
\norm{ XX^\top- \Xs\Xs^\top }{\operatorname{op}}& +  \norm{ YY^\top- \Ys\Ys^\top }{\operatorname{op}} \nonumber \\ &\leq \min\left\{ \sigma^2_{\rs}\left( \Xs \right)  - \norm{ XX^\top- \Xs\Xs^\top }{\operatorname{op}}, \sigma^2_{\rs}\left( \Ys \right)  - \norm{ YY^\top- \Ys\Ys^\top }{\operatorname{op}} \right\} \nonumber \\ & \leq \min\left\{ \sigma_{\rs}^2 \left( X \right), \sigma_{\rs}^2 \left( Y \right) \right\},
\end{align}
where the second inequality follows from Weyl's inequality.
To establish~\eqref{eq: main}, we bound 
\begin{align*}
\sigma_{\rs+1}^2 \left( X \right) + \sigma_{\rs+1}^2 \left( Y \right)  & \leq \norm{ XX^\top- \Xs\Xs^\top }{\operatorname{op}} +\norm{ YY^\top- \Ys\Ys^\top }{\operatorname{op}}  \leq \min\left\{ \sigma_{\rs}^2 \left( X \right), \sigma_{\rs}^2 \left( Y \right) \right\},
\end{align*}
where the first and second inequality follow from Weyl's and~\eqref{ineq:intermed}. This concludes the proof.
\end{proof}

We are now ready to prove Proposition~\ref{lem: asym-alignment}. 

\begin{proof}[Proof of Proposition~\ref{lem: asym-alignment}] 
We start by invoking the triangle inequality to decompose
$$\norm{\left(I - \Pi^{(X,Y)}_{k} \right) [XY^\top - \Xs\Ys^\top ]}{F} \leq \underbrace{\norm{\left(I - \Pi^{(X,Y)}_{k} \right) [XY^\top]}{F}}_{T_1:=} + \underbrace{\norm{\left(I - \Pi^{(X,Y)}_{k} \right) [\Xs\Ys^\top]}{F}}_{T_2:=}.$$ 
We will provide upper bounds for both $T_1$ and $T_2$.
We begin with $T_1$, Corollary~\ref{cor: ass-eig} yields
{ \allowdisplaybreaks
\begin{align*}
    T_1^2
    &= \norm{\sum_{i={k+1}}^{d_1d_2} U_{i} U_{i}^\top \Vect \left( XY^\top\right) }{2}^2 \\
    & =\norm{\sum_{\substack{(i,j) \mid \tau(i,j) \geq {k+1}}} \left( {\uy_{j}} \otimeskron {\ux_{i}} \right) \left( {\uy_{j}} \otimeskron {\ux_{i}} \right)^\top \Vect \left( XY^\top\right) }{2}^2 \\
    & \overset{(i)}{=} \norm{\sum_{\substack{(i,j) \in [r] \times [r] \\ \tau(i,j) \geq {k+1}}} \sigma_i^X \sigma_j^Y \inner{V_i^X}{V_j^Y} \left( {\uy_{j}} \otimeskron {\ux_{i}} \right)}{2}^2\\
    &\overset{(ii)}{=} \sum_{\substack{(i,j) \in [r] \times [r] \\ \tau(i,j) \geq {k+1}}} \norm{\sigma_i^X \sigma_j^Y \inner{V_i^X}{V_j^Y} \left( {\uy_{j}} \otimeskron {\ux_{i}} \right) }{2}^2 \\
    &\overset{(iii)}{\leq} \sum_{\substack{(i,j) \in [r] \times [r] \\ \tau(i,j) \geq {k+1}}} \frac{1}{4} \left( \sigma_{i}^2\left( X \right) + \sigma_{j}^2\left( Y \right) \right)^2 \\
    &\overset{(iv)}{\leq} \frac{1}{4} \#{\Big\{ (i,j) \in [r]^2 \mid \rank \nabla F(\Xs, \Ys) + 1 \leq \tau(i,j) \leq \rank \nabla F(X,Y)\Big\}} \sigma_{k+1}^4.
\end{align*}
}
Here, $(i)$ follows since $XY^\top= \sum_{k=1}^r \sum_{\ell=1}^r \sigma_k^X \sigma_\ell^Y \inner{V_k^X}{V_\ell^Y} {\ux_k} {\uy_{\ell}}^\top$ and using \eqref{kr_p2} with \eqref{kr_p1} we derive
\[
( {\uy_{j}}^\top \otimeskron {\ux_{i}}^\top ) \Vect \left( XY^\top\right) = \begin{cases}
    \sigma_i^X \sigma_j^Y \inner{V_i^X}{V_j^Y} & \text{if $i \le r$ and $j \le r$},\\
    0 & \text{otherwise.}
\end{cases}
\]
On the other hand, $(ii)$ follows from orthogonality and cancellation of cross-terms, $(iii)$ from the Cauchy–Schwarz inequality and from Young's inequality, and $(iv)$ follows from Lemma~\ref{lem: spectral_asymmetric}. Combining this bound with the fact that $\#\Big\{ (i,j) \in [r]^2 \mid \rank\nabla F(\Xs, \Ys) + 1 \leq \tau(i,j) \leq \rank \nabla F(X,Y)\Big\} \leq (r-\rs+1)^2$ yields $T_1 \leq \frac{1}{2}(r-\rs + 1) \sigma_{k+1}^2.$

We continue by bounding the term $T_2$. Let $(\ik,\jk)$ be the pair such that $\tau(\ik,\jk)= k+1$, then 
{\allowdisplaybreaks 
\begin{align*}
    T_2&\overset{}{=}
    \norm{ \sum_{\substack{(i,j) \mid \tau(i,j) \geq k+1}}
 \left( {\uy_{j}} \otimeskron {\ux_{i}} \right) \left( {\uy_{j}}\otimeskron {\ux_{i}}\right)^\top 
     \vect{\Xs\Ys^\top} }{2}
    \\
     &\overset{(i)}{\leq}
    \norm{ \sum_{i=\ik}^{d_1}\sum_{j=\jk}^{d_2} \left( {\uy_{j}} \otimeskron {\ux_{i}} \right) \left( {\uy_{j}}\otimeskron {\ux_{i}}\right)^\top 
     \vect{\Xs\Ys^\top}}{2}
    \\
     &\overset{(ii)}{=}
     \norm{ \left( \sum_{i=\ik}^{d_1}
        {\ux_{i}} {\ux_{i}}^\top \right) \Xs\Ys^\top
     \left( \sum_{j=\jk}^{d_2}
    {\uy_{i}} {\uy_{i}}^\top \right)}{F}
    \\
    &\overset{(iii)}{\leq} \frac{1}{\min\left\{\sigma^2_{\rs}\left( \Xs \right), \sigma^2_{\rs}\left( \Ys \right)\right\}} \norm{ \left( \sum_{i=1}^{\ik-1}
    {\ux_{i}} {\ux_{i}}^\top \right) XY^\top   \left( \sum_{j=1}^{\jk - 1}
    {\uy_{j}} {\uy_{j}}^\top \right) - \Xs\Ys^\top}{F}^2 \\
    &\overset{(iv)}{\leq }  \frac{2\left(  \norm{  \left( \sum_{i=1}^{\ik-1}
    {\ux_{i}} {\ux_{i}}^\top \right) XY^\top   \left(\sum_{j=1}^{\jk - 1}
     {\uy_{j}} {\uy_{j}}^\top \right) - XY^\top}{F}^2 + \norm{ XY^\top- \Xs\Ys^\top}{F}^2 \right)}{\min\left\{\sigma^2_{\rs}\left( \Xs \right), \sigma^2_{\rs}\left( \Ys \right)\right\}} \\ 
    & \overset{(v)}{\leq }  \underbrace{\frac{2\left(  \norm{ \left(\sum_{i=1}^{\ik-1} {\ux_{i}} {\ux_{i}}^\top \right) XY^\top    \left( \sum_{j=1}^{\jk - 1}{\uy_{j}} {\uy_{j}}^\top \right) - XY^\top}{F}^2 \right)}{\min\left\{\sigma^2_{\rs}\left( \Xs \right), \sigma^2_{\rs}\left( \Ys \right)\right\}}}_{T_3:=}  + \frac{\rho}{2} \norm{ XY^\top- \Xs\Ys^\top}{F},
\end{align*}}
where $(i)$ follows from the definition of $(\ik,\jk)$ and from Lemma~\ref{lem:adding_stuff_to_orhogonal}, $(ii)$ follows from the Kronecker product properties~\eqref{kr_p1},~\eqref{kr_p2}, and~\eqref{kr_p3},
$(iii)$ follows from Lemma~\ref{second-term-assymetric} together with Lemma~\ref{lem:distance_factors_implies_distance_op_asymmetric} and Lemma~\ref{lem:fatahi-asym-convenable-ordering}, $(iv)$ follows from adding and substracting $XY^\top$ in conjunction with Young's inequality, and $(v)$ follows from the initial condition $\norm{XY^\top - \Xs\Ys^\top}{F} \leq \rho \frac{ \min\left\{\sigma^2_{\rs}\left( \Xs \right), \sigma^2_{\rs}\left( \Ys \right)\right\}}{4}$.

Next, we provide a bound for the first term of the right-hand side of (vi). Let us denote $\mathcal I_{QQ} := \{(i,j)\in\mathbb N^{2} : \ik\le i\le r, \jk\le j\le r\}$, $\mathcal I_{PQ} := \{(i,j)\in\mathbb N^{2} : 1\le i<\ik, \jk\le j\le r\}$, $\mathcal I_{QP} := \{(i,j)\in\mathbb N^{2} : \ik\le i\le r, 1\le j<\jk\}$, and $\mathcal{I} := \mathcal{I}_{QQ} \cup \mathcal{I}_{PQ} \cup \mathcal{I}_{QP}$. Then by orthonormality of the basis  $\{U^Y_j \otimes U^X_i \mid (i,j) \in [d_1] \times [d_2] \}$ and since $\text{vec}(U^X_i {U^Y_j}^\top) = U^Y_j \otimeskron U^X_i$, we have that 
\begin{align*}
    T_3 = \underbrace{2\frac{ \sum_{(i,j)\in \mathcal{I}_{QQ}} (\sigma_i^X)^2 (\sigma_j^Y)^2 \inner{V^X_i}{V^Y_j}^2}{\min\left\{\sigma^2_{\rs}\left( \Xs \right), \sigma^2_{\rs}\left( \Ys \right)\right\}}}_{T_4:=}  + \underbrace{2\frac{\sum_{(i,j)\in \mathcal{I}_{PQ} \cup\mathcal{I}_{QP}} (\sigma_i^X)^2 (\sigma_j^Y)^2 \inner{V^X_i}{V^Y_j}^2 }{\min\left\{\sigma^2_{\rs}\left( \Xs \right), \sigma^2_{\rs}\left( \Ys \right)\right\}}}_{T_5:=} .
     \end{align*}
We next bound $T_4$ and $T_5$. For $T_4$, \begin{align*}
T_4 &\leq 2\frac{ \sigma_{\rs}(X) \sigma_{\rs}(Y) \sum_{i=\ik}^r \sum_{j=\jk}^r \sigma_i^X\sigma_j^Y  }{\min\left\{\sigma^2_{\rs}\left( \Xs \right), \sigma^2_{\rs}\left( \Ys \right)\right\}}
\\ &\leq \frac{1}{2}\frac{ \sigma_{\rs}^2 \left( X \right) + \sigma_{\rs}^2 \left( Y \right) }{\min\left\{\sigma^2_{\rs}\left( \Xs \right), \sigma^2_{\rs}\left( \Ys \right)\right\}} \sum_{i=\ik}^r \sum_{j=\jk}^r \left( \sigma_{i}^2 \left( X \right) + \sigma_{j}^2 \left( Y \right) \right )
\\ &\leq 2\frac{ \sigma^2_{\rs}\left( \Xs \right) + \sigma^2_{\rs}\left( \Ys \right) }{\min\left\{\sigma^2_{\rs}\left( \Xs \right), \sigma^2_{\rs}\left( \Ys \right)\right\}} (r-\rs+1)^2 \sigma_{k+1}^2 ,
\end{align*}
where the first inequality is due to Cauchy–Schwarz, the second is due to Young's inequality applied twice, and the third is due to Corollary~\ref{cor: ass-eig} and the initial conditions in conjunction with Weyl's inequality. For the $T_5$, we define $R:=\argmin_{\Rprime\in O(r-\rs)} \norm{\left( V^Y_{ \{\rs + 1 \ldots r\}} - V^{\Ys}_{\{ \rs + 1 \ldots r \}}\Rprime \right)}{F}^2$. Next we only bound the terms in $T_5$ associated with the indices in $\cI_{PQ}$, that is
{\allowdisplaybreaks \begin{align*}
&\frac{ \sum_{(i,j)\in \mathcal{I}_{PQ}} (\sigma_i^X)^2 (\sigma_j^Y)^2 \inner{V^X_i}{V^Y_j}^2}{\min\left\{\sigma^2_{\rs}\left( \Xs \right), \sigma^2_{\rs}\left( \Ys \right)\right\}} \\ 
& \overset{(i)}{=} 2\frac{ \sum_{i=1}^{\rs}  \sum_{j=\jk}^{r}   (\sigma_i^X)^2 (\sigma_j^Y)^2 \inner{V^X_i}{V^Y_j}^2 + \sum_{i=\rs+1}^{\ik-1}  \sum_{j=\jk}^{r}   (\sigma_i^X)^2 (\sigma_j^Y)^2 \inner{V^X_i}{V^Y_j}^2 }{\min\left\{\sigma^2_{\rs}\left( \Xs \right), \sigma^2_{\rs}\left( \Ys \right)\right\}}  \\ 
& \overset{(ii)}{\leq} 2\frac{ \sum_{i=1}^{\rs}  \sum_{j=\jk}^{r}   (\sigma_i^X)^2 (\sigma_j^Y)^2 \inner{V^X_i}{V^Y_j}^2 + \sigma^2_{\rs}\left( \Xs \right)  \sigma_{\jk}^2 \left( Y \right) \sum_{i=\rs+1}^{\ik-1}  \sum_{j=\jk}^{r}   1 }{\min\left\{\sigma^2_{\rs}\left( \Xs \right), \sigma^2_{\rs}\left( \Ys \right)\right\}}  \\ 
& \overset{(iii)}{\leq} 2\frac{  \left(\sigma_1^X\right)^2 \left(\sigma_{\jk}^Y\right)^2 \norm{\left( V^X_{\{1\ldots \rs\}} \right)^\top V^Y_{ \{\rs + 1 \ldots r\}}   }{F}^2   + (r-\rs+1)^2\sigma^2_{\rs}\left( \Xs \right)  \sigma_{\jk}^2 \left( Y \right) }{\min\left\{\sigma^2_{\rs}\left( \Xs \right), \sigma^2_{\rs}\left( \Ys \right)\right\}}
\\ & \overset{(iv)}{\leq} 2\frac{  2\left(\sigma_1^X\right)^2 \left(\sigma_{\jk}^Y\right)^2 \left( \norm{\left( V^X_{\{1\ldots \rs\}} \right)^\top V^\Xs_{\{ \rs + 1 \ldots r\}} R }{F}^2+ \norm{\left( V^X_{\{1\ldots \rs\}} \right)^\top \left( V^Y_{ \{\rs + 1 \ldots r\}} - V^{\Ys}_{\{ \rs + 1 \ldots r \}} R \right)   }{F}^2 \right)}{\min\left\{\sigma^2_{\rs}\left( \Xs \right), \sigma^2_{\rs}\left( \Ys \right)\right\}} 
\\ & \qquad + 2 \frac{(r-\rs+1)^2\sigma^2_{\rs}\left( \Xs \right)  \sigma_{\jk}^2 \left( Y \right) }{\min\left\{\sigma^2_{\rs}\left( \Xs \right), \sigma^2_{\rs}\left( \Ys \right)\right\}} 
\\ & \overset{(v)}{\leq} 2\frac{  2\left(\sigma_1^X\right)^2 \left(\sigma_{\jk}^Y\right)^2 \norm{\left( V^X_{\{1\ldots \rs\}} \right)^\top V^\Xs_{\{ \rs + 1 \ldots r\}}}{F}^2  + 2\left(\sigma_1^X\right)^2 \left(\sigma_{\jk}^Y\right)^2 \norm{ V^Y_{ \{\rs + 1 \ldots r\}} - V^{\Ys}_{\{ \rs + 1 \ldots r \}}R }{F}^2}{\min\left\{\sigma^2_{\rs}\left( \Xs \right), \sigma^2_{\rs}\left( \Ys \right)\right\}} 
\\ & \qquad + 2 \frac{(r-\rs+1)^2\sigma^2_{\rs}\left( \Xs \right)  \sigma_{\jk}^2 \left( Y \right) }{\min\left\{\sigma^2_{\rs}\left( \Xs \right), \sigma^2_{\rs}\left( \Ys \right)\right\}},
\end{align*}}
where $(i)$ follows from rearranging, $(ii)$ follows from Weyl's inequality applied to $\sigma_{\rs+1}^X$ and from the Cauchy--Schwarz inequality, $(iii)$ follows from Lemma~\ref{lem:fatahi-asym-convenable-ordering} and from adding nonnegative components to the Frobenius norm, $(iv)$ follows from Young's inequality and from the assumption that $V^{\Xs} = V^{\Ys}$. Finally, $(v)$ follows from the Courant--Fisher theorem applied to orthogonal operators. We further bound, 
{\allowdisplaybreaks \begin{align*}
&2\frac{  2\left(\sigma_1^X\right)^2 \left(\sigma_{\jk}^Y\right)^2 \norm{\left( V^X_{\{1\ldots \rs\}} \right)^\top V^\Xs_{\{ \rs + 1 \ldots r\}}}{F}^2  + 2\left(\sigma_1^X\right)^2 \left(\sigma_{\jk}^Y\right)^2 \norm{ V^Y_{ \{\rs + 1 \ldots r\}} - V^{\Ys}_{\{ \rs + 1 \ldots r \}}R  }{F}^2}{\min\left\{\sigma^2_{\rs}\left( \Xs \right), \sigma^2_{\rs}\left( \Ys \right)\right\}} 
\\ & \qquad + 2 \frac{(r-\rs+1)^2\sigma^2_{\rs}\left( \Xs \right)  \sigma_{\jk}^2 \left( Y \right) }{\min\left\{\sigma^2_{\rs}\left( \Xs \right), \sigma^2_{\rs}\left( \Ys \right)\right\}}
\\ & \overset{(i)}{\leq} 2\frac{  2\left(\sigma_1^X\right)^2 \left(\sigma_{\jk}^Y\right)^2 \frac{4\norm{X-\Xs}{F}^2}{\min\left\{\sigma^2_{\rs}\left( \Xs \right), \sigma^2_{\rs}\left( \Ys \right)\right\}}  + 2\left(\sigma_1^X\right)^2 \left(\sigma_{\jk}^Y\right)^2 \frac{8\norm{Y-\Ys}{F}^2}{\min\left\{\sigma^2_{\rs}\left( \Xs \right), \sigma^2_{\rs}\left( \Ys \right)\right\}} }{\min\left\{\sigma^2_{\rs}\left( \Xs \right), \sigma^2_{\rs}\left( \Ys \right)\right\}} 
\\ & \qquad + 2 \frac{(r-\rs+1)^2\sigma^2_{\rs}\left( \Xs \right)  \sigma_{\jk}^2 \left( Y \right) }{\min\left\{\sigma^2_{\rs}\left( \Xs \right), \sigma^2_{\rs}\left( \Ys \right)\right\}}
\\ & \overset{}{\leq} \frac{  32 \left(\sigma_1^X\right)^2 \left(\sigma_{\jk}^Y\right)^2 }{\min\left\{\sigma^4_{\rs}\left( \Xs \right), \sigma^4_{\rs}\left( \Ys \right)\right\}} \norm{(X,Y) - (\Xs,\Ys)}{F}^2 
+ 2 \frac{(r-\rs+1)^2\sigma^2_{\rs}\left( \Xs \right)  \sigma_{\jk}^2 \left( Y \right) }{\min\left\{\sigma^2_{\rs}\left( \Xs \right), \sigma^2_{\rs}\left( \Ys \right)\right\}} 
\\ &\overset{(ii)}{ \leq }  \frac{1}{16}\frac{  \left(\sigma_1^X\right)^2\left(\sigma_{\jk}^Y\right)^2  }{\max \left\{ {\sigma_1^2\left(\Xs\right)}, {\sigma_1^2\left(\Ys\right)}  \right\} } 
+ 2 \frac{(r-\rs+1)^2\sigma^2_{\rs}\left( \Xs \right)  \sigma_{\jk}^2 \left( Y \right) }{\min\left\{\sigma^2_{\rs}\left( \Xs \right), \sigma^2_{\rs}\left( \Ys \right)\right\}} 
\\ &\overset{(iii)}{ \leq }  %
\frac{\sigma_1^2\left(\Xs\right) +  \sigma_{\rs}^2\left(\Xs\right)}{8 \max \left\{ {\sigma_1^2\left(\Xs\right)}, {\sigma_1^2\left(\Ys\right)}  \right\} } \left(\sigma_{\jk}^Y\right)^2 
+ 2 \frac{(r-\rs+1)^2\sigma^2_{\rs}\left( \Xs \right)  \sigma_{\jk}^2 \left( Y \right) }{\min\left\{\sigma^2_{\rs}\left( \Xs \right), \sigma^2_{\rs}\left( \Ys \right)\right\}} 
\\ &\overset{}{ \leq }  \frac{1}{4} \left(\sigma_{\jk}^Y\right)^2 + 2 \frac{(r-\rs+1)^2\sigma^2_{\rs}\left( \Xs \right)  \sigma_{\jk}^2 \left( Y \right) }{\min\left\{\sigma^2_{\rs}\left( \Xs \right), \sigma^2_{\rs}\left( \Ys \right)\right\}}.
\\ &\overset{}{ \leq }  \frac{1}{4} \left(\sigma_{\jk}^Y\right)^2 + 2 \frac{(r-\rs+1)^2\sigma^2_{\rs}\left( \Xs \right)  \sigma_{\jk}^2 \left( Y \right) }{\min\left\{\sigma^2_{\rs}\left( \Xs \right), \sigma^2_{\rs}\left( \Ys \right)\right\}}
\\ &\overset{}{ \leq }  \frac{9}{4} \frac{(r-\rs+1)^2\sigma^2_{\rs}\left( \Xs \right) }{\min\left\{\sigma^2_{\rs}\left( \Xs \right), \sigma^2_{\rs}\left( \Ys \right)\right\}} \sigma_{\jk}^2 \left( Y \right).
\end{align*}}
Here, $(i)$ follows from a combination of results: first, we rewrite $\norm{\left( V^X_{\{1\ldots \rs\}} \right)^\top V^\Xs_{\{ \rs + 1 \ldots r\}}}{F}$ using Lemma~\ref{lem:orthogonal-matrices-sine}; second, we bound $\norm{ V^Y_{ \{\rs + 1 \ldots r\}} - V^{\Ys}_{\{ \rs + 1 \ldots r \}}R }{F}$ using Lemma~\ref{lem:procrustes-sin-theta}; and third, we invoke Wedin's theorem \cite[Theorem 2.9]{Chen_2021} on both of these terms, which is applicable since $\norm{(X,Y) - (\Xs, \Ys)}{F} \leq \epsassymmatrix \leq \frac{1}{4} \min\left\{ \sigma_{\rs}\left(\Xs\right), \sigma_{\rs}(\Ys) \right\}$. %
 Finally, inequalities $(ii)$ and $(iii)$ are due to the bound on the condition $\norm{(X,Y)-(\Xs,\Ys)}{F}$ %
together with Weyl's and Young's inequality.
Using a similar argument, one can bound the rest of the terms in $T_5$ by 
$$
2\frac{ \sum_{(i,j)\in \mathcal{I}_{QP}} (\sigma_i^X)^2 (\sigma_j^Y)^2 \inner{V^X_i}{V^Y_j}^2}{\min\left\{\sigma^2_{\rs}\left( \Xs \right), \sigma^2_{\rs}\left( \Ys \right)\right\}} \leq \left( \frac{9}{4}\frac{(r-\rs+1)^2\sigma^2_{\rs}\left( \Ys \right)}{\min\left\{\sigma^2_{\rs}\left( \Xs \right), \sigma^2_{\rs}\left( \Ys \right)\right\}} \right) \sigma_{\ik}^2 \left( X \right),
$$
so that
\begin{align*}
T_5 &\leq \left( \frac{9}{4}\frac{(r-\rs+1)^2\max \left\{ \sigma^2_{\rs}\left( \Xs \right) , \sigma^2_{\rs}\left( \Ys \right) \right\}}{\min\left\{\sigma^2_{\rs}\left( \Xs \right), \sigma^2_{\rs}\left( \Ys \right)\right\}} \right) \left( \sigma_{\ik}^2 \left( X \right) + \sigma_{\jk}^2 \left( Y \right)\right) 
\\ &=\left( \frac{9}{4}\frac{(r-\rs+1)^2\max \left\{ \sigma^2_{\rs}\left( \Xs \right) , \sigma^2_{\rs}\left( \Ys \right) \right\}}{\min\left\{\sigma^2_{\rs}\left( \Xs \right), \sigma^2_{\rs}\left( \Ys \right)\right\}} \right)\sigma_{k+1}^2,
\end{align*}
and thus, adding $T_4$ yields
$$
T_4 + T_5  \leq  \frac{17}{4}\frac{ \sigma^2_{\rs}\left( \Xs \right) + \sigma^2_{\rs}\left( \Ys \right) }{\min\left\{\sigma^2_{\rs}\left( \Xs \right), \sigma^2_{\rs}\left( \Ys \right)\right\}} (r-\rs+1)^2 \sigma_{k+1}^2.
$$

To conclude, since $T_1 \leq \frac{1}{2}(r-\rs + 1) \sigma_{k+1}^2 \leq \frac{1}{2} \frac{ \sigma^2_{\rs}\left( \Xs \right) + \sigma^2_{\rs}\left( \Ys \right) }{\min\left\{\sigma^2_{\rs}\left( \Xs \right), \sigma^2_{\rs}\left( \Ys \right)\right\}} (r-\rs+1)^2 \sigma_{k+1}^2$, we obtain 
\begin{align*}
    \| &\left(I - \Pi^{(X,Y)}_{k} \right) [XY^\top - \Xs\Ys^\top ] \|_F  \\ &\leq  \left(5\frac{ \sigma^2_{\rs}\left( \Xs \right) + \sigma^2_{\rs}\left( \Ys \right) }{\min\left\{\sigma^2_{\rs}\left( \Xs \right), \sigma^2_{\rs}\left( \Ys \right)\right\}} (r-\rs+1)^2 \right) \sigma^2_{k+1} + \frac{\rho}{2} \norm{XY^\top - \Xs\Ys^\top}{F}.  
\end{align*}
Taking the square of both sides and applying Young's inequality yields the desired result.

\end{proof}

\subsubsection{Proof of Lemma~\ref{lem: smooth-sensing-satisfies-assum}} \label{proof:lem-RIP-implies-assumptions}
The objective $h$ is a composition of a linear map with a convex function, and so it's convex, proving Item~\ref{item:smoothpenalty:convexity} of Assumption~\ref{assum:smoothpenalty}. To establish Items~\ref{item:smoothpenalty:uniqueminimizer} and \ref{item:smoothpenalty:sharpness}, we establish quadratic growth with $\zs = \Ms$. Notice that $\Ima \c$ corresponds with the set of rank $r$ matrices. Let an arbitrary $M\in \Ima \c$. Applying the reverse triangle inequality yields
  \begin{align}
  \label{quadgrowth}
      \frac{1}{2} \norm{\mathcal{A}(M) - b}{2}^2 -   \frac{1}{2} \norm{\mathcal{A}(\Ms) - b}{2}^2 %
      \geq \frac{1}{2} \norm{\mathcal{A}\left(M - \Ms\right)}{2}^2 %
      \geq \frac{1}{2} (1-\delta)\norm{M-\Ms}{2}^2,
  \end{align}
  where %
  the second inequality follows from (\ref{def:RIP}) since $M-\Ms$ has rank at most 2r.

We will use the following lemma for our proof of Item~\ref{item:smoothpenalty:lip} in Assumption~\ref{assum:smoothpenalty}.
\begin{lemma}[Lemma 3.3 in \cite{candes2011tight} and Lemma 31 in \cite{tong2021accelerating}]\label{lem: candes-lemma}
Assume that $\mathcal{A}$ satisfies~\eqref{def:RIP} for matrices of rank at most $2r$. Then one has
$$
\inner{\mathcal{A}(M)}{\mathcal{A}(\Mprime)} \leq \delta \norm{M}{F} \norm{\Mprime}{F} + \abs{\inner{M}{\Mprime}},
$$
for any matrices $M$ and $\Mprime$ of rank at most $r$.
\end{lemma}

We use $\Pi$ as a shorthand for the projection $\Pi^{(X, Y)}$ onto the image of $\nabla \c(X,Y)$. Recall from \eqref{action: nablac-and-nablacT-asym} that this image lies within the set of matrices of rank at most $2r$ and so \begin{equation} \label{eq: proj_subset_low_rank}
  \Ima {\Pi} \subseteq \{ Z\in \RR^{d_1\times d_2} \mid \rankk{Z} \leq 2r\}.
  \end{equation}
Therefore,  {\allowdisplaybreaks \begin{align*}
  \norm{\Pi\nabla h(M)}{2}  &= \norm{\Pi \left[\mathcal{A}^\ast \mathcal{A}\left(M-M^\star\right] \right)}{F}   
  \\ &=\sup_{W \in \RR^{d_1\times d_2} \mid \norm{W}{F}=1} \inner{\Pi \left(\mathcal{A}^\ast \mathcal{A}\left(M-M^\star\right) \right)}{W}  \\ &=\sup_{W \in \RR^{d_1\times d_2} \mid \norm{W}{F}=1} \inner{\mathcal{A}^\ast \mathcal{A}\left(M-M^\star\right)}{\Pi[W]}  \\ &\leq \sup_{\substack{W \in \RR^{d_1\times d_2} \mid \norm{W}{F}=1 \\ \rankk{W} \leq 2r}} \inner{\mathcal{A}^\ast \mathcal{A}\left(M-M^\star\right)}{W} \\ & = \sup_{\substack{W \in \RR^{d_1\times d_2} \mid \norm{W}{F}=1 \\ \rankk{W} \leq 2r}} \inner{\mathcal{A}\left(M-M^\star\right)}{\mathcal{A}\left(W\right)} \\ &\leq \delta\norm{M-\Ms}{F}  + \sup_{W \in \RR^{d_1\times d_2} \mid \norm{W}{F}=1}\inner{M-\Ms}{W} \\&=(1+\delta) \norm{M-\Ms}{F} ,
   \end{align*}}
   where the first inequality follows from~\eqref{eq: proj_subset_low_rank}, and the second inequality uses Lemma~\ref{lem: candes-lemma}. Using this result in tandem with~\eqref{quadgrowth} yields 
    $$
    \frac{(1-\delta)}{2(1+\delta)^2} \norm{\Pi \nabla h(M)}{2}^2 \leq   h(M) - h(M^\star), 
    $$
    thus, establishing part (a) of Item~\ref{item:smoothpenalty:lip}. For part (b), recall from the definition of $P((X, Y), \lambda)$, \eqref{eq:matrix-P}, that $\Ima (I - P((X,Y), \lambda)) = \Ima \nabla \c(X,Y)$ and so $(I - P((X,Y), \lambda))[M] \in M + \Ima \nabla \c(X, Y)$ and consequently \begin{equation}\label{eq:rank-bound}\rank((I - P((X,Y), \lambda))[M]) \leq \rank M + 2r.
    \end{equation}
 Hence,
    \begin{align*}
        &\inner{\mathcal{A}^\ast \mathcal{A}\left(M-\Ms\right)}{(I - P((X,Y), \lambda))[M-\Ms]} \\ &\hspace{.5cm} =   \inner{ \mathcal{A}\left(M-\Ms\right)}{\mathcal{A}\left((I - P((X,Y), \lambda))[M-\Ms]\right)}
        \\& \hspace{.5cm}\leq \delta\norm{M-\Ms}{F} \norm{(I - P((X,Y), \lambda))[M-\Ms]}{F} +\abs{\inner{M-\Ms}{(I - P((X,Y), \lambda))[M-\Ms]}} \\ & \hspace{.5cm}\leq (1+\delta) \norm{M-\Ms}{F} \norm{(I - P((X,Y), \lambda))(M-\Ms)}{F},
    \end{align*}
    where the first inequality follows from Lemma~\ref{lem: candes-lemma}, which applies due to~\eqref{eq:rank-bound}, and the second inequality follows from Cauchy–Schwarz.
    The proof concludes by taking $\mus =(1-\delta)$ and $\Lhs = \max\left\{ \frac{(1+\delta)^2}{(1-\delta)}, (1+\delta) \right\} = \frac{(1+\delta)^2}{(1-\delta)}$.
\subsubsection{Proof of Lemma~\ref{lem:restricted-regularity-implies-ugly-regularity}}\label{app:proof-restricted-regularity-implies-ugly-regularity}    Items~\ref{item:nonsmoothpenalty:uniqueminimizer},~\ref{item:nonsmoothpenalty:convexity} and ~\ref{item:nonsmoothpenalty:sharpness} hold automatically and so we focus on proving Item~\ref{item:nonsmoothpenalty:lip}. Take $M\in \RR^{d_1\times d_2}$ and $V\in \partial f(M)$. A key ingredient to this proof is the fact that for any matrix $W$ of rank at most $2r$ we have
    \begin{align}\label{eq:keyprop}
     \inner{V}{W} &= \inner{V}{(W+M)-M} \nonumber \\ &\leq f(W+M) - f(M)\nonumber\\
     &\leq L\norm{W}{F},
     \end{align}
where the first inequality holds since $f$ is convex and $V\in \partial f(M)$, and the second holds since $f$ satisfies restricted $L$ Lipschitzness. 

To establish Item~\ref{item:nonsmoothpenalty:lip_a} of the Assumption, we have
\begin{align*}
    \norm{\Pi^{(X, Y)} [V]}{F} &\leq \sup_{\substack{W \mid \rankk{W}\leq 2r,\\ \norm{W}{F}=1}} \inner{W}{V} \\& \leq \sup_{W\mid \norm{W}{F}=1} L \norm{W}{F} = 1.
\end{align*}
Recall that $\Pi^{(X, Y)}$ denotes the projection onto $\operatorname{range} \nabla F(X,Y)$, which is a subset of $\{ {W} \mid W \in \RR^{d_1\times d_2} \text{ and } \rankk{W}\leq 2r \}$. Thus, establishing the first inequality, while the second inequality follows from~\eqref{eq:keyprop}.

To prove Item~\ref{item:nonsmoothpenalty:lip_b}, the matrix $W = (I-P((X,Y), \lambda))[M - M^{\star}]$ has rank $4r$ due to \eqref{eq:rank-bound}. We can further decompose $W$ into two rank-$2r$ matrices $W_{1}$ and $W_{2}$ such that $\dotp{W_{1},W_{2}} = 0.$ Therfore,
\begin{align*}
  \abs{\inner{V}{W} }^{2} &= \abs{\inner{V}{W_{1}} + \inner{V}{W_{2}}}^{2} \\&= \abs{\inner{V}{W_{1}}}^{2} + 2 \abs{\inner{V}{W_{1}}\inner{V}{W_{2}}} +  \abs{\inner{V}{W_{2}}}^{2}
  \\ &\leq L^{2}\norm{W_{1}}{F}^{2} +2 \norm{W_{1}}{F} \norm{W_{2}}{F}+L^{2} \norm{W_{2}}{F}^{2}
\\& = L^{2}\left(\norm{W_{1}}{F}^{2} + \norm{W_{2}}{F}\right)^{2} \\&= L^{2} \norm{W}{F}^{2} ,
\end{align*}
where the inequality follows from \eqref{eq:keyprop} and the last equality uses fact that $\norm{W}{F}^{2} = \norm{W_{1}}{F}^{2} + \norm{W_{2}}{F}^{2}$, since $W_{1}$ and $W_{2}$ are orthogonal.
This concludes the proof.
\subsection{Proofs from Section \ref{sec: tensor}}
We start with a few explicit definitions that will play a role in our arguments.
In what follows, we use $M_{j:}$ and $M_{i}$ to refer to the $j$-th row and $i$-th column, respectively.

\begin{definition}[Column-major vectorization of matrices and tensors]
Let $M\in\RR^{d_1\times d_2}$ and $T\in\RR^{d_1\times d_2\times d_3}$. The vectors $\vect{M}\in\RR^{d_1d_2}$ and $\vect{T}\in\RR^{d_1d_2d_3}$ are defined by
\begin{align*}
  \vect{M}_{(i_2-1)d_1+i_1}
    &= M_{i_1,i_2}& \text{for }
    i_1\in [d_1],\, i_2\in [d_2],\\
  \vect{T}_{(i_3-1)d_1d_2+(i_2-1)d_1+i_1}
    &= T_{i_1,i_2,i_3}& \text{for }
    i_1\in [d_1],\, i_2\in [d_2],\, i_3\in [d_3].
\end{align*}
\end{definition}

\begin{definition}[Matricization of tensors]
Let $T\in\RR^{d_1\times d_2\times d_3}$. The mode-1 matricization $\cM_1(T)\in\RR^{d_1\times(d_2d_3)}$ is given by
\begin{align*}
  \cM_1(T)_{i_1,\,(i_3-1)d_2+i_2}
    &= T_{i_1,i_2,i_3}& \text{for }
    i_1\in [d_1],\, i_2\in [d_2],\, i_3\in [d_3].
\end{align*}
Similarly, the mode-2 and mode-3 matricizations $\cM_2(T)\in\RR^{d_2\times(d_1d_3)}$ and $\cM_3(T)\in\RR^{d_3\times(d_1d_2)}$ are
\begin{align*}
  \cM_2(T)_{i_2,\,(i_3-1)d_1+i_1}
    &= T_{i_1,i_2,i_3}& \text{for }
    i_1\in [d_1],\, i_2\in [d_2],\, i_3\in [d_3],\\
  \cM_3(T)_{i_3,\,(i_2-1)d_1+i_1}
    &= T_{i_1,i_2,i_3}& \text{for }
    i_1\in [d_1],\, i_2\in [d_2],\, i_3\in [d_3].
\end{align*}
\end{definition}

\begin{definition}[Permutations of tensor matricization]\label{Pis}
Given $i\in\{2,3\}$ and $T\in\RR^{d_1\times d_2\times d_3}$, let $P_i\in\RR^{d_1d_2d_3\times d_1d_2d_3}$ be the permutation matrix such that
\begin{align*}
  P_i\,\vect{\cM_i(T)} &= \vect{\cM_1(T)} = \vect{T}.
\end{align*}
\end{definition}

\begin{definition}[Kronecker product for matrices and vectors]
Given $M\in\mathbb{R}^{d_1\times d_2}$ and $N\in\mathbb{R}^{d_3\times d_4}$, their Kronecker product $M\otimeskron N\in\mathbb{R}^{d_1d_3\times d_2d_4}$ is defined entrywise by
\begin{align*}
  (M\otimeskron N)_{(i_1-1)d_3+i_3,\,(i_2-1)d_4+i_4}
    &= M_{i_1,i_2}\,N_{i_3,i_4}& \text{for }
    i_1\in [d_1],\, i_2\in [d_2],\, i_3\in [d_3],\, i_4\in [d_4].
\end{align*}
Moreover, for vectors $u\in\mathbb{R}^{d_1}$ and $v\in\mathbb{R}^{d_2}$, one has
$$
(u\otimeskron v)_{(i_1-1)d_2+i_2} = u_{i_1}v_{i_2} \quad \text{for } i_1\in[d_1], i_2\in[d_2].
$$
\end{definition}

\begin{definition}\label{psi}
    Given any two matrices $M \in\RR^{d_1\times r}$ and $N\in \RR^{d_2\times r}$, we define 
    \begin{equation*}
      \ext{M}{N} =  [\, M_1 \otimeskron N_1 \;\cdots\; M_r \otimeskron N_r\,] \in \RR^{d_1d_2 \times r}\text { and } \Psi(M,N) = \sum_{j=1}^r M_j M_j^\top \otimeskron N_j N_j^\top \in \RR^{d_1d_2 \times d_1d_2},
\end{equation*}
      where $M_{j}$ and $N_{j}$ denote the $j$th columns of $M$ and $N$, respectively.
\end{definition}
\subsubsection{Proof of Theorem \ref{thm-tensor-symmetric-reg-conditions}}\label{proof-tensor-symmetric-reg-conditions}
We start with the actions of the Jacobian and its adjoint.  
\begin{lemma}
\label{lemma: action-nablac-symmetric}
Let $X\in \RR^{d\times r}$. Then, the action of $\nabla \Fsym(X) $ on a direction $D \in \RR ^{d\times r}$ is given by
\begin{align*}
    \nabla \Fsym(X) \vect{D} &= \sum_{\ell=1}^r \left(D_\ell \otimeskron X_\ell \otimeskron X_\ell +   X_\ell \otimeskron D_\ell \otimeskron X_\ell +  X_\ell \otimeskron X_\ell \otimeskron D_\ell\right) \\ & = \left( I + P_3 + P_2 \right) \sum_{\ell=1}^r D_\ell \otimeskron X_\ell \otimeskron X_\ell 
    \in \RR^{d^3}.
\end{align*} Moreover, the action of the adjoint $\nabla \Fsym(X)^\top$ on a rank-$1$ tensor $a \otimeskron b \otimeskron c \in \RR^{d^3}$ is given by
    $$
    \nabla \Fsym(X)^\top \left( a \otimeskron b \otimeskron c \right)=\vect{ \left( a  (b^\top \otimeskron c^\top) \;+\; b  (a^\top \otimeskron c^\top)\;+\; c  (a^\top \otimeskron b^\top) \right)  \extone{X}} \in \RR^{dr}.
    $$
\end{lemma} 
The proof is deferred to Appendix \ref{proof: action-nablac-symmetric}. Given this result, it is straightforward to establish Assumption~\ref{assum:local_lip_jacob}. Consider an arbitrary pair $X,\Xprime \in \RR^{d \times r}$ satisfying $\max\left\{ \norm{X-\Xs}{F}, \norm{\Xprime-\Xs}{F} \right\} \leq \norm{\Xs}{F}$. 
Using the variational characterization of the operator norm, we have
\begin{subequations}
\begin{align*}
    \norm{ \nabla \Fsym(X) - \nabla \Fsym(\Xprime) }{\operatorname{op}} &= \sup_{A \in \RR^{d\times r} ,  \norm{A}{F} = 1} \norm{\left(\nabla \Fsym(X) - \nabla \Fsym(\Xprime)\right)  \vect{A}}{2}\nonumber \\
    & \overset{(i)}{=}\sup_{\norm{A}{F} = 1} \norm{ \left(I + P_2 + P_3 \right)  \left( \sum_{\ell=1}^r A_\ell \otimeskron X_\ell \otimeskron X_\ell -  A_\ell \otimeskron \Xprime_\ell \otimeskron \Xprime_\ell \right)  }{2}\nonumber\\
    &\overset{(ii)}{\leq} 3 \sup_{\norm{A}{F} = 1}  \norm{\sum_{\ell=1}^r A_\ell \otimeskron \left( X_\ell \otimeskron X_\ell -  \Xprime_\ell \otimeskron \Xprime_\ell\right) }{2} \nonumber\\
    &\overset{(iii)}{=} 3 \sup_{\norm{A}{F} = 1}  \norm{ \left(  \extone{X} - \extone{\Xprime} \right) A^\top  }{F}\nonumber \\
        &\overset{(iv)}{\leq} 3 \norm{\extone{X} - \extone{\Xprime}}{F\nonumber},
    \end{align*}
        \end{subequations}
     where $(i)$ follows from the Lemma~\ref{lemma: action-nablac-symmetric},  $(ii)$ follows from the fact that permutation matrices have operator norm one, $(iii)$ follows from $\sum_{\ell=1}^r A_\ell \otimeskron B_\ell = \vect{B A^\top}$ for any $A$ and $B$,
     and $(iv)$ follows from the submultiplicativity of the Frobenius norm. 
     Leveraging the fact that $\sum_{\ell=1}^r \norm{v_\ell}{2}^2 \leq \left( \sum_{\ell=1}^r \norm{v_\ell}{2} \right)^2$, we have
    \begin{subequations}\begin{align*}
           \norm{ \nabla \Fsym(X) - \nabla \Fsym(\Xprime) }{\operatorname{op}} &\leq 3\sum_{\ell=1}^r \norm{X_\ell \otimeskron X_\ell - \Xprime_\ell \otimeskron \Xprime_\ell}{2} \nonumber \\
    &= 3\sum_{\ell=1}^r \norm{X_\ell \otimeskron X_\ell - X_\ell \otimeskron \Xprime_\ell + X_\ell \otimeskron \Xprime_\ell - \Xprime_\ell \otimeskron \Xprime_\ell}{2}\nonumber \\
    &\overset{(i)}{\le} 3\sum_{\ell=1}^r \left( \norm{X_\ell \otimeskron (X_\ell - \Xprime_\ell)}{2} + \norm{(X_\ell - \Xprime_\ell) \otimeskron \Xprime_\ell}{2} \right)\nonumber\\
    &\overset{(ii)}{\le} 3\left( \norm{X}{F} + \norm{\Xprime}{F} \right)  \norm{X - \Xprime}{F} \\ 
        &\overset{(iii)}{\le} 12\norm{\Xs}{F} \norm{X - \Xprime}{F},
\end{align*}
\end{subequations} 
where $(i)$ follows from the triangle inequality and bilinearity of the Kronecker product, $(ii)$ follows from the fact that $\norm{a \otimes b}{2} = \norm{a}{2} \norm{b}{2}$ for any $a$ and $b$, together with the Cauchy-Schwarz inequality, and $(iii)$ holds since by assumption $\max \left\{ \norm{X-\Xs}{F}, \norm{\Xprime - \Xs}{F} \right\} \leq \norm{\Xs}{F}$, which implies $\max \left\{ \norm{X}{F}, \norm{\Xprime}{F} \right\} \leq 2\norm{\Xs}{F}$ by the reverse triangle inequality.

 We now proceed to proving strong alignment for this map. Let $\Ts \in \RR^{d\times d\times d}$ be an arbitrary tensor and $\Xs \in \RR^{d\times r}$ be any full-rank matrix with $\Ts=\Fsym(\Xs)$. To establish alignment, we use the following result. \begin{proposition}[Constant rank of the symmetric CP map]\label{prop: cpsym constant rank}
 For any full rank matrix $X\in \RR^{d\times r}$,
    \[
    \rank \left(\nabla \Fsym(X) \right) = dr.
    \]    
\end{proposition}
Given this result, invoking Lemma~\ref{constant_rank_implies_strong_alignment} gives us that the map $\Fsym$ satisfies Assumption \ref{ass: local-strong-alignment} with $j=dr$ and with the quantities 
 \begin{align*}\alignvarepsilon = \epssymtensor, \quad \rloc(\rho) = \frac{\rho}{C}, \quad \text{ and } \quad
    \sig = \frac{1}{2}\sigma_{dr}\left( \nabla \Fsym(\Xs) \right)\end{align*}
for some positive constants $R$ and $C$ that depend only on $\Xs$. This establishes the result, provided that we show Proposition~\ref{prop: cpsym constant rank}; we now proceed to prove it.

\begin{proof}[Proof of Proposition~\ref{prop: cpsym constant rank}]
Let $X\in \RR^{d\times r}$ be of rank $r$. The proof consists of two steps. Firstly, we will construct a set $H$ of probing vectors, i.e., $dr$ linearly independent vectors in $\RR^{d^3}$. Secondly, we will prove that the probing set remains linearly independent after applying $\nabla \Fsym(X)^\top$. This shows a lower bound $\rank (\nabla \Fsym(X)^\top)  \geq dr$. Since $ \nabla \Fsym(X) \in \RR^{d^3 \times dr}$, the rank of $\nabla \Fsym(X)$ is at most $dr$, we must have $\rank(\nabla \Fsym(X)) = dr$.

We take the full SVD of $X$ as $\svd{X}$, where $U^X \in \RR^{d\times d}$, $\Sigma^X \in \RR^{d\times r} \text{ and } (V^X)^\top \in \RR^{r\times r}$. Consider the extended SVD $$\Tilde{\Sigma}^X := \diagg{ \sigma^X_1, \ldots, \sigma^X_r, 1, \ldots 1} \in \RR^{d \times d}, \text{ and } \Tilde{V}^X := \begin{pmatrix} V_X & 0 \\ 0 & I \end{pmatrix} \in \RR^{d \times d}.$$ Observe that with this notation we have $X_{\ell} = U^X\tilde{\Sigma}^X{(\tilde{V}^X_{\ell:})^\top}$ for all $\ell \in [r].$

\paragraph{Probing set.}
First, since $X$ is full rank, the vectors $$T_k := U^X \left(\tilde{\Sigma}^X\right)^{-1}(\tilde{V}^X _{k:})^\top \in \RR^{d} \quad \text{are well-defined for all } k \in [d].$$ We construct $H= \{ T_i \otimeskron T_j \otimeskron T_j \mid i \in [d], j \in [r] \}$. By~\eqref{kr_p3}, Kronecker products of invertible matrices are invertible. Then, the matrix $U^X \left(\tilde{\Sigma}^X\right)^{-1} \left(\tilde{V}^X\right)^\top \otimeskron U^X \left(\tilde{\Sigma}^X\right)^{-1} \left(\tilde{V}^X\right)^\top \otimeskron U^X \left(\tilde{\Sigma}^X\right)^{-1} \left(\tilde{V}^X\right)^\top \in \RR^{d^3 \times d^3}$ is invertable. Since $H$ is a subset of columns of this matrix, the vectors in $H$ are linearly independent.
\paragraph{Rank lower bound.} Let $i\in [d], j\in[r]$. By Lemma \ref{lemma: action-nablac-symmetric}, we have that 
{\allowdisplaybreaks
\begin{align*}
&\nabla \Fsym(X)^\top \left( T_i \otimeskron T_j \otimeskron T_j\right)
\\&= \vect{\left( T_i\left( T_j^\top\otimeskron T_j^\top \right)
+ T_j\left( T_i^\top\otimeskron T_j^\top \right)
+ T_j\left( T_i^\top\otimeskron T_j^\top \right) \right)
\begin{bmatrix} X_1 \otimeskron X_1 & \cdots & X_r \otimeskron X_r \end{bmatrix}} \\
&\overset{(i)}{=} \vect{T_i\begin{bmatrix}
\left\langle T_j, X_1\right\rangle^2 \\
\vdots \\
\left\langle T_j, X_r\right\rangle^2
\end{bmatrix}^\top
+ 2T_j\begin{bmatrix}
\left\langle T_i, X_1 \right\rangle\left\langle T_j, X_1 \right\rangle \\
\vdots \\
\left\langle T_i, X_r \right\rangle\left\langle T_j, X_r \right\rangle
\end{bmatrix}^\top} \\
&\overset{(ii)}{=} \vect{T_i\begin{bmatrix}
\inner{\tilde{V}^X_{j:}}
{\tilde{V}^X_{1:}}^2 \\
\vdots \\
\inner{\tilde{V}^X_{j:}}
{\tilde{V}^X_{r:}}^2
\end{bmatrix}^\top + 2T_j\begin{bmatrix}
\inner{\tilde{V}^X_{i:}}
{\tilde{V}^X_{1:}}
\inner{\tilde{V}^X_{j:}}
{\tilde{V}^X_{1:}} \\
\vdots \\
\inner{\tilde{V}^X_{i:}}
{\tilde{V}^X_{r:}}
\inner{\tilde{V}^X_{j:}}
{\tilde{V}^X_{r:}}
\end{bmatrix}^\top} \\
&\overset{(iii)}{=} e_j \otimeskron T_i
+ 2\mathbbm{1}_{i=j }\left(e_i \otimeskron T_j\right) \\
&= \begin{cases}
3 e_i \otimeskron T_i & \text{if } i=j, \\[6pt]
e_j \otimeskron T_i & \text{otherwise}
\end{cases} \in \RR^{dr},
\end{align*}
}
where $(i)$ follows from \eqref{kr_p3}, $(ii)$ follows from $
\inner{T_i}{X_\ell} = \tilde{V}^X_{i:} (\tilde{\Sigma}^X)^{-1} (U^X)^\top U^X (\tilde{\Sigma}^X) \left(\tilde{V}^X_{\ell:}\right)^{\top} = \left\langle \tilde{V}^X_{i:},\tilde{V}^X_{\ell:} \right\rangle,
$ and $(iii)$ follows from~\eqref{kr_p4}, here $\{ e_j \}_{j\in[r]}\subseteq \RR^{r}$ denotes the canonical basis for $\RR^{r}$.
The resulting $dr$ vectors are scaled versions of different columns of the invertible matrix $I_r \otimeskron  \left( U^X \left( \tilde{\Sigma}^X \right)^{-1} \left( \tilde{V}^X \right)^\top \right)$, thus they are linearly independent, which completes the proof of Proposition~\ref{prop: cpsym constant rank}.
\end{proof}
This concludes the proof of Theorem~\ref{thm-tensor-symmetric-reg-conditions}.

\subsubsection{Proof of Theorem~\ref{thm-tensor-asymmetric-reg-conditions}} \label{proof-tensor-asymmetric-reg-conditions}

Let $W\in \RR^{d_1 \times r}, X \in \RR^{d_2 \times r} \text{ and } Y \in \RR^{d_3 \times r}$ be arbitrary matrices. Denote the full SVD factorization of each one of these matrices as
\begin{equation}\label{svds-cp-proof}
    W = \svd{W}, \quad X = \svd{X} \quad \text{ and } \quad Y = \svd{Y}.
\end{equation}
 Moreover, have $\nabla \Fasym(W,X,Y) = U \Sigma V^\top \in \RR^{d_1d_2d_3 \times (d_1+d_2+d_3)r}$. We start with a basic result for the analytical expression of the Jacobian.
\begin{lemma}[Corollary 4.2 in \cite{Acar2011}, Lemma 1 in \cite{Karim_2024}]
        \label{jac: cp}
    The Jacobian of $\Fasym$ is given by        \begin{equation*}
        \nabla \Fasym(W, X, Y) =
            \begin{pmatrix}
            J^{W} & J^{X} & J^{Y}
            \end{pmatrix}  \in \RR^{d_1d_2d_3 \times (d_1+d_2+d_3)r}
    \end{equation*}
    where
    \begin{align*}
        J^{W} &=
            I_{d_1} \otimeskron
            \psi(X,Y) \in \RR^{d_1d_2d_3 \times d_1r}, \\
            J^{X} &=
            P_2 \left( I_{d_2} \otimeskron
           \psi(W,Y)\right)  \in \RR^{d_1d_2d_3 \times d_2r}, \text{ and} \\
            J^{Y} &=
            P_3\left( I_{d_3} \otimeskron
           \psi(W,X) \right) \in \RR^{d_1d_2d_3 \times d_3r},
    \end{align*}
    with $P_i$ and $\psi$ introduced in Definitions~\ref{Pis} and~\ref{psi}, respectively.
\end{lemma}

 Consider an arbitrary pair $(W,X,Y), (\Wprime, \Xprime, \Zprime) \in \RR^{d_1 \times r} \times \RR^{d_2 \times r} \times \RR^{d_3 \times r}$ satisfying  $$\max\left\{ \norm{(W,X,Y)-(\Ws,\Xs,\Ys)}{F}, \norm{(\Wprime,\Xprime,\Yprime)-(\Ws,\Xs,\Ys)}{F} \right\} \leq \norm{(\Ws,\Xs,\Ys)}{F}.$$
Using the variational characterization of the operator norm, we have
  {\allowdisplaybreaks
\begin{align}
\label{smoothmaptensorasymm}
      & \norm{ \nabla \Fasym(W,X,Y) - \nabla \Fasym(\Wprime, \Xprime, \Yprime) }{\operatorname{op}}\nonumber\\ &= \sup_{\norm{A}{F} = 1} \norm{\left(\nabla \Fasym(W,X,Y) - \nabla \Fasym(\Wprime, \Xprime, \Yprime)\right)  \vect{A}}{2}\nonumber \\
            &\overset{}{=} \sup_{\substack{(A_1,A_2,A_3)\in\RR^{d_1\times r}\times\RR^{d_2\times r}\times \RR^{d_3\times r} \\ \norm{(A_1,A_2,A_3)}{F}=1 } }\norm{  \left(J^W-J^{\Wprime}\right)\vect{A_1}   +  \left(J^X-J^{\Xprime}\right)\vect{A_2}  +  \left(J^Y-J^{\Yprime}\right)\vect{A_3}}{2}\nonumber \\
    & \overset{(i)}{\leq}  \sup_{\substack{(A_1,A_2,A_3)\in\RR^{d_1\times r}\times\RR^{d_2\times r}\times \RR^{d_3\times r} \\ \norm{(A_1,A_2,A_3)}{F}=1 } }\Bigg(\norm{  \vect{\left(\ext{X}{Y} - \ext{\Xprime}{\Yprime} \right)  A_1^\top}}{2}\nonumber\\
      &\hspace{170pt}
        + \norm{ P_2 \vect{\left( \ext{W}{Y} - \ext{\Wprime}{\Yprime} \right)  A_2^\top}}{2}
        \nonumber \\
      &\hspace{200pt}
        + \norm{P_3 \vect{\left(\ext{W}{X} -\ext{\Wprime}{\Xprime} \right)  A_3^\top}}{2} \Bigg), \nonumber
\end{align}
}
where $(i)$ follows from applying Lemma \ref{jac: cp} in tandem with the triangle inequality and from \eqref{kr_p1}. 
Leveraging the fact that the operator norm of a permutation matrix is at most one, we derive
{\allowdisplaybreaks
\begin{align*}
  & \norm{ \nabla \Fasym(W,X,Y) - \nabla \Fasym(\Wprime, \Xprime, \Yprime) }{\operatorname{op}} \nonumber\\
  & \leq \sup_{ \norm{(A_1, A_2, A_3)}{F} = 1} \Bigg( \norm{\left(\ext{X}{Y} - \ext{\Xprime}{\Yprime} \right)  A_1^\top}{F} \\
  &\hspace{120pt} + \norm{\left( \ext{W}{Y} - \ext{\Wprime}{\Yprime} \right)  A_2^\top}{F} \nonumber \\
  &\hspace{140pt} + \norm{\left(\ext{W}{X} -\ext{\Wprime}{\Xprime} \right)  A_3^\top}{F} \Bigg) \nonumber \\
      &  \leq\norm{ \ext{X}{Y} - \ext{\Xprime}{\Yprime}}{F}
            + \norm{\ext{W}{Y} - \ext{\Wprime}{\Yprime}}{F}
            + \norm{\ext{W}{X} - \ext{\Wprime}{\Xprime}}{F} \nonumber\\
  &\overset{}{=} \sum_{\ell=1}^r \Bigg( \norm{ \left( X_{\ell} - \Xprime_{\ell} \right) \otimeskron Y_{\ell} + \Xprime_{\ell} \otimeskron \left( Y_{\ell} - \Yprime_{\ell} \right) }{2} \nonumber \\ &\hspace{40pt}+ \norm{ \left( W_{\ell} - \Wprime_{\ell} \right) \otimeskron Y_{\ell} + \Wprime_{\ell} \otimeskron \left( Y_{\ell} - \Yprime_{\ell} \right) }{2}\nonumber \\ & \hspace{80pt}+ \norm{ \left( W_{\ell} - \Wprime_{\ell} \right) \otimeskron X_{\ell} + \Wprime_{\ell} \otimeskron  \left( X_{\ell} - \Xprime_{\ell} \right) }{2} \Bigg)\nonumber\\ & \overset{(i)}{\leq}
  \left( \norm{W}{F} + \norm{X}{F}  + \norm{Y}{F}  + \norm{\Wprime}{F} + \norm{\Xprime}{F} + \norm{\Yprime}{F} \right) \left( \norm{W-\Wprime}{F} + \norm{X-\Xprime}{F} + \norm{Y-\Yprime}{F} \right) \nonumber \\ & \overset{(ii)}{\leq} \sqrt{3}\left( \norm{({W},{X},{Y})}{F} + \norm{({\Wprime},{\Xprime},{\Yprime})}{F} \right)  \norm{({W},{X},{Y}) -({\Wprime},{\Xprime},{\Yprime})}{F} \nonumber
  \\ & \overset{(iii)}{\leq} 4\sqrt{3}\norm{(\Ws,\Xs,\Ys)}{F}\norm{({W},{X},{Y}) -({\Wprime},{\Xprime},{\Yprime})}{F},
\end{align*}
}
where $(i)$ follows from the fact that $\norm{v \otimeskron w}{2} = \norm{v}{2} \norm{w}{2}$ for any vectors $v$ and $w$ together with Cauchy-Schwarz and $(ii)$  uses that $ \left(\abs{a} + \abs{b} + \abs{c} \right)^2  \overset{}{\leq}3\left(\abs{a}^2 + \abs{b}^2 + \abs{c}^2 \right) $ for $a,b,c \in \RR$, and  $(iii)$ holds since by assumption $\max \left\{ \norm{(W,X,Y)-(\Ws,\Xs,\Ys)}{F}, \norm{(\Wprime,\Xprime,\Yprime) - (\Ws,\Xs,\Ys)}{F} \right\} \leq \norm{(\Ws,\Xs,\Ys)}{F}$, which implies $\max \left\{ \norm{(W,X,Y)}{F}, \norm{(\Wprime,\Xprime,\Yprime)}{F} \right\} \leq 2\norm{(\Ws,\Xs,\Ys)}{F}$ by the reverse triangle inequality.

We now proceed to prove Assumption~\ref{ass: local-strong-alignment} for this map. Let $\Ts \in \RR^{d_1\times d_2\times d_2}$ be an arbitrary tensor and let $\Ws \in \RR^{d_1\times r}, \Xs \in \times \RR^{d_2\times r} \text{ and }\Ys\in \RR^{d_3\times r}$ be any full-rank matrices such that $\Ts=\Fasym(\Ws,\Xs,\Ys)$. %
\begin{proposition}[Constant Rank of Asymmetric Canonical Polyadic Map] \label{prop: cp constant rank}
 For any full rank matrices \( W \in \mathbb{R}^{d_1 \times r}, X\in \mathbb{R}^{d_2 \times r} \text{ and } Y\in \mathbb{R}^{d_3 \times r}  \),  we have
    \[
    \rank \left(\nabla \Fasym(W,X,Y) \right) = (d_1 + d_2 + d_3 - 2)r.
    \]
\end{proposition} 
Equipped with this constant rank result, we invoke Lemma~\ref{constant_rank_implies_strong_alignment} to establish that the map $\Fasym$ satisfies Assumption~\ref{ass: local-strong-alignment} with the  quantities

\begin{align*}&\alignvarepsilon = \epsAsymtensor, \qquad \rloc(\rho) = \frac{\rho}{C} \\&  \qquad \qquad \qquad \qquad\text{ and }\qquad
    \sig = \frac{1}{2}\sigma_{(d_1+d_2+d_3-2)r}\left( \nabla \Fasym(\Ws,\Xs,\Ys) \right)\end{align*}
for some positive constants $R$ and  $C$ that depend only on the solution $(\Ws,\Xs,\Ys) $ and with $ j = \rank(\nabla \Fasym(\Ws,\Xs,\Ys)) = (d_1+d_2+d_3-2)r$.
This establishes the result, if we prove Proposition~\ref{prop: cp constant rank}.

\begin{proof}[Proof of Proposition~\ref{prop: cp constant rank}]
The proof structure involves three steps. For the first step, we will construct two sets $H$ and $H^c$ of vectors with cardinalities $(d_1+d_3+d_3-2)r$ and  $d_1d_2d_3 - (d_1+d_3+d_3-2)r$, respectively, such that $H\cup H^c$ forms a basis for $\RR^{d_{1} d_{2}d_{3}}.$ For the second step, we will show that the set $H^c \subseteq \nulll{ \nabla \Fasym(W,X,Y) \nabla \Fasym(W,X,Y)^\top} $. This establishes an upper bound of $(d_1+d_3+d_3-2)r$ on $\rankk{\nabla \Fasym(W,X,Y)}$. For the final step, we will prove that the set $H$ remains linearly independent after applying $\nabla \Fasym(W,X,Y) \nabla \Fasym(W,X,Y)^\top$. This establishes a lower bound of $(d_1+d_3+d_3-2)r$ on $\rankk{\nabla \Fasym(W,X,Y)}$, finishing the proof.

To start, we introduce some notation. Define the index sets
\begin{align}\label{eq:I_indices}
  \begin{split}
     I_0 &:= \left\{ (l,l,l)\mid l \in [r] \right\}, \\
     I_1 &:=  \left\{  (l,l,k) \mid l \in [r],\ k \in [d_3],\ l \neq k \right\},  \\
     I_2 &:=  \left\{(l,k,l) \mid l \in [r],\  k \in [d_2],\ l \neq k \right\},\quad \text{and} \\
    I_3 &:=  \left\{ (k,l,l) \mid l \in [r],\ k \in [d_1],\ l \neq k \right\}.
          \end{split}
\end{align}
In what follows, we use $M$ and $N$ as placeholders for $W, X$, or $Y$. Further, $d_{M}$ and $d_{N}$ denote the number of rows of $M$ and $N$, respectively. Define the sets
\begin{align}
    \label{eq: diag-set}
  \begin{split}
    \mathcal{T}_{\mathrm{off}}^{M, N}&:= \{ (i,j) \in [d_M]\times[d_N] \mid i > r \text{ or } \ j > r \text{ or }\ i \neq j \}, \quad \text{and}  \\
    \mathcal{T}_{\mathrm{on}}^{M, N}&:= (\mathcal{T}_{\mathrm{off}}^{d_M, d_N})^c =\{ (i,j) \in [d_M]\times[d_N] \mid i=j \text{ and } i\leq r\}.
                                      \end{split}
\end{align}
The sets in \eqref{eq:I_indices} form a partition of $I:=\{(i,j,k) \mid (i,j) \in  \mathcal{T}_{\mathrm{on}}^{W, X} \text{ or }  (j,k) \in \mathcal{T}_{\mathrm{on}}^{X, Y} \text{ or } (i,k)\in  \mathcal{T}_{\mathrm{on}}^{W, Y}\}$.

\paragraph{Probing sets.} For $M \in \{W, X, Y\}$, define
\begin{equation}
    \label{eq: vectors_row_pseudoinverse}
    T_i^M = U^M \left(\tilde{\Sigma}^M\right)^{-1} \left(\tilde{V}^M_{i:}\right)^\top,
\end{equation}
 where $\Tilde{\Sigma}^M = \diagg{  \sigma^M_1, \ldots, \sigma^M_r, 1, \ldots 1} \in \RR^{d_M \times d_M}, \text{ and } \Tilde{V}^M = \begin{pmatrix} V^M & 0 \\ 0 & I \end{pmatrix} \in \RR^{d_M \times d_M}$, where $U^M$ and $(V^M)^\top$ are the left and right eigenvectors of $M$, and $\sigma_i$'s are its singular values. %
 We construct $H:= \{ T_i^W \otimeskron T_j^X \otimeskron T_k^Y \mid (i,j,k) \in I \}$ and $H^c:=  \{ T_i^W \otimeskron T_j^X \otimeskron T_k^Y \mid (i,j,k) \in [d_1] \times [d_2] \times [d_3] \setminus I\}$. These sets are linearly independent since they correspond to column vectors of the matrix
$$U^W \left(\tilde{\Sigma}^W\right)^{-1} \left(\tilde{V}^W\right)^\top  \otimeskron U^X \left(\tilde{\Sigma}^X\right)^{-1} \left(\tilde{V}^X\right)^\top \otimeskron U^Y \left(\tilde{\Sigma}^Y\right)^{-1} \left(\tilde{V}^Y\right)^\top \in \RR^{d_1d_2d_3 \times d_1d_2d_3} $$ 
which again is invertible by~\eqref{kr_p3}. A relevant property about these vectors is
\begin{equation}\label{eq:orto-probing}
  \dotp{M_{i},T_{j}^{M}} = \mathbbm{1}_{i=j}.
  \end{equation}

\paragraph{Upper bound.} We use the following lemma; whose proof is deferred to Appendix~\ref{proof: cp_on_ts}.
\begin{lemma}\label{lem: cp_on_ts}
Let $(W, X, Y) \in \RR^{d_1 \times r} \times \RR^{d_2 \times r} \times \RR^{d_3 \times r}$ be full-rank matrices. Then,
{\allowdisplaybreaks
  \begin{flalign*}
    & \nabla  \Fasym(W,X,Y)  \nabla \Fasym(W,X,Y)^\top \left( T^W_{i} \otimeskron T^X_{j} \otimeskron T^Y_{k} \right) \\&= \mathbbm{1}_{(i,j)\in \mathcal{T}_{\mathrm{on}}^{W, X}} \left( W_{i} \otimeskron X_{j} \otimeskron T^Y_{k} \right) 
  +\; \mathbbm{1}_{(i,k)\in \mathcal{T}_{\mathrm{on}}^{W, Y}} \left( W_{i} \otimeskron T^X_{j} \otimeskron Y_{k} \right)
  + \mathbbm{1}_{(j,k)\in \mathcal{T}_{\mathrm{on}}^{X, Y}} \left( T^W_{i} \otimeskron X_{j} \otimeskron Y_{k} \right).
\end{flalign*}
}
\end{lemma}

By Lemma~\ref{lem: cp_on_ts},
it is easy to see that $\nabla  \Fasym(W,X,Y)  \nabla \Fasym(W,X,Y)^\top \left( T^W_{i} \otimeskron T^X_{j} \otimeskron T^Y_{k} \right)$
is  zero
when $
T^W_{i} \otimeskron T^X_{j} \otimeskron T^Y_{k} \in H^c
$. Therefore, $$
\spann{H^{c}} \subseteq \nulll{ \nabla \Fasym(W,X,Y) \nabla \Fasym(W,X,Y) ^\top}.$$ Since $H^c$ is linearly independant, and since $\card{H^c} = d_1d_2d_3 - (d_1+d_2+d_3-2)r$, then $\dimm{\spann{H^{c}}} =  d_1d_2d_3 - (d_1+d_2+d_3-2)r$. Therefore,
\begin{equation*}%
    d_1d_2d_3 - (d_1+d_2+d_3-2)r \leq \dimm{ \nulll{\nabla \Fasym(W,X,Y)\nabla \Fasym(W,X,Y)^\top} }.
\end{equation*}
Then, the rank-nullity theorem yields $$\rankk{\nabla \Fasym(W,X,Y)}=\rankk{\nabla \Fasym(W,X,Y)\nabla \Fasym(W,X,Y)^\top}  \leq (d_1+d_2+d_3 -2)r.$$
\paragraph{Lower bound.} We will show that the elements in $H$ remain linearly indepent after applying $\nabla \Fasym(W,X,Y) \nabla \Fasym(W,X,Y)^\top$, which shows that $\rankk{\nabla \Fasym(W,X,Y)\nabla \Fasym(W,X,Y)^\top}  \ge (d_1+d_2+d_3 -2)r$. By Lemma \ref{lem: cp_on_ts}, we have that elements of such a set are given by
\begin{align*}
&\nabla \Fasym(W,X,Y)  \nabla \Fasym(W,X,Y)^\top  \left( T_i^W \otimeskron  T_j^X \otimeskron  T_k^Y  \right) \\
&\qquad =
\begin{cases}
W_i \otimeskron X_i \otimeskron T_{i}^Y + W_i \otimeskron T_{i}^X \otimeskron Y_i + T_{i}^W \otimeskron X_i \otimeskron Y_i, & \text{if } (i,j,k) \in I_0, \\
W_i \otimeskron X_i \otimeskron T_{k}^Y, & \text{if } (i,j,k) \in I_1, \\
W_i \otimeskron T_{j}^X \otimeskron Y_i, & \text{if } (i,j,k) \in I_2, \\
  T_{i}^W \otimeskron X_j \otimeskron Y_j, & \text{if } (i,j,k) \in I_3, \\
  0 & \text{otherwise}.
\end{cases}
\end{align*}
Recall that $T_{i}^W \otimeskron T_j^X \otimeskron T_k^Y \in H$ if $(i,j,k) \in I = I_0 \cup I_1\cup I_2 \cup I_3.$ To show that these vectors are linearly independent, we will show that any linear combination of them that equates to zero has to have zero coefficients. Thus, suppose that we have vectors of coefficients $\alpha, \beta, \gamma, \delta$ such that the following linear combination is equal to zero
\begin{align}
  \label{eq:super-simple-fourth-of-july}
  \begin{split}
L^H := & \sum_{l=1 }^r \alpha_l \left( W_l \otimeskron X_l \otimeskron T_{l}^Y
+ W_l \otimeskron T_{l}^X \otimeskron Y_l
+ T_{l}^W \otimeskron X_l \otimeskron Y_l \right) \\
&\quad + \sum_{l \in [r],\, k \in [d_3] \setminus \{l\} } \beta_{l,k}  \left( W_l \otimeskron X_l \otimeskron T_{k}^Y \right) \\
&\quad\quad + \sum_{l \in [r],\, k \in [d_2] \setminus \{l\} } \gamma_{l,k}  \left( W_l \otimeskron T_{k}^X \otimeskron Y_l \right)\\
&\quad \quad\quad + \sum_{l \in [r],\, k \in [d_1] \setminus \{l\} } \delta_{l,k}  \left( T_{k}^W \otimeskron X_l \otimeskron Y_l\right) = 0.
  \end{split}
\end{align}
For the rest of the proof, we focus on showing that the coefficients $\alpha_l, \beta_{l,k}, \gamma_{l,k}, \delta_{l,k} \in \RR$ are all zero. To this end, we probe the equality above with several linear maps, which allows us to derive conclusions for specific coefficients. In particular, for fixed $i\in [r],\ j \in [d_3]\setminus \{i\}$,
we apply the linear transformation
$
I_{d_1} \otimeskron T^X_i \otimeskron T^Y_j
$
on both sides of the equality, which yields
 {\allowdisplaybreaks \begin{align}\label{eqn:linear_indep}
    \begin{split}
0 &= \left( I_{d_1} \otimeskron T^X_i \otimeskron T^Y_j \right)  L^H \\
&\overset{(i)}{=} \sum_{l=1}^r \alpha_l \left( \inner{X_l}{T_{i}^X} \inner{Y_l}{T_{j}^Y}
+ \inner{X_l}{T_{i}^X} \inner{T_l^Y}{T_{j}^Y}  + \inner{T_l^X}{T_{i}^X} \inner{Y_l}{T_{j}^Y}  \right) W_l \\
  &\qquad + \sum_{l \in [r],\, k \in [d_3] \setminus \{l\} } \beta_{l,k} \left( \inner{X_l}{T_{i}^X} \inner{T_{k}^Y}{T_{j}^Y}\right) W_l\\
      &\qquad \qquad + \sum_{l \in [r],\, k \in [d_2] \setminus \{l\} } \gamma_{l,k} \left( \inner{T_{k}^X}{T_{i}^X} \inner{Y_l}{T_{j}^Y}\right) W_l \\
&\qquad\qquad \qquad + \sum_{l \in [r],\, k \in [d_1] \setminus \{l\} } \delta_{l,k} \left( \inner{X_l}{T_{i}^X} \inner{Y_l}{T_{j}^Y}\right) T_{k}^W\\
  &\overset{(ii)}{=} \inner{\alpha_i T_{i}^Y}{T_{j}^Y}  W_i +  \inner{\alpha_j  T_{j}^X}{T_{i}^X}  W_j  \mathbbm{1}_{j \leq r}  \\
      & \qquad + \inner{\sum_{k \in [d_3] \setminus \{ i\}} \beta_{i,k}T_{k}^Y}{T_{j}^Y} W_i + \inner{\sum_{k \in [d_2] \setminus \{ j \}} \gamma_{j,k}T_{k}^X}{T_{i}^X} W_j  \mathbbm{1}_{j \leq r},
    \end{split}
    \end{align}
    } where $(i)$ follows from \eqref{kr_p3} and $(ii)$ use the choice of indices together with \eqref{eq:orto-probing}. Since $W$ has full column rank, the vectors  $\{W_k \mid k \in [r]\}$ are linearly independent. Thus, \eqref{eqn:linear_indep} implies that the coefficient associated with the vector $W_i$ is zero. Consequently,
$
\inner{ \alpha_i T^Y_i + \sum_{k \in [d_3] \setminus \{ i\}} \beta_{i,k} T_{k}^Y}{T_{j}^Y } = 0,
$
which can be rewritten equivalently as
\begin{align*}
  \inner{ T^Y \omega^{\alpha, \beta, i}}{ T_{j}^Y } = 0 \qquad \text{where} \qquad  \omega^{\alpha,\beta,i} :=\left(
\beta_{i,1} ,
  \dots ,
  \beta_{i, i-1},
  \alpha_i ,
  \beta_{i, i+1},
\dots ,
\beta_{i,d_3}\right)^{\top}.
\end{align*}
Since this holds for all $j \in [d_3] \setminus \{ i \}$, then $T^Y \omega^{\alpha, \beta, i}$ is orthogonal to all columns of $T^Y$ except the $i$-th column. The vector $Y_i$ is also orthogonal to all these $T^Y_j$ due to \eqref{eq:orto-probing}. Thus, $T^Y\omega^{\alpha, \beta, i}$ and $Y_i$ are orthogonal to the same $d_3-1$ dimensional subspace, consequently colinear. Hence, for any $i \in [r]$, there exists $c^{\alpha, \beta, i} \in \RR \setminus \{ 0 \}$ such that $T^Y \omega^{\alpha, \beta, i} = c^{\alpha, \beta, i}$ and, consequently,
$$\omega^{\alpha, \beta, i} = c^{\alpha, \beta, i} \left( T^Y \right)^{-1} Y_i = c_i^{\alpha, \beta} \tilde{V}^Y \left({\tilde{\Sigma}^Y} \right)^{2} \left( \tilde{V}^Y_{i:} \right)^\top .$$
By an equivalent argument, for any $i \in [r]$,  there exist $c^{\alpha, \gamma, i}, c^{\alpha, \delta, i} \in \RR$ such that \begin{equation} \label{closed_form}\omega^{\alpha, \gamma, i}  = c^{\alpha, \gamma,i} \tilde{V}^X \left( {\tilde{\Sigma}^X} \right)^{2} \left( \tilde{V}_{i:}^X \right)^\top,\qquad \text{and}\qquad \omega^{\alpha, \delta, i}  = c^{\alpha, \delta,i} \tilde{V}^W \left( {\tilde{\Sigma}^W}\right)^{2} \left( \tilde{V}^W_{i:} \right)^\top.\end{equation}
If we prove these vectors are zero, then the original coefficients in \eqref{eq:super-simple-fourth-of-july} would also be zero. Equipped with this closed-form expression for the coefficients, we derive
{\allowdisplaybreaks
\begin{align}
\label{lh_linear}
L^H &= \sum_{l=1}^r \left(W_l \otimeskron  X_l \otimeskron\left(\alpha_l T_l^Y + \sum_{k \in [d_3] \setminus \{l\} }   \beta_{l,k} T_k^Y \right) \right)  \nonumber \\ & \qquad+  \sum_{l=1}^r \left( W_l \otimeskron  \left(  \alpha_l T_l^X + \sum_{k \in [d_2] \setminus \{l\} }   \gamma_{l,k} T_k^X \right) \otimeskron Y_l \right) \nonumber \\
&\qquad \qquad  + \sum_{l=1}^r\left(\left(\alpha_l T_l^W +  \sum_{k \in [d_1] \setminus \{l\} }   \delta_{l,k} T_k^W \right) \otimeskron  X_l \otimeskron Y_l\right) \nonumber \\
&= \sum_{l=1}^r \left(W_l \otimeskron X_l \otimeskron T^Y  \omega^{\alpha,\beta, l} \right) + \left( W_l \otimeskron T^X  \omega^{\alpha,\gamma, l} \otimeskron Y_l \right) + \left( T^W  \omega^{\alpha,\delta,l} \otimeskron X_l \otimeskron Y_l \right)  \nonumber \\
&= \sum_{l=1}^r \; \left( c^{\alpha, \beta, l} + c^{\alpha,\gamma,l} + c^{\alpha,\delta,l} \right)  \left(  W_l \otimeskron X_l \otimeskron Y_l \right),
\end{align}
}
where the first equality follows from distributing the sum over $l$ and rearranging, the second equality rewrites $\alpha_l T_l^Y + \sum_{k \in [d_3] \setminus \{l\}} \beta_{l,k}T^Y_k$  as $T^Y\omega^{\alpha,\beta, l}$, and the third equality follows from the definition of $c^{\alpha, \beta, l}, c^{\alpha, \gamma, l},$ and $c^{\alpha, \delta,  l}$ and the multilinearity of the Kronecker product.

Since the vectors $\{ W_l \otimeskron X_l \otimeskron Y_l \mid l \in [r] \}$ are linearly independent, we conclude that $c^{\alpha, \beta, l} + c^{\alpha,\gamma,l} +c^{\alpha,\delta,l} = 0 $ for all $l$. Let's show that each term is also zero. Using~\eqref{closed_form}, we can extract the $l$th component from $\omega^{\alpha, \beta, l}, \omega^{\alpha, \gamma, l}$, and $\omega^{\alpha, \delta, l}$ via
\begin{align*}
\alpha_l = c^{\alpha,\beta,l} \, \tilde V^Y_{l:} \left(\tilde\Sigma^Y\right)^2 \left(\tilde V^Y_{l:}\right)^\top= c^{\alpha,\gamma,l}  \tilde V^X_{l:}\left(\tilde \Sigma^X\right)^2 \left(\tilde V^X_{l:}\right)^\top = c^{\alpha,\delta,l}  \tilde V^W_{l:}  \left(\tilde \Sigma^W\right)^2 \left(\tilde V^W_{l:}\right)^\top.
\end{align*}
Observe that all the quadratic forms in these equalities are strictly positive because $\tilde \Sigma^W$, $\tilde \Sigma^X$, and $\tilde \Sigma^Y$ are positive definite matrices.
Thus $c^{\alpha, \beta, l}$, $c^{\alpha,\gamma, l}$ and $c^{\alpha,\delta, l}$ have the same sign, which in turn implies that $c^{\alpha, \beta, l} + c^{\alpha,\gamma,l} +c^{\alpha,\delta,l} = 0 $ if, and only if, $ c^{\alpha, \beta, l} = c^{\alpha,\gamma,l} =c^{\alpha,\delta,l} = 0 $. Thus, by~\eqref{closed_form}, that $\omega^{\alpha,\beta, l} = \omega^{\alpha,\gamma, l} = \omega^{\alpha,\delta,l} = 0$ and, consequently all the coefficients in \eqref{eq:super-simple-fourth-of-july} are also zero. Hence, the vectors $\{\nabla \Fasym(W,X,Y)^\top \nabla \Fasym(W,X,Y)v \mid v \in H \}$ are linearly independent; this finishes the proof of Proposition~\ref{prop: cp constant rank}.
\end{proof}
This concludes the proof of Theorem~\ref{thm-tensor-asymmetric-reg-conditions}.

    \section{Auxiliary proofs and results}
    In this section, we summarize some auxiliary results that we use throughout the paper.

\begin{lemma}[Properties of the Kronecker Product]
\label{lem:kronecker_properties}[\cite{schacke2004kronecker}]
Let \(A, B, C, D\) be matrices (or vectors) of compatible dimensions. Then, the Kronecker product satisfies the following properties.
\begin{enumerate}%
    \item Multiplication from the right and left can be succinctly written as
    \begin{equation} \label{kr_p1}
        \vect{A  B  C}
        =
        \left(C^\top \otimeskron A\right)\,\vect{B}.
    \end{equation}
    \item The transpose commutes with the Kronecker product
    \begin{equation} \label{kr_p2}
        (A \otimeskron B)^\top
        \;=\;
        A^\top \otimeskron B^\top.
    \end{equation}
    \item The matrix product commutes with the Kronecker product
    \begin{equation} \label{kr_p3}
        (A  B) \otimeskron (C  D)
        =
        (A \otimeskron C)(B \otimeskron D).
    \end{equation}
    \item
    Trivially, if \(x \in \mathbb{R}^n\) and \(y \in \mathbb{R}^m\), the vectorized outer product is equal to the Kronecker product
    \begin{equation} \label{kr_p4}
        y \otimeskron x
        =
        \vect{xy^\top}.
    \end{equation}
\end{enumerate}
\end{lemma}

\begin{lemma}\label{lem:adding_stuff_to_orhogonal}
Let $U\in \RR^{d\times d}$ be an orthogonal matrix. Then for any $I\subset J \subseteq [d]$ and vector $v \in \RR^{d}$ we have
$$
\norm{U_{I}U_{I}^\top v}{2} \leq \norm{U_{J}U_{J}^\top v}{2}. 
$$
\end{lemma}
\begin{proof}
We have $$\norm{U_{J} U_{J}^{\top} v}{2}^2 %
= \inner{v}{U_{J} U_{J}^{\top} v} %
= \inner{v}{U_{J \setminus I} U_{J \setminus I}^{\top} v} + \inner{v}{U_{I} U_{ I}^{\top} v} 
\geq \norm{U_I U_I^{\top} v}{2}^2,$$ where the first equality follows since orthogonal projections are symmetric and idempotent. 
\end{proof}
\begin{lemma}[Lemma 2.5 in \cite{Chen_2021}]
\label{lem:orthogonal-matrices-sine}
Suppose $U = [U_0, U_1]$ and $V = [V_0, V_1]$ are square orthonormal matrices, where $U_0, V_0 \in \mathbb{R}^{d \times k}$. Let $\theta_1, \ldots, \theta_k$ denote the principal angles between $\operatorname{span}(U_0)$ and $\operatorname{span}(V_0)$, and denote $\Theta(U_0,V_0) = \diagg{\theta_1, \ldots, \theta_k}$. Then
\begin{equation}
\norm{U_0^\top V_1}{F} = \norm{\sin \Theta(U_0,V_0)}{F}.
\end{equation}
\end{lemma}

\begin{lemma}[Lemma 2.6 in \cite{Chen_2021}]
\label{lem:procrustes-sin-theta}
Let \(U,V\in\RR^{d\times k}\;(k\le d)\) have orthonormal columns. Write the principal angles between \(\operatorname{span}(U)\) and \(\operatorname{span}(V)\) as
\(\theta_1,\dots,\theta_k\) and denote
\(\Theta(U,V)=\diagg{\theta_1,\dots,\theta_k}\).
Then, $$\min_{Q\in O(k)} \norm{U-VQ}{F}
        \leq \sqrt{2}\norm{\sin\Theta(U,V)}{F}.$$
\end{lemma}

\subsection{Proof of Lemma~\ref{lem:spectral-bm}}\label{sec:proof-spectral-bm}

\paragraph{Eigenpairs.}
For any pair of indexes $(i, j)$ with $i \leq j$, label $Z = \ux_{i}{ \ux_{j} }^\top + \ux_{j}{\ux_{i}  }^\top$. %
Then,{\allowdisplaybreaks
        \begin{align*}
            \nabla \Fsym(X) \nabla \Fsym(X)^\top [Z] &= 2\nabla \Fsym(X) \left[ Z X \right] \\
            &= 2ZXX^\top + 2XX^\top Z \\
            &= 2\left(Z\sum_{k=1}^n \sigma_{k}^2\left( X \right)  {\ux_k} { \ux_k }^\top + \sum_{k=1}^n \sigma_{k}^2\left( X \right) {\ux_k} { \ux_k }^\top Z\right) \\
            &= 2\left(\sigma_{i}^2\left( X \right)  \left(\ux_{i}{ \ux_{j} }^\top + \ux_{j}{\ux_{i}  }^\top \right) + \sigma_{j}^2\left( X \right)  \left(\ux_{i}{ \ux_{j} }^\top + \ux_{j}{\ux_{i}  }^\top\right)\right) \\
            &= 2\left(\sigma_{i}^2\left( X \right) + \sigma_{j}^2\left( X \right)\right)Z,
        \end{align*}
        }
        where the first two lines follow from~\eqref{action: nablac-and-nablacT-bm} and the last two lines follow by definition of $Z$.
         \paragraph{Orthonormal basis.}
Recall that the image of $\nabla \c$ is $\SSS^{d},$ thus the number of eigenvectors above matches the number of dimensions of $\SSS^{d}.$ It suffices to prove that they are orthogonal. Let $U'_{i,j}$ be a placeholder for
\[
 U'_{i,j} =
    \text{vec}\left( \ux_{i}{ \ux_{j} }^\top + \ux_{j}{\ux_{i}  }^\top \right).
\]
Let $\left(i,j\right), \left(k,\ell\right) \in \left[d\right] \times \left[d\right]$ with $i \leq j$ and $k \leq \ell$. We shall show that $U'_{i,j}$ is nonzero, and if $\left(i, j\right) \neq \left(k, \ell\right),$ then $\inner{ U'_{i,j}} {U'_{k,\ell} } = 0$. Then,
    {\allowdisplaybreaks \begin{align*}
        \inner{ U'_{ij}}{U'_{k,\ell} } &= \inner{ \ux_{i}{ \ux_{j} }^\top + \ux_{j}{\ux_{i}  }^\top} { {\ux_k} {\ux_\ell}^\top + \ux_\ell { \ux_k }^\top } \\
        &= \trace{\ux_{i}{ \ux_{j} }^\top \ux_\ell { \ux_k }^\top } + \trace{\ux_{i}{ \ux_{j} }^\top {\ux_k} {\ux_\ell}^\top } \\
        &\qquad + \trace{\ux_{j}{\ux_{i}  }^\top \ux_\ell { \ux_k }^\top } + \trace{\ux_{j}{\ux_{i}  }^\top {\ux_k} {\ux_\ell}^\top } \\
        &= 2 \inner{ \ux_{i}} { {\ux_k} } \inner{ \ux_{j}} { \ux_\ell } + 2 \inner{ \ux_{i}} { \ux_\ell } \inner{ \ux_{j}} { {\ux_k} }.
    \end{align*}}
    First if $\left(i,j\right) = \left(k,\ell\right)$, then $\norm{U'_{i,j}}{2}^2 = \begin{cases} 2, & i \neq j \\ 4, & \text{otherwise} \end{cases} \neq 0$ and so $U'_{i,j}$ is nonzero. Second, if $\left(i,j\right) \neq \left(k,\ell\right)$, then either $i \neq k$ or $j \neq \ell$. Without loss of generality, assume that $i \neq k$; thus, the first term in the last expression is zero. Seeking contradiction, assume that the second term is nonzero. Since $\left\{\ux_{i}\right\}$ forms an orthonormal basis, we derive that $i = \ell \geq  k = j$ and by assumption $i \leq j$, therefore, $i = j = k = \ell$, which is a contradiction. %
    This completes the proof of the lemma.

\subsection{Proof of Lemma \ref{lem:distance_factors_implies_distance_op_asymmetric}} \label{proof: lem-distance_factors_implies_distance_op_asymmetric}
Recall that we denote by $\rs$ the ranks of $\Xs$ and $\Ys$. One has 
{\allowdisplaybreaks \begin{align*}
    \norm{XX-\Xs\Xs^\top }{F} &\leq \norm{\left( X-\Xs \right)X^\top +  \Xs \left( X- \Xs \right)^\top }{F}  
    \\ &\overset{(i)}{\leq} \norm{\left( X-\Xs \right)X^\top }{F} + \norm{\Xs \left( X-\Xs \right)^\top }{F}      
        \\ &\overset{(ii)}{\leq} 2 \max \left\{ \sigma_1(\Xs), \sigma_1(X) \right\} \norm{ X-\Xs  }{F} 
        \\ &\overset{(iii)}{\leq} 2  \left( \sigma_1(\Xs) + \norm{X-\Xs}{F} \right) \norm{X-\Xs}{F} 
        \\ &\overset{(iv)}{\leq}\frac{\min\left\{ \sigma_{\rs}^2(\Xs), \sigma_{\rs}^2(\Ys)\right\} }{8 \sqrt{2}} + \frac{\min\left\{ \sigma_{\rs}^2(\Xs), \sigma_{\rs}^2(\Ys)\right\}}{8\sqrt{2} \sigma_1(\Xs)} \norm{X-\Xs}{F}  
        \\ &\overset{(v)}{\leq}\frac{\min\left\{ \sigma_{\rs}^2(\Xs), \sigma_{\rs}^2(\Ys)\right\} }{4\sqrt{2}},
\end{align*}}
where $(i)$ follows from the triangle inequality, $(ii)$ follows from the variational characterization of singular values, $(iii)$ follows from Weyl's inequality, $(iv)$ holds since $ \norm{X - \Xs}{F} \leq \frac{\min\left\{ \sigma_{\rs}^2(\Xs), \sigma_{\rs}^2(\Ys)\right\}}{16\sqrt{2} \sigma_1(\Xs)}$ and $(v)$ holds since $\frac{\min\left\{ \sigma_{\rs}^2(\Xs), \sigma_{\rs}^2(\Ys)\right\}}{16\sqrt{2} \sigma_1(\Xs)} \leq \sigma_1(\Xs)$. A similar argument yields
$$
\norm{YY-\Ys\Ys^\top }{F} \leq\frac{\min\left\{ \sigma_{\rs}^2(\Xs), \sigma_{\rs}^2(\Ys)\right\} }{4\sqrt{2}}.
$$
Adding both bounds, we conclude that $ \norm{XX - \Xs\Xs^\top}{F} + \norm{YY-\Ys\Ys^\top }{F}\leq \frac{\min\left\{ \sigma_{\rs}^2(\Xs), \sigma_{\rs}^2(\Ys)\right\} }{2\sqrt{2}}$.

\subsection{Proof of Lemma \ref{lem: spectral_asymmetric}}\label{sec:proof-spectral-asym}
\paragraph{Eigenpairs.} Let $i,j \in [d_1]\times [d_2]$ with corresponding eigenpairs $\left(\sigma_{i}^2\left( X \right), { \ux_{i} } \right)$ and $\left(\sigma_{j}^2\left( Y \right), { \uy_{j} }\right)$ for respectively $XX^\top$ and $YY^\top$. Let $M$ be a placeholder for  $\ux_{i} {\uy_{j}}^\top$. We have
\begin{align*}
\nabla F(X,Y) \, \nabla F(X,Y)^\top M &= M \, YY^\top + XX^\top M \\
&= \sigma_{j}^2\left( Y \right) \left( \ux_{i} {\uy_{j}}^\top \right) + \sigma_{i}^2\left( X \right) \left( \ux_{i} {\uy_{j}}^\top \right) \\
&= \left( \sigma_{i}^2\left( X \right) + \sigma_{j}^2\left( Y \right) \right) M,
\end{align*}
where the first equality follows from~\eqref{action: nablac-and-nablacT-asym} and the second equality is due to the orthonormality of the eigenvectors of both $XX^\top \text{ and } YY^\top$.

\paragraph{Orthonormal basis. } Let $(i,j), (k,\ell) \in [d_1] \times [d_2]$. We will show that if $(i,j) = (k,\ell)$, then $\inner{{ \ux_{i} }{{ \uy_{j} }}^\top }{{{\ux_k}}{\uy_{\ell}}^\top} \neq 0$ and that if   $(i,j) \neq (k,l)$, then $\inner{ { \ux_{i} }{{ \uy_{j} }}^\top }{{{\ux_k}}{\uy_{\ell}}^\top} = 0$.

 \begin{enumerate}
     \item[] \textit{Case 1.} $(i,j) = (k,l)$. Then  $\inner{{ \ux_{i} }{{ \uy_{j} }}^\top }{{ \ux_{i} }{{ \uy_{j} }}^\top} = \operatorname{tr}(  { \ux_{i} } {{ \uy_{j} }}^\top {{ \uy_{j} }} {{ \ux_{i} }}^\top ) = \inner{{{ \ux_{i} }}}{{{ \ux_{i} }}} \inner{{{ \uy_{j} }}}{{{ \uy_{j} }}} = 1 \neq 0$.
     \item[] \textit{Case 2.}$(i,j) \neq (k,l)$. Then  $\inner{ {{ \ux_{i} }}{{ \uy_{j} }}^\top }{{{\ux_k}}{\uy_{\ell}}^\top} = \inner{{{ \ux_{i} }}}{{{\ux_k}}} \inner{{{ \uy_{j} }}}{{\uy_{\ell}}} = 0$.
\end{enumerate}

\subsection{Proof of Lemma \ref{lemma: action-nablac-symmetric}} \label{proof: action-nablac-symmetric}

Fix $X, D \in \RR^{d\times r}$ arbitrary. For the action of the Jacobian, expanding the denominator in $\nabla \Fsym(X) \vect{D} = \lim_{t\downarrow 0} \frac{\Fsym(X + tD) - \Fsym(X)}{t}$ leads to \begin{equation}\label{eq: action-symcp-jac}
\nabla \Fsym(X) \vect{D} = \sum_{\ell=1}^r \left(D_\ell \otimeskron X_\ell \otimeskron X_\ell +   X_\ell \otimeskron D_\ell \otimeskron X_\ell +  X_\ell \otimeskron X_\ell \otimeskron D_\ell\right).
\end{equation}
\noindent A straight computation shows that the permutation matrices $P_2$ and $P_3$ from Definition~\ref{Pis} satisfy $P_2(a\otimeskron b \otimeskron b) = b \otimeskron b \otimeskron a$  and $P_3(a\otimeskron b \otimeskron b) = b\otimeskron a \otimeskron b $ for arbitary vectors $a,b$, so that $\nabla \Fsym(X) \vect{D}$ can be written as $ \sum_{\ell=1}^r \left( I+ P_3 + P_2 \right) \left(D_\ell \otimeskron X_\ell \otimeskron X_\ell \right).$

For the action of the adjoint, recall that by definition, we have
$$
\inner{\nabla \Fsym(X)^\top  ( a \otimeskron b \otimeskron c )}{\vect{D}} = \inner{\nabla \Fsym(X)  \vect{D}}{a\otimeskron b \otimeskron c}.
$$
Using~\eqref{eq: action-symcp-jac}, we have
{\allowdisplaybreaks
\begin{align*}
    \langle \nabla &\Fsym(X)  \vect{D}, a\otimeskron b \otimeskron c\rangle \\&= \inner{ \sum_{\ell\in[r]} \left(X_\ell \otimeskron X_\ell \otimeskron D_\ell +   X_\ell \otimeskron D_\ell \otimeskron X_\ell +  D_\ell \otimeskron X_\ell \otimeskron X_\ell\right)}{ a\otimeskron b \otimeskron c}\\
    &= \sum_{i,j,k \in [d]^3} \sum_{\ell\in[r]} \left( X_{i\ell }X_{j\ell }D_{k\ell } + X_{i\ell }D_{j\ell }X_{k\ell }+ D_{i\ell }X_{j\ell }X_{k\ell } \right) a_i b_j c_k \\
    &= \sum_{i,j,k \in [d]^3} \sum_{\ell\in[r]}  (X_{i\ell }X_{j\ell }D_{k\ell } a_i b_j c_k +  X_{i\ell }D_{j\ell }X_{k\ell }a_i b_j c_k +  D_{i\ell }X_{j\ell }X_{k\ell } a_i b_j c_k).
\end{align*}} Next, we derive an expression for the first term of the above equation. Changing the order of summation, we have
\begin{align*}
\sum_{i,j,k \in [d]^3} \sum_{\ell\in[r]}  X_{i\ell }X_{j\ell }D_{k\ell }a_i b_j c_k &= \sum_{k\in [d], \ell \in [r]} D_{k \ell} \left( c_k \sum_{i,j \in [d]^2} X_{i\ell} X_{j \ell}a_i b_j \right).
\end{align*} 
The matrix $M$ with components $M_{k\ell} = c_k \sum_{i,j \in [d]^2} X_{i\ell} X_{j \ell}a_i b_j$ can be compactly written as $M=c  ( a^\top \otimeskron b^\top) [\, X_1 \otimeskron X_1 \;\cdots\; X_r \otimeskron X_r\,]$. Therefore,
$$
\sum_{i,j,k \in [d]^3} \sum_{\ell=1}^r  X_{i\ell }X_{j\ell }D_{k\ell }a_i b_j c_k =\inner{D}{c  ( a^\top \otimeskron b^\top) [\, X_1 \otimeskron X_1 \;\cdots\; X_r \otimeskron X_r\,]}.
$$
Similarly, for the two remaining terms
$$
\sum_{i,j,k \in [d]^3} \sum_{\ell=1}^r  X_{i\ell }D_{j\ell }X_{k\ell }a_i b_j c_k  =\inner{D}{b  ( a^\top \otimeskron c^\top) [\, X_1 \otimeskron X_1 \;\cdots\; X_r \otimeskron X_r\,]},
$$  and
$$
\sum_{i,j,k \in [d]^3} \sum_{\ell=1}^r   D_{i\ell }X_{j\ell }X_{k\ell }a_i b_j c_k  =\inner{D}{a  ( b^\top \otimeskron c^\top) [\, X_1 \otimeskron X_1 \;\cdots\; X_r \otimeskron X_r\,]}.
$$ By linearity of the inner product and matrix multiplication, we obtain
$$
 \inner{\nabla \Fsym(X)^\top  ( a \otimeskron b \otimeskron c )}{\vect{D}} = \inner{  \left(c  ( a^\top \otimeskron b^\top) + b  ( a^\top \otimeskron c^\top) + a  ( b^\top \otimeskron c^\top) \right)  \extone{X}}  {D}.$$
 Since $D$ was arbitrary, one can show component-wise that
$\nabla \Fsym(X)^\top  ( a \otimeskron b \otimeskron c )$ matches $\left(c  ( a^\top \otimeskron b^\top) + b  ( a^\top \otimeskron c^\top) + a  ( b^\top \otimeskron c^\top) \right)  \extone{X}$. This concludes the proof.

\subsection{Proof of Lemma~\ref{lem: cp_on_ts}}\label{proof: cp_on_ts}

The proof relies heavily on the following two claims.
\begin{claim}\label{lemma: action nablaTnabla tensor asymmetric} Let $u\otimeskron v \otimeskron w\in \RR^{d_1d_2d_3}$ be arbitrary. Then,
\begin{align*}
\nabla \Fasym(W,X,Y) \nabla \Fasym(W,X,Y)^\top  \left( u \otimeskron v \otimeskron w \right)   & =  u \otimeskron \left[ \Psi(X,Y)  \left( v \otimeskron w \right)  \right]  \\ & \qquad + P_2  \left( v \otimeskron \left[ \Psi(W,Y)  \left( u \otimeskron w \right) \right] \right)\\& \qquad \qquad + P_3 \left( w \otimeskron \left[ \Psi(W,X)  \left( u \otimeskron v \right) \right] \right),
\end{align*}
where $P_i$ and $\Psi$ are introduced in Definitions~\ref{Pis}
 and~\ref{psi}, respectively.\end{claim}
\begin{proof}[Proof of Claim~\ref{lemma: action nablaTnabla tensor asymmetric}]
  By Lemma \ref{jac: cp}, we have
  {\allowdisplaybreaks 
\begin{align*}
 \nabla \c(W, X, Y) \nabla \c(W,X, Y)^{\top}&= J^{W}\left(J^{W}\right)^\top + J^{X}\left(J^{X}\right)^\top + J^{Y}\left(J^{Y}\right)^\top \\&=
I_{d_1} \otimeskron \left(  \sum_{l=1}^r \left( X_lX_l^\top \right) \otimeskron \left( Y_lY_l^\top \right) \right) \\&\qquad + P_2 \left( I_{d_2} \otimeskron \left(  \sum_{l=1}^r \left( W_lW_l^\top \right) \otimeskron \left( Y_lY_l^\top \right) \right) \right) P_2^\top \\ &\qquad\qquad +  P_3 \left( I_{d_3} \otimeskron \left( \sum_{l=1}^r \left( W_lW_l^\top \right) \otimeskron \left( X_lX_l^\top \right) \right)\right) P_3^\top,
\end{align*}}
where the second equality we use the Kronecker property \eqref{kr_p2} to take the transpose and \eqref{kr_p3} to simplify the product. Given a vector $u \otimeskron v \otimeskron w$ we expand
{\allowdisplaybreaks
\begin{align*}
&\nabla F(W,X,Y)\,\nabla F(W,X,Y)^\top  \left( u \otimeskron v \otimeskron w \right) \\
&= \left( I_{d_1} \otimeskron \left(\sum_{l=1}^r \left( X_lX_l^\top \right) \otimeskron \left( Y_lY_l^\top \right) \right)\right) (u \otimeskron v \otimeskron w) \\
&\quad+ P_2 \left( I_{d_2} \otimeskron \left(\sum_{l=1}^r \left( W_lW_l^\top \right) \otimeskron \left( Y_lY_l^\top \right) \right) \right)P_2^\top (u \otimeskron v \otimeskron w) \\
&\quad \quad+  P_3 \left( I_{d_3} \otimeskron \left( \sum_{l=1}^r \left( W_lW_l^\top \right) \otimeskron \left( X_lX_l^\top \right) \right)  \right) P_3^\top (u \otimeskron v \otimeskron w) \\
&\overset{(i)}{=} \left( I_{d_1} \otimeskron \left( \sum_{l=1}^r \left( X_lX_l^\top \right) \otimeskron \left( Y_lY_l^\top \right) \right) \right) (u \otimeskron v \otimeskron w) \\
&\quad+ P_2 \left( I_{d_2} \otimeskron \left(\sum_{l=1}^r \left( W_lW_l^\top \right) \otimeskron \left( Y_lY_l^\top \right) \right) \right) (v \otimeskron u \otimeskron w) \\
&\quad\quad+  P_3 \left( I_{d_3} \otimeskron \left(\sum_{l=1}^r \left( W_lW_l^\top \right) \otimeskron \left( X_lX_l^\top \right) \right)\right) (w \otimeskron u \otimeskron v) \\
&\overset{(ii)}{=} u \otimeskron \left( \left(\sum_{l=1}^r X_lX_l^\top \otimeskron Y_lY_l^\top\right) \left( v \otimeskron w \right) \right)
+ P_2 \left[ v \otimeskron \left(\left( \sum_{l=1}^r W_lW_l^\top \otimeskron Y_lY_l^\top \right) \left( u \otimeskron w \right) \right) \right] \\
&\quad+ P_3 \left[ w \otimeskron \left(\left( \sum_{l=1}^r W_lW_l^\top \otimeskron X_lX_l^\top \right) \left( u \otimeskron v \right) \right) \right]\\
&=u \otimeskron \left[ \Psi(X,Y)  \left( v \otimeskron w \right)  \right] + P_2  \left( v \otimeskron \left[ \Psi(W,Y)  \left( u \otimeskron w \right) \right] \right) + P_3 \left( w \otimeskron \left[ \Psi(W,X)  \left( u \otimeskron v \right) \right] \right),
\end{align*} 
}
where $(i)$ follows from the definition of $P_{i}$ and $(ii)$ follows from~\eqref{kr_p3}. This concludes the proof.
\end{proof}

\begin{claim}\label{lemma: nullspace of psi}
The following identity holds $$\Psi(M,N)\left( T^M_{i} \otimeskron T^N_{j} \right)=\mathbbm{1}_{(i,j) \in \mathcal{T}_{\mathrm{on}}^{M, N}}(M_i \otimeskron N_j).$$
\end{claim}
\begin{proof}[Proof of Claim~\ref{lemma: nullspace of psi}]
To establish the claim, we express the operator $\Psi(M,N)$ using the SVD of $M$ and $N$. Since $M_l = \sum_{k=1}^r \sigma_k^M V^M_{kl} U^M_k$, then $ \sum_{l=1}^r M_l M_l^\top = \sum_{l, k_1, k_2 \in [r]^3}  \sigma_{k_1}^M  \sigma_{k_2}^MV^M_{k_1l} V^M_{k_2l} U^M_{k_1} {U^M_{k_2}}^\top$. Moreover, by bilinearity of the Kronecker product, we have $$M_l M_l^\top \otimeskron N_l N_l^\top =\sum_{k_1, k_2, \ell_1, \ell_2 \in [r]^4}   \sigma_{k_1}^M \sigma_{k_2}^M \sigma_{\ell_1}^N \sigma_{\ell_2}^N V^M_{k_1l} V^M_{k_2l}   V^N_{\ell_1l}V^N_{\ell_2l}  \left( U^M_{k_1}{U^M_{k_2}}^\top \right) \otimeskron \left( U^N_{\ell_1} {U^N_{\ell_2}}^\top \right), $$
and so, taking a sum over $\ell$ yields $$
\Psi(M,N) =\sum_{l, k_1, k_2, \ell_1, \ell_2 \in [r]^5}   \sigma_{k_1}^M \sigma_{k_2}^M\sigma_{\ell_1}^N\sigma_{\ell_2}^N V^M_{k_1l} V^M_{k_2l}   V^N_{\ell_1l} V^N_{\ell_2l}  \left( U^M_{k_1}{U^M_{k_2}}^\top \right) \otimeskron \left( U^N_{\ell_1} {U^N_{\ell_2}}^\top \right).
$$
We are now ready to establish the action of this operator in a vector of the form $T_i^M \otimeskron T_j^N$. Recall the definition of $\mathcal{T}_{\mathrm{on}}^{M, N}$ given in \eqref{eq: diag-set}. First, assume that $\max\{i, j\} > r $, without loss of generality, suppose that $i  > r$. Then, $T_i^M = U_i^M$, and $U_{k_2}^\top T_i^M = 0$ for all $k_2 \in [r]$ and, consequently, $\Phi(M, N)(T_{i}^{M}\otimeskron T_{j}^{N}) = 0$. Second, assuming both $i,j \in [r]$, we have that $T_i^M \otimeskron T_j^N = U^M \left( \Sigma^M \right)^{-1}\left(V^M_{:i}\right)^\top \otimeskron  U^N \left( \Sigma^N \right)^{-1}\left(V^N_{:j}\right)^\top$. Then 
\begin{align}
\Psi&(M,N)  \left( T_i^M \otimeskron T_j^N \right) \notag \\ &=    \sum_{l, k_1, k_2, \ell_1, \ell_2 \in [r]^5}  \sigma^M_{k_1}  \sigma^M_{k_2} \sigma^N_{\ell_1} \sigma^N_{\ell_2} V^M_{ l k_1} V^M_{ l k_2} V^N_{ l \ell_1} V^N_{ l \ell_2}  \left( \um_{k_1}{\um_{k_2}}^\top U^M  (\Sigma^M)^{-1}  \left(V^M_{:i}\right)^\top \right) \notag \\ & \hspace{250pt} \otimeskron \left( \un_{\ell_1} {\un_{\ell_2}}^\top  U^N  (\Sigma^N)^{-1}   \left(V^N_{:j}\right)^\top \right)  \\
&\overset{(i)}{=} \sum_{l, k_1, k_2, \ell_1, \ell_2 \in [r]^5} \sigma^M_{k_1}  \sigma^M_{k_2} \sigma^N_{\ell_1} \sigma^N_{\ell_2}V^M_{ l k_1} V^M_{ l k_2} V^N_{ l \ell_1} V^N_{ l \ell_2}  V^M_{i k_2}  (\sigma^{M}_{k_2})^{-1}   V^N_{ j \ell_2} (\sigma^{N}_{\ell_2})^{-1}  \left( \um_{k_1}\otimeskron \un_{\ell_1}  \right) \notag \\
&\overset{(ii)}{=} \sum_{l, k_1, k_2, \ell_1, \ell_2 \in [r]^5}  \sigma^M_{k_1}  \sigma^N_{\ell_1}  V^M_{ l k_1} V^M_{ l k_2} V^N_{ l \ell_1} V^N_{ l \ell_2}  V^M_{i k_2}  V^N_{ j \ell_2}  \left( \um_{k_1}\otimeskron \un_{\ell_1}  \right) \notag\\
&\overset{(iii)}{=} \sum_{l, k_1, \ell_1 \in [r]^3}  \sigma^M_{k_1}  \sigma^N_{\ell_1}
V^M_{ l k_1} V^N_{ l \ell_1} \inner{V^M_{:l}}{V^M_{:i}} \inner{V^N_{:l}}{V^N_{:j}}   \left( \um_{k_1}\otimeskron \un_{\ell_1}  \right) \\
&= \mathbbm{1}_{i=j} \sum_{k_1, \ell_1 \in [r]^2}  \sigma^M_{k_1}  \sigma^N_{\ell_1}  V^M_{ i k_1} V^N_{ j \ell_1}   \left( \um_{k_1}\otimeskron \un_{\ell_1}  \right)  \notag \\
&= \mathbbm{1}_{(i,j)\in \mathcal{T}_{\mathrm{on}}^{M, N}}\left(M_i \otimeskron N_j \right) \notag,
\end{align}
 where $(i)$ follows since ${\um_{k_2}}^\top U^M  (\Sigma^M)^{-1}  \left(V^M_{:i}\right)^\top = \left(\sigma_{k_2}^M\right)^{-1} V_{ik_2}^M$ and ${\um_{\ell_2}}^\top U^N  (\Sigma^N)^{-1}  \left(V^N_{:i}\right)^\top = \left(\sigma_{\ell_2}^N \right)^{-1} V_{i\ell_2}^N$, $(ii)$ follows from the fact that the columns of $U^M$ and $U^{N}$ form orthonormal bases, and $(iii)$ follows from factorizing out the dot product. This finishes the proof of the claim. 
\end{proof}

We apply Claim~\ref{lemma: action nablaTnabla tensor asymmetric} with
\(
u = T^W_{i},
v = T^X_{j},\text{ and }
w = T^Y_{k},
\)
to derive \begin{align*}
\nabla \Fasym(W,X,Y)\,\nabla F(W,X,Y)^\top 
  \left( T^W_{i} \otimeskron T^X_{j} \otimeskron T^Y_{k} \right)
       & =
T^W_{i} \otimes\left[\Psi\left(X,Y\right)
\left(T^X_{j} \otimes T^Y_{k}\right)\right] \\
& \quad \quad +
P_2 \left[
T^X_{j} \otimes \Psi\left(W,Y\right)\left(T^W_{i} \otimes T^Y_{k}\right)
\right]
  \\ & \quad \quad \quad
       +
P_3 \left[
T^Y_{k} \otimes\Psi\left(W,X\right)\left(T^W_{i} \otimes T^X_{j}\right)
\right].
\end{align*}
Using Lemma~\ref{lemma: nullspace of psi} in tandem with Definition~\ref{Pis} of \(P_2\) and \(P_3\) we obtain
\begin{align*}
\nabla F(W,X,Y)\,\nabla F(W,X,Y)^\top 
\left( T^W_{i} \otimeskron T^X_{j} \otimeskron T^Y_{k} \right)
& = \mathbbm{1}_{(i,j)\in \mathcal{T}_{\mathrm{on}}^{W, X}} \left( W_{i} \otimeskron X_{j} \otimeskron T^Y_{k} \right) 
    \\ & \quad\quad
  +\; \mathbbm{1}_{(i,k)\in \mathcal{T}_{\mathrm{on}}^{W, Y}} \left( W_{i} \otimeskron T^X_{j} \otimeskron Y_{k} \right)
\\ & \qquad \quad
  + \mathbbm{1}_{(j,k)\in \mathcal{T}_{\mathrm{on}}^{X, Y}} \left( T^W_{i} \otimeskron X_{j} \otimeskron Y_{k} \right).
\end{align*}
This completes the argument.

\subsection{Alignment lemmas}

\begin{lemma}\label{second-term-assymetric}

    Let $X,\Xs \in \mathbb{R}^{d_1 \times r}$ and $Y, \Ys \in \mathbb{R}^{d_2 \times r}$ with $\rank\left(\Xs\right)=\rank \left( \Ys \right) =\rs$. Let $k_1\in\{\rs \ldots r\}$, $k_2 \in\{\rs \ldots r \}$,  and define $\mathcal{Q}_X:= \sum_{i=k_1+1}^{d_1} U^X_i {U^X_i}^\top \quad \text{and} \quad \mathcal{Q}_Y := \sum_{i=k_2 + 1}^{d_2} U^Y_i {U^Y_i}^\top$. Under the assumption that $V^\Xs= V^\Ys$, if $$\norm{XX^\top - \Xs \Xs^\top}{F} + \norm{YY^\top - \Ys \Ys^\top}{F} \leq  \frac{1}{2\sqrt{2}} \min\left\{ \sigma^2_{\rs}\left( \Xs \right),  \sigma^2_{\rs}\left( \Ys \right) \right\},$$ one has
\[
 \norm{ \cQ_{X} \Xs{\Ys}^\top \cQ_{Y}    }{F} \leq  \frac{1}{\min\left\{ \sigma^2_{\rs}\left( \Xs \right),  \sigma^2_{\rs}\left( \Ys \right) \right\}}  \norm{ (I-\cQ_{X}){X}{Y}^\top (I-\cQ_{Y}) - {\Xs}{\Ys}^\top }{F}^2.
\]
Moreover, if $k_1=k_2$ and $V^X = V^Y$, one has 
\[
 \norm{ \cQ_{X} \Xs{\Ys}^\top \cQ_{Y}    }{F} \leq  \frac{1}{\min\left\{ \sigma^2_{\rs}\left( \Xs \right),  \sigma^2_{\rs}\left( \Ys \right) \right\}}  \norm{{X}{Y}^\top  - {\Xs}{\Ys}^\top }{F}^2.
\]

\end{lemma}
\begin{proof}
We will denote $E:=(I-\cQ_{X}){X}{Y}^\top (I-\cQ_{Y}) - \Xs\Ys^\top$. We can decompose the error $E$ as
{\allowdisplaybreaks 
\begin{align*}
E = (I - \cQ_X) E (I - \cQ_Y) + \cQ_X E (I-\cQ_Y) + (I - \cQ_X) E \cQ_Y +  \cQ_X E \cQ_Y.
\end{align*} 
 Observe that all the terms in this sum are pairwise orthogonal in the Frobenius inner product. Therefore,
\begin{align*}
\|E\|_F^2 &=\left\|(I - \cQ_X) E (I - \cQ_Y)\right\|_F^2 +\left\|\cQ_X E (I-\cQ_Y)\right\|_F^2 +\left\|(I - \cQ_X) E \cQ_Y\right\|_F^2 + \left\|\cQ_X E \cQ_Y\right\|_F^2 \nonumber \\
&\geq\left\|\cQ_X E (I-\cQ_Y)\right\|_F^2 +\left\|(I - \cQ_X) E \cQ_Y\right\|_F^2 \nonumber \\
 \\
&\overset{(i)}=\left\|\cQ_X \Xs\Ys^\top (I-\cQ_Y)\right\|_F^2 +\left\|(I - \cQ_X) \Xs \Ys^\top \cQ_Y\right\|_F^2\nonumber \\ 
&\overset{(ii)}{=} \left\|\cQ_X \left(\Xs V^{\Xs}\right)_{\{ 1,\ldots, \rs \}} \left( \left(\Ys V^{\Ys}\right)_{\{ 1,\ldots, \rs \}} \right)^\top (I-\cQ_Y) \right\|_F^2 \\ & \qquad +\left\|(I - \cQ_X) \left(\Xs V^{\Xs}\right)_{\{ 1,\ldots, \rs \}} \left( \left(\Ys V^{\Ys}\right)_{\{ 1,\ldots, \rs \}} \right)^\top  \cQ_Y\right\|_F^2\nonumber \\ 
&\overset{(iii)}{\geq}  \sigma_{\rs}^2\left( (I-\cQ_Y) \left(\Ys V^{\Ys}\right)_{\{ 1,\ldots, \rs \}} \right) \norm{\cQ_X \left(\Xs V^{\Xs}\right)_{\{ 1,\ldots, \rs \}}}{F}^2 \\
& \qquad +  \sigma_{\rs}^2\left((I-\cQ_X) \left(\Xs V^{\Xs}\right)_{\{ 1,\ldots, \rs \}} \right) \norm{\cQ_Y \left(\Ys V^{\Ys}\right)_{\{ 1,\ldots, \rs \}}}{F}^2   \nonumber 
\\ &\overset{(iv)}{\geq}\frac{1}{2} \min\left\{ \sigma^2_{\rs}\left( \Xs \right),  \sigma^2_{\rs}\left( \Ys \right) \right\} \left( \norm{\cQ_X \left(\Xs V^{\Xs}\right)_{\{ 1,\ldots, \rs \}}}{F}^2 + \norm{\cQ_Y \left(\Ys V^{\Ys}\right)_{\{ 1,\ldots, \rs \}}}{F}^2 \right) \\ &
\overset{(v)}{\geq} \min\left\{ \sigma^2_{\rs}\left( \Xs \right),  \sigma^2_{\rs}\left( \Ys \right) \right\} \norm{\cQ_X X^\star \Ys^\top \cQ_Y}{F},
\end{align*}}
where $(i)$ follows from $\mathcal{Q}_X(I - \mathcal{Q}_X) = 0$ and $\mathcal{Q}_Y(I - \mathcal{Q}_Y) = 0$; $(ii)$ follows from $V^\Xs = V^\Ys$ and $\rs = \rs$; $(iii)$ follows from the variational characterization of singular values; $(iv)$ follows from Corollary~\ref{cor:eigenproj-vs-eigen} and $(v)$ follows from Young's inequality in conjunction with the Cauchy--Schwarz inequality. Moreover, if $V^X= V^Y$ and $k_1=k_2$, then the same argument holds for $E:={X}{Y}^\top - \Xs\Ys^\top$, since equality $(i)$ holds. 
This concludes the proof.  
\end{proof}

\begin{lemma}[Lemma 33 in \cite{zhang2021preconditioned}] \label{lem:asymmetric_eigenvalue_projected_bound} Let $\Xs \in \RR^{d \times \rs}$ of rank $\rs$ and let $X \in \RR^{d\times r}$. Assume that $\norm{XX^\top - \Xs \Xs^\top}{F} \leq  \frac{1}{2\sqrt{2}} \sigma^2_{\rs}\left( \Xs \right)$. Let $\cQ_{X}= \sum_{i=k_1+1}^{d_1} \ux_i {\ux_i}^\top$ with $k_1\geq \rs$. Then %
$$
\lambda_{\rs} \left(\Xs^\top (I-\cQ_X) \Xs \right) \geq \lambda_{1} \left(\Xs^\top \cQ_X \Xs \right).
$$ 
\end{lemma}
\begin{proof}
We will use as placeholders
\(
\alpha_{1} := \lambda_{\rs}\left(\Xs^{\top}(I-\cQ_X) \Xs \right) \text{ and } 
\alpha_{2} := \lambda_{1}\left(\Xs^{\top}\cQ_X\Xs\right).
\)
Our argument follows by contradiction, we will prove that \( \alpha_{1} < \alpha_{2} \) implies \( \frac{\norm{XX^\top - \Xs \Xs^\top}{F}}{\sigma_{\rs}^2\left(\Xs\right)} \geq \frac{1}{\sqrt{2}} > \frac{1}{2\sqrt{2}} \) which contradicts the hypothesis of the lemma. %
We bound
{\allowdisplaybreaks \begin{align}
\label{ineq:decompose_metric_with_projection}
\norm{XX^\top - \Xs \Xs^\top}{F}^2
& \overset{(i)}{=}\left\|(I-\cQ_X) X X^{\top} (I-\cQ_X) - (I-\cQ_X) \Xs \Xs^{\top} (I-\cQ_X)\right\|_F^2 \\
& \qquad + 2\left\|(I-\cQ_X) \Xs \Xs^{\top} \cQ_X \right\|_{F}^{2} \nonumber\\
& \qquad \qquad +\left\|\cQ_X X X^{\top} \cQ_X - \cQ_X \Xs \Xs^{\top} \cQ_X\right\|_{F}^{2}\nonumber \\
& \overset{(ii)}{\geq} \left\|(I-\cQ_X) X X^{\top} (I-\cQ_X) - (I-\cQ_X) \Xs \Xs^{\top} (I-\cQ_X)\right\|_{F}^{2} \nonumber
\\
& \qquad +\left\|\cQ_X X X^{\top} \cQ_X - \cQ_X \Xs \Xs^{\top} \cQ_X \right\|_{F}^{2}+ 2 \sigma_{\rs}^2\left((I-\cQ_{X})\Xs\right) \sigma_{1}^2\left(\cQ_{X}\Xs\right)  \\
& \overset{(iii)}{=} \left\|(I-\cQ_X) X X^{\top} (I-\cQ_X) - (I-\cQ_X) \Xs \Xs^{\top} (I-\cQ_X)\right\|_{F}^{2} \nonumber
\\
& \qquad +\left\|\cQ_X X X^{\top} \cQ_X - \cQ_X \Xs \Xs^{\top} \cQ_X \right\|_{F}^{2}+ 2 \alpha_1 \alpha_2,
\end{align}} \noindent where $(i)$ follows by expanding the square and using orthogonality and $(ii)$ follows from the variational characterization of singular values, and $(iii)$ holds since $\sigma_{k}^2\left(P\Xs \right) = \lambda_{k}\left( {\Xs}^\top P^\top P \Xs  \right) = \lambda_{k} \left( {\Xs}^\top P \Xs \right)$ for any $k \in \{ 1\ldots \rs \}$ and any orthogonal projection matrix $P \in \RR^{d_1\times d_1}$. We claim 
\begin{align}
\|(I-\cQ_X) X X^{\top} & (I-\cQ_X) - (I-\cQ_X) \Xs \Xs^{\top} (I-\cQ_X)\|_{F}^{2} 
+\left\|\cQ_X X X^{\top} \cQ_X - \cQ_X \Xs \Xs^{\top} \cQ_X \right\|_F^2 \notag \\
\ge 
& \min_{\beta_1, \beta_2 \in \mathbb{R}_+ \mid \beta_1 \geq \beta_2} 
\left\{ (\beta_1 - \alpha_1)^2 + (\beta_2 - \alpha_2)^2 \right\},
\label{eq:optim_alpha}
\end{align} 
let us defer the proof of this inequality until after we establish the result. Given \eqref{eq:optim_alpha}, if \( \alpha_1 < \alpha_2 \), then the optimal solution of the lower bound occurs at \( \beta_1 = \beta_2 = \frac{\alpha_1+\alpha_2}{2}\), so the minimum value becomes \( \frac{1}{2}(\alpha_1 - \alpha_2)^2 \). Substituting this into \eqref{ineq:decompose_metric_with_projection} gives
\[
\norm{XX^\top - \Xs \Xs^\top}{F}^2 \ge \frac{1}{2}(\alpha_1 - \alpha_2)^2 + 2 \alpha_1 \alpha_2 
= \frac{1}{2}(\alpha_1 + \alpha_2)^2 \geq \frac{1}{2} \sigma_{\rs}^4\left(\Xs\right),
\]
where we used Weyl's inequality in the last step to bound $$\alpha_1 + \alpha_2 \ge \lambda_{\rs} \left( \Xs^\top (I-\cQ_X) \Xs  + \Xs^\top \cQ_X\Xs  \right) = \lambda_{\rs} \left( \Xs^\top \Xs \right) = \sigma^2_{\rs}\left( \Xs \right).$$ Taking the square roots on both side implies \( \norm{XX^\top - \Xs \Xs^\top}{F}^2 \ge  \frac{1}{\sqrt{2}} \sigma^2_{\rs}\left( \Xs \right)\), a contradicts the radius hypothesis.
We turn to proving \eqref{eq:optim_alpha}, we have that %
\begin{align*}
\Big\|(I-\cQ_X) &X X^{\top} (I-\cQ_X) - (I-\cQ_X) \Xs \Xs^{\top} (I-\cQ_X)\Big\|_{F}^{2} 
+ \left\|\cQ_X X X^{\top} \cQ_X - \cQ_X \Xs \Xs^{\top} \cQ_X\right\|_{F}^{2} \\ 
& \overset{(i)}{\geq}\min_{S_1 \succeq 0, S_2 \succeq 0 \mid \sigma_{\rs}(S_1) \geq \sigma_{1}(S_2)} \left\| S_1- (I-\cQ_X) \Xs \Xs^{\top} (I-\cQ_X) \right\|_{F}^{2}  + \left\| S_2 - \cQ_X \Xs \Xs^{\top} \cQ_X \right\|_{F}^{2} \notag \\ & \overset{(ii)}{\geq} \min_{\beta_1,\beta_2 \in \RR_{+} \mid \beta_1 \geq \beta_2}  (\beta_1 - \alpha_1)^2 + (\beta_2 - \alpha_2)^2,\notag
\label{eq:opt_problem}
\end{align*}
where $(i)$ is due to $\sigma_{\rs}\left( (I-\cQ_X) X X^{\top} (I-\cQ_X) \right) \geq \sigma_{1} \left( \cQ_X X X^{\top} \cQ_X \right)$ and $(ii)$ follows from the Hoffman-Wielandt Theorem \cite[Problem III.6.15]{bhatia2013matrix}. %
This concludes the proof. 
\end{proof}

\begin{corollary}
    \label{cor:eigenproj-vs-eigen}
    Let $\Xs \in \RR^{d \times \rs}$ of rank $\rs$. If $\norm{XX^\top - \Xs \Xs^\top}{F} \leq  \frac{1}{2\sqrt{2}} \sigma^2_{\rs}\left( \Xs \right)$, we have that 
     $$\lambda_{\rs}({\Xs}^\top {\Xs}) \leq 2 \lambda_{\rs}\left({\Xs}^\top (I-\cQ_{X}) \Xs\right).$$
\end{corollary}
\begin{proof}
    One has {\allowdisplaybreaks \begin{align*}
\lambda_{\rs}\left({\Xs}^\top {\Xs}\right) &= \lambda_{\rs}\left({\Xs}^\top (I-\cQ_{X}) \Xs + {\Xs}^\top \cQ_{X}^\top \cQ_{X} \Xs\right) \\
&\overset{(i)}{\leq} \lambda_{\rs}\left({\Xs}^\top (I-\cQ_{X}) \Xs\right) + \lambda_{1}\left({\Xs}^\top \cQ_{X} \Xs\right) \\
&\overset{(ii)}{\leq} 2 \lambda_{\rs}\left({\Xs}^\top (I-\cQ_{X}) \Xs\right),
\end{align*}}where $(i)$ follows from Weyl's inequality and $(ii)$ follows from Lemma~\ref{lem:asymmetric_eigenvalue_projected_bound}.
\end{proof}

\section{Computing the preconditioner}
\label{efficient}

In this section, we elaborate on how to compute the preconditioners $$
\left(\nabla \c(x)^\top \nabla \c(x) + \lambda I\right)^{-1}  g
$$
given a fixed $g \in \RR^{d}$. For this task, we use the conjugate gradients method (CG), which
converges linearly at a rate that depends on the condition number of the matrix $P(x,\lambda) = \nabla \c(x)^\top \nabla \c(x) + \lambda I.$ We have found empirically that executing around ten iterations of CG suffices to obtain fast convergence of LMM. The main subroutine necessary for CG is the matrix-vector product $y \mapsto P(x, \lambda)y$; in what follows, we study the complexity of this subroutine in the examples we studied.

\subsection{Square-variable map}
For the component-wise square, we have that $m=d$, and that $\nabla \c(x) =2 \diagg{x}$. Thus, we have that the quantity $P(x, \lambda) y = 4 x \odot x \odot y + \lambda y$ which can be computed with $\mathcal{O}(d)$ flops.

\subsection{Burer-Monteiro factorization}
For Burer-Monteiro map, given inputs $X \in \RR^{d \times r}$ we have
$$
P(X, \lambda) [\Xprime] = \vect{ \Xprime X^\top  X +  X\Xprime^\top X + \lambda \Xprime}  \qquad \text{for } \Xprime \in \RR^{d\times r}.
$$
The computation follows from~\eqref{action: nablac-and-nablacT-bm}. This action can be computed with $\mathcal{O}(dr^2)$ flops.

\subsection{Asymmetric matrix factorization}
For the asymmetric matrix factorization map, given inputs $X \in \RR^{d_1 \times r}$ and $Y \in \RR^{d_2\times r}$ we have
$$
P((X, Y), \lambda) \begin{bmatrix} \Xprime \\ \Yprime \end{bmatrix}  =  \begin{pmatrix} X {\Yprime}^\top Y + \Xprime Y^\top Y \\[5pt] Y {\Xprime}^\top X + \Yprime X^\top X \end{pmatrix}  \qquad \text{for }\Xprime \in \RR^{d_1 \times r} \text{ and }\Yprime \in \RR^{d_2 \times r},
$$
where the computation follows from~\eqref{action: nablac-and-nablacT-asym}. This action can be computed with $\mathcal{O}\left((d_1 + d_2)r^2\right)$ flops.

\subsection{Symmetric CP factorization}
For the symmetric canonical polyadic map, given an input $X\in \RR^{d\times r}$ we have
\begin{align*}
P\left( X, \lambda \right) [\Xprime] &
= 3\Xprime  \left( X^\top X \odot X^\top X\right) + 6 X  \left( {\Xprime}^\top X \odot X^\top X \right) + \lambda \Xprime \qquad \text{for }\Xprime \in \RR^{d\times r}.
\end{align*}
This computation follows from Lemma \ref{lemma: action-nablac-symmetric}. As with the matrix case, this action can be computed with $\mathcal{O}(dr^2)$ flops.
\subsection{CP factorization}
For the canonical polyadic map, given inputs
\(
W \in \RR^{d_1 \times r},\;
X \in \RR^{d_2 \times r},\;
Y \in \RR^{d_3 \times r},
\) we have

\begin{align*}
 P((W,X, Y), \lambda) \begin{bmatrix} \Wprime \\ \Xprime \\ \Yprime \end{bmatrix} =
  \begin{pmatrix}
    \Wprime \left( X^{\top} X \odot Y^{\top} Y\right) + W \left(\Xprime^{\top} X \odot Y^{\top} Y + X^{\top}X \odot \Yprime^\top Y \right) + \lambda \Wprime\\
    \Xprime \left( W^{\top} W \odot Y^{\top} Y\right) + X \left(\Wprime^{\top} W \odot Y^{\top} Y + W^{\top}W \odot \Yprime^\top Y \right) + \lambda \Xprime\\
\Yprime \left( W^{\top} W \odot X^{\top} X\right) + Y \left(\Wprime^{\top} W \odot X^{\top} X + W^{\top}W \odot \Xprime^\top X \right) + \lambda \Yprime
    \end{pmatrix}
  \end{align*}
for \(\Wprime \in \RR^{d_1 \times r},\; \Xprime \in \RR^{d_2 \times r},\; \Yprime \in \RR^{d_3 \times r}.\) This computation follows from Lemma~\ref{jac: cp}. Once more, this action can be computed with $\mathcal{O}((d_1+d_2+d_3)r^{2})$ flops.

\endgroup
	
\end{document}